\documentclass[11pt,reqno]{amsart}
\pdfoutput=1
\usepackage{amsmath,amssymb,amscd,mathrsfs,epic,wasysym,latexsym,tikz,mathrsfs,cite,hyperref,stmaryrd}
\usepackage{pb-diagram}
\usepackage[matrix,arrow]{xy}
\usepackage[margin=2cm]{geometry}
\usepackage[skip=3pt,font=footnotesize,labelfont=bf]{caption}
\usepackage{longtable,multirow,setspace}

\usepackage[enableskew]{youngtab}
\usepackage{ytableau}
\usepackage{color}
\usepackage{tikz-cd}

\usepackage{comment}
\usepackage[new]{old-arrows}

\definecolor{mediumblue}{rgb}{0.0, 0.0, 0.8}
\colorlet{darkgreen}{green!50!black}
\usepackage{doi}
\usepackage{hyperref}

\hypersetup{
  colorlinks=true,
  linktoc=page,
  citecolor=darkgreen,
  linkcolor=blue,
  urlcolor=mediumblue,
  unicode
}

\usepackage[capitalise]{cleveref}

\AddToHook{env/Definition/begin}{\crefalias{equation}{Definition}}
\AddToHook{env/Theorem/begin}{\crefalias{equation}{Theorem}}
\AddToHook{env/Proposition/begin}{\crefalias{equation}{Proposition}}
\AddToHook{env/Lemma/begin}{\crefalias{equation}{Lemma}}
\AddToHook{env/Corollary/begin}{\crefalias{equation}{Corollary}}
\AddToHook{env/Conjecture/begin}{\crefalias{equation}{Conjecture}}
\AddToHook{env/Example/begin}{\crefalias{equation}{Example}}
\AddToHook{env/Remark/begin}{\crefalias{equation}{Remark}}

\crefname{Definition}{Definition}{Definitions}
\crefname{Theorem}{Theorem}{Theorems}
\crefname{Proposition}{Proposition}{Propositions}
\crefname{Lemma}{Lemma}{Lemmas}
\crefname{Corollary}{Corollary}{Corollaries}
\crefname{Conjecture}{Conjecture}{Conjectures}
\crefname{section}{\S\hspace*{-0.12cm}}{\S\hspace*{-0.1cm}}
\crefname{subsection}{\S\hspace*{-0.12cm}}{\S\hspace*{-0.1cm}}
\crefname{Example}{Example}{Examples}

\Crefname{Definition}{Definition}{Definitions}
\Crefname{Theorem}{Theorem}{Theorems}
\Crefname{Proposition}{Proposition}{Propositions}
\Crefname{Lemma}{Lemma}{Lemmas}
\Crefname{Corollary}{Corollary}{Corollaries}
\Crefname{Conjecture}{Conjecture}{Conjectures}
\Crefname{section}{\S\hspace*{-0.12cm}}{\S\hspace*{-0.1cm}}
\Crefname{subsection}{\S\hspace*{-0.12cm}}{\S\hspace*{-0.1cm}}
\Crefname{Example}{Example}{Examples}

\definecolor{answercolor}{RGB}{0, 112, 48}
\excludecomment{answer}      

\makeatletter

\numberwithin{equation}{section}
\def\theequation{\arabic{section}.\arabic{equation}}

\newtheorem{Proposition}[equation]{Proposition}
\newtheorem{Lemma}[equation]{Lemma}
\newtheorem{Theorem}[equation]{Theorem}
\newtheorem{Corollary}[equation]{Corollary}
\theoremstyle{definition}  
\newtheorem{Definition}[equation]{Definition}
\newtheorem{Remark}[equation]{Remark}

\newtheorem{Example}[equation]{Example}

\newtheorem{Conjecture}[equation]{Conjecture}

\let\<\langle
\let\>\rangle

\newcommand\Comment[2][\relax]{\space\par\medskip\noindent%
   \fbox{\begin{minipage}{\textwidth}\textbf{Comment\ifx\relax#1\else---#1\fi}\newline%
        #2\end{minipage}}\medskip
}

\def\bi{\text{\boldmath$i$}}

\def\bj{\text{\boldmath$j$}}
\def\bm{\text{\boldmath$m$}}
\def\bk{\text{\boldmath$k$}}

\def\b1{\text{\boldmath$1$}}

\def\bb{\text{\boldmath$b$}}
\def\bla{\text{\boldmath$\lambda$}}

\def\pmod#1{\text{ }(\text{\rm mod } #1)\,}

\def\bijection{\overset{\sim}{\longrightarrow}}

\newcommand{\End}{\operatorname{End}}

\newcommand{\im}{\operatorname{im}}

\newcommand{\res}{\operatorname{res}}
\newcommand{\Std}{\operatorname{Std}}

\newcommand{\NEarrow}{\mathbin{\rotatebox[origin=c]{45}{$\Rightarrow$}}}

\newcommand{\Z}{\mathbb{Z}}

\newcommand{\N}{\mathcal{N}}

\def\phi{{\varphi}}

\newcommand\RoCK{\operatorname{RoCK}}

\newcommand{\la}{\lambda}

\newcommand{\remrib}{\mathcal{R}_{\brho}}
\newcommand{\addrib}{\mathcal{A}_{\brho}}

\renewcommand\geq\geqslant
\renewcommand\leq\leqslant

\renewcommand\succeq\succcurlyeq
\renewcommand\preceq\preccurlyeq

\renewcommand{\trianglerighteq}{\trianglerighteqslant}
\renewcommand{\trianglelefteq}{\trianglelefteqslant}

\DeclareRobustCommand\longtwoheadrightarrow
     {\relbar\joinrel\twoheadrightarrow}

\newcommand\dhxrightarrow[2][]{%
  \mathrel{\ooalign{$\xrightarrow[#1\mkern4mu]{#2\mkern4mu}$\cr%
  \hidewidth$\rightarrow\mkern4mu$}}
}

\newcommand{\Ind}{{\mathrm {Ind}}}

\newcommand{\tr}{{\mathrm {tr}}}

\newcommand{\Res}{{\mathrm {Res}}}

\newcommand{\CC}{{\mathbb C}}
\newcommand{\asl}{\hat{\mathfrak{sl}}}

\newcommand{\Q}{{\mathbb Q}}

\newcommand{\RR}{{\mathbb R}}
\newcommand{\QQ}{{\mathbb Q}}
\newcommand{\ZZ}{{\mathbb Z}}

\newcommand{\blam}{\boldsymbol{\lambda}}
\newcommand{\balpha}{\boldsymbol{\alpha}}
\newcommand{\bkap}{\boldsymbol{\kappa}}
\newcommand{\bmu}{\boldsymbol{\mu}}
\newcommand{\bnu}{\boldsymbol{\nu}}
\newcommand{\brho}{\boldsymbol{\rho}}
\newcommand{\bxi}{\boldsymbol{\xi}}
\newcommand{\btau}{\boldsymbol{\tau}}
\newcommand{\bzeta}{\boldsymbol{\zeta}}

\newcommand{\zS}{\boldsymbol{\mathsf{S}}}

\def\fin{{\operatorname{fin}}}
\def\cont{{\operatorname{cont}}}

\def\b{\mathfrak{b}}
\def\k{\Bbbk}

\def\height{{\operatorname{ht}}}

\def\op{{\mathrm{op}}}
\def\re{{\mathrm{re}}}
\def\im{{\mathrm{im}\,}}

\def\B{{\mathcal B}}

\def\END{\operatorname{END}}

{\catcode`\|=\active
  \gdef\set#1{\mathinner{\lbrace\,{\mathcode`\|"8000%
  \let|\midvert #1}\,\rbrace}}
}
\def\midvert{\egroup\mid\bgroup}

\colorlet{darkgreen}{green!50!black}
\tikzset{dots/.style={very thick,loosely dotted},
         greendot/.style={fill,circle,color=darkgreen,inner sep=1.5pt,outer sep=0}
}
\def\greendot(#1,#2){\node[greendot] at(#1,#2){}}

\newenvironment{braid}{
  \begin{tikzpicture}[baseline=6mm,blue,line width=1pt, scale=0.4,
                      draw/.append style={rounded corners},
                      every node/.append style={font=\fontsize{5}{5}\selectfont}]%
  }{\end{tikzpicture}
}

\def\Grid(#1,#2){
  \draw[very thin,gray,step=2mm] (0,0)grid(#1,#2);
  \draw[very thin,darkgreen,step=10mm] (0,0)grid(#1,#2);
}

\newcommand\Tableau[2][\relax]{
  \begin{tikzpicture}[scale=0.5,draw/.append style={thick,black}]
    \ifx\relax#1\relax%
    \else 
      \foreach\box in {#1} { \filldraw[blue!30]\box+(-.5,-.5)rectangle++(.5,.5); }
    \fi
    \newcount\row\newcount\col
    \row=0
    \foreach \Row in {#2} {
       \col=1
       \foreach\k in \Row {
          \draw(\the\col,\the\row)+(-.5,-.5)rectangle++(.5,.5);
          \draw(\the\col,\the\row)node{\k};
          \global\advance\col by 1
       }
       \global\advance\row by -1
    }
  \end{tikzpicture}
}

\newcommand{\hackcenter}[1]{
 \xy (0,0)*{#1}; \endxy}

\newcommand\YoungDiagram[2][\relax]{
  \begin{tikzpicture}[scale=0.5,draw/.append style={thick,black}]
    \ifx\relax#1\relax%
    \else 
    \foreach\box in {#1} {
      \filldraw[blue!30]\box rectangle ++(1,1);
    }
    \fi
    \newcount\row
    \row=0
    \foreach \col in {#2} {
       \draw(1,\the\row)grid ++(\col,1);
       \global\advance\row by -1
    }
  \end{tikzpicture}
}

\colorlet{darkgreen}{green!50!black}
\tikzset{dots/.style={very thick,loosely dotted},
         greendot/.style={fill,circle,color=darkgreen,inner sep=1.5pt,outer sep=0},
         blackdot/.style={fill,circle,color=black,inner sep=2pt,outer sep=0},
         graydot/.style={fill,circle,color=gray,inner sep=1.1pt,outer sep=0},
         whitedot/.style={fill,circle,color=white,inner sep=3pt,outer sep=0},
}
\def\greendot(#1,#2){\node[greendot] at(#1,#2){}}
\def\blackdot(#1,#2){\node[blackdot] at(#1,#2){}}
\def\graydot(#1,#2){\node[graydot] at(#1,#2){}}
\def\whitedot(#1,#2){\node[whitedot] at(#1,#2){}}

\def\Grid(#1,#2){
  \draw[very thin,gray,step=2mm] (0,0)grid(#1,#2);
  \draw[very thin,darkgreen,step=10mm] (0,0)grid(#1,#2);
}

\theoremstyle{plain} 
\newcommand{\thistheoremname}{}
\newtheorem*{genericthm*}{\thistheoremname}
\newenvironment{namedthm*}[1]
  {\renewcommand{\thistheoremname}{#1}%
   \begin{genericthm*}}
  {\end{genericthm*}}

\begin{document}


\title[Skew Specht modules, RoCK blocks, and cuspidal systems]{{\bf A skew Specht perspective of RoCK blocks and\\ cuspidal systems for KLR algebras in affine type A}}

\author{\sc Robert Muth}
\address{Department of Mathematics and Computer Science\\ Duquesne University\\ Pittsburgh\\ PA, USA 15282}
\email{muthr@duq.edu}

\author{\sc Thomas Nicewicz}
\address{Department of Mathematics \\ Washington \& Jefferson College \\ Washington\\ PA, USA 15301}
\email{Tnicewicz@protonmail.com}

\author{\sc Liron Speyer}
\address{Okinawa Institute of Science and Technology \\ Okinawa \\ Japan 904-0495}
\email{liron.speyer@oist.jp}

\author{\sc Louise Sutton}
\address{Okinawa Institute of Science and Technology \\ Okinawa \\ Japan 904-0495}
\email{louise.sutton@oist.jp}

\begin{abstract}
Cuspidal systems parameterize KLR algebra representations via root partitions $\pi$, where simple modules $L(\pi)$ arise as heads of proper standard modules. Working in affine type A with an arbitrary convex preorder, we construct explicit skew diagrams $\zeta(\pi)$ such that the skew Specht module $S^{\zeta(\pi)}$ has simple head $L(\pi)$ and a filtration by proper standard modules. A key ingredient in this construction is the development of `core-truncation' functors, which take module categories of level one RoCK blocks to the category of imaginary semicuspidal KLR modules. Every simple imaginary semicuspidal module arises in the image of these functors.
 This result stems from an in-depth study of the combinatorial interplay between cuspidal systems and RoCK cyclotomic KLR algebras, in which we characterize core blocks and RoCK blocks in arbitrary level via cuspidal tiling properties of multipartitions in these blocks.
\end{abstract}

\maketitle

\section{Introduction}\label{sec:intro}

In this paper we study combinatorial and algebraic interactions between cuspidal systems for Khovanov--Lauda--Rouquier (KLR) algebras, RoCK cyclotomic KLR algebras and skew Specht modules. We apply these connections to construct simple, semicuspidal, and proper standard KLR modules as (quotients of) explicit skew Specht modules.

\subsection{Cuspidal systems for KLR algebras}
Fix \(e \in \ZZ_{>1}\), and let \(\Phi_+ = \Phi_+^\re \sqcup \Phi_+^\im\) be the positive root system of type \({\tt A}^{(1)}_{e-1}\), where \(I = \{\alpha_0, \ldots, \alpha_{e-1}\}\) is the set of simple roots, \(\Phi_+^\re\) is the set of real roots, and \(\Phi_+^\im = \{d \delta \mid d \in \ZZ_{>0}\}\) is the set of imaginary roots, with \(\delta = \alpha_0 + \cdots + \alpha_{e-1}\) being the null root. We additionally write \(\Psi := \Phi_+^\re \sqcup \{ \delta\}\) for the set of indivisible roots. This root system data is fundamental in the representation theory of the Kac--Moody Lie algebra \(\widehat{\mathfrak{sl}}_e(\CC)\) \cite{kac}.

For any field \(\k\) and \(\omega \in \ZZ_{\geq 0}I\), there is an associated KLR algebra \(R_\omega\) over \(\k\).
This family of algebras categorifies the positive part of the quantum group \(U_q(\widehat{\mathfrak{sl}}_e(\CC))\), with the product in \(U_q(\widehat{\mathfrak{sl}}_e(\CC))\) lifting to the induced tensor product of modules over KLR algebras, see \cite{kl09, Rouq}.
The representation theory of KLR algebras is studied via {\em cuspidal systems}, as in \cite{klesh14, McN17, km17a, TW}, 
which are associated with PBW bases for the quantum group \cite{BeckConBas}. 
 We will now briefly explain the setup.

Fix a convex preorder \(\succeq\) on \(\Phi_+\). 
For \(\omega \in \ZZ_{\geq 0}I\), a {\em Kostant partition} of \(\omega\) is a tuple of non-negative integers \(\boldsymbol{K} = (K_\beta)_{\beta \in \Psi}\) such that \(\sum_{\beta \in \Psi} K_\beta \beta = \omega\). If \(\beta_1 \succ \cdots \succ \beta_t\) are the members of \(\Psi\) such that \(K_{\beta_i} \neq 0\), then we write \(\boldsymbol{K}\) in the form
\(
\boldsymbol{K} = 
(\beta_1^{K_{\beta_1}} \mid \dots \mid \beta_t^{K_{\beta_t}}).
\)
The convex preorder \(\succeq\) induces a bilexicographic partial order \(\geq_{\textup{b}}\) on \(\Xi(\omega)\), the set of all Kostant partitions of \(\omega\).
We write \(\Pi(\omega)\) for the set of all {\em root partitions} of \(\omega\); these are pairs \(\pi = (\boldsymbol{K}, \bnu)\), where
\begin{align*}
\boldsymbol{K} &= (\beta_1^{K_{\beta_1}} \mid \dots \mid \beta_u^{K_{\beta_u}} \mid \delta^{K_\delta} \mid \beta_{u+1}^{K_{\beta_{u+1}}} \mid \dots \mid \beta_t^{K_{\beta_t}}) \in \Xi(\omega), \textup{and}\\
\bnu &= (\nu^{(1)} \mid \cdots \mid \nu^{(e-1)} )  \textup{ is an \((e-1)\)-multipartition of \(K_\delta\)}.
\end{align*}
We refine the bilexicographic partial order \(\geq_{\textup{b}}\) to a partial order \(\geq_{\textup{bd}}\) on \(\Pi(\omega)\) by incorporating the dominance order on \((e-1)\)-multipartitions. 

To each \(\beta \in \Phi_+^\re\), we associate a simple {\em cuspidal} \(R_\beta\)-module \(L(\beta)\), and to each \((e-1)\)-multipartition \(\bnu\) of \(d \in \ZZ_{>0}\), we associate a simple {\em semicuspidal} \(R_{d \delta}\)-module \(L(\bnu)\). Then, to each 
\(\pi \in \Pi(\omega)\), we associate a proper standard module \(\bar\Delta(\pi)\) which is an ordered induction product of these simple semicuspidal modules. The module \(\bar\Delta(\pi)\) has a self-dual simple head \(L(\pi)\), and 
\(
\{ L(\pi) \mid \pi \in \Pi(\omega)\}
\)
is a complete and irredundant set of simple \(R_\omega\)-modules up to isomorphism and grading shift.
To be precise, if \(\pi\) is as above,
then \(\bar\Delta(\pi)\) is (up to grading shift) the induction product of the following semicuspidal simple modules:
\begin{align*}
\bar\Delta(\pi) = L(\beta_1)^{\circ K_{\beta_1}} \circ \dots \circ L(\beta_u)^{\circ {K_{\beta_u}}} \circ
L( \bnu) \circ L(\beta_{u+1})^{\circ {K_{\beta_{u+1}}}} \circ \dots \circ L(\beta_t)^{\circ {K_{\beta_t}}}.
\end{align*}
In the literature, the semicuspidal modules are not presented directly -- rather, their existence is established via categorification (c.f.~\cite{klesh14,McN17}), or they are constructed through Morita equivalences with symmetric groups and Schur algebras (c.f.~\cite{km17,km17a}).
In this paper our primary goal is to use {\em skew Specht modules} to render a more direct and combinatorial-flavored description of semicuspidal and simple \(R_\omega\)-modules. 

\subsection{Semicuspidal skew Specht modules} 
A skew diagram \(\btau\) is a set of nodes that can arise as a set difference of multipartitions. 
Nodes in \(\btau\) have residues in \(\ZZ_e\), and the skew diagram \(\btau\) has content in \( \ZZ_{\geq 0}I\),0 see \cref{subsec:content}.
Associated to \(\btau\) of content \(\omega\) is a skew Specht module over \(R_\omega\), denoted \(\zS^{\btau}\) (see \cref{Spechtsec}). These are natural generalizations of Specht modules \(S^{\blam}\) and \(S^{\blam/\bmu}\) indexed by multipartitions \cite{KMR} and skew multipartitions \cite{muthskew}, respectively. Specht modules are key objects in the representation theory of {\em cyclotomic} KLR algebras, Hecke algebras and symmetric group algebras, via the connection between these algebras proved in \cite{bkisom}. The skew Specht module \(\zS^{\btau}\) is a concrete and highly combinatorial object -- it has a direct presentation via generators and relations, and an explicit basis indexed by standard \(\btau\)-tableaux.

We now explain how we use Specht modules to make semicuspidal representation theory of KLR algebras more concrete.
In \cite{muthtiling} it was shown that, for all \(\beta \in \Phi_+^\re\), there exists an explicit ribbon \(\zeta(\beta)\) of content \(\beta\) such that \(\zS^{\zeta(\beta)} \approx L(\beta)\) (we use \(\approx\) to signify isomorphism up to grading shift). We restate the construction of \(\zeta(\beta)\) in \cref{def:zetaribs}. Our first main result in this paper establishes an analogous result for the thornier {\em imaginary} simple semicuspidal modules.

To each \((e-1)\)-multipartition \(\bnu\) of \(d\), we construct an explicit skew diagram \(\zeta(\bnu)\) of content \(d\delta\) in \cref{def:zetanu}, and establish the following result, which shows that imaginary simple semicuspidal modules arise as heads of skew Specht modules, proving a conjecture made in \cite[\S8.7.3]{muthtiling}.

\begin{namedthm*}{Theorem A}\hypertarget{thm:A}
Let \(\bnu, \bmu\) be \((e-1)\)-multipartitions of \(d\). Then \(\zS^{\zeta(\bnu)}\) is an indecomposable semicuspidal \(R_{d \delta}\)-module, with simple semicuspidal head \(\textup{hd}(\zS^{\zeta(\bnu)}) \approx L(\bnu)\), and \([\zS^{\zeta(\bnu)}: L(\bmu)] = d^{\textup{RoCK}}_{\bnu, \bmu}\).
\end{namedthm*}

This appears as \cref{thm:Asimp} in the body of the paper. The decomposition numbers \(d^{\textup{RoCK}}_{\bnu, \bmu}\) above refer to certain well-studied decomposition numbers of Specht modules in (level one) RoCK blocks of defect \(d\); see (\ref{defdRoCKintro}). 

\begin{Example}\label{firstexintro}
In order to help visualize the constructions in this introduction, we will make use of the following ongoing example. 
Take \(e=4\), and the convex preorder \(\succeq\) on \(\Phi_+\) described in \cref{3coretileex}.
To each \(\beta \in \Phi_+^\re\), we associate the ribbon \(\zeta(\beta)\) of content \(\beta\) via the algorithm described in \cref{def:zetaribs} such that \( \zS^{\zeta(\beta)} \approx L(\beta) \). For instance, we have:
\begin{align*}
\zeta(\alpha_2 \hspace{-0.5mm}+\hspace{-0.5mm} \alpha_3 \hspace{-0.5mm}+\hspace{-0.5mm} \alpha_0) = 
{}
\hackcenter{
\begin{tikzpicture}[scale=0.29]
\draw[thick,fill=gray!30]  (0,0)--(0,2)--(2,2)--(2,1)--(1,1)--(1,0)--(0,0);
\node at (0.5,0.5){$\scriptstyle 2$};
\node at (0.5,1.5){$\scriptstyle 3$};
\node at (1.5,1.5){$\scriptstyle 0$};
\draw[thick,  gray, dotted] (0,1)--(1,1);
\draw[thick,  gray, dotted] (1,1)--(1,2);
\end{tikzpicture}
}
\quad
\zeta(2 \delta \hspace{-0.5mm} + \hspace{-0.5mm} \alpha_0  \hspace{-0.5mm}+ \hspace{-0.5mm} \alpha_1  \hspace{-0.5mm}+ \hspace{-0.5mm} \alpha_2) = 
{}
\hackcenter{
\begin{tikzpicture}[scale=0.29]
\draw[thick,fill=gray!30]  (0,0)--(0,2)--(1,2)--(1,3)--(2,3)--(2,4)--(3,4)--(3,5)--(7,5)--(7,4)--(4,4)--(4,3)--(3,3)--(3,2)--(2,2)--(2,1)--(1,1)--(1,0)--(0,0);
\node at (0.5,0.5){$\scriptstyle 0$};
\node at (0.5,1.5){$\scriptstyle 1$};
\node at (1.5,1.5){$\scriptstyle 2$};
\node at (1.5,2.5){$\scriptstyle 3$};
\node at (2.5,2.5){$\scriptstyle 0$};
\node at (2.5,3.5){$\scriptstyle 1$};
\node at (3.5,3.5){$\scriptstyle 2$};
\node at (3.5,4.5){$\scriptstyle 3$};
\node at (4.5,4.5){$\scriptstyle 0$};
\node at (5.5,4.5){$\scriptstyle 1$};
\node at (6.5,4.5){$\scriptstyle 2$};
\draw[thick,  gray, dotted] (0,1)--(1,1);
\draw[thick,  gray, dotted] (1,2)--(2,2);
\draw[thick,  gray, dotted] (2,3)--(3,3);
\draw[thick,  gray, dotted] (3,4)--(4,4);
\draw[thick,  gray, dotted] (1,1)--(1,2);
\draw[thick,  gray, dotted] (2,2)--(2,3);
\draw[thick,  gray, dotted] (3,3)--(3,4);
\draw[thick,  gray, dotted] (4,4)--(4,5);
\draw[thick,  gray, dotted] (5,4)--(5,5);
\draw[thick,  gray, dotted] (6,4)--(6,5);
\end{tikzpicture}
}
\quad
\zeta(\delta  \hspace{-0.5mm}+ \hspace{-0.5mm} \alpha_2  \hspace{-0.5mm}+ \hspace{-0.5mm} \alpha_3) = 
{}
\hackcenter{
\begin{tikzpicture}[scale=0.29]
\draw[thick,fill=gray!30]  (0,0)--(0,4)--(1,4)--(1,5)--(2,5)--(2,3)--(1,3)--(1,0)--(0,0);
\node at (0.5,0.5){$\scriptstyle 2$};
\node at (0.5,1.5){$\scriptstyle 3$};
\node at (0.5,2.5){$\scriptstyle 0$};
\node at (0.5,3.5){$\scriptstyle 1$};
\node at (1.5,3.5){$\scriptstyle 2$};
\node at (1.5,4.5){$\scriptstyle 3$};
\draw[thick,  gray, dotted] (0,1)--(1,1);
\draw[thick,  gray, dotted] (0,2)--(1,2);
\draw[thick,  gray, dotted] (0,3)--(1,3);
\draw[thick,  gray, dotted] (1,3)--(1,4);
\draw[thick,  gray, dotted] (1,4)--(2,4);
\end{tikzpicture}
}
\quad
\zeta(\delta  \hspace{-0.5mm}+ \hspace{-0.5mm} \alpha_1) = 
{}
\hackcenter{
\begin{tikzpicture}[scale=0.29]
\draw[thick,fill=gray!30]  (0,0)--(0,1)--(1,1)--(1,2)--(2,2)--(2,3)--(3,3)--(3,1)--(2,1)--(2,0)--(0,0);
\node at (0.5,0.5){$\scriptstyle 1$};
\node at (1.5,0.5){$\scriptstyle 2$};
\node at (1.5,1.5){$\scriptstyle 3$};
\node at (2.5,1.5){$\scriptstyle 0$};
\node at (2.5,2.5){$\scriptstyle 1$};
\draw[thick,  gray, dotted] (1,1)--(2,1);
\draw[thick,  gray, dotted] (2,2)--(3,2);
\draw[thick,  gray, dotted] (1,0)--(1,1);
\draw[thick,  gray, dotted] (2,1)--(2,2);
\end{tikzpicture}
}.
\end{align*}
Recalling that \(e-1=3\), to each \(i \in \{1,2,3\}\) we associate a distinct ribbon \(\zeta_i\) of content \(\delta\) via the algorithm described in \cref{def:zetaribs}:
\begin{align*}
{}
\zeta_1=
{}
\hackcenter{
\begin{tikzpicture}[scale=0.29]
\draw[thick,fill=green!30] (7+0,20+0)--(9+0,20+0)--(9+0,21+0)--(10+0,21+0)--(10+0,22+0)--(8+0,22+0)--(8+0,21+0)--(7+0,21+0)--(7+0,20+0);
\node at (7.5+0,20.5+0){$\scriptstyle 1$};
\node at (8.5+0,20.5+0){$\scriptstyle 2$};
\node at (8.5+0,21.5+0){$\scriptstyle 3$};
\node at (9.5+0,21.5+0){$\scriptstyle 0$};
\draw[thick, gray, dotted] (8+0,20+0)--(8+0,21+0);
\draw[thick, gray, dotted] (8+0,21+0)--(9+0,21+0);
\draw[thick, gray, dotted] (9+0,21+0)--(9+0,22+0);
\end{tikzpicture}
}
\qquad
\zeta_2=
\hackcenter{
\begin{tikzpicture}[scale=0.29]
\draw[thick,fill=red!30] (2+0,12+0)--(3+0,12+0)--(3+0,13+0)--(4+0,13+0)--(4+0,15+0)--(3+0,15+0)--(3+0,14+0)--(2+0,14+0)--(2+0,12+0);
\node at (2.5+0,12.5+0){$\scriptstyle 0$};
\node at (2.5+0,13.5+0){$\scriptstyle 1$};
\node at (3.5+0,13.5+0){$\scriptstyle 2$};
\node at (3.5+0,14.5+0){$\scriptstyle 3$};
\draw[thick, gray, dotted] (2+0,13+0)--(3+0,13+0);
\draw[thick, gray, dotted] (3+0,13+0)--(3+0,14+0);
\draw[thick, gray, dotted] (3+0,14+0)--(4+0,14+0);
\end{tikzpicture}
}
\qquad
\zeta_3=
\hackcenter{
\begin{tikzpicture}[scale=0.29]
\draw[thick,fill=cyan!30] (0+0,0+0)--(1+0,0+0)--(1+0,4+0)--(0+0,4+0)--(0+0,0+0);
\node at (0.5+0,0.5+0){$\scriptstyle 2$};
\node at (0.5+0,1.5+0){$\scriptstyle 3$};
\node at (0.5+0,2.5+0){$\scriptstyle 0$};
\node at (0.5+0,3.5+0){$\scriptstyle 1$};
\draw[thick, gray, dotted] (0+0,1+0)--(1+0,1+0);
\draw[thick, gray, dotted] (0+0,2+0)--(1+0,2+0);
\draw[thick, gray, dotted] (0+0,3+0)--(1+0,3+0);
\end{tikzpicture}
},
\end{align*}
and to any multipartition \(\bnu = (\nu^{(1)} \mid \nu^{(2)} \mid \nu^{(3)})\), we associate the skew diagram \(\zeta(\bnu)\) by dilating each of the nodes in \(\nu^{(i)}\) by the ribbon \(\zeta_i\) (see \cref{def:zetanu}). For instance, given the multipartition \(\bnu = ((3^2,1) \mid (2^2) \mid (2))\) of \(13\), we have:
\begin{align*}
\zeta(\bnu) = 
\zeta
\left(\;\;
{}
\hackcenter{
\begin{tikzpicture}[scale=0.29]
\draw[thick,fill=green!30] (0,0)--(0,3)--(3,3)--(3,1)--(1,1)--(1,0)--(0,0);
\draw[thick,  gray, ] (0,1)--(1,1);
\draw[thick,  gray, ] (0,2)--(3,2);
\draw[thick,  gray, ] (1,1)--(1,3);
\draw[thick,  gray, ] (2,1)--(2,3);
\end{tikzpicture}
}
\hackcenter{
\begin{tikzpicture}[scale=0.29]
\node[white] at (-0.2,0){{}};
\node[white] at (0.2,0){{}};
\draw[thin, gray,fill=gray!30]  (0,0)--(0,2);
\end{tikzpicture}
}
\hackcenter{
\begin{tikzpicture}[scale=0.29]
\draw[thick,fill=red!30] (0,0)--(0,2)--(2,2)--(2,0)--(0,0);
\draw[thick,  gray, ] (0,1)--(2,1);
\draw[thick,  gray, ] (1,0)--(1,2);
\draw[thick, white] (0,-1)--(1,-1);
\end{tikzpicture}
}
\hackcenter{
\begin{tikzpicture}[scale=0.29]
\node[white] at (-0.2,0){{}};
\node[white] at (0.2,0){{}};
\draw[thin, gray,fill=gray!30]  (0,0)--(0,2);
\end{tikzpicture}
}
\hackcenter{
\begin{tikzpicture}[scale=0.29]
\draw[thick,fill=cyan!30] (0,0)--(0,1)--(2,1)--(2,0)--(0,0);
\draw[thick,  gray, ] (1,0)--(1,1);
\draw[thick, white] (0,-2)--(1,-2);
\end{tikzpicture}
}
\;\;
\right)
=
\left(
\;\;\;
{}
\hackcenter{
\begin{tikzpicture}[scale=0.29]
\draw[thick,fill=green!30]  (0,0)--(0,1)--(1,1)--(1,2)--(2,2)--(2,3)--(3,3)--(3,4)--(4,4)--(4,5)--(5,5)--(5,6)--(8,6)--(8,7)--(11,7)--(11,8)--(13,8)--(13,7)--(12,7)--(12,6)--(11,6)--(11,5)--(10,5)--(10,4)--(7,4)--(7,3)--(4,3)--(4,2)--(3,2)--(3,1)--(2,1)--(2,0)--(0,0);
\node at (0.5,0.5){$\scriptstyle 1$};
\node at (1.5,0.5){$\scriptstyle 2$};
\node at (1.5,1.5){$\scriptstyle 3$};
\node at (2.5,1.5){$\scriptstyle 0$};
\node at (0.5+2,0.5+2){$\scriptstyle 1$};
\node at (1.5+2,0.5+2){$\scriptstyle 2$};
\node at (1.5+2,1.5+2){$\scriptstyle 3$};
\node at (2.5+2,1.5+2){$\scriptstyle 0$};
\node at (0.5+4,0.5+4){$\scriptstyle 1$};
\node at (1.5+4,0.5+4){$\scriptstyle 2$};
\node at (1.5+4,1.5+4){$\scriptstyle 3$};
\node at (2.5+4,1.5+4){$\scriptstyle 0$};
\node at (0.5+5,0.5+3){$\scriptstyle 1$};
\node at (1.5+5,0.5+3){$\scriptstyle 2$};
\node at (1.5+5,1.5+3){$\scriptstyle 3$};
\node at (2.5+5,1.5+3){$\scriptstyle 0$};
\node at (0.5+7,0.5+5){$\scriptstyle 1$};
\node at (1.5+7,0.5+5){$\scriptstyle 2$};
\node at (1.5+7,1.5+5){$\scriptstyle 3$};
\node at (2.5+7,1.5+5){$\scriptstyle 0$};
\node at (0.5+8,0.5+4){$\scriptstyle 1$};
\node at (1.5+8,0.5+4){$\scriptstyle 2$};
\node at (1.5+8,1.5+4){$\scriptstyle 3$};
\node at (2.5+8,1.5+4){$\scriptstyle 0$};
\node at (0.5+10,0.5+6){$\scriptstyle 1$};
\node at (1.5+10,0.5+6){$\scriptstyle 2$};
\node at (1.5+10,1.5+6){$\scriptstyle 3$};
\node at (2.5+10,1.5+6){$\scriptstyle 0$};
\draw[thick,  gray, dotted] (1,1)--(2,1);
\draw[thick,  gray, dotted] (2,2)--(3,2);
\draw[thick,  gray, dotted] (3,3)--(4,3);
\draw[thick,  gray, dotted] (4,4)--(7,4);
\draw[thick,  gray, dotted] (5,5)--(10,5);
\draw[thick,  gray, dotted] (8,6)--(11,6);
\draw[thick,  gray, dotted] (11,7)--(12,7);
\draw[thick,  gray, dotted] (1,0)--(1,1);
\draw[thick,  gray, dotted] (2,1)--(2,2);
\draw[thick,  gray, dotted] (3,2)--(3,3);
\draw[thick,  gray, dotted] (4,3)--(4,4);
\draw[thick,  gray, dotted] (5,3)--(5,5);
\draw[thick,  gray, dotted] (6,3)--(6,6);
\draw[thick,  gray, dotted] (7,4)--(7,6);
\draw[thick,  gray, dotted] (8,4)--(8,6);
\draw[thick,  gray, dotted] (9,4)--(9,7);
\draw[thick,  gray, dotted] (10,5)--(10,7);
\draw[thick,  gray, dotted] (11,6)--(11,7);
\draw[thick,  gray, dotted] (12,7)--(12,8);
\draw[thick] (0,0)--(0,1)--(1,1)--(1,2)--(3,2)--(3,1)--(2,1)--(2,0)--(0,0);
\draw[thick] (0+2.00,0+2.00)--(0+2.00,1+2.00)--(1+2.00,1+2.00)--(1+2.00,2+2.00)--(3+2.00,2+2.00)--(3+2.00,1+2.00)--(2+2.00,1+2.00)--(2+2.00,0+2.00)--(0+2.00,0+2.00);
\draw[thick] (0+4.00,0+4.00)--(0+4.00,1+4.00)--(1+4.00,1+4.00)--(1+4.00,2+4.00)--(3+4.00,2+4.00)--(3+4.00,1+4.00)--(2+4.00,1+4.00)--(2+4.00,0+4.00)--(0+4.00,0+4.00);
\draw[thick] (0+5.00,0+3.00)--(0+5.00,1+3.00)--(1+5.00,1+3.00)--(1+5.00,2+3.00)--(3+5.00,2+3.00)--(3+5.00,1+3.00)--(2+5.00,1+3.00)--(2+5.00,0+3.00)--(0+5.00,0+3.00);
\draw[thick] (0+8.00,0+4.00)--(0+8.00,1+4.00)--(1+8.00,1+4.00)--(1+8.00,2+4.00)--(3+8.00,2+4.00)--(3+8.00,1+4.00)--(2+8.00,1+4.00)--(2+8.00,0+4.00)--(0+8.00,0+4.00);
\draw[thick] (0+7.00,0+5.00)--(0+7.00,1+5.00)--(1+7.00,1+5.00)--(1+7.00,2+5.00)--(3+7.00,2+5.00)--(3+7.00,1+5.00)--(2+7.00,1+5.00)--(2+7.00,0+5.00)--(0+7.00,0+5.00);
\draw[thick] (0+10.00,0+6.00)--(0+10.00,1+6.00)--(1+10.00,1+6.00)--(1+10.00,2+6.00)--(3+10.00,2+6.00)--(3+10.00,1+6.00)--(2+10.00,1+6.00)--(2+10.00,0+6.00)--(0+10.00,0+6.00);
\end{tikzpicture}
}
\hackcenter{
\begin{tikzpicture}[scale=0.29]
\node[white] at (-0.2,0){{}};
\node[white] at (0.2,0){{}};
\draw[thin, gray,fill=gray!30]  (0,0)--(0,6);
\end{tikzpicture}
}
\hackcenter{
\begin{tikzpicture}[scale=0.29]
\draw[thick,fill=red!30]  (0,0)--(0,2)--(1,2)--(1,5)--(2,5)--(2,6)--(3,6)--(3,7)--(4,7)--(4,8)--(5,8)--(5,6)--(4,6)--(4,3)--(3,3)--(3,2)--(2,2)--(2,1)--(1,1)--(1,0)--(0,0);
\node at (0.5,0.5){$\scriptstyle 0$};
\node at (0.5,1.5){$\scriptstyle 1$};
\node at (1.5,1.5){$\scriptstyle 2$};
\node at (1.5,2.5){$\scriptstyle 3$};
\node at (0.5+2,0.5+2){$\scriptstyle 0$};
\node at (0.5+2,1.5+2){$\scriptstyle 1$};
\node at (1.5+2,1.5+2){$\scriptstyle 2$};
\node at (1.5+2,2.5+2){$\scriptstyle 3$};
\node at (0.5+1,0.5+3){$\scriptstyle 0$};
\node at (0.5+1,1.5+3){$\scriptstyle 1$};
\node at (1.5+1,1.5+3){$\scriptstyle 2$};
\node at (1.5+1,2.5+3){$\scriptstyle 3$};
\node at (0.5+3,0.5+5){$\scriptstyle 0$};
\node at (0.5+3,1.5+5){$\scriptstyle 1$};
\node at (1.5+3,1.5+5){$\scriptstyle 2$};
\node at (1.5+3,2.5+5){$\scriptstyle 3$};
\draw[thick,  gray, dotted] (0,1)--(1,1);
\draw[thick,  gray, dotted] (1,2)--(2,2);
\draw[thick,  gray, dotted] (1,3)--(3,3);
\draw[thick,  gray, dotted] (1,4)--(4,4);
\draw[thick,  gray, dotted] (2,5)--(4,5);
\draw[thick,  gray, dotted] (3,6)--(4,6);
\draw[thick,  gray, dotted] (4,7)--(5,7);
\draw[thick,  gray, dotted] (1,1)--(1,2);
\draw[thick,  gray, dotted] (2,2)--(2,5);
\draw[thick,  gray, dotted] (3,3)--(3,6);
\draw[thick,  gray, dotted] (4,6)--(4,7);
\draw[thick]  (0,0)--(0,2)--(1,2)--(1,3)--(2,3)--(2,1)--(1,1)--(1,0)--(0,0);
\draw[thick]  (0+2.0,0+2.0)--(0+2.0,2+2.0)--(1+2.0,2+2.0)--(1+2.0,3+2.0)--(2+2.0,3+2.0)--(2+2.0,1+2.0)--(1+2.0,1+2.0)--(1+2.0,0+2.0)--(0+2.0,0+2.0);
\draw[thick]  (0+1.0,0+3.0)--(0+1.0,2+3.0)--(1+1.0,2+3.0)--(1+1.0,3+3.0)--(2+1.0,3+3.0)--(2+1.0,1+3.0)--(1+1.0,1+3.0)--(1+1.0,0+3.0)--(0+1.0,0+3.0);
\draw[thick]  (0+3.0,0+5.0)--(0+3.0,2+5.0)--(1+3.0,2+5.0)--(1+3.0,3+5.0)--(2+3.0,3+5.0)--(2+3.0,1+5.0)--(1+3.0,1+5.0)--(1+3.0,0+5.0)--(0+3.0,0+5.0);
\end{tikzpicture}
}
\hackcenter{
\begin{tikzpicture}[scale=0.29]
\node[white] at (-0.2,0){{}};
\node[white] at (0.2,0){{}};
\draw[thin, gray,fill=gray!30]  (0,0)--(0,6);
\end{tikzpicture}
}
\hackcenter{
\begin{tikzpicture}[scale=0.29]
\draw[thick,fill=cyan!30]  (0,0)--(0,4)--(1,4)--(1,7)--(2,7)--(2,3)--(1,3)--(1,0)--(0,0);
\node at (0.5,0.5){$\scriptstyle 2$};
\node at (0.5,1.5){$\scriptstyle 3$};
\node at (0.5,2.5){$\scriptstyle 0$};
\node at (0.5,3.5){$\scriptstyle 1$};
\node at (1.5,3.5){$\scriptstyle 2$};
\node at (1.5,4.5){$\scriptstyle 3$};
\node at (1.5,5.5){$\scriptstyle 0$};
\node at (1.5,6.5){$\scriptstyle 1$};
\draw[thick,  gray, dotted] (0,1)--(1,1);
\draw[thick,  gray, dotted] (0,2)--(1,2);
\draw[thick,  gray, dotted] (0,3)--(1,3);
\draw[thick,  gray, dotted] (1,3)--(1,4);
\draw[thick,  gray, dotted] (1,4)--(2,4);
\draw[thick,  gray, dotted] (1,5)--(2,5);
\draw[thick,  gray, dotted] (1,6)--(2,6);
\draw[thick] (0,0)--(0,4)--(1,4)--(1,0)--(0,0);
\draw[thick] (0+1,0+3)--(0+1,4+3)--(1+1,4+3)--(1+1,0+3)--(0+1,0+3);
\end{tikzpicture}
}
\;\;\;
\right).
\end{align*}
Then \(\zS^{\zeta(\bnu)}\) is an indecomposable semicuspidal \(R_{13 \delta}\)-module with \(\textup{hd}(\zS^{\zeta(\bnu)}) \approx L(\bnu)\).
\end{Example}

\subsection{Specht covers of simple and proper standard modules}
More generally, for a root partition \(\pi = (\boldsymbol{K}, \bnu) \in \Pi(\omega)\), we construct an associated skew diagram \(\zeta(\pi)\) of content \(\omega\) in \cref{def:zetapi}. This is a concatenation of previously constructed semicuspidal skew diagrams with multiplicities determined by \(\boldsymbol{K}\):
\begin{align*}
\zeta(\pi) = \left(
\zeta(\beta_1)^{K_{\beta_1}} \mid \cdots \mid \zeta(\beta_u)^{K_{\beta_u}} \mid \zeta(\bnu) \mid  \zeta(\beta_{u+1})^{K_{\beta_{u+1}}} \mid \cdots \mid \zeta(\beta_t)^{K_{\beta_t}}
\right).
\end{align*}
In \cref{thm:klrsimples} we establish the following, showing that there exists a cuspidal-theoretic approach to KLR representation theory in which skew Specht modules play a role analogous to proper standard modules.

\begin{namedthm*}{Theorem B}\hypertarget{thm:B}
Let \(\pi = (\boldsymbol{K}, \bnu) \in \Pi(\omega)\). Then the following statements hold.
\begin{enumerate}
\item The skew Specht module \(\zS^{\zeta(\pi)}\) is indecomposable with simple head
\begin{align*} 
\textup{hd}(\zS^{\zeta(\pi)}) \approx L(\pi),
\end{align*} 
so that \(\{\textup{hd}(\zS^{\zeta(\pi)}) \mid \pi \in \Pi(\omega)\}\) gives a complete and irredundant set of simple \(R_\omega\)-modules up to grading shift.
\item For \(\sigma \in \Pi(\omega)\), we have \([\zS^{\zeta(\pi)} : L(\pi)] = 1\), and \([\zS^{\zeta(\pi)} : L(\sigma)] >0\) only if \(\sigma \leq_{\textup{bd}} \pi\). Moreover, for all root partitions of the form \((\boldsymbol{K}, \bmu) \in \Pi(\omega)\), we have \([\zS^{\zeta(\pi)} : L(\boldsymbol{K},\bmu)] = d^{\textup{RoCK}}_{\bnu, \bmu}\).
\item There exists a surjection \(\zS^{\zeta(\pi)} \twoheadrightarrow \bar{\Delta}(\pi)\), and \(\zS^{\zeta(\pi)}\) has a filtration by proper standard modules of the form \(\bar{\Delta}(\boldsymbol{K}, \bmu)\), where  \((\zS^{\zeta(\pi)} : \bar{\Delta}(\boldsymbol{K}, \bmu)) = d^{\textup{RoCK}}_{\bnu, \bmu}\). 
 \end{enumerate}
\end{namedthm*}

By comparison with the proper standard module \(\bar{\Delta}(\pi)\), the Specht module \(\zS^{\zeta(\pi)}\) has a direct presentation via generators and relations and an explicit combinatorial basis, which we believe will be useful in further combinatorial study of KLR representation theory. This point of view also provides the beginnings of a combinatorial link between the cuspidal system approach to KLR representation theory and the cellular approach to cyclotomic KLR representation theory (see for instance \S\ref{jamesintrosec}).
One may view the above parameterization of simple KLR modules via skew diagrams as an affine type generalization of the {\em Bernstein--Zelevinsky multisegment} parameterization of Hecke algebra representations \cite{vazirani02, Gur}.

\begin{Example}\label{secondexintro}
Continuing where we left off with \cref{firstexintro}, take for instance the root partition \(\pi \in \Pi(20\alpha_0 + 20 \alpha_1 + 22 \alpha_2 + 21 \alpha_3)\) defined as
\begin{align*}
\pi = \left(
\left(
\alpha_2 + \alpha_3 + \alpha_0 
\mid
 2\delta + \alpha_0 + \alpha_1 + \alpha_2
 \mid
 (\delta + \alpha_2 + \alpha_3)^2 
\mid
 \delta^{13}
 \mid
 \delta + \alpha_1 
 \right)
 ,
\left(
(3^2,1) \mid (2^2) \mid (2)
\right)
\right).
\end{align*}
Then, in consideration of the ribbons in \cref{firstexintro}, we have that
\begin{align*}
{}
\zeta(\pi) =
\left(
\hackcenter{
\begin{tikzpicture}[scale=0.29]
\draw[thick,fill=gray!30]  (0,0)--(0,2)--(2,2)--(2,1)--(1,1)--(1,0)--(0,0);
\node at (0.5,0.5){$\scriptstyle 2$};
\node at (0.5,1.5){$\scriptstyle 3$};
\node at (1.5,1.5){$\scriptstyle 0$};
\draw[thick,  gray, dotted] (0,1)--(1,1);
\draw[thick,  gray, dotted] (1,1)--(1,2);
\end{tikzpicture}
}
\hackcenter{
\begin{tikzpicture}[scale=0.29]
\node[white] at (-0.2,0){{}};
\node[white] at (0.2,0){{}};
\draw[thin, gray,fill=gray!30]  (0,0)--(0,6);
\end{tikzpicture}
}
\hackcenter{
\begin{tikzpicture}[scale=0.29]
\draw[thick,fill=gray!30]  (0,0)--(0,2)--(1,2)--(1,3)--(2,3)--(2,4)--(3,4)--(3,5)--(7,5)--(7,4)--(4,4)--(4,3)--(3,3)--(3,2)--(2,2)--(2,1)--(1,1)--(1,0)--(0,0);
\node at (0.5,0.5){$\scriptstyle 0$};
\node at (0.5,1.5){$\scriptstyle 1$};
\node at (1.5,1.5){$\scriptstyle 2$};
\node at (1.5,2.5){$\scriptstyle 3$};
\node at (2.5,2.5){$\scriptstyle 0$};
\node at (2.5,3.5){$\scriptstyle 1$};
\node at (3.5,3.5){$\scriptstyle 2$};
\node at (3.5,4.5){$\scriptstyle 3$};
\node at (4.5,4.5){$\scriptstyle 0$};
\node at (5.5,4.5){$\scriptstyle 1$};
\node at (6.5,4.5){$\scriptstyle 2$};
\draw[thick,  gray, dotted] (0,1)--(1,1);
\draw[thick,  gray, dotted] (1,2)--(2,2);
\draw[thick,  gray, dotted] (2,3)--(3,3);
\draw[thick,  gray, dotted] (3,4)--(4,4);
\draw[thick,  gray, dotted] (1,1)--(1,2);
\draw[thick,  gray, dotted] (2,2)--(2,3);
\draw[thick,  gray, dotted] (3,3)--(3,4);
\draw[thick,  gray, dotted] (4,4)--(4,5);
\draw[thick,  gray, dotted] (5,4)--(5,5);
\draw[thick,  gray, dotted] (6,4)--(6,5);
\end{tikzpicture}
}
\hackcenter{
\begin{tikzpicture}[scale=0.29]
\node[white] at (-0.2,0){{}};
\node[white] at (0.2,0){{}};
\draw[thin, gray,fill=gray!30]  (0,0)--(0,6);
\end{tikzpicture}
}
\hackcenter{
\begin{tikzpicture}[scale=0.29]
\draw[thick,fill=gray!30]  (0,0)--(0,4)--(1,4)--(1,5)--(2,5)--(2,3)--(1,3)--(1,0)--(0,0);
\node at (0.5,0.5){$\scriptstyle 2$};
\node at (0.5,1.5){$\scriptstyle 3$};
\node at (0.5,2.5){$\scriptstyle 0$};
\node at (0.5,3.5){$\scriptstyle 1$};
\node at (1.5,3.5){$\scriptstyle 2$};
\node at (1.5,4.5){$\scriptstyle 3$};
\draw[thick,  gray, dotted] (0,1)--(1,1);
\draw[thick,  gray, dotted] (0,2)--(1,2);
\draw[thick,  gray, dotted] (0,3)--(1,3);
\draw[thick,  gray, dotted] (1,3)--(1,4);
\draw[thick,  gray, dotted] (1,4)--(2,4);
\end{tikzpicture}
}
\hackcenter{
\begin{tikzpicture}[scale=0.29]
\node[white] at (-0.2,0){{}};
\node[white] at (0.2,0){{}};
\draw[thin, gray,fill=gray!30]  (0,0)--(0,6);
\end{tikzpicture}
}
\hackcenter{
\begin{tikzpicture}[scale=0.29]
\draw[thick,fill=gray!30]  (0,0)--(0,4)--(1,4)--(1,5)--(2,5)--(2,3)--(1,3)--(1,0)--(0,0);
\node at (0.5,0.5){$\scriptstyle 2$};
\node at (0.5,1.5){$\scriptstyle 3$};
\node at (0.5,2.5){$\scriptstyle 0$};
\node at (0.5,3.5){$\scriptstyle 1$};
\node at (1.5,3.5){$\scriptstyle 2$};
\node at (1.5,4.5){$\scriptstyle 3$};
\draw[thick,  gray, dotted] (0,1)--(1,1);
\draw[thick,  gray, dotted] (0,2)--(1,2);
\draw[thick,  gray, dotted] (0,3)--(1,3);
\draw[thick,  gray, dotted] (1,3)--(1,4);
\draw[thick,  gray, dotted] (1,4)--(2,4);
\end{tikzpicture}
}
\hackcenter{
\begin{tikzpicture}[scale=0.29]
\node[white] at (-0.2,0){{}};
\node[white] at (0.2,0){{}};
\draw[thin, gray,fill=gray!30]  (0,0)--(0,6);
\end{tikzpicture}
}
\hackcenter{
\begin{tikzpicture}[scale=0.29]
\draw[thick,fill=green!30]  (0,0)--(0,1)--(1,1)--(1,2)--(2,2)--(2,3)--(3,3)--(3,4)--(4,4)--(4,5)--(5,5)--(5,6)--(8,6)--(8,7)--(11,7)--(11,8)--(13,8)--(13,7)--(12,7)--(12,6)--(11,6)--(11,5)--(10,5)--(10,4)--(7,4)--(7,3)--(4,3)--(4,2)--(3,2)--(3,1)--(2,1)--(2,0)--(0,0);
\draw[thick] (0,0)--(0,1)--(1,1)--(1,2)--(3,2)--(3,1)--(2,1)--(2,0)--(0,0);
\draw[thick] (0+2.00,0+2.00)--(0+2.00,1+2.00)--(1+2.00,1+2.00)--(1+2.00,2+2.00)--(3+2.00,2+2.00)--(3+2.00,1+2.00)--(2+2.00,1+2.00)--(2+2.00,0+2.00)--(0+2.00,0+2.00);
\draw[thick] (0+4.00,0+4.00)--(0+4.00,1+4.00)--(1+4.00,1+4.00)--(1+4.00,2+4.00)--(3+4.00,2+4.00)--(3+4.00,1+4.00)--(2+4.00,1+4.00)--(2+4.00,0+4.00)--(0+4.00,0+4.00);
\draw[thick] (0+5.00,0+3.00)--(0+5.00,1+3.00)--(1+5.00,1+3.00)--(1+5.00,2+3.00)--(3+5.00,2+3.00)--(3+5.00,1+3.00)--(2+5.00,1+3.00)--(2+5.00,0+3.00)--(0+5.00,0+3.00);
\draw[thick] (0+8.00,0+4.00)--(0+8.00,1+4.00)--(1+8.00,1+4.00)--(1+8.00,2+4.00)--(3+8.00,2+4.00)--(3+8.00,1+4.00)--(2+8.00,1+4.00)--(2+8.00,0+4.00)--(0+8.00,0+4.00);
\draw[thick] (0+7.00,0+5.00)--(0+7.00,1+5.00)--(1+7.00,1+5.00)--(1+7.00,2+5.00)--(3+7.00,2+5.00)--(3+7.00,1+5.00)--(2+7.00,1+5.00)--(2+7.00,0+5.00)--(0+7.00,0+5.00);
\draw[thick] (0+10.00,0+6.00)--(0+10.00,1+6.00)--(1+10.00,1+6.00)--(1+10.00,2+6.00)--(3+10.00,2+6.00)--(3+10.00,1+6.00)--(2+10.00,1+6.00)--(2+10.00,0+6.00)--(0+10.00,0+6.00);
\node at (0.5,0.5){$\scriptstyle 1$};
\node at (1.5,0.5){$\scriptstyle 2$};
\node at (1.5,1.5){$\scriptstyle 3$};
\node at (2.5,1.5){$\scriptstyle 0$};
\node at (0.5+2,0.5+2){$\scriptstyle 1$};
\node at (1.5+2,0.5+2){$\scriptstyle 2$};
\node at (1.5+2,1.5+2){$\scriptstyle 3$};
\node at (2.5+2,1.5+2){$\scriptstyle 0$};
\node at (0.5+4,0.5+4){$\scriptstyle 1$};
\node at (1.5+4,0.5+4){$\scriptstyle 2$};
\node at (1.5+4,1.5+4){$\scriptstyle 3$};
\node at (2.5+4,1.5+4){$\scriptstyle 0$};
\node at (0.5+5,0.5+3){$\scriptstyle 1$};
\node at (1.5+5,0.5+3){$\scriptstyle 2$};
\node at (1.5+5,1.5+3){$\scriptstyle 3$};
\node at (2.5+5,1.5+3){$\scriptstyle 0$};
\node at (0.5+7,0.5+5){$\scriptstyle 1$};
\node at (1.5+7,0.5+5){$\scriptstyle 2$};
\node at (1.5+7,1.5+5){$\scriptstyle 3$};
\node at (2.5+7,1.5+5){$\scriptstyle 0$};
\node at (0.5+8,0.5+4){$\scriptstyle 1$};
\node at (1.5+8,0.5+4){$\scriptstyle 2$};
\node at (1.5+8,1.5+4){$\scriptstyle 3$};
\node at (2.5+8,1.5+4){$\scriptstyle 0$};
\node at (0.5+10,0.5+6){$\scriptstyle 1$};
\node at (1.5+10,0.5+6){$\scriptstyle 2$};
\node at (1.5+10,1.5+6){$\scriptstyle 3$};
\node at (2.5+10,1.5+6){$\scriptstyle 0$};
\draw[thick,  gray, dotted] (1,1)--(2,1);
\draw[thick,  gray, dotted] (2,2)--(3,2);
\draw[thick,  gray, dotted] (3,3)--(4,3);
\draw[thick,  gray, dotted] (4,4)--(7,4);
\draw[thick,  gray, dotted] (5,5)--(10,5);
\draw[thick,  gray, dotted] (8,6)--(11,6);
\draw[thick,  gray, dotted] (11,7)--(12,7);
\draw[thick,  gray, dotted] (1,0)--(1,1);
\draw[thick,  gray, dotted] (2,1)--(2,2);
\draw[thick,  gray, dotted] (3,2)--(3,3);
\draw[thick,  gray, dotted] (4,3)--(4,4);
\draw[thick,  gray, dotted] (5,3)--(5,5);
\draw[thick,  gray, dotted] (6,3)--(6,6);
\draw[thick,  gray, dotted] (7,4)--(7,6);
\draw[thick,  gray, dotted] (8,4)--(8,6);
\draw[thick,  gray, dotted] (9,4)--(9,7);
\draw[thick,  gray, dotted] (10,5)--(10,7);
\draw[thick,  gray, dotted] (11,6)--(11,7);
\draw[thick,  gray, dotted] (12,7)--(12,8);
\end{tikzpicture}
}
\hackcenter{
\begin{tikzpicture}[scale=0.29]
\node[white] at (-0.2,0){{}};
\node[white] at (0.2,0){{}};
\draw[thin, gray,fill=gray!30]  (0,0)--(0,6);
\end{tikzpicture}
}
\hackcenter{
\begin{tikzpicture}[scale=0.29]
\draw[thick,fill=red!30]  (0,0)--(0,2)--(1,2)--(1,5)--(2,5)--(2,6)--(3,6)--(3,7)--(4,7)--(4,8)--(5,8)--(5,6)--(4,6)--(4,3)--(3,3)--(3,2)--(2,2)--(2,1)--(1,1)--(1,0)--(0,0);
\draw[thick]  (0,0)--(0,2)--(1,2)--(1,3)--(2,3)--(2,1)--(1,1)--(1,0)--(0,0);
\draw[thick]  (0+2.0,0+2.0)--(0+2.0,2+2.0)--(1+2.0,2+2.0)--(1+2.0,3+2.0)--(2+2.0,3+2.0)--(2+2.0,1+2.0)--(1+2.0,1+2.0)--(1+2.0,0+2.0)--(0+2.0,0+2.0);
\draw[thick]  (0+1.0,0+3.0)--(0+1.0,2+3.0)--(1+1.0,2+3.0)--(1+1.0,3+3.0)--(2+1.0,3+3.0)--(2+1.0,1+3.0)--(1+1.0,1+3.0)--(1+1.0,0+3.0)--(0+1.0,0+3.0);
\draw[thick]  (0+3.0,0+5.0)--(0+3.0,2+5.0)--(1+3.0,2+5.0)--(1+3.0,3+5.0)--(2+3.0,3+5.0)--(2+3.0,1+5.0)--(1+3.0,1+5.0)--(1+3.0,0+5.0)--(0+3.0,0+5.0);
\node at (0.5,0.5){$\scriptstyle 0$};
\node at (0.5,1.5){$\scriptstyle 1$};
\node at (1.5,1.5){$\scriptstyle 2$};
\node at (1.5,2.5){$\scriptstyle 3$};
\node at (0.5+2,0.5+2){$\scriptstyle 0$};
\node at (0.5+2,1.5+2){$\scriptstyle 1$};
\node at (1.5+2,1.5+2){$\scriptstyle 2$};
\node at (1.5+2,2.5+2){$\scriptstyle 3$};
\node at (0.5+1,0.5+3){$\scriptstyle 0$};
\node at (0.5+1,1.5+3){$\scriptstyle 1$};
\node at (1.5+1,1.5+3){$\scriptstyle 2$};
\node at (1.5+1,2.5+3){$\scriptstyle 3$};
\node at (0.5+3,0.5+5){$\scriptstyle 0$};
\node at (0.5+3,1.5+5){$\scriptstyle 1$};
\node at (1.5+3,1.5+5){$\scriptstyle 2$};
\node at (1.5+3,2.5+5){$\scriptstyle 3$};
\draw[thick,  gray, dotted] (0,1)--(1,1);
\draw[thick,  gray, dotted] (1,2)--(2,2);
\draw[thick,  gray, dotted] (1,3)--(3,3);
\draw[thick,  gray, dotted] (1,4)--(4,4);
\draw[thick,  gray, dotted] (2,5)--(4,5);
\draw[thick,  gray, dotted] (3,6)--(4,6);
\draw[thick,  gray, dotted] (4,7)--(5,7);
\draw[thick,  gray, dotted] (1,1)--(1,2);
\draw[thick,  gray, dotted] (2,2)--(2,5);
\draw[thick,  gray, dotted] (3,3)--(3,6);
\draw[thick,  gray, dotted] (4,6)--(4,7);
\end{tikzpicture}
}
\hackcenter{
\begin{tikzpicture}[scale=0.29]
\node[white] at (-0.2,0){{}};
\node[white] at (0.2,0){{}};
\draw[thin, gray,fill=gray!30]  (0,0)--(0,6);
\end{tikzpicture}
}
\hackcenter{
\begin{tikzpicture}[scale=0.29]
\draw[thick,fill=cyan!30]  (0,0)--(0,4)--(1,4)--(1,7)--(2,7)--(2,3)--(1,3)--(1,0)--(0,0);
\draw[thick] (0,0)--(0,4)--(1,4)--(1,0)--(0,0);
\draw[thick] (0+1,0+3)--(0+1,4+3)--(1+1,4+3)--(1+1,0+3)--(0+1,0+3);
\node at (0.5,0.5){$\scriptstyle 2$};
\node at (0.5,1.5){$\scriptstyle 3$};
\node at (0.5,2.5){$\scriptstyle 0$};
\node at (0.5,3.5){$\scriptstyle 1$};
\node at (1.5,3.5){$\scriptstyle 2$};
\node at (1.5,4.5){$\scriptstyle 3$};
\node at (1.5,5.5){$\scriptstyle 0$};
\node at (1.5,6.5){$\scriptstyle 1$};
\draw[thick,  gray, dotted] (0,1)--(1,1);
\draw[thick,  gray, dotted] (0,2)--(1,2);
\draw[thick,  gray, dotted] (0,3)--(1,3);
\draw[thick,  gray, dotted] (1,3)--(1,4);
\draw[thick,  gray, dotted] (1,4)--(2,4);
\draw[thick,  gray, dotted] (1,5)--(2,5);
\draw[thick,  gray, dotted] (1,6)--(2,6);
\end{tikzpicture}
}
\hackcenter{
\begin{tikzpicture}[scale=0.29]
\node[white] at (-0.2,0){{}};
\node[white] at (0.2,0){{}};
\draw[thin, gray,fill=gray!30]  (0,0)--(0,6);
\end{tikzpicture}
}
\hackcenter{
\begin{tikzpicture}[scale=0.29]
\draw[thick,fill=gray!30]  (0,0)--(0,1)--(1,1)--(1,2)--(2,2)--(2,3)--(3,3)--(3,1)--(2,1)--(2,0)--(0,0);
\node at (0.5,0.5){$\scriptstyle 1$};
\node at (1.5,0.5){$\scriptstyle 2$};
\node at (1.5,1.5){$\scriptstyle 3$};
\node at (2.5,1.5){$\scriptstyle 0$};
\node at (2.5,2.5){$\scriptstyle 1$};
\draw[thick,  gray, dotted] (1,1)--(2,1);
\draw[thick,  gray, dotted] (2,2)--(3,2);
\draw[thick,  gray, dotted] (1,0)--(1,1);
\draw[thick,  gray, dotted] (2,1)--(2,2);
\end{tikzpicture}
}
\right).
\end{align*}
Then \(\zS^{\zeta(\pi)}\) is an indecomposable \(R_{20\alpha_0 + 20 \alpha_1 + 22 \alpha_2 + 21 \alpha_3}\textup{-module}\) with simple head \(\textup{hd}(\zS^{\zeta(\pi)}) \approx L(\pi)\).
\end{Example}

\subsection{Cuspidal ribbons and RoCK blocks}
The proofs of Theorems A and B come by way of a connection between semicuspidal modules and RoCK blocks.
RoCK blocks are `generic' cyclotomic KLR algebras that have played a key role in the modular representation theory of the symmetric group and the proof of Brou\'e's abelian defect conjecture in this setting. Initially defined by Rouquier \cite{R2}, their structure was studied up to Morita equivalence by Chuang--Kessar \cite{CK02} (in the abelian defect case) and Evseev--Kleshchev \cite{ek18} (in arbitrary defect). Recently RoCK blocks have been defined and studied for the higher-level cyclotomic KLR algebras by Lyle \cite{Lyle22}, working from a combinatorial perspective, and Webster \cite{websterScopes}, working from a Lie-theoretic perspective. 
{\em Core blocks} are the simplest types of RoCK blocks (see \cref{paramRoCK}), and their associated combinatorics have also been the focus of study by Fayers \cite{fay07core, fayerssimcore, fay06wts}, Lyle~\cite{Lyle24coreblocks} and Lyle--Ruff \cite{LyleRuff}, to name a few examples. 

In \cref{subsec:cuspsystems} we study interactions between the combinatorics of RoCK blocks, core blocks and {\em cuspidal ribbon tilings} of multipartitions in these blocks. For a {\em multicharge} \(\bkap = (\kappa_1, \dots, \kappa_\ell)\) of level \(\ell\),
every \(\ell\)-multipartition \(\blam\) has an associated content \(\textup{cont}(\blam) \in \ZZ_{\geq 0}I\). 
Given a choice of convex preorder \(\succeq\) on \(\Phi_+\), we have the associated cuspidal ribbons
\(
 \{ \zeta(\beta) \mid \beta \in \Phi_+^\re \} \cup \{\zeta_0, \ldots, \zeta_{e-1}\}
\),
as in \cref{firstexintro}.
It was shown in \cite{muthtiling} that every skew diagram \(\btau\) possesses a unique ordered tiling \(\Gamma_{\btau}\) by these ribbons, called a {\em cuspidal Kostant tiling}.  In \S\ref{cusptilesec}, we investigate the properties of cuspidal Kostant tilings in RoCK blocks and core blocks, and 
establish in \cref{thetarockblocktile,coretilingprop} the following result, which shows that these special blocks are distinguished combinatorially by their `biased' tilings.

\begin{namedthm*}{Theorem C}\hypertarget{thm:C}
A cyclotomic KLR algebra is a RoCK block (resp.~core block) if and only if there exists a convex preorder \(\succeq\) on \(\Phi_+\) such that the cuspidal Kostant tiling \(\Gamma_{\blam}\) of every multipartition \(\blam\) with \(\textup{cont}(\blam) = \beta\) consists only of tiles with content \(\succeq \delta\) (resp.~\(\succ \delta\)).
\end{namedthm*}

See \cref{3coretileex} and Figures~\ref{corebigtilings} and~\ref{exwithdeltaribsfig} for illustrative examples of cuspidal ribbon tilings of core and RoCK multipartitions. We remark that the `existence' statement of Theorem C has a constructive proof; we provide a construction of a good convex preorder associated to a RoCK block in \S\ref{cusptilesec}.
We consider RoCK multipartitions as the data of a multicore coupled with a skew diagram tiled by imaginary ribbons. In \cref{combimagsec} we investigate in detail the diagram and word combinatorics of imaginary ribbons in RoCK blocks.


\subsection{Skew cyclotomic KLR algebras}
For \(\omega, \beta \in \ZZ_{\geq 0}I\), we consider the representation theory of the `\(\omega\)-skew' cyclotomic quotient of \(R_\beta\), which we denote \(R_\beta^{\Lambda/\omega}\). This algebra arises as the homomorphic image of the composition
\begin{align*}
R_\beta \longhookrightarrow R_\omega \otimes R_\beta \longhookrightarrow R_{\omega + \beta} \dhxrightarrow{\pi^\Lambda} R_{\omega + \beta}^\Lambda,
\end{align*}
where \(\pi^\Lambda\) is the usual cyclotomic quotient map.
In \cref{wbcutthm} we establish some connections between this skew cyclotomic KLR algebra and idempotent truncations of \(R^{\Lambda}_{\omega + \beta}\). 
Of particular interest in this paper is the case where \(\omega\) admits exactly one multipartition \(\brho\) with \(\textup{cont}_{\bkap}(\brho) = \omega\). In this case, the algebra \(R^\Lambda_\omega\) is simple, and we call the multipartition \(\brho\) a \(\bkap\)-core.
In \cref{skewcell}, we prove the following result.

\begin{namedthm*}{Theorem D}\hypertarget{thm:D}
Let \(\omega, \beta \in \ZZ_{\geq 0}I\), and assume that \(\brho\) is a \(\bkap\)-core with \(\textup{cont}_{\bkap}(\brho) = \omega\).
Then the \(\omega\)-skew cyclotomic KLR algebra \(R_\beta^{\Lambda/\omega}\) is a graded cellular algebra, with cell modules given by the skew Specht modules \(\{S^{\blam/ \brho} \mid \blam \text{ an $\ell$-multipartition with } \textup{cont}_{\bkap}(\blam/\brho) = \beta\}\).
There is moreover an exact functor \(\mathcal{T}: R^{\Lambda}_{\omega + \beta}\textup{-mod} \to R_\beta^{\Lambda/\omega}\textup{-mod}\) which sends the Specht module \(S^{\blam}\) to the skew Specht module \(S^{\blam/\brho}\), and simple modules to simple modules (or zero). 
\end{namedthm*}

In \S\ref{cuspsyssec}, we apply Theorem D to the RoCK block setting.
Our next result shows that the representation theory of imaginary semicuspidal \(R_{d\delta}\)-modules can be studied via skew cyclotomic KLR algebras associated to RoCK blocks.

\begin{namedthm*}{Theorem E}\hypertarget{thm:E}
Fix a convex preorder \(\succeq\) on \(\Phi_+\), and let \(d \in \mathbb{Z}_{>0}\). There exists a level one charge \(\kappa\), core partition \(\rho\) with \(\cont(\rho) = \omega\), and RoCK block \(R^{\Lambda_{\kappa}}_{\omega + d \delta}\) such that the following holds.
\begin{enumerate}
\item The functor \(\mathcal{T}: R^{\Lambda_{\kappa}}_{\omega + d\delta}\textup{-mod} \to R^{{\Lambda_{\kappa}}/\omega}_{d\delta}\textup{-mod}\) of Theorem D is a Morita equivalence.
\item Each \((e-1)\)-multipartition \(\bnu\) of \(d\) may be associated with an \(e\)-restricted partition \(\lambda(\bnu)\) in the RoCK block \(R^{\Lambda_\kappa}_{\omega + d\delta}\) (whose \(e\)-quotient is a certain permutation of \(\bnu\)). We have \(\mathcal{T} S^{\lambda(\bnu)} \cong S^{\lambda(\bnu)/\rho} \approx \zS^{\zeta(\bnu)}\), and 
\begin{align}\label{introallcusps}
\{\textup{hd}(\zS^{\zeta(\bnu)}) \mid \bnu \text{ is an \((e-1)\)-multipartition of } d\}
\end{align}
is a complete and irredundant set of simple \(R^{{\Lambda_{\kappa}}/\omega}_{d\delta}\)-modules up to grading shift.
\item \(R^{{\Lambda_{\kappa}}/\omega}_{d\delta}\)-modules lift to imaginary semicuspidal \(R_{d\delta}\)-modules, and (\ref{introallcusps}) is a complete and irredundant set of simple imaginary semicuspidal \(R_{d\delta}\)-modules up to grading shift.
\end{enumerate}
\end{namedthm*}

This result appears as \cref{Morcut,semicuspthm} in the body of the paper; the reader should
see \cref{twoex} for illustrative examples. For the decomposition numbers in the RoCK block in Theorem E, we write
\begin{align}\label{defdRoCKintro}
d^{\textup{RoCK}}_{\bnu, \bmu} = [S^{\lambda(\bnu)} : \textup{hd}(S^{\lambda(\bmu)})].
\end{align}
The decomposition numbers in Theorems A and B arise as a direct result of this theorem. 
Importantly, these numbers are well-studied and have explicit formulae when \(\text{char}(\k) = 0\) or \(\text{char}(\k) > d\), see for instance \cite{jlm}. 


\subsection{A cuspidal `regularization' theorem for Specht modules}\label{jamesintrosec}
We end by briefly demonstrating an application of Theorem B to the representation theory of Specht modules.
James's celebrated regularization theorem \cite[Theorem A]{j76} gives, for any partition \(\lambda\), a purely combinatorial rule for describing another partition \(\bar{\lambda}\) such that 
\(
[S^{\lambda}: D^{\bar{\lambda}}] = 1
\)
and
\(\mu \trianglelefteq^D \bar{\lambda}\) whenever \([S^\lambda : D^{\mu}] \neq 0\). (Here \(D^{\mu}\) denotes \( \textup{hd}(S^{\mu})\) and \(\trianglelefteq^D\) is the dominance order on partitions.)
We establish the following analogous result for simple factors labeled by {\em root partitions}.
\begin{namedthm*}{Theorem F}\hypertarget{thm:F}
Let \(\btau\) be a multipartition of content \(\omega \in \ZZ_{\geq 0}I\), and assume that \(\pi \in \Pi(\omega)\) is a root partition such that \(\btau\) has a \(\zeta(\pi)\)-tiling.
Then there exists a nonzero homomorphism of \(R_\omega\)-modules \(\zS^{\zeta(\pi)} \to \zS^{\btau}\), and we have \([\zS^{\btau} : L(\pi)] = 1\) and \([\zS^{\btau} : L(\sigma)] > 0\) only if \(\sigma \leq_{\textup{bd}} \pi\).
In particular, if \(\btau = \tau\) is an \(e\)-restricted (level one) partition, the unique cuspidal Kostant tiling \(\Gamma_\tau\) defines such a root partition \(\pi \in \Pi(\omega)\). 
\end{namedthm*}
This appears as \cref{F1,F2} in the body of the paper.
Theorem F is a refined version of \cite[Theorem 8.5]{muthtiling}, which operates at the coarser level of Kostant partition labels. We remark that in general, the information provided by James's regularization theorem and this theorem are generally distinct and unrelated.
It would be advantageous to understand how the simple labels given by multipartitons coming from the {\em cellular} theory of  cyclotomic KLR algebras may be connected with the root partition labels of the cuspidal system theory of the KLR algebra. We view Theorem F as a step in this direction. 

\begin{Example}
Continuing with \(e=4\) and the convex preorder in \cref{3coretileex}, 
consider the partition \(\tau = (7^3,6,3)\) in the block \(R^{\Lambda_0}_{8 \alpha_0 + 8 \alpha_1 + 8 \alpha_2 + 6 \alpha_3}\). The unique cuspidal Kostant tiling \(\Gamma_\tau\) is shown below.
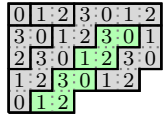
\begin{align*}
{}
\hackcenter{
\begin{tikzpicture}[scale=0.29]
\draw[thick,fill=gray!30]  (0,0)--(0,5)--(7,5)--(7,2)--(6,2)--(6,1)--(3,1)--(3,0)--(0,0);
\draw[thick,fill=green!30]  (1,0)--(3,0)--(3,1)--(4,1)--(4,2)--(5,2)--(5,3)--(6,3)--(6,4)--(4,4)--(4,3)--(3,3)--(3,2)--(2,2)--(2,1)--(1,1)--(1,0);
\node at (0.5,0.5){$\scriptstyle 0$};
\node at (1.5,0.5){$\scriptstyle 1$};
\node at (2.5,0.5){$\scriptstyle 2$};
\node at (0.5,1.5){$\scriptstyle 1$};
\node at (1.5,1.5){$\scriptstyle 2$};
\node at (2.5,1.5){$\scriptstyle 3$};
\node at (3.5,1.5){$\scriptstyle 0$};
\node at (4.5,1.5){$\scriptstyle 1$};
\node at (5.5,1.5){$\scriptstyle 2$};
\node at (0.5,2.5){$\scriptstyle 2$};
\node at (1.5,2.5){$\scriptstyle 3$};
\node at (2.5,2.5){$\scriptstyle 0$};
\node at (3.5,2.5){$\scriptstyle 1$};
\node at (4.5,2.5){$\scriptstyle 2$};
\node at (5.5,2.5){$\scriptstyle 3$};
\node at (6.5,2.5){$\scriptstyle 0$};
\node at (0.5,3.5){$\scriptstyle 3$};
\node at (1.5,3.5){$\scriptstyle 0$};
\node at (2.5,3.5){$\scriptstyle 1$};
\node at (3.5,3.5){$\scriptstyle 2$};
\node at (4.5,3.5){$\scriptstyle 3$};
\node at (5.5,3.5){$\scriptstyle 0$};
\node at (6.5,3.5){$\scriptstyle 1$};
\node at (0.5,4.5){$\scriptstyle 0$};
\node at (1.5,4.5){$\scriptstyle 1$};
\node at (2.5,4.5){$\scriptstyle 2$};
\node at (3.5,4.5){$\scriptstyle 3$};
\node at (4.5,4.5){$\scriptstyle 0$};
\node at (5.5,4.5){$\scriptstyle 1$};
\node at (6.5,4.5){$\scriptstyle 2$};
\draw[thick,  gray, dotted] (0,1)--(3,1);
\draw[thick,  gray, dotted] (0,2)--(6,2);
\draw[thick,  gray, dotted] (0,3)--(7,3);
\draw[thick,  gray, dotted] (0,4)--(7,4);
\draw[thick,  gray, dotted] (0,5)--(7,5);
\draw[thick,  gray, dotted] (1,0)--(1,5);
\draw[thick,  gray, dotted] (2,0)--(2,5);
\draw[thick,  gray, dotted] (3,1)--(3,5);
\draw[thick,  gray, dotted] (4,1)--(4,5);
\draw[thick,  gray, dotted] (5,1)--(5,5);
\draw[thick,  gray, dotted] (6,2)--(6,5);
\draw[thick,]  (0,0)--(0,5)--(7,5)--(7,2)--(6,2)--(6,1)--(3,1)--(3,0)--(0,0);
\draw[thick,] (0,0)--(0,2)--(1,2)--(1,3)--(2,3)--(2,4)--(3,4)--(3,5)--(7,5)--(7,4)--(4,4)--(4,3)--(3,3)--(3,2)--(2,2)--(2,1)--(1,1)--(1,0)--(0,0);
\draw[thick,]  (0,2)--(1,2)--(1,3)--(2,3)--(2,4)--(0,4)--(0,2);
\draw[thick,]  (0,4)--(1,4)--(1,5)--(0,5)--(0,4);
\draw[thick,]  (1,4)--(3,4)--(3,5)--(1,5)--(1,4);
\draw[thick,]  (4,1)--(6,1)--(6,2)--(7,2)--(7,4)--(6,4)--(6,3)--(5,3)--(5,2)--(4,2)--(4,1);
\draw[thick,]  (3,2)--(4,2);
\end{tikzpicture}
}
\end{align*}
Taking unions of same-content tiles defines a root partition \(\pi \in \Pi(8 \alpha_0 + 8 \alpha_1 + 8 \alpha_2 + 6 \alpha_3)\), where
\begin{align*}
\pi = \left(
\left(
\alpha_0
\mid
\alpha_1 + \alpha_2
 \mid
\alpha_2 + \alpha_3 + \alpha_0
\mid
2\delta + \alpha_0 + \alpha_1 + \alpha_2
 \mid
\delta^2 
\mid
\delta + \alpha_1
 \right)
 ,
\left(
(1^2) \mid  \varnothing \mid \varnothing
\right)
\right),
\end{align*}
and
\begin{align*}
{}
\zeta(\pi) =
\left(
\hackcenter{
\begin{tikzpicture}[scale=0.29]
\draw[thick,fill=gray!30]  (0,0)--(1,0)--(1,1)--(0,1)--(0,0);
\node at (0.5,0.5){$\scriptstyle 0$};
\draw[thick,  gray, dotted] (0,1)--(1,1);
\end{tikzpicture}
}
\hackcenter{
\begin{tikzpicture}[scale=0.29]
\node[white] at (-0.2,0){{}};
\node[white] at (0.2,0){{}};
\draw[thin, gray,fill=gray!30]  (0,0)--(0,6);
\end{tikzpicture}
}
\hackcenter{
\begin{tikzpicture}[scale=0.29]
\draw[thick,fill=gray!30]  (0,0)--(2,0)--(2,1)--(0,1)--(0,0);
\node at (0.5,0.5){$\scriptstyle 1$};
\node at (1.5,0.5){$\scriptstyle 2$};
\draw[thick,  gray, dotted] (1,0)--(1,1);
\end{tikzpicture}
}
\hackcenter{
\begin{tikzpicture}[scale=0.29]
\node[white] at (-0.2,0){{}};
\node[white] at (0.2,0){{}};
\draw[thin, gray,fill=gray!30]  (0,0)--(0,6);
\end{tikzpicture}
}
\hackcenter{
\begin{tikzpicture}[scale=0.29]
\draw[thick,fill=gray!30]  (0,0)--(0,2)--(2,2)--(2,1)--(1,1)--(1,0)--(0,0);
\node at (0.5,0.5){$\scriptstyle 2$};
\node at (0.5,1.5){$\scriptstyle 3$};
\node at (1.5,1.5){$\scriptstyle 0$};
\draw[thick,  gray, dotted] (0,1)--(1,1);
\draw[thick,  gray, dotted] (1,1)--(1,2);
\end{tikzpicture}
}
\hackcenter{
\begin{tikzpicture}[scale=0.29]
\node[white] at (-0.2,0){{}};
\node[white] at (0.2,0){{}};
\draw[thin, gray,fill=gray!30]  (0,0)--(0,6);
\end{tikzpicture}
}
\hackcenter{
\begin{tikzpicture}[scale=0.29]
\draw[thick,fill=gray!30]  (0,0)--(0,2)--(1,2)--(1,3)--(2,3)--(2,4)--(3,4)--(3,5)--(7,5)--(7,4)--(4,4)--(4,3)--(3,3)--(3,2)--(2,2)--(2,1)--(1,1)--(1,0)--(0,0);
\node at (0.5,0.5){$\scriptstyle 0$};
\node at (0.5,1.5){$\scriptstyle 1$};
\node at (1.5,1.5){$\scriptstyle 2$};
\node at (1.5,2.5){$\scriptstyle 3$};
\node at (2.5,2.5){$\scriptstyle 0$};
\node at (2.5,3.5){$\scriptstyle 1$};
\node at (3.5,3.5){$\scriptstyle 2$};
\node at (3.5,4.5){$\scriptstyle 3$};
\node at (4.5,4.5){$\scriptstyle 0$};
\node at (5.5,4.5){$\scriptstyle 1$};
\node at (6.5,4.5){$\scriptstyle 2$};
\draw[thick,  gray, dotted] (0,1)--(1,1);
\draw[thick,  gray, dotted] (1,2)--(2,2);
\draw[thick,  gray, dotted] (2,3)--(3,3);
\draw[thick,  gray, dotted] (3,4)--(4,4);
\draw[thick,  gray, dotted] (1,1)--(1,2);
\draw[thick,  gray, dotted] (2,2)--(2,3);
\draw[thick,  gray, dotted] (3,3)--(3,4);
\draw[thick,  gray, dotted] (4,4)--(4,5);
\draw[thick,  gray, dotted] (5,4)--(5,5);
\draw[thick,  gray, dotted] (6,4)--(6,5);
\end{tikzpicture}
}
\hackcenter{
\begin{tikzpicture}[scale=0.29]
\node[white] at (-0.2,0){{}};
\node[white] at (0.2,0){{}};
\draw[thin, gray,fill=gray!30]  (0,0)--(0,6);
\end{tikzpicture}
}
\hackcenter{
\begin{tikzpicture}[scale=0.29]
\draw[thick,fill=green!30] (1,0)--(3,0)--(3,1)--(4,1)--(4,2)--(5,2)--(5,3)--(6,3)--(6,4)--(4,4)--(4,3)--(3,3)--(3,2)--(2,2)--(2,1)--(1,1)--(1,0);
\draw[thick,] (3,2)--(4,2);
\node at (1.5,0.5){$\scriptstyle 1$};
\node at (2.5,0.5){$\scriptstyle 2$};
\node at (2.5,1.5){$\scriptstyle 3$};
\node at (3.5,1.5){$\scriptstyle 0$};
\node at (1.5+2,0.5+2){$\scriptstyle 1$};
\node at (2.5+2,0.5+2){$\scriptstyle 2$};
\node at (2.5+2,1.5+2){$\scriptstyle 3$};
\node at (3.5+2,1.5+2){$\scriptstyle 0$};
\draw[thick,  gray, dotted] (2,0)--(2,1);
\draw[thick,  gray, dotted] (3,1)--(3,2);
\draw[thick,  gray, dotted] (4,2)--(4,3);
\draw[thick,  gray, dotted] (5,3)--(5,4);
\draw[thick,  gray, dotted] (2,1)--(3,1);
\draw[thick,  gray, dotted] (3,2)--(4,2);
\draw[thick,  gray, dotted] (4,3)--(5,3);
\end{tikzpicture}
}
\hackcenter{
\begin{tikzpicture}[scale=0.29]
\node[white] at (-0.2,0){{}};
\node[white] at (0.2,0){{}};
\draw[thin, gray,fill=gray!30]  (0,0)--(0,6);
\end{tikzpicture}
}
\hackcenter{
\begin{tikzpicture}[scale=0.29]
\draw[thick,fill=gray!30]  (0,0)--(0,1)--(1,1)--(1,2)--(2,2)--(2,3)--(3,3)--(3,1)--(2,1)--(2,0)--(0,0);
\node at (0.5,0.5){$\scriptstyle 1$};
\node at (1.5,0.5){$\scriptstyle 2$};
\node at (1.5,1.5){$\scriptstyle 3$};
\node at (2.5,1.5){$\scriptstyle 0$};
\node at (2.5,2.5){$\scriptstyle 1$};
\draw[thick,  gray, dotted] (1,1)--(2,1);
\draw[thick,  gray, dotted] (2,2)--(3,2);
\draw[thick,  gray, dotted] (1,0)--(1,1);
\draw[thick,  gray, dotted] (2,1)--(2,2);
\end{tikzpicture}
}
\right).
\end{align*}
Then \(L(\pi)\) arises as a factor in \(\zS^\tau\) once, and \(\sigma \leq_{\textup{bd}} \pi\) for all other simple factors \(L(\sigma)\) of \(\zS^\tau\).
\end{Example}

\subsection{ArXiv Version}\label{SS:ArxivVersion} We relegate several of the more routine calculations to the \texttt{arXiv} version of the paper.  The reader interested in seeing these additional details can download the \LaTeX~source file from the \texttt{arXiv}.  Near the beginning of the file is a toggle which allows one to compile the paper with these calculations included.

\subsection{Acknowledgments} We thank Matt Fayers and Ben Webster for helpful contributions. In particular we thank Fayers for the idea behind the proof of \cref{bigmultiprop}(iv) which is drastically simpler than our original proof, and Webster for the idea behind the proof of \cref{addribbonlemma}, which we were previously stumped by.
We thank the referee for many helpful comments on a previous version of the paper.
The first author is partially supported by an AMS-Simons PUI Research Enhancement Grant.
The third author is partially supported by JSPS Kakenhi grant number 23K03043.
The fourth author is partially supported by JSPS Kakenhi grant number 23K12964.

\section{Root systems and multipartition combinatorics}
\subsection{Root systems}\label{posrootsec}
Throughout the paper, we now fix some choice of \(e \in \ZZ_{>1}\).
Associated to \(e\) is the affine root system of type \({\tt A}_{e-1}^{(1)}\), corresponding to the affine Dynkin diagram shown in Figure~\ref{fig:dynkin}
(see \cite[\S 4, Table Aff 1]{kac}). This root system plays a crucial role in the representation theory of the Kac--Moody algebra \(\widehat{\mathfrak{sl}}_e(\CC)\) (see \S\ref{slsec}) and its associated quantum group, as well as the modular representation theory of the symmetric group.
We describe the root system directly below.

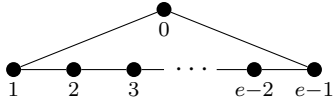
\begin{figure}[h]
\begin{align*}
{}
\hackcenter{
\begin{tikzpicture}[scale=0.4]
%
\coordinate (0) at (5,1);
\coordinate (1) at (0,-1);
\coordinate (2) at (2,-1);
\coordinate (3) at (4,-1);
\coordinate (4) at (6,-1);
\coordinate (5) at (8,-1);
\coordinate (6) at (10,-1);
\coordinate (0a) at (5,0.9);
\coordinate (1a) at (0,-1.1);
\coordinate (2a) at (2,-1.1);
\coordinate (3a) at (4,-1.1);
\coordinate (4a) at (6,-1.1);
\coordinate (5a) at (8,-1.1);
\coordinate (6a) at (10,-1.1);
\draw [thin, black,shorten <= 0.1cm, shorten >= 0.1cm]   (0) to (1);
\draw [thin, black,shorten <= 0.1cm, shorten >= 0.1cm]   (1) to (2);
\draw [thin, black,shorten <= 0.1cm, shorten >= 0.1cm]   (2) to (3);
\draw [thin, black,shorten <= 0.1cm, shorten >= 0.4cm]   (3) to (4);
\draw [thin, black,shorten <= 0.3cm, shorten >= 0.1cm]   (4) to (5);
\draw [thin, black,shorten <= 0.1cm, shorten >= 0.1cm]   (5) to (6);
\draw [thin, black,shorten <= 0.1cm, shorten >= 0.1cm]   (0) to (6);
\draw (0a) node[below]{$\scriptstyle{0}$};
\draw (1a) node[below]{$\scriptstyle{1}$};
\draw (2a) node[below]{$ \scriptstyle{2}$};
\draw (3a) node[below]{$\scriptstyle{3}$};
\draw (5a) node[below]{$ \scriptstyle{e-2}$};
\draw (6a) node[below]{$\scriptstyle{e-1}$};
\blackdot(5,1);
\blackdot(0,-1);
\blackdot(2,-1);
\blackdot(4,-1);
\blackdot(8,-1);
\blackdot(10,-1);
\draw(6,-1) node{$\cdots$};
\end{tikzpicture}
}
\end{align*}
\caption{Dynkin diagram of type \({\tt A}^{(1)}_{e-1}\)}
\label{fig:dynkin}       
\end{figure}

We write \(\ZZ_e := \ZZ/e\ZZ\), and we will indicate elements of \(\ZZ_e\) with barred integers, i.e. \(\overline t = t + e\ZZ\) for \(t \in \ZZ\), freely omitting bars for \(t \in [0,e-1]\) when the context is clear. We will also abuse notation somewhat, writing \(\overline{t} = x\) where \(x \in [0,e-1]\) and \(t \equiv x \pmod{e}\).
Let \(\ZZ I\) be the free \(\ZZ\)-module of rank \(e\), with basis \(I = \{\alpha_i \mid i \in \ZZ_e\}\). For \(\beta = \sum_{i \in \Z_e} c_i\alpha_i \in \ZZ_{\geq 0}I\), we write \(\height(\beta) := \sum_{i \in \Z_e} c_i \in \ZZ_{\geq 0}\) for the {\em height} of \(\beta\).
For any \(t \in \ZZ\), \(L \in \ZZ_{>0}\), we define \(\alpha(t,L) \in \ZZ_{\geq 0}I\) via
\begin{align*}
\alpha(t,L):= \alpha_{\bar t} + \alpha_{\overline{t+1}} + \dots + \alpha_{\overline{t+L-1}}.
\end{align*}
We have then that \(\height(\alpha(t,L)) = L\). The element \(\alpha(t,L)\) corresponds to the sum of simple roots in a anticlockwise path through \(L\) vertices in the Dynkin diagram in Figure~\ref{fig:dynkin} which begins at the vertex labeled \(\bar t\). Of particular importance is the {\em null root} \(\delta\) of height \(e\):
\begin{align*}
\delta := \alpha_{ 0} + \dots + \alpha_{e-1} = \alpha(t,e)\qquad (\textup{any }t \in \ZZ_e).
\end{align*}
For \(\beta, \beta' \in \ZZ_{\geq 0}I\), we will write \(\beta \subseteq \beta'\) provided that \(\beta' - \beta \in \ZZ_{\geq 0}I\).

We call an element of the form \(\bi = i_1 \dots i_m \in I^m\) a {\em word}, and for \(\omega \in \ZZ_{\geq 0}I\), we set
\begin{align*}
I^{\omega}:= \{ \bi \in I^{\height(\omega)} \mid i_1 + \dots + i_{\height(\omega)} = \omega\}.
\end{align*}
For convenience, we will freely associate the set \(I\) with \([0,e-1]\) when writing words; using \(3112\) to indicate the word \(\alpha_3 \alpha_1 \alpha_1 \alpha_2\), for instance.

\begin{Definition} \(\)
\begin{enumerate}
\item
We say \(\beta \in Q_+\) is a {\em positive root} if \(\beta = \alpha(t,L)\) for some \(t \in \ZZ, L \in \ZZ_{>0}\). We write \(\Phi_+\) for the set of all positive roots, so we have
\begin{align*}
\Phi_+ := \{ \alpha(t,L) \mid t \in \ZZ_e, L \in \ZZ_{>0}\} \subset \ZZ I.
\end{align*}
\item We say \(\beta \in \Phi_+\) is {\em real} if \(\overline{\height(\beta)} \neq \overline 0\). Writing \(\Phi_+^\re\) for the set of all real positive roots, we have
\begin{align*}
\Phi_+^\re :=  \{ \alpha(t,L) \mid t \in \ZZ_e, L \in \ZZ_{>0}, \overline L \neq \overline 0\} \subset \Phi_+.
\end{align*}
\item We say \(\beta \in \Phi_+\) is {\em imaginary} if \(\overline{\height(\beta)} = \overline 0\).  Writing \(\Phi_+^\im\) for the set of imaginary positive roots, we have
\begin{align}\label{allimag}
\Phi_+^\im :=  \{ \alpha(t,L) \mid t \in \ZZ_e, L \in \ZZ_{>0}, \overline L = \overline 0\} = \{m \delta \mid m \in\ZZ_{>0} \} \subset \Phi_+.
\end{align}
\item We say a positive root \(\beta\) is {\em divisible} if there exists \(\beta' \in \Phi_+\), \(m \in \ZZ_{>1}\) such that \(\beta= m \beta'\), and {\em indivisible} if not. Writing \(\Psi\) for the set of indivisible roots, we have
\begin{align}\label{PsiDef}
\Psi := \Phi_+^\re \sqcup \{ \delta\}.
\end{align}
\item We write \(I^\fin = \{\alpha_i \mid i \in \ZZ_e \backslash \{0\}\}\) and \(\Phi^\fin \subseteq \ZZ I^\fin\) for the underlying roots for finite type \({\tt A}_{e-1}\):
\begin{align*}
\Phi^\fin := \{ \pm \alpha(t, L) \mid t \in [1, e-1], L \in [ 1   , e-t]\}
\end{align*}
We write \(p:\Z I \to \Z I^\fin \cong \Z I/ \Z \delta\) for the \(\Z\)-module map given by modding out by \(\delta\), noting that \(p(\Phi_+) = \Phi^\fin \cup \{0\}\).
Specifically, \(p\) maps \(\alpha_0\) to \(-(\alpha_1 + \dots + \alpha_{e-1})\) and all other \(\alpha_i\) to \(\alpha_i\).
\end{enumerate}
\end{Definition}
We will assume the reader has a basic familiarity with the combinatorics of finite type \({\tt A}\) root systems, such as outlined for instance in \cite[\S III]{humLie}.

\subsection{Multipartition combinatorics}\label{subsec:mptns}
\subsubsection{Level one combinatorics}\label{levonedefs}
Let \(\N = \ZZ \times \ZZ\). We refer to elements of \(\N\) as {\em nodes}, and by convention, we visually represent nodes as boxes in a \(\ZZ \times \ZZ\) array, so that the node \((x,y)\) is a box in the \(x\)th row and \(y\)th column of the array. In this orientation, a positive increase in the \(x\) component corresponds to a southward move, and a positive increase in the \(y\) component corresponds to an eastward move. We write
\begin{align*}
(x,y) \searrow (x',y') \textup{ provided } x'\geq x \textup{ and } y' \geq y.
\end{align*}
We define the single-unit north, east, south and west translations of nodes, respectively, by setting 
\begin{align*}
{\tt N}(x,y) := (x-1,y), \qquad {\tt E}(x,y) := (x,y+1),\qquad {\tt S}(x,y) :=(x+1,y),\qquad {\tt W}(x,y)  := (x,y-1).
\end{align*}
We define a {\em residue} function \(\textup{res}: \N \to \ZZ_e\) on nodes by setting:
\begin{align*}
\textup{res}((x,y)) := \overline{y-x}.
\end{align*}

A {\em skew diagram} is a finite subset \(\tau \subseteq \N\) such that \((x,y) \searrow (x',y') \searrow (x'',y'')\) and \((x,y), (x'',y'') \in \tau\) implies \((x',y') \in \tau\). 

A {\em partition \(\lambda\) of {\em charge} \(\kappa \in \ZZ\)} is a skew diagram which is either empty, or contains the {\em corner} node \((1,\kappa+1)\), such that \((1,\kappa+1) \searrow (x,y)\) for all \((x,y) \in \lambda\). Pictorially, \(\lambda\) is a northwest-aligned array of boxes (often called a {\em Young diagram}) with the node \((1,\kappa + 1)\) at the northwest corner. Partitions of a given charge are in bijection with integer partitions; writing
\begin{align*}
\lambda_a := |\{(a,x) \mid x \in \ZZ_{\geq 0}\} \cap \lambda|
\end{align*}
for all \(a \in \mathbb{Z}_{>0}\), we have that \(\lambda_1 \geq \lambda_2 \geq \cdots\), and \(\sum \lambda_a = |\lambda|\). Thus we will often abuse notation and write \(\lambda = (\lambda_1, \lambda_2, \ldots)\). We collect repeated parts in partitions, writing $(4,2^3,1^2)$ as shorthand for the partition $(4,2,2,2,1,1)$, for example. We remark that the empty set \(\varnothing\) is a partition of charge \(\kappa\) for any \(\kappa \in \ZZ\). If we refer to a partition without mentioning a specific charge, we will take the charge to be \(0\) by default. 

For partitions \(\mu \subseteq \lambda\) of charge \(\kappa\), we set \(\lambda/\mu \subseteq \N\) to be the set difference \(\lambda \backslash \mu\), and call this a {\em skew partition of charge \(\kappa\)}. We remark that \(\lambda/\mu\) is a skew diagram, and consider the choice of \(\lambda\) and \(\mu\) to be part of the data of \(\lambda/\mu\). Every skew diagram can be realized as (some translation of) a skew partition. 

\subsubsection{Higher level combinatorics}\label{highlevdefs}

More generally, for a {\em level} \(\ell \in \mathbb{Z}_{>0}\), we write 
\begin{align*}
\N_\ell := \bigsqcup_{t \in [1,\ell]} \N = \N^{(1)} \sqcup \dots \sqcup \N^{(\ell)},
\end{align*}
labeling subsets and nodes in the constituent copies of \(\N\) in \(\N_\ell\) with parenthesized superscripts, so that \((x,y)^{(r)} \in \N^{(r)}\).

We extend the definitions above to higher levels as follows. We will say that \(\btau = (\tau^{(1)} \mid \cdots \mid \tau^{(\ell)})\) is a {\em skew \(\ell\)-diagram} provided that each component \(\tau^{(i)}\) is a skew diagram.

For a {\em multicharge} \(\bkap = (\kappa_1 \mid \cdots \mid \kappa_\ell) \in \ZZ^{\ell}\), we say that \(\blam = (\lambda^{(1)} \mid \dots \mid \lambda^{(\ell)})\) is an {\em \(\ell\)-multipartition of multicharge \(\bkap\)} provided that each component \(\lambda^{(i)}\) is a partition of charge \(\kappa_i\). We extend this to define {\em skew \(\ell\)-multipartitions of multicharge \(\bkap\)} in the obvious way. We will often omit the `\(\ell\)-' in our nomenclature, when the level is irrelevant or clear from context.
We will visually depict skew diagrams, multipartitions, and skew multipartitions in rows, with component labels increasing from left to right (see Figure~\ref{figmult1}).

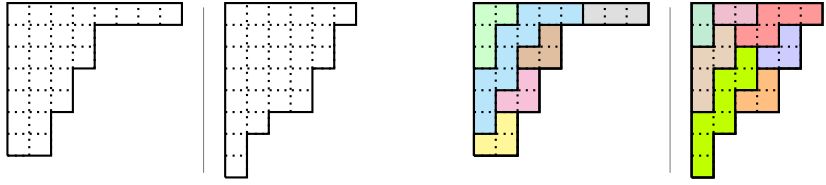
\begin{figure}
\begin{align*}
	{}
		\hackcenter{
			\begin{tikzpicture}[scale=0.29]
				\draw[thick, dotted] (0,1)--(1,1);
				\draw[thick, dotted] (0,2)--(2,2);
				\draw[thick, dotted] (0,3)--(2,3);
				\draw[thick, dotted] (0,4)--(3,4);
				\draw[thick, dotted] (0,5)--(3,5);
				\draw[thick, dotted] (0,6)--(4,6);
				\draw[thick, dotted] (0,7)--(4,7);
				\draw[thick, dotted] (1,1)--(1,8);
				\draw[thick, dotted] (2,1)--(2,8);
				\draw[thick, dotted] (3,5)--(3,8);
				\draw[thick, dotted] (4,5)--(4,8);
				\draw[thick, dotted] (5,7)--(5,8);
				\draw[thick, dotted] (6,7)--(6,8);
				\draw[thick, dotted] (7,7)--(7,8);
				\draw[thick]  (0,1)--(2,1)--(2,3)--(3,3)--(3,5)--(4,5)--(4,7)--(8,7)--(8,8)--(0,8)--(0,1);
				\phantom{\draw[thick,fill=blue] (0,0)--(0,1)--(1,1)--(1,0)--(0,0);}
			\end{tikzpicture}
		}
		\hackcenter{
			\begin{tikzpicture}[scale=0.29]
				\node[white] at (-0.5,0){{}};
				\node[white] at (0.5,0){{}};
				\draw[thin, gray,fill=gray!30]  (0,0)--(0,7.6);
				\draw[thin, white]  (0,7.6)--(0,8);
			\end{tikzpicture}
		}
		\hackcenter{
			\begin{tikzpicture}[scale=0.29]
				\draw[thick, dotted] (0,1)--(1,1);
				\draw[thick, dotted] (0,2)--(2,2);
				\draw[thick, dotted] (0,3)--(2,3);
				\draw[thick, dotted] (0,4)--(4,4);
				\draw[thick, dotted] (0,5)--(5,5);
				\draw[thick, dotted] (0,6)--(5,6);
				\draw[thick, dotted] (0,7)--(5,7);
				\draw[thick, dotted] (1,0)--(1,8);
				\draw[thick, dotted] (2,2)--(2,8);
				\draw[thick, dotted] (3,3)--(3,8);
				\draw[thick, dotted] (4,5)--(4,8);
				\draw[thick, dotted] (5,5)--(5,8);
				\draw[thick]  (0,0)--(1,0)--(1,2)--(2,2)--(2,3)--(4,3)--(4,5)--(5,5)--(5,7)--(6,7)--(6,8)--(0,8)--(0,0);
			\end{tikzpicture}
		}
	\qquad
	\qquad
		\hackcenter{
			\begin{tikzpicture}[scale=0.29]
				\draw[thick,fill=cyan!25]  (0,1)--(2,1)--(2,3)--(3,3)--(3,5)--(4,5)--(4,7)--(8,7)--(8,8)--(0,8)--(0,1);
				\draw[thick, fill=yellow!50] (0,1)--(2,1)--(2,3)--(1,3)--(1,2)--(0,2)--(0,1);
				\draw[thick, fill=magenta!30] (1,3)--(3,3)--(3,5)--(2,5)--(2,4)--(1,4)--(1,3);
				\draw[thick, fill=brown!50] (2,5)--(4,5)--(4,7)--(3,7)--(3,6)--(2,6)--(2,5);
				\draw[thick, fill=green!20] (0,5)--(1,5)--(1,7)--(2,7)--(2,8)--(0,8)--(0,5);
				\draw[thick,  fill=lightgray!50] (5,7)--(8,7)--(8,8)--(5,8)--(5,7);
				\draw[thick, dotted] (0,1)--(1,1);
				\draw[thick, dotted] (0,2)--(2,2);
				\draw[thick, dotted] (0,3)--(2,3);
				\draw[thick, dotted] (0,4)--(3,4);
				\draw[thick, dotted] (0,5)--(3,5);
				\draw[thick, dotted] (0,6)--(4,6);
				\draw[thick, dotted] (0,7)--(4,7);
				\draw[thick, dotted] (1,1)--(1,8);
				\draw[thick, dotted] (2,1)--(2,8);
				\draw[thick, dotted] (3,5)--(3,8);
				\draw[thick, dotted] (4,5)--(4,8);
				\draw[thick, dotted] (5,7)--(5,8);
				\draw[thick, dotted] (6,7)--(6,8);
				\draw[thick, dotted] (7,7)--(7,8);
				\draw[thick]  (0,1)--(2,1)--(2,3)--(3,3)--(3,5)--(4,5)--(4,7)--(8,7)--(8,8)--(0,8)--(0,1);
				\phantom{\draw[thick,fill=blue] (0,0)--(0,1)--(1,1)--(1,0)--(0,0);}
			\end{tikzpicture}
		}
		\hackcenter{
			\begin{tikzpicture}[scale=0.29]
				\node[white] at (-0.5,0){{}};
				\node[white] at (0.5,0){{}};
				\draw[thin, gray,fill=gray!30]  (0,0)--(0,7.6);
				\draw[thin, white]  (0,7.6)--(0,8);
			\end{tikzpicture}
		}
		\hackcenter{
			\begin{tikzpicture}[scale=0.29]
				\draw[thick, fill=red!40]  (0,0)--(1,0)--(1,2)--(2,2)--(2,3)--(4,3)--(4,5)--(5,5)--(5,7)--(6,7)--(6,8)--(0,8)--(0,0);
				\draw[thick, fill=lime]  (0,0)--(1,0)--(1,2)--(2,2)--(2,4)--(3,4)--(3,6)--(2,6)--(2,5)--(1,5)--(1,3)--(0,3)--(0,0);
				\draw[thick, fill=orange!50] (2,3)--(4,3)--(4,5)--(3,5)--(3,4)--(2,4)--(2,3);
				\draw[thick, fill=blue!20] (3,5)--(5,5)--(5,7)--(4,7)--(4,6)--(3,6)--(3,5);
				\draw[thick,  fill=brown!35] (0,3)--(1,3)--(1,5)--(2,5)--(2,7)--(1,7)--(1,6)--(0,6)--(0,3);
				\draw[thick, fill=blue!30!green!25] (0,6)--(1,6)--(1,8)--(0,8)--(0,6);
				\draw[thick, fill=purple!25] (1,7)--(3,7)--(3,8)--(1,8)--(1,7);
				\draw[thick, dotted] (0,1)--(1,1);
				\draw[thick, dotted] (0,2)--(2,2);
				\draw[thick, dotted] (0,3)--(2,3);
				\draw[thick, dotted] (0,4)--(4,4);
				\draw[thick, dotted] (0,5)--(5,5);
				\draw[thick, dotted] (0,6)--(5,6);
				\draw[thick, dotted] (0,7)--(5,7);
				\draw[thick, dotted] (1,0)--(1,8);
				\draw[thick, dotted] (2,2)--(2,8);
				\draw[thick, dotted] (3,3)--(3,8);
				\draw[thick, dotted] (4,5)--(4,8);
				\draw[thick, dotted] (5,5)--(5,8);
				\draw[thick]  (0,0)--(1,0)--(1,2)--(2,2)--(2,3)--(4,3)--(4,5)--(5,5)--(5,7)--(6,7)--(6,8)--(0,8)--(0,0);
			\end{tikzpicture}
		}
\end{align*}
\caption{On the left, a diagram corresponding to the 2-multipartition \(\blam = ((8,4^2,3^2,2^2)\mid(6,5^2,4^2,2,1^2))\). On the right, a ribbon tiling of \(\blam\).
}
\label{figmult1}      
\end{figure}

We denote by \(\trianglelefteq^D\) the usual dominance order on \(\ell\)-multipartitions:
\begin{align*}
\blam \trianglelefteq^D \bmu \qquad \iff \qquad \sum_{s=1}^{r-1} |\lambda^{(s)}| + \sum_{a=1}^t \lambda_a^{(r)} \leq \sum_{s=1}^{r-1} |\mu^{(s)}| + \sum_{a=1}^t \mu_a^{(r)}\qquad (\textup{for all } r \in [1,\ell],t \in \mathbb{Z}_{>0})
\end{align*}
so \(\blam \trianglelefteq^D \bmu\) indicates that \(\bmu\) may be obtained from \(\blam\) by progressively `moving nodes northeasterly or to an earlier component'.

For \(u = (x,y)^{(r)} \in \N^{(r)} \subseteq \N_\ell\) with \(x,y>0\), we set \(\textup{rect}_u = \{v \in \N^{(r)} \mid (1,1)^{(r)} \searrow v \searrow u\} \subseteq \N_\ell\); that is, \(\textup{rect}_u\) is the \(\ell\)-multipartition of multicharge \(0^\ell\) which is empty in every component save the one containing \(u\), and in that component it is a rectangular partition whose southeasternmost node is \(u\).

If \(\bxi, \bxi' \subseteq \N_\ell\) we write \(\bxi \NEarrow \bxi'\) if for every \(u' \in \bxi'\) and every \(u\in\bxi\), \(u'\) is either in an earlier component than \(u\), or is in the same component and is strictly northeast of \(u\).

\subsubsection{Ribbons} We say that a nonempty skew diagram \(\xi = \xi^{(r)} \subseteq \N^{(r)} \subseteq \N_\ell \) is a {\em ribbon} provided that \(\xi\) may be written 
\begin{align*}
\xi = \{(x_1,y_1)^{(r)}, \dots, (x_{|\xi|}, y_{|\xi|})^{(r)}\}, \textup{ with } (x_{a+1}, y_{a+1})^{(r)} \in \{(x_a -1, y_a)^{(r)}, (x_a, y_a +1)^{(r)}\}
\end{align*}
for all \(a \in [1,|\xi|-1]\). In other words, \(\xi\) is a nonempty connected skew diagram, contained within one component of \(\N_\ell\), containing no \(2 \times 2\) squares.

\subsection{Tilings and tableaux}\label{tiletabsec} Let \(\btau \subseteq \N_\ell\) be a skew diagram. We say that a set \(\Gamma\) of disjoint skew diagrams in \(\N_\ell\) is a {\em skew tiling} of \(\btau\) provided that \(\btau = \bigsqcup_{\bzeta \in \Gamma} \bzeta\). We refer to the constituent members of \(\Gamma\) as {\em tiles}. A {\em skew tableau} of \(\btau\) is the data \((\Gamma, {\tt t})\) of a skew tiling \(\Gamma\) of \(\btau\), together with a bijective ordering \({\tt t}: [1,|\Gamma|] \to \Gamma\), such that for all \(r \in [1,\ell]\), \(a,b \in [1,  |\Gamma|]\), we have
\begin{align*}
{\tt t}(a) \ni (x,y)^{(r)} \searrow (x',y')^{(r)}\in {\tt t}(b) \implies a \leq b.
\end{align*}
We may indicate the data of a skew tableau with the ordered tuple \(({\tt t}(1) , \dots , {\tt t}(|\Gamma|)) \in (\N_\ell)^{|\Gamma|}\). 

We say that a skew diagram \(\bmu\) is {\em {\tt SE}-removable} (resp. {\em {\tt NW}-removable}) for \(\btau\) provided that there exists a 2-part skew tableau \(({\tt t}(1), {\tt t}(2))\) for \(\btau\) such that \({\tt t}(2) = \bmu\) (resp. \({\tt t}(1) = \bmu\)). We say that \(\bmu\) is {\em {\tt SE}-addable} (resp. {\tt NW}-addable) for \(\btau\) provided that \(\bmu \cap \btau = \varnothing\) and \(\bmu \sqcup \btau\) is a skew diagram in which \(\bmu\) is {\tt SE}-removable (resp. {\tt NW}-removable).

A {\em ribbon tiling} (resp.~{\em ribbon tableau}) is a skew tiling (resp.~tableau) in which all tiles are ribbons (see Figure~\ref{figmult1} for an example). 
A {\em standard tableau} for \(\btau\) is a skew tableau in which all tiles consist of single nodes. Note that in this setting, we have \(\Gamma = \btau\), so we may omit the data of \(\Gamma\) and refer to \({\tt t}\) itself as the standard tableau. We write \(\Std(\btau)\) for the set of all standard tableaux of \(\btau\).
We write \({\tt t}^{\btau} \in \Std(\btau)\) for the {\em (row)-leading standard tableau} of \(\btau\); this is the standard tableau defined by stipulating that whenever \({\tt t}^{\btau}(a) = (x,y)^{(r)}\), \({\tt t}^{\btau}(b) = (x',y')^{(r')}\) for some \(a<b\), then 
\begin{align*}
r < r', \qquad \textup{or} \qquad r=r', x<x', \qquad \textup{or} \qquad r=r',x=x', y<y'.
\end{align*}
In other words, \({\tt t}^{\btau}\) labels the nodes of \(\btau\) left-to-right along the first row, the second row, and so on, of the first component, and then moving in like fashion through the second component, and so on.

If \({\tt T}\) is a standard tableau for a multipartition \(\bmu\), and \({\tt t}\) is a standard tableau for a skew multipartition \(\blam/\bmu\), then we define the standard tableau \({\tt Tt}\) for \(\blam\) by setting
\begin{align*}
{\tt Tt}(a) := 
\begin{cases}
{\tt T}(a) & \textup{if }a \in [1, |\bmu|];\\
{\tt t}(a-|\bmu|) & \textup{if }a \in [|\bmu| + 1, |\blam|],
\end{cases}
\end{align*}
for all \(a \in [1, |\blam|]\), i.e.~\({\tt Tt}\) first labels all the nodes of \(\bmu\) in the same order as \({\tt T}\), and then continues on, labelling all the nodes of \(\blam/\bmu\) in the same order as \({\tt t}\).

If \(\blam \subseteq \N_\ell\) is an \(\ell\)-multipartition, \({\tt T} \in \Std(\blam)\) and \(1 \leq n \leq |\blam|\), we write 
\begin{align*}
\textup{sh}^\downarrow_n({\tt T}) := \{{\tt T}(x) \mid x \in [1,n]\} \subseteq \N_{\ell}
\end{align*}
for the \(\ell\)-multipartition formed by the first \(n\) nodes in \({\tt T}\).

\subsubsection{\(\mathfrak{S}_n\)-action on tableaux}\label{subsubsec:Snaction}
Let \(\btau \in \N_\ell\) be a skew diagram, and \(|\btau| = n\). Let \(\textup{Tab}(\btau)\) be the set of bijections \({\tt t}:[1,n] \xrightarrow{\sim} \btau \), so that \(\Std(\btau) \subseteq\textup{Tab}(\btau)\). 
There is a left action of the symmetric group \(\mathfrak{S}_n\) on \(\textup{Tab}(\btau)\) given by \((\sigma {\tt t})(x) := {\tt t}(\sigma^{-1}x)\), for all \(x \in [1,n]\), \(\sigma \in \mathfrak{S}_n\), \({\tt t} \in \textup{Tab}(\btau)\). 
For \({\tt u} \in \textup{Tab}(\btau)\), we let \(w^{\tt u}\) denote the unique element in \(\mathfrak{S}_n\) such that \(w^{\tt u} {\tt t}^{\btau} = {\tt u}\).

\subsection{Content}\label{subsec:content}
For a finite set \(S \subseteq \N_\ell\), the {\em content} of \(S\) is defined by:
\begin{align*}
\cont(S): = \sum_{u \in S} \alpha_{\res(u)} \in \ZZ_{\geq 0}I.
\end{align*}
We note that when \(\xi\) is a ribbon, we have \(\cont(\xi) \in \Phi_+\).

Let \(\bkap = (\kappa_1 \mid \dots \mid \kappa_\ell) \in \ZZ^\ell\) be a multicharge of level \(\ell\).  For a fixed multipartition \(\brho \in \N_{\ell}\), we write
\begin{align*}
\Lambda^\ell(\omega) &:= \{ \textup{skew diagram } \btau \subseteq \N_\ell \mid \cont(\btau) = \omega\},\\
\Lambda_+^{\bkap}(\omega)&:= \{ \textup{multipartition } \blam \subset \N_\ell  \textup{ of multicharge } \bkap \mid \cont(\blam) = \omega\} \subseteq \Lambda^\ell(\omega),\\
\Lambda_{+/+}^{\bkap}(\omega) &:= \{ \textup{skew multipartition } \btau \subset \N_{\ell}  \textup{ of multicharge } \bkap \mid \cont(\btau) = \omega\} \supseteq \Lambda_+^{\bkap}(\omega),\\
\Lambda_{+/\brho}^{\bkap}(\omega) &:= \{ \textup{skew multipartition } \blam/\brho \subset \N_{\ell}  \textup{ of multicharge } \bkap \mid \cont(\blam/\brho) = \omega\} \subseteq \Lambda_{+/+}^{\bkap}(\omega).
\end{align*}
We also write \(\Lambda^{\bkap}_+ := \bigsqcup_{\omega \in \ZZ_{\geq 0}I} \Lambda_+^{\bkap}(\omega)\), and so forth. We call \(\Lambda_+^{\bkap}(\omega)\) a {\em block}. Multipartitions in these blocks correspond to cells of cyclotomic KLR algebras of type \({\tt A}^{(1)}_{e-1}\), or, equivalently, Ariki--Koike algebra blocks, justifying our use of the term, see \S\ref{subsec:cycKLR}.

\subsubsection{Similarity}\label{similaritysec}
Let \(\btau \in \Lambda^\ell(\omega)\) be a nonempty skew diagram. Assume that \(\btau\) consists of the nonempty connected components \(\btau = \phi_1 \sqcup \dots \sqcup \phi_m\), with
\begin{align*}
\phi_m \NEarrow \dots \NEarrow \phi_1.
\end{align*}
Then we define the {\em distillation} of \(\btau\) to be the skew diagram
\begin{align*}
\textup{dist}(\btau) :=  ( \textup{dist}(\btau)^{(1)} \mid \dots \mid  \textup{dist}(\btau)^{(m)}) = (\phi_1 \mid \dots \mid \phi_m)  \in \Lambda^m(\omega).
\end{align*}
For \(\btau \in \Lambda^\ell(\omega)\) and \(\btau' \in \Lambda^{\ell'}(\omega)\), we write \(\btau \sim \btau'\) provided that \(\textup{dist}(\btau), \textup{dist}(\btau') \in \Lambda^m(\omega)\) for some \(m\), and \(\textup{dist}(\btau')^{(i)}\) is a residue-preserving translation of \(\textup{dist}(\btau)^{(i)}\) for all \(i \in [1,m]\).

\subsubsection{Standard tableau words}\label{subsubsec:tabwords}
For \(\btau \in \Lambda^\ell(\omega)\), \({\tt t} \in \Std(\btau)\), we define the associated {\em tableau word} (or {\em residue sequence}, or {\em content sequence}) of \({\tt t}\):
\begin{align*}
\bi^{\tt t}:= \cont({\tt t}(1)) \dots \cont({\tt t}(|\btau|)) \in I^\omega.
\end{align*}
If \(\bi = \bi^{\tt t}\) for some \(\tt t \in \Std(\btau)\), then we say that {\em \(\bi\) is a word in \(\btau\)}.
We moreover define \(\bi^{\btau}:= \bi^{{\tt t}^{\btau}}\) for the word corresponding to the leading tableau \({\tt t }^{\btau}\) of \(\btau\).

\subsection{Separability}

\begin{Definition}\label{kapsepdef}
For a multicharge \(\bkap\), 
we say that the pair \((\omega, \beta) \in (\ZZ_{\geq 0}I)^2\) is \emph{\(\bkap\)-separable} provided that for every \(\brho \in \Lambda^{\bkap}_+(\omega)\) and \(0 \subsetneq \beta' \subseteq \beta\), there do not exist \(\bmu \subseteq \brho \subseteq \bnu \in \Lambda^{\bkap}_+\) such that \(\cont(\brho/ \bmu) = \cont(\bnu/\brho) = \beta'\). 
\end{Definition}

\begin{Example}
Let $e=4$ and $\bkap=(0,0)$. There are two bipartitions lying in $\Lambda^{\bkap}_+(\omega)$ with $\omega=2\alpha_0+2\alpha_1+\alpha_3$, namely $\brho=((2,1)\mid(2))$ and $\brho'=((2)\mid(2,1))$. Observe that $(\omega,\alpha_0+\alpha_1)$ is $\bkap$-separable, but $(\omega,\alpha_2+\alpha_3)$ is not. For instance,
\[
\bmu=
\
\hackcenter{
	\begin{tikzpicture}[scale=0.29]
		\draw[thick, dotted] (1,0)--(1,1);
		\draw[thick]  (0,0)--(2,0)--(2,1)--(0,1)--(0,0);
		\node at (0.5,0.5){$\scriptstyle 0$};
		\node at (1.5,0.5){$\scriptstyle 1$};
		\phantom{\node at (0.5,-0.5){$\scriptstyle 3$};}
	\end{tikzpicture}
}
\hackcenter{
	\begin{tikzpicture}[scale=0.29]
		\node[white] at (-0.2,0){{}};
		\node[white] at (0.2,0){{}};
		\draw[thin, gray,fill=gray!30]  (0,0)--(0,1.8);
		\draw[thin, white]  (0,1.8)--(0,2);
	\end{tikzpicture}
}
\hackcenter{
	\begin{tikzpicture}[scale=0.29]
		\draw[thick, dotted] (1,0)--(1,1);
		\draw[thick]  (0,0)--(2,0)--(2,1)--(0,1)--(0,0);
		\node at (0.5,0.5){$\scriptstyle 0$};
		\node at (1.5,0.5){$\scriptstyle 1$};
		\phantom{\node at (0.5,-0.5){$\scriptstyle 3$};}
	\end{tikzpicture}
}
\
\subset
\
\hackcenter{
	\begin{tikzpicture}[scale=0.29]
		\fill[fill=green!30](0,3)--(1,3)--(1,2)--(0,2)--(0,3);
		\draw[thick, dotted] (1,3)--(1,4);
		\draw[thick, dotted] (0,3)--(1,3);
		\draw[thick]  (0,2)--(1,2)--(1,3)--(2,3)--(2,4)--(0,4)--(0,2);
		\node at (0.5,3.5){$\scriptstyle 0$};
		\node at (1.5,3.5){$\scriptstyle 1$};
		\node at (0.5,2.5){$\scriptstyle 3$};
	\end{tikzpicture}
}
\hackcenter{
	\begin{tikzpicture}[scale=0.29]
		\node[white] at (-0.2,0){{}};
		\node[white] at (0.2,0){{}};
		\draw[thin, gray,fill=gray!30]  (0,0)--(0,1.8);
		\draw[thin, white]  (0,1.8)--(0,2);
	\end{tikzpicture}
}
\hackcenter{
	\begin{tikzpicture}[scale=0.29]
		\draw[thick, dotted] (1,0)--(1,1);
		\draw[thick]  (0,0)--(2,0)--(2,1)--(0,1)--(0,0);
		\node at (0.5,0.5){$\scriptstyle 0$};
		\node at (1.5,0.5){$\scriptstyle 1$};
		\phantom{\node at (0.5,-0.5){$\scriptstyle 3$};}
	\end{tikzpicture}
}
\
\subset
\
\hackcenter{
	\begin{tikzpicture}[scale=0.29]
		\fill[fill=green!30](0,3)--(1,3)--(1,2)--(0,2)--(0,3);
		\draw[thick, dotted] (1,3)--(1,4);
		\draw[thick, dotted] (0,3)--(1,3);
		\draw[thick]  (0,2)--(1,2)--(1,3)--(2,3)--(2,4)--(0,4)--(0,2);
		\node at (0.5,3.5){$\scriptstyle 0$};
		\node at (1.5,3.5){$\scriptstyle 1$};
		\node at (0.5,2.5){$\scriptstyle 3$};
	\end{tikzpicture}
}
\hackcenter{
	\begin{tikzpicture}[scale=0.29]
		\node[white] at (-0.2,0){{}};
		\node[white] at (0.2,0){{}};
		\draw[thin, gray,fill=gray!30]  (0,0)--(0,1.8);
		\draw[thin, white]  (0,1.8)--(0,2);
	\end{tikzpicture}
}
\hackcenter{
	\begin{tikzpicture}[scale=0.29]
		\fill[fill=green!30](0,3)--(1,3)--(1,2)--(0,2)--(0,3);
		\draw[thick, dotted] (1,3)--(1,4);
		\draw[thick, dotted] (0,3)--(1,3);
		\draw[thick]  (0,2)--(1,2)--(1,3)--(2,3)--(2,4)--(0,4)--(0,2);
		\node at (0.5,3.5){$\scriptstyle 0$};
		\node at (1.5,3.5){$\scriptstyle 1$};
		\node at (0.5,2.5){$\scriptstyle 3$};
	\end{tikzpicture}
}
\
=\bnu,
\]
where $\cont(\brho/\bmu)=\cont(\bnu/\brho)=\alpha_3$.
\end{Example}

\subsection{Defect}\label{subsec:def} 
Let \(\bkap = (\kappa_1 \mid \cdots \mid \kappa_\ell)\) be a multicharge of level \(\ell\).
Extrapolating slightly from \cite[\S2.1]{fay06wts} we define the {\em defect} \(\textup{def}_{\bkap}(\btau)\) for a skew diagram \(\btau \in \Lambda^{\ell}\) as follows, setting:
\begin{align*}
c_i(\btau) &:= \#\{u \in \btau \mid \res(u) = i\} \qquad (\textup{for }i \in \ZZ_e)
\\
\textup{def}_{\bkap}(\btau) &:= \sum_{r=1}^\ell c_{\overline{\kappa}_r}(\btau)
-
\frac{1}{2}\sum_{i \in \ZZ_e}  \left(c_i(\btau) - c_{i+1}(\btau)\right)^2.
\end{align*}
When \(\btau\) is a multipartition, it is shown in \cite[Corollary 3.9]{fay06wts} that \(\textup{def}_{\bkap}(\btau) \geq 0\), but this is not necessarily the case for skew diagrams in general. 

\subsection{Beta numbers}\label{cdefs1} A useful way to describe a multipartition \(\blam \in \Lambda^{\bkap}_+\) is in terms of its {\em \(\bkap\)-beta numbers}.
\begin{Definition}
Let \(\blam = (\lambda^{(1)}\mid \dots\mid \lambda^{(\ell)}) \in \Lambda^{\bkap}_+\). 
\begin{enumerate}
\item The {\em \(\bkap\)-beta numbers} for \(\blam\) with multicharge \(\bkap\) 
are defined by setting
\begin{align*}
\B(\blam, \bkap) := (\B^1(\blam, \bkap) \mid \B^2(\blam, \bkap) \mid \dots \mid \B^\ell(\blam, \bkap)),
\end{align*}
where
\begin{align*}
\B^r(\blam, \bkap) := \{ \overline{\kappa_r} + \lambda_a^{(r)} - a  \mid a \in \mathbb{Z}_{>0}\} \subset \Z.
\end{align*}
\item For \(r \in [1,\ell]\), \(i,j\in [0, e-1]\), we set:
\begin{align*}
\B^r_i(\blam, \bkap) &:= \{x \in \B^r(\blam, \bkap) \mid \bar x =  i\},\\
M^r_i(\blam,\bkap) &:= \max\{ b \in \B_i^r(\blam,\bkap)\},\\
h^r_{i,j}(\blam, \bkap) &:= \left\lfloor  \frac{M^r_{j}(\blam, \bkap) - M^r_i(\blam,\bkap)}{e}\right\rfloor.
\end{align*}\end{enumerate}
\end{Definition}

When it is clear from context, we will occasionally omit the \(\blam\) and \( \bkap\) from the above notation.
It is possible to recover \(\blam\) from the \(\bkap\)-beta numbers \(\B^r(\blam, \bkap)\) as follows. For all \(r \in [1,\ell]\), define \(m^r_a\) to be the \(a\)th greatest element of \(\B^r(\blam, \bkap)\). Then
\begin{align*}
\lambda_a^{(r)} = | \ZZ_{< m^r_a} \backslash \B^r(\blam, \bkap)|.
\end{align*}

\subsubsection{Bead and abacus representation of \(\bkap\)-beta numbers}
The \(\bkap\)-beta numbers for a multipartition \(\blam\) are most usefully visualized as an array of beads oriented along the number line, see Figure~\ref{figmult2} and \cref{multiex5}.
Readers may also be familiar with abacus configurations of beads, 
as explained, for instance, in \cite[\S2.2]{Lyle22}. 
Although we will not use abacus arguments in proofs in this paper, we will occasionally include additional abacus displays and explanations from this point of view for the benefit of readers accustomed to thinking about abaci.
We define the \emph{$e$-runner abacus} to be $e$ infinite runners from left to right, with marked positions strictly increasing in $\mathbb{Z}$ from left to right and top to bottom.
We set the $0$th position to be recorded on the leftmost runner, so that the runners correspond to the $e$-residues $\overline{0},\overline{1},\dots,\overline{e-1}$, from left to right.
An example of an $e$-runner abacus is given in Figure~\ref{abacus}.
The abacus display for \(\lambda^{(r)}\) is obtained from the set \(\B^r(\blam, \bkap)\) by placing a bead on the $e$-runner abacus for each beta number in \(\B^r(\blam, \bkap)\).
The \emph{abacus display} for \(\blam\) is the $\ell$-tuple of abacus displays for each component \(\lambda^{(r)}\).
An example is given in Figure~\ref{abacus}.


\begin{figure}[h]
\begin{align*}
	\begin{array}{c}
		{}\bla={}
	\hackcenter{
		\begin{tikzpicture}[scale=0.29]
			\draw[thick, dotted] (0,1)--(1,1);
			\draw[thick, dotted] (0,2)--(2,2);
			\draw[thick, dotted] (0,3)--(2,3);
			\draw[thick, dotted] (0,4)--(3,4);
			\draw[thick, dotted] (0,5)--(3,5);
			\draw[thick, dotted] (0,6)--(4,6);
			\draw[thick, dotted] (0,7)--(4,7);
			\draw[thick, dotted] (1,1)--(1,8);
			\draw[thick, dotted] (2,1)--(2,8);
			\draw[thick, dotted] (3,5)--(3,8);
			\draw[thick, dotted] (4,5)--(4,8);
			\draw[thick, dotted] (5,7)--(5,8);
			\draw[thick, dotted] (6,7)--(6,8);
			\draw[thick, dotted] (7,7)--(7,8);
			\node at (0.5,7.5){$\scriptstyle 2$};
			\node at (1.5,7.5){$\scriptstyle 0$};
			\node at (2.5,7.5){$\scriptstyle 1$};
			\node at (3.5,7.5){$\scriptstyle 2$};
			\node at (4.5,7.5){$\scriptstyle 0$};
			\node at (5.5,7.5){$\scriptstyle 1$};
			\node at (6.5,7.5){$\scriptstyle 2$};
			\node at (7.5,7.5){$\scriptstyle 0$};
			\node at (0.5,6.5){$\scriptstyle 1$};
			\node at (1.5,6.5){$\scriptstyle 2$};
			\node at (2.5,6.5){$\scriptstyle 0$};
			\node at (3.5,6.5){$\scriptstyle 1$};
			\node at (0.5,5.5){$\scriptstyle 0$};
			\node at (1.5,5.5){$\scriptstyle 1$};
			\node at (2.5,5.5){$\scriptstyle 2$};
			\node at (3.5,5.5){$\scriptstyle 0$};
			\node at (0.5,4.5){$\scriptstyle 2$};
			\node at (1.5,4.5){$\scriptstyle 0$};
			\node at (2.5,4.5){$\scriptstyle 1$};
			\node at (0.5,3.5){$\scriptstyle 1$};
			\node at (1.5,3.5){$\scriptstyle 2$};
			\node at (2.5,3.5){$\scriptstyle 0$};
			\node at (0.5,2.5){$\scriptstyle 0$};
			\node at (1.5,2.5){$\scriptstyle 1$};
			\node at (0.5,1.5){$\scriptstyle 2$};
			\node at (1.5,1.5){$\scriptstyle 0$};
			\draw[thick]  (0,1)--(2,1)--(2,3)--(3,3)--(3,5)--(4,5)--(4,7)--(8,7)--(8,8)--(0,8)--(0,1);
		\end{tikzpicture}
	}
	\\
	\\
	\\
	\hackcenter{
	\begin{tikzpicture}[scale=0.45]
		\draw[thick, lightgray, dotted] (10,0)--(11,0);
		\draw[thick, black, dotted] (-9.5,0)--(-10.5,0);
		\draw[thick, lightgray ] (-9.5,0)--(10,0);
		\blackdot(9,0);
		\blackdot(4,0);
		\blackdot(3,0);
		\blackdot(1,0);
		\blackdot(0,0);
		\blackdot(-2,0);
		\blackdot(-3,0);
		\blackdot(-6,0);
		\blackdot(-7,0);
		\blackdot(-8,0);
		\blackdot(-9,0);
		\node[below] at (-9,-0.2){$\scriptstyle \textup{-}9$};
		\node[below] at (-8,-0.2){$\scriptstyle \textup{-}8$};
		\node[below] at (-7,-0.2){$\scriptstyle \textup{-}7$};
		\node[below] at (-6,-0.2){$\scriptstyle \textup{-}6$};
		\node[below] at (-5,-0.2){$\scriptstyle \textup{-}5$};
		\node[below] at (-4,-0.2){$\scriptstyle \textup{-}4$};
		\node[below] at (-3,-0.2){$\scriptstyle \textup{-}3$};
		\node[below] at (-2,-0.2){$\scriptstyle \textup{-}2$};
		\node[below] at (-1,-0.2){$\scriptstyle \textup{-}1$};
		\node[below] at (0,-0.2){$\scriptstyle 0$};
		\node[below] at (1,-0.2){$\scriptstyle 1$};
		\node[below] at (2,-0.2){$\scriptstyle 2$};
		\node[below] at (3,-0.2){$\scriptstyle 3$};
		\node[below] at (4,-0.2){$\scriptstyle 4$};
		\node[below] at (5,-0.2){$\scriptstyle 5$};
		\node[below] at (6,-0.2){$\scriptstyle 6$};
		\node[below] at (7,-0.2){$\scriptstyle 7$};
		\node[below] at (8,-0.2){$\scriptstyle 8$};
		\node[below] at (9,-0.2){$\scriptstyle 9$};
		\node[below] at (10,-0.2){$\scriptstyle 10$};
	\end{tikzpicture}
	}
	\end{array}
	\qquad
	\hackcenter{
		\begin{tikzpicture}[scale=0.45]
			\draw[thick, lightgray] (0,0.5)--(0,-8.5);
			\draw[thick, lightgray] (1,0.5)--(1,-8.5);
			\draw[thick, lightgray] (2,0.5)--(2,-8.5);
			\draw[thick, lightgray, dotted] (0,-8.6)--(0,-9.2);
			\draw[thick, lightgray, dotted] (1,-8.6)--(1,-9.2);
			\draw[thick, lightgray, dotted] (2,-8.6)--(2,-9.2);
			\draw[thick, black, dotted] (0,0.5)--(0,1);
			\draw[thick, black, dotted] (1,0.5)--(1,1);
			\draw[thick, black, dotted] (2,0.5)--(2,1);
			\node[above] at (0,1){$\scriptstyle \overline{0}$};
			\node[above] at (1,1){$\scriptstyle \overline{1}$};
			\node[above] at (2,1){$\scriptstyle \overline{2}$};
			\whitedot(0,0);
			\whitedot(0,-1);
			\whitedot(0,-2);
			\whitedot(0,-3);
			\whitedot(0,-4);
			\whitedot(0,-5);
			\whitedot(0,-6);
			\whitedot(0,-7);
			\whitedot(0,-8);
			\node at (0,0){$\scriptstyle -12\phantom{-}$};
			\node at (0,-1){$\scriptstyle -9\phantom{-}$};
			\node at (0,-2){$\scriptstyle -6\phantom{-}$};
			\node at (0,-3){$\scriptstyle -3\phantom{-}$};
			\node at (0,-4){$\scriptstyle 0$};
			\node at (0,-5){$\scriptstyle 3$};
			\node at (0,-6){$\scriptstyle 6$};
			\node at (0,-7){$\scriptstyle 9$};
			\node at (0,-8){$\scriptstyle 12$};
			\whitedot(1,0);
			\whitedot(1,-1);
			\whitedot(1,-2);
			\whitedot(1,-3);
			\whitedot(1,-4);
			\whitedot(1,-5);
			\whitedot(1,-6);
			\whitedot(1,-7);
			\whitedot(1,-8);
			\node at (1,0){$\scriptstyle -11\phantom{-}$};
			\node at (1,-1){$\scriptstyle -8\phantom{-}$};
			\node at (1,-2){$\scriptstyle -5\phantom{-}$};
			\node at (1,-3){$\scriptstyle -2\phantom{-}$};
			\node at (1,-4){$\scriptstyle 1$};
			\node at (1,-5){$\scriptstyle 4$};
			\node at (1,-6){$\scriptstyle 7$};
			\node at (1,-7){$\scriptstyle 10$};
			\node at (1,-8){$\scriptstyle 13$};
			\whitedot(2,0);
			\whitedot(2,-1);
			\whitedot(2,-2);
			\whitedot(2,-3);
			\whitedot(2,-4);
			\whitedot(2,-5);
			\whitedot(2,-6);
			\whitedot(2,-7);
			\whitedot(2,-8);
			\node at (2,0){$\scriptstyle -10\phantom{-}$};
			\node at (2,-1){$\scriptstyle -7\phantom{-}$};
			\node at (2,-2){$\scriptstyle -4\phantom{-}$};
			\node at (2,-3){$\scriptstyle -1\phantom{-}$};
			\node at (2,-4){$\scriptstyle 2$};
			\node at (2,-5){$\scriptstyle 5$};
			\node at (2,-6){$\scriptstyle 8$};
			\node at (2,-7){$\scriptstyle 11$};
			\node at (2,-8){$\scriptstyle 14$};
		\end{tikzpicture}
	}
	\qquad
	\hackcenter{
		\begin{tikzpicture}[scale=0.45]
			\draw[thick, lightgray] (0,0.5)--(0,-8.5);
			\draw[thick, lightgray] (1,0.5)--(1,-8.5);
			\draw[thick, lightgray] (2,0.5)--(2,-8.5);
			\draw[thick, lightgray, dotted] (0,-8.6)--(0,-9.2);
			\draw[thick, lightgray, dotted] (1,-8.6)--(1,-9.2);
			\draw[thick, lightgray, dotted] (2,-8.6)--(2,-9.2);
			\draw[thick, black, dotted] (0,0.5)--(0,1);
			\draw[thick, black, dotted] (1,0.5)--(1,1);
			\draw[thick, black, dotted] (2,0.5)--(2,1);
			\node[above] at (0,1){$\scriptstyle \overline{0}$};
			\node[above] at (1,1){$\scriptstyle \overline{1}$};
			\node[above] at (2,1){$\scriptstyle \overline{2}$};
			\blackdot(0,0);
			\blackdot(0,-1);
			\blackdot(0,-2);
			\blackdot(0,-3);
			\blackdot(0,-4);
			\blackdot(0,-5);
			\draw[thick, lightgray] (-0.2,-6)--(0.2,-6);
			\blackdot(0,-7);
			\draw[thick, lightgray] (-0.2,-8)--(0.2,-8);
			\blackdot(1,0);
			\blackdot(1,-1);
			\draw[thick, lightgray] (0.8,-2)--(1.2,-2);
			\blackdot(1,-3);
			\blackdot(1,-4);
			\blackdot(1,-5);
			\draw[thick, lightgray] (0.8,-6)--(1.2,-6);
			\draw[thick, lightgray] (0.8,-7)--(1.2,-7);
			\draw[thick, lightgray] (0.8,-8)--(1.2,-8);
			\blackdot(2,0);
			\blackdot(2,-1);
			\draw[thick, lightgray] (1.8,-2)--(2.2,-2);
			\draw[thick, lightgray] (1.8,-3)--(2.2,-3);
			\draw[thick, lightgray] (1.8,-4)--(2.2,-4);
			\draw[thick, lightgray] (1.8,-5)--(2.2,-5);
			\draw[thick, lightgray] (1.8,-6)--(2.2,-6);
			\draw[thick, lightgray] (1.8,-7)--(2.2,-7);
			\draw[thick, lightgray] (1.8,-8)--(2.2,-8);
		\end{tikzpicture}
	}
\end{align*}
\caption{Let $e=3$, \(\bkap = \kappa_1 = 2\), and $\bla=\la^{(1)}=(8,4^2,3^2,2^2)$. On the top left, the Young diagram of $\bla$ is displayed, and on the bottom left, the set of beta numbers are displayed on the infinity runner. In the middle, the positions of a $3$-runner abacus are displayed. On the right, the abacus display for $\bla$ is given.
}
\label{abacus}      
\end{figure}
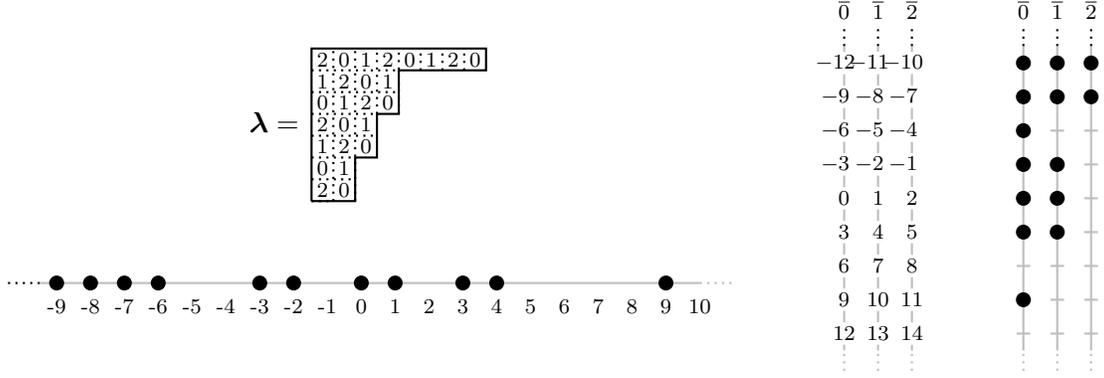

\subsubsection{Removable ribbons and beta numbers}\label{remribsexp}
Let \(\blam \subseteq \N_\ell\) be an \(\ell\)-multipartition. 
Removable ribbons for \(\blam\) are easily classified in terms of the \(\bkap\)-beta numbers for \(\lambda\). There is a bijection 
\begin{align*}
\{(x,y, r)\mid r \in [1,\ell], x \in \ZZ \backslash \B^r(\blam, \bkap), y \in  \B^r(\blam, \bkap), x < y\} \longleftrightarrow \{\textup{Removable ribbons for } \blam\}
\end{align*}
given by associating \((x,y,r)\) with the ribbon \(\blam/\bmu\), where \(\bmu\) is the multipartition with 
\begin{align*}
\B^s(\bmu, \bkap) = 
\begin{cases}
\B^r(\blam, \bkap)\backslash\{y\} \cup \{x\} & \textup{if }s=r;\\
\B^s(\blam, \bkap) & \textup{otherwise}.
\end{cases}
\end{align*}
 In other words, moving the bead \(y \in \B^r(\blam, \bkap)\) to the unoccupied lower position \(x\) is equivalent to removing the ribbon \(\blam/\bmu\) from \(\blam\). Moreover, the content of the ribbon \(\blam/\bmu\) is given by:
\begin{align*}
\cont(\blam/\bmu) = \alpha(x+1, y-x).
\end{align*}

\subsubsection{Addable ribbons and beta numbers}\label{addribsexp}
Let \(\blam \subseteq \N_\ell\) be an \(\ell\)-multipartition. 
Addable ribbons for \(\blam\) are also easily classified in terms of the \(\bkap\)-beta numbers for \(\lambda\). There is a bijection 
\begin{align*}
\{(x,y, r)\mid r \in [1,\ell], x \in \B^r(\blam, \bkap), y \in   \ZZ \backslash\B^r(\blam, \bkap), x < y\} \longleftrightarrow \{\textup{Addable ribbons for } \blam\}
\end{align*}
given by associating \((x,y,r)\) with the ribbon \(\bmu/\blam\), where \(\bmu\) is the multipartition with 
\begin{align*}
\B^s(\bmu, \bkap) = 
\begin{cases}
\B^r(\blam, \bkap)\backslash\{x\} \cup \{y\} & \textup{if }s=r;\\
\B^s(\blam, \bkap) & \textup{otherwise}.
\end{cases}
\end{align*}
 In other words, moving the bead \(x \in \B^r(\blam, \bkap)\) to the unoccupied greater position \(y\) is equivalent to adding the ribbon \(\bmu/\blam\) to \(\blam\). Moreover, the content of the ribbon \(\bmu/\blam\) is given by:
\begin{align*}
\cont(\bmu/\blam) = \alpha(x+1, y-x).
\end{align*}

\begin{figure}[h]
\begin{align*}
	{}
	\bla\ =\
				\hackcenter{
			\begin{tikzpicture}[scale=0.29]
				\fill[fill=white] (0,0)--(0,1)--(1,1)--(1,0)--(0,0);
				\fill[fill=green!30] (2,3)--(3,3)--(3,5)--(4,5)--(4,7)--(3,7)--(3,6)--(2,6)--(2,3);
				\draw[thick, dotted] (0,1)--(1,1);
				\draw[thick, dotted] (0,2)--(2,2);
				\draw[thick, dotted] (0,3)--(2,3);
				\draw[thick, dotted] (0,4)--(3,4);
				\draw[thick, dotted] (0,5)--(3,5);
				\draw[thick, dotted] (0,6)--(4,6);
				\draw[thick, dotted] (0,7)--(4,7);
				\draw[thick, dotted] (1,1)--(1,8);
				\draw[thick, dotted] (2,1)--(2,8);
				\draw[thick, dotted] (3,5)--(3,8);
				\draw[thick, dotted] (4,5)--(4,8);
				\draw[thick, dotted] (5,7)--(5,8);
				\draw[thick, dotted] (6,7)--(6,8);
				\draw[thick, dotted] (7,7)--(7,8);
				\node at (0.5,7.5){$\scriptstyle 2$};
				\node at (1.5,7.5){$\scriptstyle 0$};
				\node at (2.5,7.5){$\scriptstyle 1$};
				\node at (3.5,7.5){$\scriptstyle 2$};
				\node at (4.5,7.5){$\scriptstyle 0$};
				\node at (5.5,7.5){$\scriptstyle 1$};
				\node at (6.5,7.5){$\scriptstyle 2$};
				\node at (7.5,7.5){$\scriptstyle 0$};
				\node at (0.5,6.5){$\scriptstyle 1$};
				\node at (1.5,6.5){$\scriptstyle 2$};
				\node at (2.5,6.5){$\scriptstyle 0$};
				\node at (3.5,6.5){$\scriptstyle 1$};
				\node at (0.5,5.5){$\scriptstyle 0$};
				\node at (1.5,5.5){$\scriptstyle 1$};
				\node at (2.5,5.5){$\scriptstyle 2$};
				\node at (3.5,5.5){$\scriptstyle 0$};
				\node at (0.5,4.5){$\scriptstyle 2$};
				\node at (1.5,4.5){$\scriptstyle 0$};
				\node at (2.5,4.5){$\scriptstyle 1$};
				\node at (0.5,3.5){$\scriptstyle 1$};
				\node at (1.5,3.5){$\scriptstyle 2$};
				\node at (2.5,3.5){$\scriptstyle 0$};
				\node at (0.5,2.5){$\scriptstyle 0$};
				\node at (1.5,2.5){$\scriptstyle 1$};
				\node at (0.5,1.5){$\scriptstyle 2$};
				\node at (1.5,1.5){$\scriptstyle 0$};
				\draw[thick]  (0,1)--(2,1)--(2,3)--(3,3)--(3,5)--(4,5)--(4,7)--(8,7)--(8,8)--(0,8)--(0,1);
				\phantom{\draw[thick,fill=blue] (0,0)--(0,1)--(1,1)--(1,0)--(0,0);}
			\end{tikzpicture}
		}
		\hackcenter{
			\begin{tikzpicture}[scale=0.29]
				\node[white] at (-0.5,0){{}};
				\node[white] at (0.5,0){{}};
				\draw[thin, gray,fill=gray!30]  (0,0)--(0,7.6);
				\draw[thin, white]  (0,7.6)--(0,8);
			\end{tikzpicture}
		}
				\hackcenter{
			\begin{tikzpicture}[scale=0.29]
				\fill[fill=red!30] (1,1)--(2,1)--(3,1)--(3,2)--(5,2)--(5,5)--(4,5)--(4,3)--(2,3)--(2,2)--(1,2)--(1,1);
				\draw[thick, dotted] (0,1)--(1,1);
				\draw[thick, dotted] (0,2)--(2,2);
				\draw[thick, dotted] (0,3)--(2,3);
				\draw[thick, dotted] (0,4)--(4,4);
				\draw[thick, dotted] (0,5)--(5,5);
				\draw[thick, dotted] (0,6)--(5,6);
				\draw[thick, dotted] (0,7)--(5,7);
				\draw[thick, dotted] (1,0)--(1,8);
				\draw[thick, dotted] (2,2)--(2,8);
				\draw[thick, dotted] (3,3)--(3,8);
				\draw[thick, dotted] (4,5)--(4,8);
				\draw[thick, dotted] (5,5)--(5,8);
				\node at (0.5,7.5){$\scriptstyle 1$};
				\node at (1.5,7.5){$\scriptstyle 2$};
				\node at (2.5,7.5){$\scriptstyle 0$};
				\node at (3.5,7.5){$\scriptstyle 1$};
				\node at (4.5,7.5){$\scriptstyle 2$};
				\node at (5.5,7.5){$\scriptstyle 0$};
				\node at (0.5,6.5){$\scriptstyle 0$};
				\node at (1.5,6.5){$\scriptstyle 1$};
				\node at (2.5,6.5){$\scriptstyle 2$};
				\node at (3.5,6.5){$\scriptstyle 0$};
				\node at (4.5,6.5){$\scriptstyle 1$};
				\node at (0.5,5.5){$\scriptstyle 2$};
				\node at (1.5,5.5){$\scriptstyle 0$};
				\node at (2.5,5.5){$\scriptstyle 1$};
				\node at (3.5,5.5){$\scriptstyle 2$};
				\node at (4.5,5.5){$\scriptstyle 0$};
				\node at (0.5,4.5){$\scriptstyle 1$};
				\node at (1.5,4.5){$\scriptstyle 2$};
				\node at (2.5,4.5){$\scriptstyle 0$};
				\node at (3.5,4.5){$\scriptstyle 1$};
				\node at (4.5,4.5){$\scriptstyle 2$};
				\node at (0.5,3.5){$\scriptstyle 0$};
				\node at (1.5,3.5){$\scriptstyle 1$};
				\node at (2.5,3.5){$\scriptstyle 2$};
				\node at (3.5,3.5){$\scriptstyle 0$};
				\node at (4.5,3.5){$\scriptstyle 1$};
				\node at (0.5,2.5){$\scriptstyle 2$};
				\node at (1.5,2.5){$\scriptstyle 0$};
				\node at (2.5,2.5){$\scriptstyle 1$};
				\node at (3.5,2.5){$\scriptstyle 2$};
				\node at (4.5,2.5){$\scriptstyle 0$};
				\node at (0.5,1.5){$\scriptstyle 1$};
				\node at (1.5,1.5){$\scriptstyle 2$};
				\node at (2.5,1.5){$\scriptstyle 0$};
				\node at (0.5,0.5){$\scriptstyle 0$};
				\draw[thick]  (0,0)--(1,0)--(1,2)--(2,2)--(2,3)--(4,3)--(4,5)--(5,5)--(5,7)--(6,7)--(6,8)--(0,8)--(0,0);
			\end{tikzpicture}
		}
		\end{align*}
		\begin{align*}
		{}	
		\begin{array}{l}
		\hackcenter{
			\begin{tikzpicture}[scale=0.45]
				\node at (-13,0) {$\B^1(\blam, \bkap):$};
				\draw[thick, ->, green!70] (4,0) arc (0:180:2.5cm);
				\draw[thick, lightgray, dotted] (10,0)--(11,0);
				\draw[thick, black, dotted] (-9.5,0)--(-10.5,0);
				\draw[thick, lightgray ] (-9.5,0)--(10,0);
				\blackdot(9,0);
				\blackdot(4,0);
				\blackdot(3,0);
				\blackdot(1,0);
				\blackdot(0,0);
				\blackdot(-2,0);
				\blackdot(-3,0);
				\blackdot(-6,0);
				\blackdot(-7,0);
				\blackdot(-8,0);
				\blackdot(-9,0);
				\node[below] at (-9,-0.2){$\scriptstyle \textup{-}9$};
				\node[below] at (-8,-0.2){$\scriptstyle \textup{-}8$};
				\node[below] at (-7,-0.2){$\scriptstyle \textup{-}7$};
				\node[below] at (-6,-0.2){$\scriptstyle \textup{-}6$};
				\node[below] at (-5,-0.2){$\scriptstyle \textup{-}5$};
				\node[below] at (-4,-0.2){$\scriptstyle \textup{-}4$};
				\node[below] at (-3,-0.2){$\scriptstyle \textup{-}3$};
				\node[below] at (-2,-0.2){$\scriptstyle \textup{-}2$};
				\node[below] at (-1,-0.2){$\scriptstyle \textup{-}1$};
				\node[below] at (0,-0.2){$\scriptstyle 0$};
				\node[below] at (1,-0.2){$\scriptstyle 1$};
				\node[below] at (2,-0.2){$\scriptstyle 2$};
				\node[below] at (3,-0.2){$\scriptstyle 3$};
				\node[below] at (4,-0.2){$\scriptstyle 4$};
				\node[below] at (5,-0.2){$\scriptstyle 5$};
				\node[below] at (6,-0.2){$\scriptstyle 6$};
				\node[below] at (7,-0.2){$\scriptstyle 7$};
				\node[below] at (8,-0.2){$\scriptstyle 8$};
				\node[below] at (9,-0.2){$\scriptstyle 9$};
				\node[below] at (10,-0.2){$\scriptstyle 10$};
			\end{tikzpicture}
		}
		\\
		\\
		\hackcenter{
			\begin{tikzpicture}[scale=0.45]
				\node at (-13,0) {$\B^2(\blam, \bkap):$};
				\draw[thick, <-, red!70] (2,0) arc (0:180:3.5cm);
				\draw[thick, lightgray, dotted] (10,0)--(11,0);
				\draw[thick, black, dotted] (-9.5,0)--(-10.5,0);
				\draw[thick, lightgray ] (-9.5,0)--(10,0);
				\blackdot(6,0);
				\blackdot(4,0);
				\blackdot(3,0);
				\blackdot(1,0);
				\blackdot(0,0);
				\blackdot(-3,0);
				\blackdot(-5,0);
				\blackdot(-6,0);
				\blackdot(-8,0);
				\blackdot(-9,0);
				\node[below] at (-9,-0.2){$\scriptstyle \textup{-}9$};
				\node[below] at (-8,-0.2){$\scriptstyle \textup{-}8$};
				\node[below] at (-7,-0.2){$\scriptstyle \textup{-}7$};
				\node[below] at (-6,-0.2){$\scriptstyle \textup{-}6$};
				\node[below] at (-5,-0.2){$\scriptstyle \textup{-}5$};
				\node[below] at (-4,-0.2){$\scriptstyle \textup{-}4$};
				\node[below] at (-3,-0.2){$\scriptstyle \textup{-}3$};
				\node[below] at (-2,-0.2){$\scriptstyle \textup{-}2$};
				\node[below] at (-1,-0.2){$\scriptstyle \textup{-}1$};
				\node[below] at (0,-0.2){$\scriptstyle 0$};
				\node[below] at (1,-0.2){$\scriptstyle 1$};
				\node[below] at (2,-0.2){$\scriptstyle 2$};
				\node[below] at (3,-0.2){$\scriptstyle 3$};
				\node[below] at (4,-0.2){$\scriptstyle 4$};
				\node[below] at (5,-0.2){$\scriptstyle 5$};
				\node[below] at (6,-0.2){$\scriptstyle 6$};
				\node[below] at (7,-0.2){$\scriptstyle 7$};
				\node[below] at (8,-0.2){$\scriptstyle 8$};
				\node[below] at (9,-0.2){$\scriptstyle 9$};
				\node[below] at (10,-0.2){$\scriptstyle 10$};
			\end{tikzpicture}
		}
		\\
		\\
	\end{array}
\end{align*}	
\caption{On the top, 
the residues are shown for the 2-multipartition \(\blam = ((8,4^2,3^2,2^2)\mid(6,5^2,4^2,2,1^2))\) with \(e=3\) and multicharge \(\bkap=(2,1)\), considered in \cref{multiex5}. A removable ribbon and an addable ribbon are highlighted. The \(\bkap\)-beta numbers for \(\blam\) are depicted below, together with the bead moves corresponding to adding/removing the indicated ribbons.
}
\label{figmult2}      
\end{figure}
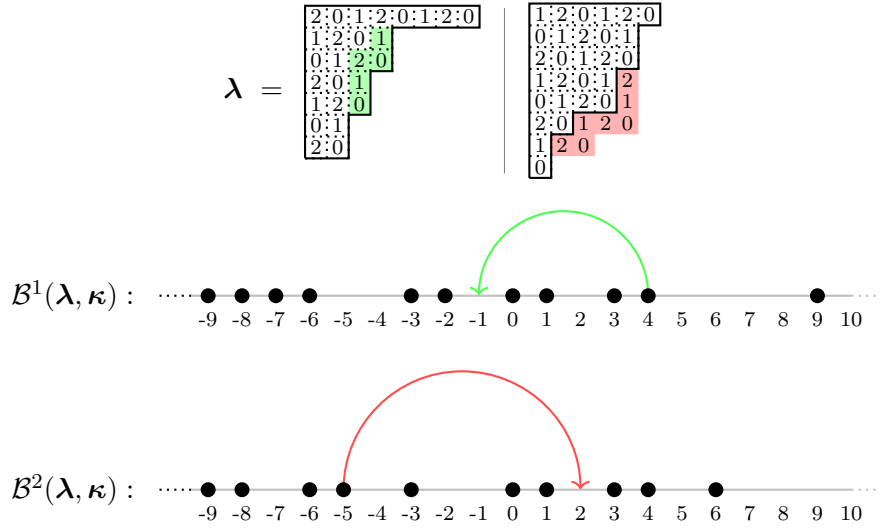

\begin{Example}\label{multiex5}
On the left of Figure~\ref{figmult2},
the residues are shown for the example:
\begin{align*}
e=3, \qquad \bkap = (2,1), \qquad \blam = ((8,4^2,3^2,2^2)\mid(6,5^2,4^2,2,1^2)),
\end{align*}
and the \(\bkap\)-beta numbers are displayed on the right.
We have content \(\textup{cont}(\blam)= 21 \alpha_0 + 17\alpha_1 + 16 \alpha_2\), defect \(\textup{def}_{\bkap}(\blam) = 12\), and leading row word:
\begin{align*}
\bi^{\blam} = 20120120 1201 0120 201 120 01 20 120120 01201 20120 1201 0120 20 1 0.
\end{align*}
The remaining associated data for \(\blam\) is as follows:
\begin{align*}
\begin{array}{llllll}
M_0^1(\blam, \bkap) = 9 & &M_1^1(\blam, \bkap) = 4 & &M_2^1(\blam, \bkap) = -7\\
\\
M_0^2(\blam, \bkap) = 6 & &M_1^2(\blam, \bkap) = 4 & &M_2^2(\blam, \bkap) = -10\\
\end{array}
\end{align*}
\begin{align*}
\begin{array}{llllll}
h^1_{0,1}(\blam, \bkap) = -2 & h^1_{0,2}(\blam, \bkap) = -6 & h^1_{1,2}(\blam, \bkap) = -4 & h^1_{1,0}(\blam, \bkap) =  1 & h^1_{2,0}(\blam, \bkap) = 5 & h^1_{2,1}(\blam, \bkap) = 3\\
\\
h^2_{0,1}(\blam, \bkap) = -1 & h^2_{0,2}(\blam, \bkap) = -6 & h^2_{1,2}(\blam, \bkap) = -5 & h^2_{1,0}(\blam, \bkap) =  0 & h^2_{2,0}(\blam, \bkap) = 5 & h^2_{2,1}(\blam, \bkap) = 4\\
\end{array}
\end{align*}
The highlighted green removable ribbon has content \(2\alpha_0 + 2 \alpha_1 + \alpha_2\). Removing this ribbon would correspond to moving the bead in position \(4\) in \(\B^1(\blam, \bkap)\) into the unoccupied \(-1\) position. The highlighted red addable ribbon has content \(2 \alpha_0 + 2 \alpha_1 + 3 \alpha_2\). Adding this ribbon would correspond to moving the bead in position \(-5\) in \(\B^2(\blam, \bkap)\) into the unoccupied \(2\) position.
\end{Example}

\subsubsection{Ribbons with real positive root content}

\begin{Lemma}\label{iteraddrem}
Let \(\beta \in \Phi_+^\re\), and let \(\blam \in \Lambda_+^{\bkap}\). Then:
\begin{enumerate}
\item It is possible to iteratively remove \(k\) ribbons of content \(\beta\) from \(\blam\) if and only if \(\blam\) has \(k\) distinct removable ribbons of content \(\beta\).
\item It is possible to iteratively add \(k\) ribbons of content \(\beta\) to \(\blam\) if and only if \(\blam\) has \(k\) distinct addable ribbons of content \(\beta\).
\item Only finitely many ribbons of content \(\beta\) may be iteratively added to \(\blam\).
\end{enumerate}
\end{Lemma}
\begin{proof}
We prove (i) first. Let \(\beta = \alpha(t+1,L)\). Then, in consideration of \S\ref{remribsexp}, the removable ribbons of content \(\beta\) for \(\blam\) correspond to pairs \((x,r)\) such that \(x \in \ZZ \backslash \B^r(\blam, \bkap)\), \(x+L \in \B^r(\blam, \bkap)\), and \(\bar x = t\). Note that, since \(\beta \in \Phi_+^\re\), we have that \(\overline{x+L} \neq t\).
Let \(\xi\) be a removable ribbon of content \(\beta\) for \(\blam\) which corresponds to the pair \((x',r')\). Then the beta numbers for \(\blam\backslash \xi\) are given by deleting \(x' + L\) from \( \B^r(\blam, \bkap)\) and replacing them with \(x'\). Then we have:
\begin{align*}
\{(x,r) \mid \bar x = t, & \,x \in \ZZ \backslash \B^r(\blam\backslash \xi, \bkap), x+L \in \B^r(\blam\backslash \xi, \bkap)\}\\
&= \{(x,r) \mid \bar x = t, x \in \ZZ \backslash \B^r(\blam, \bkap), x+L \in \B^r(\blam, \bkap)\}\backslash \{(x',r')\}.
\end{align*}
Thus \(\blam \backslash \xi\) has one less removable ribbon of content \(\beta\) than \(\blam\). The result follows by induction on \(k\).

The proof for (ii) proceeds similarly, by noting that if \(\xi\) is an addable ribbon of content \(\beta\) for \(\blam\), then \(\blam \sqcup \xi\) has one less addable ribbon of content \(\beta\) than \(\blam\), and again the claim follows by induction.
Finally, (iii) is an immediate corollary to (ii).
\end{proof}

\subsection{The Kac--Moody algebra \(\asl_e(\CC)\) and the Fock space} \label{slsec}
The combinatorics of multipartitions, ribbon tilings, and the root system \(\Phi_+\) interact directly in the action of the Kac--Moody algebra \(\asl_e(\CC)\) on Fock space. While we will not need the full details in this paper (and refer the reader to \cite{kac, LLT97} for more complete treatments), we briefly mention key features of this action which are combinatorially relevant.

We have the Cartan subalgebra \(\mathfrak{h} = \CC\{\alpha_i^\vee = E_{i,i} - E_{i+1,i+1} \mid i \in [1,e-1]\} \subseteq \mathfrak{sl}_e(\CC)\) and root vectors \(e_\beta = E_{i,j+1}, f_{\beta} = E_{j+1,i} \) associated to every positive root \(\beta = \alpha_i + \alpha_{i+1} + \dots + \alpha_j \in \Phi^\textup{fin}\).
We take the realization of \(\asl_e(\CC)\) given by 
\begin{align*}
\asl_e(\CC) = \left( \mathfrak{sl}_e(\CC) \otimes \CC[t^{\pm 1}]\right) \oplus \CC c \oplus \CC d,
\end{align*}
with commutator defined by:
\begin{align*}
[x \otimes t^a + \lambda c + \mu d, y \otimes t^b + \lambda' c + \mu ' d] =
[x,y] \otimes t^{a+b} + \mu b y \otimes t^b - \mu' a x \otimes t^a + a \delta_{a+b,0}(x,y)c,
\end{align*}
where \((x,y) = \tr(xy)\) is the (renormalized) Killing form on \(\mathfrak{sl}_e(\CC)\). The Cartan subalgebra \(\hat{\mathfrak{h}}\) for  \(\asl_e(\CC)\) has basis \(\{ \alpha_1^\vee, \dots, \alpha_{e-1}^\vee, c, d\}\), and we let \(\alpha_0^\vee := c - \alpha_1^\vee - \dots - \alpha_{e-1}^\vee\).  In \(\hat{\mathfrak{h}}^*\) we have the simple roots \(\alpha_0, \alpha_1, \dots, \alpha_{e-1}\), and fundamental weights \(\Lambda_0, \Lambda_1, \dots, \Lambda_{e-1}\), where:
\begin{align*}
\alpha_i(\alpha_j^\vee) = 2 \delta_{ij} - \delta_{i, \overline{j \pm 1}}, \quad \alpha_i(c) = 0, \quad \alpha_i(d) = \delta_{i,0}, \quad
\Lambda_i(\alpha_j^\vee) = \delta_{ij}, \quad \Lambda_i(c) = 1, \quad \Lambda_i(d) = 0.
\end{align*}

For \(\beta \in \Phi_+^\re\), the one-dimensional real root vector spaces are \(\asl_e(\CC)_\beta = \CC\{e_\beta\}\), \(\asl_e(\CC)_{-\beta} = \CC\{f_\beta\}\), where
\begin{align*}
e_{m \delta + \alpha} = e_\alpha \otimes t^m,
\qquad
f_{m \delta + \alpha} = f_\alpha \otimes t^{-m},
\qquad
e_{n \delta - \alpha} = f_\alpha \otimes t^n,
\qquad
f_{n \delta - \alpha} = e_\alpha \otimes t^{-n}
\end{align*}
for all \(\alpha \in \Phi_+^\fin, m \in \ZZ_{\geq 0}, n \in \ZZ_{>0}\). For any \(\beta \in \Phi_+^\re\), the triple \(\{e_\beta, f_\beta, [e_\beta, f_\beta]\}\) generates a subalgebra isomorphic to \(\mathfrak{sl}_2(\CC)\).  For \(m \neq 0\), we have the \((e-1)\)-dimensional imaginary root space
\(
\asl_e(\CC)_{m \delta} = \{ \alpha^\vee_i \otimes t^m \mid i \in [1,e-1]\}
\). This yields the root space decomposition 
\begin{align*}
\asl_e(\CC) = \bigoplus_{\beta \in \Phi_+} \asl_e(\CC)_\beta \oplus \hat{ \mathfrak{h}} \oplus \bigoplus_{\beta \in \Phi_+} \asl_e(\CC)_{-\beta}.
\end{align*}

\subsubsection{The Fock space}
Given a multicharge \(\bkap = (\kappa_1, \dots, \kappa_\ell)\) of level \(\ell\), let \(\Lambda = \Lambda_{\kappa_1} + \dots + \Lambda_{\kappa_\ell} \in \mathfrak{h}^*\). The associated Fock space is the \(\CC\)-vector space \(\CC \Lambda_+^{\bkap}\), with basis given by all multipartitions of multicharge \(\bkap\).
The Fock space is an \(\asl_e(\CC)\)-module of highest-weight \(\Lambda\). For \(\omega \in \Z_{\geq 0} I\), \(\CC \Lambda_+^{\bkap}(\omega)\) has weight \(\Lambda- \omega\) with respect to the action of \(\hat{\mathfrak{h}}\), so the weight space decomposition of the Fock space is given by \(\CC \Lambda^{\bkap}_+ = \bigoplus_{\omega \in \ZZ_{\geq 0}I} \CC \Lambda_+^{\bkap}(\omega)\). 

For \(\blam \in \Lambda_+^{\bkap}(\omega)\) and \(\beta \in \Phi_+^\re\), the action of real root vectors on the Fock space is given by:
\begin{align}\label{efaction}
f_\beta \cdot \blam = \sum_\xi (-1)^{\#\textup{columns}(\xi) + 1} \blam \sqcup \xi
\qquad
e_{\beta} \cdot \blam = \sum_\zeta (-1)^{\#\textup{columns}(\zeta) + 1} \blam \backslash \zeta,
\end{align}
where the sums range over all addable ribbons \(\xi\) (resp.~removable ribbons \(\zeta\)) of content \(\beta\) for \(\blam\). While the action of \(\asl_e(\CC)\) is usually described in terms of the simple roots \(\alpha_i\), the action (\ref{efaction}) follows immediately from the {\em semi-infinite wedge} construction of Fock space, see for instance \cite{LLT97}.

\begin{Lemma}\label{addribbonlemma}
Let \(\beta \in \Phi_+\), and \(\Lambda^{\bkap}_+(\omega), \Lambda^{\bkap}_+(\omega + \beta) \neq \varnothing\). Then there exists \(\bmu \in \Lambda^{\bkap}_+(\omega + \beta)\) with a removable ribbon \(\xi \subseteq \bmu\) such that \(\cont(\xi) = \beta\).
\end{Lemma}
\begin{proof}
If \(\beta  \in \Phi_+^{\textup{im}}\), then \(\beta = m\delta\) for some \(m>1\). Then for any \(\blam \in \Lambda^{\bkap}_+(\omega)\), adding a row \(\xi\) of \(me\) nodes to the top row of \(\mu^{(1)}\) yields a multipartition \(\bmu \in \Lambda^{\bkap}_+(\omega + \beta)\), as required.

Now assume \(\beta   \in \Phi_+^{\textup{re}}\), so \(\beta = m\delta + \alpha\) for some \(\alpha \in \Phi^\textup{fin}\). By way of contradiction, assume that for every \(\bmu \in \Lambda_+^{\bkap}(\omega + \beta)\) there is no removable ribbon \(\xi \subseteq \bmu\) such that \(\cont(\xi) = \beta\). Then it follows from consideration of the Fock space action (\ref{efaction}), that \(f_\beta \cdot \CC \Lambda_+^{\bkap}(\omega) = 0\) and \(e_\beta \cdot \CC\Lambda_+^{\bkap}(\omega + \beta) = 0\). Identifying \( \{ e_\beta, f_\beta, [e_\beta, f_\beta]\}\) with the copy of \(\mathfrak{sl}_2(\CC)\) these elements span, we have that the subspaces
\begin{align}\label{tworeps}
\bigoplus_{n \in \ZZ_{\leq 0}} \CC \Lambda_+^{\bkap}(\omega + n\beta) \qquad \textup{and} \qquad
\bigoplus_{n \in \ZZ_{> 0}} \CC \Lambda_+^{\bkap}(\omega +n \beta)
\end{align}
of the Fock space are \(\mathfrak{sl}_2(\CC)\)-modules, where the space \(\CC \Lambda_+^{\bkap}(\omega + n\beta)\) has \(\mathfrak{sl}_2(\CC)\)-weight
\begin{align*}
(\Lambda - \omega + n \beta)([e_\beta, f_\beta])=
(\Lambda - \omega + mn\delta +  n \alpha)(\alpha^\vee + mc) = (\Lambda- \omega)(\alpha^\vee) + m \ell + 2n.
\end{align*}
Note that the representation on the left in (\ref{tworeps}) is finite-dimensional, with nonzero lowest weight space \(\CC \Lambda_+^{\bkap}(\omega)\). Hence by \(\mathfrak{sl}_2(\CC)\) representation theory, it must be that this lowest weight \((\Lambda- \omega)(\alpha^\vee) + m \ell\) of this module is non-positive. Thus, letting \(\blam \in \CC\Lambda_+^{\bkap}(\omega + \beta)\), it follows that \(\CC\{ f_\beta^k \cdot \blam \mid k \in \ZZ_{\geq 0}\}\) is an \(\mathfrak{sl}_2(\CC)\)-module whose highest weight is the negative integer \((\Lambda- \omega)(\alpha^\vee) + m \ell -2\), and hence \(f_\beta^k \cdot \blam \neq 0\) for all \(k \in \ZZ_{\geq 0}\). But, on the other hand, it follows from consideration of (\ref{efaction}) and \cref{iteraddrem}(iii) that \(f_\beta^k \cdot \blam = 0\) for \(k\gg 0\), yielding the desired contradiction.
\end{proof}

\section{Multicores and RoCK multipartitions}\label{basiccoreRoCKsec}

\subsection{Special multipartitions} 
We now define some multipartitions with special properties. Throughout this section we fix a level \(\ell \in \mathbb{Z}_{>0}\), and a multicharge \(\bkap = (\kappa_1, \dots, \kappa_\ell)\). We will call a permutation \(\theta = (\theta_1, \dots, \theta_e)\) of \([0,e-1]\) a {\em residue permutation}.
\begin{Definition}\label{specialdefs}
Let \(\omega \in \ZZ_{\geq 0}I\), and \(\blam \in \Lambda_+^{\bkap}(\omega)\).
\begin{enumerate}
\item We say that \(\blam\) is a {\em \(\bkap\)-core} if \(\blam\) is the {\em only} \(\ell\)-multipartition of multicharge \(\bkap\) with its content, i.e.
\begin{align*}
\Lambda^{\bkap}_+(\omega) = \{ \blam\}.
\end{align*}
\item We say that \(\blam\) is a {\em multicore} provided that \(\blam\) has no removable ribbons of size \(e\), i.e.~\(x \in \B^r(\blam, \bkap)\) implies \(x-e \in \B^r(\blam, \bkap)\).
We say that \(\Lambda_+^{\bkap}(\omega)\) is a {\em core block} if every multipartition in \( \Lambda_+^{\bkap}(\omega)\) is a multicore.
\item For a residue permutation \(\theta\), we say that \(\blam\) is {\em  \((\bkap,\theta)\)-RoCK} provided that whenever \(y \in \B^r(\blam, \bkap)\), \(x \notin \B^r(\blam, \bkap)\), for some \(x<y\) with \(\overline{x} = \theta_a\), \( \overline{y} = \theta_b\), we have \(a \leq b\).
We say that \(\Lambda^{\bkap}_+(\omega)\) is a {\em \(\theta\)-RoCK block} if every multipartition in \(\Lambda^{\bkap}_+(\omega)\) is \((\bkap,\theta)\)-\(\RoCK\).
We say that \(\Lambda^{\bkap}_+(\omega)\) is a {\em RoCK block} if it is a \(\theta\)-RoCK block for some \(\theta\).
\end{enumerate}
\end{Definition}

\begin{Remark}
As is common in the literature, we abuse notation somewhat in using the term `block' above to refer to the set of multipartitions \(\Lambda_+^{\bkap}(\omega)\), for the following reason. The set \(\Lambda_+^{\bkap}(\omega)\) is associated with the cellular structure of a cyclotomic KLR algebra `block' (in the sense of being an indecomposable algebra), and important properties and connections between these algebras are directly correlated with the combinatorics of the set \(\Lambda_+^{\bkap}(\omega)\), see \cref{subsec:rockblocks,cycKLRsec}.
\end{Remark}

\begin{Remark}\label{rem:l1core}
When \(\ell=1\), it is well known that a partition \(\lambda\) is a \(\kappa\)-core if and only if it is an \(e\)-core. Hence in higher levels, \(\blam\) is a multicore when each constituent partition \(\lambda^{(r)}\) is a \(\kappa_r\)-core, for \(r \in [1,\ell]\). 
\end{Remark}

\begin{figure}[h]
\begin{align*}
	{}
	\begin{array}{l}
		\brho\ =\
		\hackcenter{
			\begin{tikzpicture}[scale=0.29]
				%
				\draw[thick, gray, dotted] (0,1)--(1,1);
				\draw[thick,  gray, dotted] (0,2)--(1,2);
				\draw[thick,  gray, dotted] (0,3)--(1,3);
				\draw[thick,  gray, dotted] (0,4)--(2,4);
				\draw[thick,  gray, dotted] (0,5)--(2,5);
				\draw[thick,  gray, dotted] (0,6)--(2,6);
				\draw[thick,  gray, dotted] (0,7)--(4,7);
				\draw[thick,  gray, dotted] (0,8)--(5,8);
				\draw[thick,  gray, dotted] (0,9)--(6,9);
				\draw[thick,  gray, dotted] (0,10)--(7,10);
				\draw[thick,  gray, dotted] (0,11)--(8,11);
				\draw[thick,  gray, dotted] (0,12)--(10,12);
				\draw[thick,  gray, dotted] (1,13)--(1,0);
				\draw[thick,  gray, dotted] (2,13)--(2,3);
				\draw[thick,  gray, dotted] (3,13)--(3,6);
				\draw[thick,  gray, dotted] (4,13)--(4,7);
				\draw[thick,  gray, dotted] (5,13)--(5,8);
				\draw[thick,  gray, dotted] (6,13)--(6,9);
				\draw[thick, gray,  dotted] (7,13)--(7,10);
				\draw[thick,  gray, dotted] (8,13)--(8,11);
				\draw[thick,  gray, dotted] (9,13)--(9,11);
				\draw[thick,  gray, dotted] (10,13)--(10,11);
				\draw[thick, gray,  dotted] (11,13)--(11,12);
				\draw[thick,  gray, dotted] (12,13)--(12,12);
				\node at (0.5,12.5){$\scriptstyle 2$};
				\node at (1.5,12.5){$\scriptstyle 3$};
				\node at (2.5,12.5){$\scriptstyle 0$};
				\node at (3.5,12.5){$\scriptstyle 1$};
				\node at (4.5,12.5){$\scriptstyle 2$};
				\node at (5.5,12.5){$\scriptstyle 3$};
				\node at (6.5,12.5){$\scriptstyle 0$};
				\node at (7.5,12.5){$\scriptstyle 1$};
				\node at (8.5,12.5){$\scriptstyle 2$};
				\node at (9.5,12.5){$\scriptstyle 3$};
				\node at (10.5,12.5){$\scriptstyle 0$};
				\node at (11.5,12.5){$\scriptstyle 1$};
				\node at (12.5,12.5){$\scriptstyle 2$};
				\node at (0.5,11.5){$\scriptstyle 1$};
				\node at (1.5,11.5){$\scriptstyle 2$};
				\node at (2.5,11.5){$\scriptstyle 3$};
				\node at (3.5,11.5){$\scriptstyle 0$};
				\node at (4.5,11.5){$\scriptstyle 1$};
				\node at (5.5,11.5){$\scriptstyle 2$};
				\node at (6.5,11.5){$\scriptstyle 3$};
				\node at (7.5,11.5){$\scriptstyle 0$};
				\node at (8.5,11.5){$\scriptstyle 1$};
				\node at (9.5,11.5){$\scriptstyle 2$};
				\node at (0.5,10.5){$\scriptstyle 0$};
				\node at (1.5,10.5){$\scriptstyle 1$};
				\node at (2.5,10.5){$\scriptstyle 2$};
				\node at (3.5,10.5){$\scriptstyle 3$};
				\node at (4.5,10.5){$\scriptstyle 0$};
				\node at (5.5,10.5){$\scriptstyle 1$};
				\node at (6.5,10.5){$\scriptstyle 2$};
				\node at (0.5,9.5){$\scriptstyle 3$};
				\node at (1.5,9.5){$\scriptstyle 0$};
				\node at (2.5,9.5){$\scriptstyle 1$};
				\node at (3.5,9.5){$\scriptstyle 2$};
				\node at (4.5,9.5){$\scriptstyle 3$};
				\node at (5.5,9.5){$\scriptstyle 0$};
				\node at (0.5,8.5){$\scriptstyle 2$};
				\node at (1.5,8.5){$\scriptstyle 3$};
				\node at (2.5,8.5){$\scriptstyle 0$};
				\node at (3.5,8.5){$\scriptstyle 1$};
				\node at (4.5,8.5){$\scriptstyle 2$};
				\node at (0.5,7.5){$\scriptstyle 1$};
				\node at (1.5,7.5){$\scriptstyle 2$};
				\node at (2.5,7.5){$\scriptstyle 3$};
				\node at (3.5,7.5){$\scriptstyle 0$};
				\node at (0.5,6.5){$\scriptstyle 0$};
				\node at (1.5,6.5){$\scriptstyle 1$};
				\node at (2.5,6.5){$\scriptstyle 2$};
				\node at (0.5,5.5){$\scriptstyle 3$};
				\node at (1.5,5.5){$\scriptstyle 0$};
				\node at (0.5,4.5){$\scriptstyle 2$};
				\node at (1.5,4.5){$\scriptstyle 3$};
				\node at (0.5,3.5){$\scriptstyle 1$};
				\node at (1.5,3.5){$\scriptstyle 2$};
				\node at (0.5,2.5){$\scriptstyle 0$};
				\node at (0.5,1.5){$\scriptstyle 3$};
				\node at (0.5,0.5){$\scriptstyle 2$};
				\draw[thick]  (0,0)--(1,0)--(1,3)--(2,3)--(2,6)--(3,6)--(3,7)--(4,7)--(4,8)--(5,8)--(5,9)--(6,9)--(6,10)--(7,10)--(7,11)--(10,11)--(10,12)--(13,12)--(13,13)--(0,13)--(0,0);
			\end{tikzpicture}
		}
		\hackcenter{
			\begin{tikzpicture}[scale=0.29]
				\node[white] at (-0.5,0){{}};
				\node[white] at (0.5,0){{}};
				\draw[thin, gray,fill=gray!30]  (0,0)--(0,12.6);
				\draw[thin, white]  (0,12.6)--(0,13);
			\end{tikzpicture}
		}
			\hackcenter{
			\begin{tikzpicture}[scale=0.29]
							\phantom{
					\draw[thick,fill=blue]  (0,0)--(0,-2)--(1,-2)--(1,0)--(0,0);
				}
				\draw[thick, gray, dotted] (0,1)--(1,1);
				\draw[thick,  gray, dotted] (0,2)--(1,2);
				\draw[thick,  gray, dotted] (0,3)--(1,3);
				\draw[thick,  gray, dotted] (0,4)--(2,4);
				\draw[thick,  gray, dotted] (0,5)--(2,5);
				\draw[thick,  gray, dotted] (0,6)--(2,6);
				\draw[thick,  gray, dotted] (0,7)--(4,7);
				\draw[thick,  gray, dotted] (0,8)--(5,8);
				\draw[thick,  gray, dotted] (0,9)--(6,9);
				\draw[thick,  gray, dotted] (0,10)--(8,10);
				\draw[thick,  gray, dotted] (0,11)--(8,11);
				\draw[thick,  gray, dotted] (1,11)--(1,0);
				\draw[thick,  gray, dotted] (2,11)--(2,3);
				\draw[thick,  gray, dotted] (3,11)--(3,6);
				\draw[thick,  gray, dotted] (4,11)--(4,7);
				\draw[thick,  gray, dotted] (5,11)--(5,8);
				\draw[thick,  gray, dotted] (6,11)--(6,9);
				\draw[thick, gray,  dotted] (7,11)--(7,9);
				\draw[thick,  gray, dotted] (8,11)--(8,9);
				\draw[thick,  gray, dotted] (9,11)--(9,10);
				\draw[thick,  gray, dotted] (10,11)--(10,10);
				\draw[thick, gray,  dotted] (11,11)--(11,10);
				\node at (0.5,10.5){$\scriptstyle 0$};
				\node at (1.5,10.5){$\scriptstyle 1$};
				\node at (2.5,10.5){$\scriptstyle 2$};
				\node at (3.5,10.5){$\scriptstyle 3$};
				\node at (4.5,10.5){$\scriptstyle 0$};
				\node at (5.5,10.5){$\scriptstyle 1$};
				\node at (6.5,10.5){$\scriptstyle 2$};
				\node at (7.5,10.5){$\scriptstyle 3$};
				\node at (8.5,10.5){$\scriptstyle 0$};
				\node at (9.5,10.5){$\scriptstyle 1$};
				\node at (10.5,10.5){$\scriptstyle 2$};
				\node at (0.5,9.5){$\scriptstyle 3$};
				\node at (1.5,9.5){$\scriptstyle 0$};
				\node at (2.5,9.5){$\scriptstyle 1$};
				\node at (3.5,9.5){$\scriptstyle 2$};
				\node at (4.5,9.5){$\scriptstyle 3$};
				\node at (5.5,9.5){$\scriptstyle 0$};
				\node at (6.5,9.5){$\scriptstyle 1$};
				\node at (7.5,9.5){$\scriptstyle 2$};
				\node at (0.5,8.5){$\scriptstyle 2$};
				\node at (1.5,8.5){$\scriptstyle 3$};
				\node at (2.5,8.5){$\scriptstyle 0$};
				\node at (3.5,8.5){$\scriptstyle 1$};
				\node at (4.5,8.5){$\scriptstyle 2$};
				\node at (0.5,7.5){$\scriptstyle 1$};
				\node at (1.5,7.5){$\scriptstyle 2$};
				\node at (2.5,7.5){$\scriptstyle 3$};
				\node at (3.5,7.5){$\scriptstyle 0$};
				\node at (0.5,6.5){$\scriptstyle 0$};
				\node at (1.5,6.5){$\scriptstyle 1$};
				\node at (2.5,6.5){$\scriptstyle 2$};
				\node at (0.5,5.5){$\scriptstyle 3$};
				\node at (1.5,5.5){$\scriptstyle 0$};
				\node at (0.5,4.5){$\scriptstyle 2$};
				\node at (1.5,4.5){$\scriptstyle 3$};
				\node at (0.5,3.5){$\scriptstyle 1$};
				\node at (1.5,3.5){$\scriptstyle 2$};
				\node at (0.5,2.5){$\scriptstyle 0$};
				\node at (0.5,1.5){$\scriptstyle 3$};
				\node at (0.5,0.5){$\scriptstyle 2$};
				\draw[thick]  (0,0)--(1,0)--(1,3)--(2,3)--(2,6)--(3,6)--(3,7)--(4,7)--(4,8)--(5,8)--(5,9)--(6,9)--(7,9)--(8,9)--(8,10)--(11,10)--(11,11)--(0,11)--(0,0);
			\end{tikzpicture}
		}
		\\
		\\
		\bzeta\ = \
		\hackcenter{
			\begin{tikzpicture}[scale=0.29]
				\draw[thick, gray, dotted] (0,1)--(1,1);
				\draw[thick,  gray, dotted] (0,2)--(1,2);
				\draw[thick,  gray, dotted] (0,3)--(1,3);
				\draw[thick,  gray, dotted] (0,4)--(2,4);
				\draw[thick,  gray, dotted] (0,5)--(2,5);
				\draw[thick,  gray, dotted] (0,6)--(4,6);
				\draw[thick,  gray, dotted] (0,7)--(5,7);
				\draw[thick,  gray, dotted] (0,8)--(6,8);
				\draw[thick,  gray, dotted] (0,9)--(9,9);
				\draw[thick,  gray, dotted] (0,10)--(10,10);
				\draw[thick,  gray, dotted] (0,11)--(11,11);
				\draw[thick,  gray, dotted] (0,12)--(12,12);
				\draw[thick,  gray, dotted] (1,13)--(1,0);
				\draw[thick,  gray, dotted] (2,13)--(2,3);
				\draw[thick,  gray, dotted] (3,13)--(3,5);
				\draw[thick,  gray, dotted] (4,13)--(4,6);
				\draw[thick,  gray, dotted] (5,13)--(5,7);
				\draw[thick,  gray, dotted] (6,13)--(6,8);
				\draw[thick, gray,  dotted] (7,13)--(7,8);
				\draw[thick,  gray, dotted] (8,13)--(8,8);
				\draw[thick,  gray, dotted] (9,13)--(9,8);
				\draw[thick,  gray, dotted] (10,13)--(10,9);
				\draw[thick, gray,  dotted] (11,13)--(11,11);
				\draw[thick,  gray, dotted] (12,13)--(12,12);
				\node at (0.5,12.5){$\scriptstyle 2$};
				\node at (1.5,12.5){$\scriptstyle 3$};
				\node at (2.5,12.5){$\scriptstyle 0$};
				\node at (3.5,12.5){$\scriptstyle 1$};
				\node at (4.5,12.5){$\scriptstyle 2$};
				\node at (5.5,12.5){$\scriptstyle 3$};
				\node at (6.5,12.5){$\scriptstyle 0$};
				\node at (7.5,12.5){$\scriptstyle 1$};
				\node at (8.5,12.5){$\scriptstyle 2$};
				\node at (9.5,12.5){$\scriptstyle 3$};
				\node at (10.5,12.5){$\scriptstyle 0$};
				\node at (11.5,12.5){$\scriptstyle 1$};
				\node at (12.5,12.5){$\scriptstyle 2$};
				\node at (0.5,11.5){$\scriptstyle 1$};
				\node at (1.5,11.5){$\scriptstyle 2$};
				\node at (2.5,11.5){$\scriptstyle 3$};
				\node at (3.5,11.5){$\scriptstyle 0$};
				\node at (4.5,11.5){$\scriptstyle 1$};
				\node at (5.5,11.5){$\scriptstyle 2$};
				\node at (6.5,11.5){$\scriptstyle 3$};
				\node at (7.5,11.5){$\scriptstyle 0$};
				\node at (8.5,11.5){$\scriptstyle 1$};
				\node at (9.5,11.5){$\scriptstyle 2$};
				\node at (10.5,11.5){$\scriptstyle 3$};
				\node at (11.5,11.5){$\scriptstyle 0$};
				\node at (0.5,10.5){$\scriptstyle 0$};
				\node at (1.5,10.5){$\scriptstyle 1$};
				\node at (2.5,10.5){$\scriptstyle 2$};
				\node at (3.5,10.5){$\scriptstyle 3$};
				\node at (4.5,10.5){$\scriptstyle 0$};
				\node at (5.5,10.5){$\scriptstyle 1$};
				\node at (6.5,10.5){$\scriptstyle 2$};
				\node at (7.5,10.5){$\scriptstyle 3$};
				\node at (8.5,10.5){$\scriptstyle 0$};
				\node at (9.5,10.5){$\scriptstyle 1$};
				\node at (10.5,10.5){$\scriptstyle 2$};
				\node at (0.5,9.5){$\scriptstyle 3$};
				\node at (1.5,9.5){$\scriptstyle 0$};
				\node at (2.5,9.5){$\scriptstyle 1$};
				\node at (3.5,9.5){$\scriptstyle 2$};
				\node at (4.5,9.5){$\scriptstyle 3$};
				\node at (5.5,9.5){$\scriptstyle 0$};
				\node at (6.5,9.5){$\scriptstyle 1$};
				\node at (7.5,9.5){$\scriptstyle 2$};
				\node at (8.5,9.5){$\scriptstyle 3$};
				\node at (9.5,9.5){$\scriptstyle 0$};
				\node at (0.5,8.5){$\scriptstyle 2$};
				\node at (1.5,8.5){$\scriptstyle 3$};
				\node at (2.5,8.5){$\scriptstyle 0$};
				\node at (3.5,8.5){$\scriptstyle 1$};
				\node at (4.5,8.5){$\scriptstyle 2$};
				\node at (5.5,8.5){$\scriptstyle 3$};
				\node at (6.5,8.5){$\scriptstyle 0$};
				\node at (7.5,8.5){$\scriptstyle 1$};
				\node at (8.5,8.5){$\scriptstyle 2$};
				\node at (0.5,7.5){$\scriptstyle 1$};
				\node at (1.5,7.5){$\scriptstyle 2$};
				\node at (2.5,7.5){$\scriptstyle 3$};
				\node at (3.5,7.5){$\scriptstyle 0$};
				\node at (4.5,7.5){$\scriptstyle 1$};
				\node at (5.5,7.5){$\scriptstyle 2$};
				\node at (0.5,6.5){$\scriptstyle 0$};
				\node at (1.5,6.5){$\scriptstyle 1$};
				\node at (2.5,6.5){$\scriptstyle 2$};
				\node at (3.5,6.5){$\scriptstyle 3$};
				\node at (4.5,6.5){$\scriptstyle 0$};
				\node at (0.5,5.5){$\scriptstyle 3$};
				\node at (1.5,5.5){$\scriptstyle 0$};
				\node at (2.5,5.5){$\scriptstyle 1$};
				\node at (3.5,5.5){$\scriptstyle 2$};
				\node at (0.5,4.5){$\scriptstyle 2$};
				\node at (1.5,4.5){$\scriptstyle 3$};
				\node at (0.5,3.5){$\scriptstyle 1$};
				\node at (1.5,3.5){$\scriptstyle 2$};
				\node at (0.5,2.5){$\scriptstyle 0$};
				\node at (0.5,1.5){$\scriptstyle 3$};
				\node at (0.5,0.5){$\scriptstyle 2$};
				\draw[thick]  (0,0)--(1,0)--(1,3)--(2,3)--(2,5)--(4,5)--(4,6)--(5,6)--(5,7)--(6,7)--(6,8)--(9,8)--(9,9)--(10,9)--(10,10)--(11,10)--(11,11)--(12,11)--(12,12)--(13,12)--(13,13)--(0,13)--(0,0);
			\end{tikzpicture}
		}
		\hackcenter{
			\begin{tikzpicture}[scale=0.29]
				\node[white] at (-0.5,0){{}};
				\node[white] at (0.5,0){{}};
				\draw[thin, gray,fill=gray!30]  (0,0)--(0,12.6);
				\draw[thin, white]  (0,12.6)--(0,13);
			\end{tikzpicture}
		}
		\hackcenter{
			\begin{tikzpicture}[scale=0.29]
								\phantom{
					\draw[thick,fill=blue]  (0,0)--(0,-2)--(1,-2)--(1,0)--(0,0);
				}
				\draw[thick, gray, dotted] (0,1)--(1,1);
				\draw[thick,  gray, dotted] (0,2)--(2,2);
				\draw[thick,  gray, dotted] (0,3)--(3,3);
				\draw[thick,  gray, dotted] (0,4)--(3,4);
				\draw[thick,  gray, dotted] (0,5)--(3,5);
				\draw[thick,  gray, dotted] (0,6)--(4,6);
				\draw[thick,  gray, dotted] (0,7)--(5,7);
				\draw[thick,  gray, dotted] (0,8)--(6,8);
				\draw[thick,  gray, dotted] (0,9)--(7,9);
				\draw[thick,  gray, dotted] (0,10)--(12,10);
				\draw[thick,  gray, dotted] (0,11)--(15,11);
				\draw[thick,  gray, dotted] (1,11)--(1,0);
				\draw[thick,  gray, dotted] (2,11)--(2,2);
				\draw[thick,  gray, dotted] (3,11)--(3,5);
				\draw[thick,  gray, dotted] (4,11)--(4,7);
				\draw[thick,  gray, dotted] (5,11)--(5,7);
				\draw[thick,  gray, dotted] (6,11)--(6,8);
				\draw[thick, gray,  dotted] (7,11)--(7,8);
				\draw[thick,  gray, dotted] (8,11)--(8,9);
				\draw[thick,  gray, dotted] (9,11)--(9,9);
				\draw[thick,  gray, dotted] (10,11)--(10,9);
				\draw[thick, gray,  dotted] (11,11)--(11,9);
				\draw[thick, gray,  dotted] (12,11)--(12,9);
				\draw[thick, gray,  dotted] (13,11)--(13,10);
				\draw[thick, gray,  dotted] (14,11)--(14,10);
				\node at (0.5,10.5){$\scriptstyle 0$};
				\node at (1.5,10.5){$\scriptstyle 1$};
				\node at (2.5,10.5){$\scriptstyle 2$};
				\node at (3.5,10.5){$\scriptstyle 3$};
				\node at (4.5,10.5){$\scriptstyle 0$};
				\node at (5.5,10.5){$\scriptstyle 1$};
				\node at (6.5,10.5){$\scriptstyle 2$};
				\node at (7.5,10.5){$\scriptstyle 3$};
				\node at (8.5,10.5){$\scriptstyle 0$};
				\node at (9.5,10.5){$\scriptstyle 1$};
				\node at (10.5,10.5){$\scriptstyle 2$};
				\node at (11.5,10.5){$\scriptstyle 3$};
				\node at (12.5,10.5){$\scriptstyle 0$};
				\node at (13.5,10.5){$\scriptstyle 1$};
				\node at (14.5,10.5){$\scriptstyle 2$};
				\node at (0.5,9.5){$\scriptstyle 3$};
				\node at (1.5,9.5){$\scriptstyle 0$};
				\node at (2.5,9.5){$\scriptstyle 1$};
				\node at (3.5,9.5){$\scriptstyle 2$};
				\node at (4.5,9.5){$\scriptstyle 3$};
				\node at (5.5,9.5){$\scriptstyle 0$};
				\node at (6.5,9.5){$\scriptstyle 1$};
				\node at (7.5,9.5){$\scriptstyle 2$};
				\node at (8.5,9.5){$\scriptstyle 3$};
				\node at (9.5,9.5){$\scriptstyle 0$};
				\node at (10.5,9.5){$\scriptstyle 1$};
				\node at (11.5,9.5){$\scriptstyle 2$};
				\node at (0.5,8.5){$\scriptstyle 2$};
				\node at (1.5,8.5){$\scriptstyle 3$};
				\node at (2.5,8.5){$\scriptstyle 0$};
				\node at (3.5,8.5){$\scriptstyle 1$};
				\node at (4.5,8.5){$\scriptstyle 2$};
				\node at (5.5,8.5){$\scriptstyle 3$};
				\node at (6.5,8.5){$\scriptstyle 0$};
				\node at (0.5,7.5){$\scriptstyle 1$};
				\node at (1.5,7.5){$\scriptstyle 2$};
				\node at (2.5,7.5){$\scriptstyle 3$};
				\node at (3.5,7.5){$\scriptstyle 0$};
				\node at (4.5,7.5){$\scriptstyle 1$};
				\node at (5.5,7.5){$\scriptstyle 2$};
				\node at (0.5,6.5){$\scriptstyle 0$};
				\node at (1.5,6.5){$\scriptstyle 1$};
				\node at (2.5,6.5){$\scriptstyle 2$};
				\node at (3.5,6.5){$\scriptstyle 3$};
				\node at (0.5,5.5){$\scriptstyle 3$};
				\node at (1.5,5.5){$\scriptstyle 0$};
				\node at (2.5,5.5){$\scriptstyle 1$};
				\node at (3.5,5.5){$\scriptstyle 2$};
				\node at (0.5,4.5){$\scriptstyle 2$};
				\node at (1.5,4.5){$\scriptstyle 3$};
				\node at (2.5,4.5){$\scriptstyle 0$};
				\node at (0.5,3.5){$\scriptstyle 1$};
				\node at (1.5,3.5){$\scriptstyle 2$};
				\node at (2.5,3.5){$\scriptstyle 3$};
				\node at (0.5,2.5){$\scriptstyle 0$};
				\node at (1.5,2.5){$\scriptstyle 1$};
				\node at (2.5,2.5){$\scriptstyle 2$};
				\node at (0.5,1.5){$\scriptstyle 3$};
				\node at (1.5,1.5){$\scriptstyle 0$};
				\node at (0.5,0.5){$\scriptstyle 2$};
				\draw[thick]  (0,0)--(1,0)--(1,1)--(2,1)--(2,2)--(3,2)--(3,5)--(4,5)--(4,7)--(6,7)--(6,8)--(7,8)--(7,9)--(12,9)--(12,10)--(15,10)--(15,11)--(0,11)--(0,0);
			\end{tikzpicture}
		}
	\end{array}
\end{align*}
\caption{
Two 2-multipartitions \( \brho\) and \(\bzeta\), for \(e=4\) and multicharge \(\bkap=(2,0)\), considered in \cref{ExBigRoCKnot}.
}
\label{fig123}      
\end{figure}
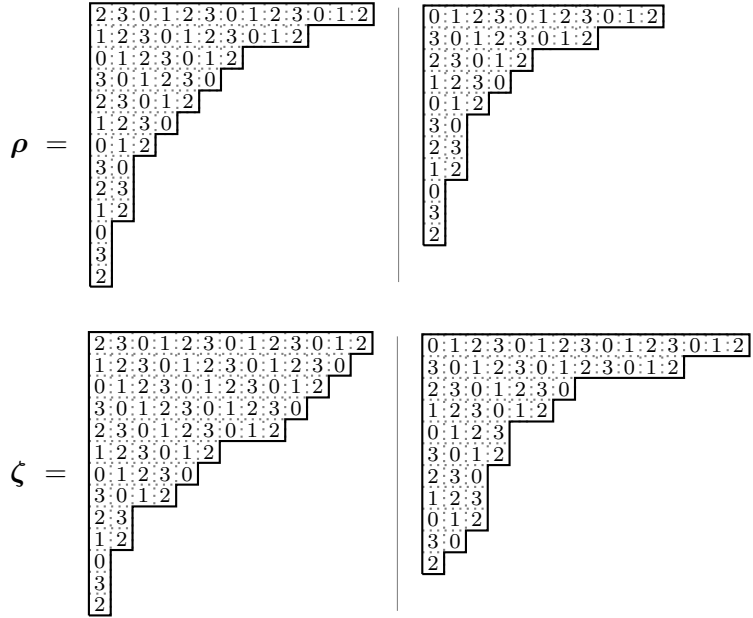

\begin{Example}\label{ExBigRoCKnot}
Let \(e=4\), \(\bkap = (2,0)\), and take the residue permutation \(\theta = (1,3,0,2)\).
In Figure~\ref{fig123}, the multipartitions
\begin{align*}
\brho &= ((13,10,7,6,5,4,3,2^3,1^3)\mid (11,8,5,4,3,2^3,1^3))\\
\bzeta &= ((13,12,11,10,9,6,5,4,2^2,1^3)\mid(15, 12,7,6,4^2,3^3,2,1))
\end{align*}
are shown. The \(\bkap\)-beta numbers for \(\brho\) are:
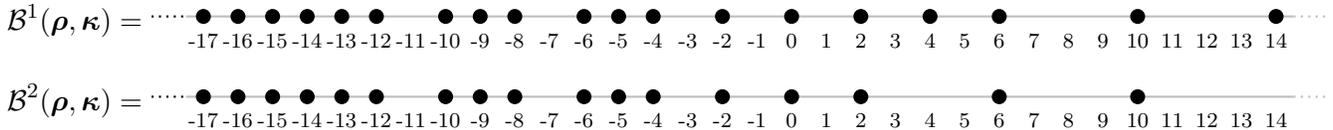
\begin{align*}
{}
\B^1(\brho, \bkap) &=
\hackcenter{
\begin{tikzpicture}[scale=0.46]
\draw[thick, white] (0,0.8);
\draw[thick, lightgray, dotted] (14.5,0)--(15.5,0);
\draw[thick, black, dotted] (-18.5,0)--(-17.5,0);
\draw[thick, lightgray ] (-17.5,0)--(14.5,0);
\blackdot(14,0);
\blackdot(10,0);
\blackdot(6,0);
\blackdot(4,0);
\blackdot(2,0);
\blackdot(0,0);
\blackdot(-2,0);
\blackdot(-4,0);
\blackdot(-5,0);
\blackdot(-6,0);
\blackdot(-8,0);
\blackdot(-9,0);
\blackdot(-10,0);
\blackdot(-12,0);
\blackdot(-13,0);
\blackdot(-14,0);
\blackdot(-15,0);
\blackdot(-16,0);
\blackdot(-17,0);
\node[below] at (-17,-0.2){$\scriptstyle \textup{-}17$};
\node[below] at (-16,-0.2){$\scriptstyle \textup{-}16$};
\node[below] at (-15,-0.2){$\scriptstyle \textup{-}15$};
\node[below] at (-14,-0.2){$\scriptstyle \textup{-}14$};
\node[below] at (-13,-0.2){$\scriptstyle \textup{-}13$};
\node[below] at (-12,-0.2){$\scriptstyle \textup{-}12$};
\node[below] at (-11,-0.2){$\scriptstyle \textup{-}11$};
\node[below] at (-10,-0.2){$\scriptstyle \textup{-}10$};
\node[below] at (-9,-0.2){$\scriptstyle \textup{-}9$};
\node[below] at (-8,-0.2){$\scriptstyle \textup{-}8$};
\node[below] at (-7,-0.2){$\scriptstyle \textup{-}7$};
\node[below] at (-6,-0.2){$\scriptstyle \textup{-}6$};
\node[below] at (-5,-0.2){$\scriptstyle \textup{-}5$};
\node[below] at (-4,-0.2){$\scriptstyle \textup{-}4$};
\node[below] at (-3,-0.2){$\scriptstyle \textup{-}3$};
\node[below] at (-2,-0.2){$\scriptstyle \textup{-}2$};
\node[below] at (-1,-0.2){$\scriptstyle \textup{-}1$};
\node[below] at (0,-0.2){$\scriptstyle 0$};
\node[below] at (1,-0.2){$\scriptstyle 1$};
\node[below] at (2,-0.2){$\scriptstyle 2$};
\node[below] at (3,-0.2){$\scriptstyle 3$};
\node[below] at (4,-0.2){$\scriptstyle 4$};
\node[below] at (5,-0.2){$\scriptstyle 5$};
\node[below] at (6,-0.2){$\scriptstyle 6$};
\node[below] at (7,-0.2){$\scriptstyle 7$};
\node[below] at (8,-0.2){$\scriptstyle 8$};
\node[below] at (9,-0.2){$\scriptstyle 9$};
\node[below] at (10,-0.2){$\scriptstyle 10$};
\node[below] at (11,-0.2){$\scriptstyle 11$};
\node[below] at (12,-0.2){$\scriptstyle 12$};
\node[below] at (13,-0.2){$\scriptstyle 13$};
\node[below] at (14,-0.2){$\scriptstyle 14$};
\end{tikzpicture}
}\\
{}
\B^2(\brho, \bkap) &=
\hackcenter{
\begin{tikzpicture}[scale=0.46]
\draw[thick, white] (0,0.8);
\draw[thick, lightgray, dotted] (14.5,0)--(15.5,0);
\draw[thick, black, dotted] (-18.5,0)--(-17.5,0);
\draw[thick, lightgray ] (-17.5,0)--(14.5,0);
\blackdot(10,0);
\blackdot(6,0);
\blackdot(2,0);
\blackdot(0,0);
\blackdot(-2,0);
\blackdot(-4,0);
\blackdot(-5,0);
\blackdot(-6,0);
\blackdot(-8,0);
\blackdot(-9,0);
\blackdot(-10,0);
\blackdot(-12,0);
\blackdot(-13,0);
\blackdot(-14,0);
\blackdot(-15,0);
\blackdot(-16,0);
\blackdot(-17,0);
\node[below] at (-17,-0.2){$\scriptstyle \textup{-}17$};
\node[below] at (-16,-0.2){$\scriptstyle \textup{-}16$};
\node[below] at (-15,-0.2){$\scriptstyle \textup{-}15$};
\node[below] at (-14,-0.2){$\scriptstyle \textup{-}14$};
\node[below] at (-13,-0.2){$\scriptstyle \textup{-}13$};
\node[below] at (-12,-0.2){$\scriptstyle \textup{-}12$};
\node[below] at (-11,-0.2){$\scriptstyle \textup{-}11$};
\node[below] at (-10,-0.2){$\scriptstyle \textup{-}10$};
\node[below] at (-9,-0.2){$\scriptstyle \textup{-}9$};
\node[below] at (-8,-0.2){$\scriptstyle \textup{-}8$};
\node[below] at (-7,-0.2){$\scriptstyle \textup{-}7$};
\node[below] at (-6,-0.2){$\scriptstyle \textup{-}6$};
\node[below] at (-5,-0.2){$\scriptstyle \textup{-}5$};
\node[below] at (-4,-0.2){$\scriptstyle \textup{-}4$};
\node[below] at (-3,-0.2){$\scriptstyle \textup{-}3$};
\node[below] at (-2,-0.2){$\scriptstyle \textup{-}2$};
\node[below] at (-1,-0.2){$\scriptstyle \textup{-}1$};
\node[below] at (0,-0.2){$\scriptstyle 0$};
\node[below] at (1,-0.2){$\scriptstyle 1$};
\node[below] at (2,-0.2){$\scriptstyle 2$};
\node[below] at (3,-0.2){$\scriptstyle 3$};
\node[below] at (4,-0.2){$\scriptstyle 4$};
\node[below] at (5,-0.2){$\scriptstyle 5$};
\node[below] at (6,-0.2){$\scriptstyle 6$};
\node[below] at (7,-0.2){$\scriptstyle 7$};
\node[below] at (8,-0.2){$\scriptstyle 8$};
\node[below] at (9,-0.2){$\scriptstyle 9$};
\node[below] at (10,-0.2){$\scriptstyle 10$};
\node[below] at (11,-0.2){$\scriptstyle 11$};
\node[below] at (12,-0.2){$\scriptstyle 12$};
\node[below] at (13,-0.2){$\scriptstyle 13$};
\node[below] at (14,-0.2){$\scriptstyle 14$};
\end{tikzpicture}
}
\end{align*}
The \(\bkap\)-beta numbers for \(\bzeta\) are:
\begin{align*}
{}
\B^1(\bzeta, \bkap) &=
\hackcenter{
\begin{tikzpicture}[scale=0.46]
\draw[thick, white] (0,0.8);
\draw[thick, lightgray, dotted] (14.5,0)--(15.5,0);
\draw[thick, black, dotted] (-18.5,0)--(-17.5,0);
\draw[thick, lightgray ] (-17.5,0)--(14.5,0);
\blackdot(14,0);
\blackdot(12,0);
\blackdot(10,0);
\blackdot(8,0);
\blackdot(6,0);
\blackdot(2,0);
\blackdot(0,0);
\blackdot(-2,0);
\blackdot(-5,0);
\blackdot(-6,0);
\blackdot(-8,0);
\blackdot(-9,0);
\blackdot(-10,0);
\blackdot(-12,0);
\blackdot(-13,0);
\blackdot(-14,0);
\blackdot(-15,0);
\blackdot(-16,0);
\blackdot(-17,0);
\node[below] at (-17,-0.2){$\scriptstyle \textup{-}17$};
\node[below] at (-16,-0.2){$\scriptstyle \textup{-}16$};
\node[below] at (-15,-0.2){$\scriptstyle \textup{-}15$};
\node[below] at (-14,-0.2){$\scriptstyle \textup{-}14$};
\node[below] at (-13,-0.2){$\scriptstyle \textup{-}13$};
\node[below] at (-12,-0.2){$\scriptstyle \textup{-}12$};
\node[below] at (-11,-0.2){$\scriptstyle \textup{-}11$};
\node[below] at (-10,-0.2){$\scriptstyle \textup{-}10$};
\node[below] at (-9,-0.2){$\scriptstyle \textup{-}9$};
\node[below] at (-8,-0.2){$\scriptstyle \textup{-}8$};
\node[below] at (-7,-0.2){$\scriptstyle \textup{-}7$};
\node[below] at (-6,-0.2){$\scriptstyle \textup{-}6$};
\node[below] at (-5,-0.2){$\scriptstyle \textup{-}5$};
\node[below] at (-4,-0.2){$\scriptstyle \textup{-}4$};
\node[below] at (-3,-0.2){$\scriptstyle \textup{-}3$};
\node[below] at (-2,-0.2){$\scriptstyle \textup{-}2$};
\node[below] at (-1,-0.2){$\scriptstyle \textup{-}1$};
\node[below] at (0,-0.2){$\scriptstyle 0$};
\node[below] at (1,-0.2){$\scriptstyle 1$};
\node[below] at (2,-0.2){$\scriptstyle 2$};
\node[below] at (3,-0.2){$\scriptstyle 3$};
\node[below] at (4,-0.2){$\scriptstyle 4$};
\node[below] at (5,-0.2){$\scriptstyle 5$};
\node[below] at (6,-0.2){$\scriptstyle 6$};
\node[below] at (7,-0.2){$\scriptstyle 7$};
\node[below] at (8,-0.2){$\scriptstyle 8$};
\node[below] at (9,-0.2){$\scriptstyle 9$};
\node[below] at (10,-0.2){$\scriptstyle 10$};
\node[below] at (11,-0.2){$\scriptstyle 11$};
\node[below] at (12,-0.2){$\scriptstyle 12$};
\node[below] at (13,-0.2){$\scriptstyle 13$};
\node[below] at (14,-0.2){$\scriptstyle 14$};
\end{tikzpicture}
}\\
{}
\B^2(\bzeta, \bkap) &=
\hackcenter{
\begin{tikzpicture}[scale=0.46]
\draw[thick, white] (0,0.8);
\draw[thick, lightgray, dotted] (14.5,0)--(15.5,0);
\draw[thick, black, dotted] (-18.5,0)--(-17.5,0);
\draw[thick, lightgray ] (-17.5,0)--(14.5,0);
\blackdot(14,0);
\blackdot(10,0);
\blackdot(4,0);
\blackdot(2,0);
\blackdot(-1,0);
\blackdot(-2,0);
\blackdot(-4,0);
\blackdot(-5,0);
\blackdot(-6,0);
\blackdot(-8,0);
\blackdot(-10,0);
\blackdot(-12,0);
\blackdot(-13,0);
\blackdot(-14,0);
\blackdot(-15,0);
\blackdot(-16,0);
\blackdot(-17,0);
\node[below] at (-17,-0.2){$\scriptstyle \textup{-}17$};
\node[below] at (-16,-0.2){$\scriptstyle \textup{-}16$};
\node[below] at (-15,-0.2){$\scriptstyle \textup{-}15$};
\node[below] at (-14,-0.2){$\scriptstyle \textup{-}14$};
\node[below] at (-13,-0.2){$\scriptstyle \textup{-}13$};
\node[below] at (-12,-0.2){$\scriptstyle \textup{-}12$};
\node[below] at (-11,-0.2){$\scriptstyle \textup{-}11$};
\node[below] at (-10,-0.2){$\scriptstyle \textup{-}10$};
\node[below] at (-9,-0.2){$\scriptstyle \textup{-}9$};
\node[below] at (-8,-0.2){$\scriptstyle \textup{-}8$};
\node[below] at (-7,-0.2){$\scriptstyle \textup{-}7$};
\node[below] at (-6,-0.2){$\scriptstyle \textup{-}6$};
\node[below] at (-5,-0.2){$\scriptstyle \textup{-}5$};
\node[below] at (-4,-0.2){$\scriptstyle \textup{-}4$};
\node[below] at (-3,-0.2){$\scriptstyle \textup{-}3$};
\node[below] at (-2,-0.2){$\scriptstyle \textup{-}2$};
\node[below] at (-1,-0.2){$\scriptstyle \textup{-}1$};
\node[below] at (0,-0.2){$\scriptstyle 0$};
\node[below] at (1,-0.2){$\scriptstyle 1$};
\node[below] at (2,-0.2){$\scriptstyle 2$};
\node[below] at (3,-0.2){$\scriptstyle 3$};
\node[below] at (4,-0.2){$\scriptstyle 4$};
\node[below] at (5,-0.2){$\scriptstyle 5$};
\node[below] at (6,-0.2){$\scriptstyle 6$};
\node[below] at (7,-0.2){$\scriptstyle 7$};
\node[below] at (8,-0.2){$\scriptstyle 8$};
\node[below] at (9,-0.2){$\scriptstyle 9$};
\node[below] at (10,-0.2){$\scriptstyle 10$};
\node[below] at (11,-0.2){$\scriptstyle 11$};
\node[below] at (12,-0.2){$\scriptstyle 12$};
\node[below] at (13,-0.2){$\scriptstyle 13$};
\node[below] at (14,-0.2){$\scriptstyle 14$};
\end{tikzpicture}
}
\end{align*}
It is straightforward to check from these \(\bkap\)-beta numbers that \(\blam\) and \(\bzeta\) are \((\bkap, \theta)\)-RoCK multipartitions. It also may be seen that \(\brho\) satisfies \cref{bigmultiprop}(iii) and thus is a \(\bkap\)-core, and so 
\begin{align*}
\{\brho\} = \Lambda^{\bkap}_+(24 \alpha_0 + 22 \alpha_1 + 29 \alpha_2 + 22 \alpha_3)
\end{align*}
is a \(\theta\)-RoCK block. 

On the other hand, take
\begin{align*}
\btau = ((13,12,11,10,9,6,5,4,3^3,1^2)\mid(15,8,7,6,4^2,3^3,2,1)),
\end{align*}
so the \(\bkap\)-beta numbers for \(\btau\) are:
\begin{align*}
{}
\B^1(\btau, \bkap) &=
\hackcenter{
\begin{tikzpicture}[scale=0.46]
\draw[thick, white] (0,0.8);
\draw[thick, lightgray, dotted] (14.5,0)--(15.5,0);
\draw[thick, black, dotted] (-18.5,0)--(-17.5,0);
\draw[thick, lightgray ] (-17.5,0)--(14.5,0);
\blackdot(14,0);
\blackdot(12,0);
\blackdot(10,0);
\blackdot(8,0);
\blackdot(6,0);
\blackdot(2,0);
\blackdot(0,0);
\blackdot(-2,0);
\blackdot(-4,0);
\blackdot(-5,0);
\blackdot(-6,0);
\blackdot(-9,0);
\blackdot(-10,0);
\blackdot(-12,0);
\blackdot(-13,0);
\blackdot(-14,0);
\blackdot(-15,0);
\blackdot(-16,0);
\blackdot(-17,0);
\node[below] at (-17,-0.2){$\scriptstyle \textup{-}17$};
\node[below] at (-16,-0.2){$\scriptstyle \textup{-}16$};
\node[below] at (-15,-0.2){$\scriptstyle \textup{-}15$};
\node[below] at (-14,-0.2){$\scriptstyle \textup{-}14$};
\node[below] at (-13,-0.2){$\scriptstyle \textup{-}13$};
\node[below] at (-12,-0.2){$\scriptstyle \textup{-}12$};
\node[below] at (-11,-0.2){$\scriptstyle \textup{-}11$};
\node[below] at (-10,-0.2){$\scriptstyle \textup{-}10$};
\node[below] at (-9,-0.2){$\scriptstyle \textup{-}9$};
\node[below] at (-8,-0.2){$\scriptstyle \textup{-}8$};
\node[below] at (-7,-0.2){$\scriptstyle \textup{-}7$};
\node[below] at (-6,-0.2){$\scriptstyle \textup{-}6$};
\node[below] at (-5,-0.2){$\scriptstyle \textup{-}5$};
\node[below] at (-4,-0.2){$\scriptstyle \textup{-}4$};
\node[below] at (-3,-0.2){$\scriptstyle \textup{-}3$};
\node[below] at (-2,-0.2){$\scriptstyle \textup{-}2$};
\node[below] at (-1,-0.2){$\scriptstyle \textup{-}1$};
\node[below] at (0,-0.2){$\scriptstyle 0$};
\node[below] at (1,-0.2){$\scriptstyle 1$};
\node[below] at (2,-0.2){$\scriptstyle 2$};
\node[below] at (3,-0.2){$\scriptstyle 3$};
\node[below] at (4,-0.2){$\scriptstyle 4$};
\node[below] at (5,-0.2){$\scriptstyle 5$};
\node[below] at (6,-0.2){$\scriptstyle 6$};
\node[below] at (7,-0.2){$\scriptstyle 7$};
\node[below] at (8,-0.2){$\scriptstyle 8$};
\node[below] at (9,-0.2){$\scriptstyle 9$};
\node[below] at (10,-0.2){$\scriptstyle 10$};
\node[below] at (11,-0.2){$\scriptstyle 11$};
\node[below] at (12,-0.2){$\scriptstyle 12$};
\node[below] at (13,-0.2){$\scriptstyle 13$};
\node[below] at (14,-0.2){$\scriptstyle 14$};
\end{tikzpicture}
}\\
{}
\B^2(\btau, \bkap) &=
\hackcenter{
\begin{tikzpicture}[scale=0.46]
\draw[thick, white] (0,0.8);
\draw[thick, lightgray, dotted] (14.5,0)--(15.5,0);
\draw[thick, black, dotted] (-18.5,0)--(-17.5,0);
\draw[thick, lightgray ] (-17.5,0)--(14.5,0);
\blackdot(14,0);
\blackdot(6,0);
\blackdot(4,0);
\blackdot(2,0);
\blackdot(-1,0);
\blackdot(-2,0);
\blackdot(-4,0);
\blackdot(-5,0);
\blackdot(-6,0);
\blackdot(-8,0);
\blackdot(-10,0);
\blackdot(-12,0);
\blackdot(-13,0);
\blackdot(-14,0);
\blackdot(-15,0);
\blackdot(-16,0);
\blackdot(-17,0);
\node[below] at (-17,-0.2){$\scriptstyle \textup{-}17$};
\node[below] at (-16,-0.2){$\scriptstyle \textup{-}16$};
\node[below] at (-15,-0.2){$\scriptstyle \textup{-}15$};
\node[below] at (-14,-0.2){$\scriptstyle \textup{-}14$};
\node[below] at (-13,-0.2){$\scriptstyle \textup{-}13$};
\node[below] at (-12,-0.2){$\scriptstyle \textup{-}12$};
\node[below] at (-11,-0.2){$\scriptstyle \textup{-}11$};
\node[below] at (-10,-0.2){$\scriptstyle \textup{-}10$};
\node[below] at (-9,-0.2){$\scriptstyle \textup{-}9$};
\node[below] at (-8,-0.2){$\scriptstyle \textup{-}8$};
\node[below] at (-7,-0.2){$\scriptstyle \textup{-}7$};
\node[below] at (-6,-0.2){$\scriptstyle \textup{-}6$};
\node[below] at (-5,-0.2){$\scriptstyle \textup{-}5$};
\node[below] at (-4,-0.2){$\scriptstyle \textup{-}4$};
\node[below] at (-3,-0.2){$\scriptstyle \textup{-}3$};
\node[below] at (-2,-0.2){$\scriptstyle \textup{-}2$};
\node[below] at (-1,-0.2){$\scriptstyle \textup{-}1$};
\node[below] at (0,-0.2){$\scriptstyle 0$};
\node[below] at (1,-0.2){$\scriptstyle 1$};
\node[below] at (2,-0.2){$\scriptstyle 2$};
\node[below] at (3,-0.2){$\scriptstyle 3$};
\node[below] at (4,-0.2){$\scriptstyle 4$};
\node[below] at (5,-0.2){$\scriptstyle 5$};
\node[below] at (6,-0.2){$\scriptstyle 6$};
\node[below] at (7,-0.2){$\scriptstyle 7$};
\node[below] at (8,-0.2){$\scriptstyle 8$};
\node[below] at (9,-0.2){$\scriptstyle 9$};
\node[below] at (10,-0.2){$\scriptstyle 10$};
\node[below] at (11,-0.2){$\scriptstyle 11$};
\node[below] at (12,-0.2){$\scriptstyle 12$};
\node[below] at (13,-0.2){$\scriptstyle 13$};
\node[below] at (14,-0.2){$\scriptstyle 14$};
\end{tikzpicture}
}
\end{align*}
Note that \(-5 \in \B^1(\btau, \bkap)\), and \(-8 \notin \B^1(\btau, \bkap)\). Since \(\overline{-5} = 3 = \theta_2\) and \(\overline{-8} = 0 = \theta_3\), \(\btau\) fails the condition of \cref{specialdefs}(iii) and hence is not a \((\bkap, \theta)\)-RoCK multipartition. Noting that 
\begin{align*}
\bzeta, \btau \in \Lambda^{\bkap}_+(34 \alpha_0 + 32 \alpha_1 + 39 \alpha_2 + 32 \alpha_3),
\end{align*}
it follows that this block is {\em not} \(\theta\)-RoCK, although it does contain the \((\bkap, \theta)\)-RoCK multipartition \(\bzeta\).
\end{Example}

\subsection{RoCK blocks}\label{subsec:rockblocks}
We now explain the connection between the RoCK blocks defined above, and the RoCK/Rouquier blocks defined in the literature.
It is straightforward to check that, if \(\theta = (0,1,\dots, e-1)\) is the trivial residue permutation, then the condition of a block \(\Lambda^{\bkap}_+(\omega)\) (or equivalently, the associated cyclotomic KLR algebra) being a \(\theta\)-RoCK block is equivalent to \(\Lambda^{\bkap}_+(\omega)\) being {\em Rouquier} in the sense of \cite{CK02} (in level one), and \cite{Lyle22} (in higher levels).
Webster \cite{websterScopes} has explained that, for arbitrary levels, Rouquier blocks belong to a large Scopes equivalence class (see \cite{scopes, d'ascopes}) of blocks called {\em RoCK blocks}.
Below we briefly explain that these RoCK blocks (in the sense of Webster) are exactly our RoCK blocks (in the sense of Definition~\ref{specialdefs}), justifying our use of the moniker, see \cref{RoCKthetaeq}.

\begin{Definition}\label{specpars}
Let \(\blam \in \Lambda_+^{\bkap}\), and \(i \in \ZZ_e\). We define a new multipartition \(\Psi_i \blam \in \Lambda_+^{\bkap}\) by setting:
\begin{align*}
\B^r_j(\Psi_i \blam, \bkap) =
\begin{cases}
\B^r_j( \blam, \bkap) & \textup{if }j \neq i, i+1;\\
\B^r_{i+1}( \blam, \bkap) -1 &\textup{if } j=i;\\
\B^r_{i}( \blam, \bkap) +1 &\textup{if } j=i+1,
\end{cases}
\end{align*}
for all \(r \in [1,\ell]\), \(j \in \ZZ_e\).
Here, by \(\B^r_{i+1}( \blam, \bkap) -1\), we mean the set of beta numbers obtained from \(\B^r_{i+1}( \blam, \bkap)\) by subtracting $1$ from each element, and similarly for \(\B^r_{i}( \blam, \bkap) +1\).
In effect, \(\Psi_i \blam\) is achieved by removing all removable nodes of residue \(i+1\) from \(\blam\), whilst adding all addable nodes of residue \(i+1\). We note that \(\Psi_i \blam\) does not necessarily have the same content as \(\blam\).
\end{Definition}

The following result follows from a \(\bkap\)-beta number reinterpretation of \cite[Lemma 3.2, \S4.2.1]{websterScopes}.

\begin{Proposition}\label{Psiblock}
Let \(\Lambda^{\bkap}_+(\omega)\) be a block, and assume that \(i \in \ZZ_e\) is such that for all \(\blam \in \Lambda^{\bkap}_+(\omega)\) and \(r \in [1,\ell]\), \(x \in \B^r_{i+1}(\blam, \bkap)\) implies \(x-1 \in \B^r_{i}(\blam, \bkap)\). Then
\begin{align*}
\Psi_i \Lambda^{\bkap}_+(\omega) = \{ \Psi_i \blam \mid \blam \in \Lambda^{\bkap}_+(\omega)\} = \Lambda^{\bkap}_+(\omega') 
\end{align*}
for some block \(\Lambda^{\bkap}_+(\omega') \).
\end{Proposition}

We write \(\Lambda^{\bkap}_+(\omega) \sim_i \Lambda^{\bkap}_+(\omega')\) to indicate the relationship between blocks under $\Psi_i$ described in \cref{Psiblock}. In view of \S\ref{remribsexp}, the hypotheses of \cref{Psiblock} apply precisely when there are no removable nodes of residue \(i+1\) in any multipartition in the block \(\Lambda^{\bkap}_+(\omega)\).

\begin{Lemma}\label{transtheta}
Suppose \(\omega, \omega' \in  \ZZ_{\geq 0}I\) and \(\Lambda^{\bkap}_+(\omega) \sim_i \Lambda^{\bkap}_+(\omega')\). Let \(\theta\) be a residue permutation such that \(\theta_a = i\) and \(\theta_b = i+1\). Let \((a,b) \in \mathfrak{S}_e\). 
Then:
\begin{enumerate}
\item If \(\Lambda^{\bkap}_+(\omega)\) is \(\theta\)-RoCK and \(a<b\), then \(\Lambda^{\bkap}_+(\omega)\) is \(((a,b)\cdot \theta)\)-RoCK as well.
\item If \(\Lambda^{\bkap}_+(\omega)\) is \(\theta\)-RoCK and \(a>b\), then \(\Lambda^{\bkap}_+(\omega')\) is \(((a,b)\cdot \theta)\)-RoCK.
\item If \(\Lambda^{\bkap}_+(\omega')\) is \(\theta\)-RoCK, then \(\Lambda^{\bkap}_+(\omega)\) is \(((a,b)\cdot \theta)\)-RoCK.
\end{enumerate}
\end{Lemma}
\begin{proof}
It is straightforward, albeit tedious, to check that these implications follow from \cref{specpars}(iii) and the hypotheses of \cref{Psiblock}. Full details can be found in the \texttt{arXiv} version of the paper as explained in \S\ref{SS:ArxivVersion}.
\end{proof}

\begin{answer}
Details: \\

(i)
Throughout the proof, we suppress \((\blam, \bkap)\) in our notation for beta numbers and \(M^r_i(\blam, \bkap)\) whenever it is clear from context.
Assume \(\Lambda^{\bkap}_+(\omega)\) is \(\theta\)-RoCK, and \(a<b\). Set \(\theta' = (a,b) \cdot \theta\). We now show that \(\Lambda^{\bkap}_+(\omega)\) is \(\theta'\)-RoCK as well.
Then \(\theta'_a = \theta_b\), \(\theta'_b = \theta_a\), and \(\theta_c = \theta'_c\) for all \(c \neq a,b\).
Assume by way of contradiction that there exist \(\la \in \Lambda^{\bkap}_+(\omega)\), \(y \in \B^r\), and \(x \notin \B^r\), where \(x<y\), \(\bar x = \theta'_z\), and \(\bar y = \theta'_w\), with \(z>w\). Since \( \Lambda^{\bkap}_+(\omega)\) is \(\theta\)-RoCK and \(\theta_c = \theta'_c\) for all \(c \neq a,b\), it must be that \(\{z,w\} \cap \{a,b\} \neq \varnothing\). 
We note that it follows from the \(\theta\)-RoCK conditions and the hypotheses of \cref{Psiblock} that 
\begin{align}\label{Mstuff1}
M^r_i & \in \{M^r_{i+1} - 1, M^r_{i+1} + e -1\}, \qquad
\B^r_i = M_r^i - e\ZZ_{\geq 0} \qquad
\B^r_{i+1} = M^r_{i+1} - e\ZZ_{\geq 0}.
\end{align}

There are only a few possibilities to consider.

(1) Say \(z=b, a = w\). Then \(\bar x = \theta'_z = \theta'_b = \theta_a\) and \(\bar y = \theta'_w = \theta'_a = \theta_b\). Then \(M_i^r < x <y \leq M_{i+1}^r\), so \(M_i^r < M_{i+1}^r - 1\), a contradiction of (\ref{Mstuff1}).

(2) Say \(z= a\), \(w \neq b\). Then \(\bar x = \theta'_z = \theta'_a = \theta_b\), and \(\bar y = \theta'_w = \theta_w\) and \(b>a>w\). This contradicts the fact that \(\Lambda_+^{\bkap}(\omega)\) is \(\theta\)-RoCK.

(3) Say \(z = b\), \(w \neq a\). Then \(\bar x = \theta'_z = \theta'_b = \theta_a\), and \(\bar y = \theta'_w = \theta_w\). Since \(x \notin \B^r_i\), we know \(x-e \geq M^r_i\). Since \(b=z>w\), it must be that \(t \in \B_{i+1}^r\) for all \(\overline{t} = i+1\) with \(t<y\). But then \(x+1 \in \B^r_{i+1}\), so \(M^r_{i+1} \geq x+1 \geq M_i^r + e + 1\). Then \(M_i^r \leq M^r_{i+1} - e - 1\), which contradicts (\ref{Mstuff1}).

(4) Say \(z \neq a\), \(w = b\). Then \(\bar x = \theta'_z = \theta_z\) and \(\bar y = \theta'_w = \theta'_b = \theta_a\). Since \(z>w = b > a\), this contradicts the fact that \(\Lambda_+^{\bkap}(\omega)\) is \(\theta\)-RoCK.

(5) Say \(z \neq b\), \(w = a\). Then \(z>w = a\). Then \(\bar x = \theta'_z = \theta_z\) and \(\bar y = \theta'_w = \theta'_a = \theta_b\). Since \(y \in \B^r_{i+1}\), we know \(M_{i+1}^r \geq y\). Since \(z>a\), it must be that \(t \notin \B^r_i\) for all \( \bar t = i\), \(t>x\). Then \(M_i^r < x < y \leq M^r_{i+1}\). Then \(M_i^r < M_{i+1}^r - 1\), which contradicts (\ref{Mstuff1}).

Thus, in any case, we have that \(\Lambda^{\bkap}_+(\omega)\) is \(\theta'\)-RoCK.

(ii) Now assume \(\Lambda^{\bkap}_+(\omega)\) is \(\theta\)-RoCK, \(\theta_a = i\), \(\theta_b = i+1\), and \(b < a\). We show that \(\Psi_i \Lambda^{\bkap}_+(\omega)\) is \(((a, b) \cdot \theta)\)-RoCK. Take \(\theta' = (a, b) \cdot \theta\). Then \(\theta'_a = \theta_b = i+1\) and \(\theta'_b = \theta_a = i\). Consider some \(\bmu = \Psi_i \blam\). Assume by way of contradiction that there exist \(y \in \B^r(\bmu, \bkap)\) and \(x \notin \B^r(\bmu,\bkap)\), where \(\bar x = \theta'_z\) and \(\bar y = \theta'_w\), with \(z>w\). Again, it must be that \(\{z,w\} \cap \{a,b\} \neq \varnothing\).

Cases to consider:

(1) Say \(z=a\), \(w = b\). Then \(\bar x = \theta'_a = \theta_b = i+1\) and \(\bar y = \theta'_b = \theta_a = i\). Then \(x=x'+1\) and \(y = y'-1\) for some \(x' \notin \B^r(\blam,\bkap)\) and \(y' \in \B^r(\blam,\bkap)\). Then \(\overline {x'} = i = \theta_a\) and \(\overline {y'} = i+1 = \theta_b\). But \(x'<y'\), which contradicts that \(\Lambda^{\bkap}_+(\omega)\) is \(\theta\)-RoCK.

(2) Say \(z=a\), \(w \neq b\). Then \(\bar x = \theta'_a = \theta_b = i+1\) and \(\bar y = \theta'_w = \theta_w\). Then \(a=z>w\). We have \(x = x' +1\) for some \(x' \notin \B^r(\blam,\bkap)\) with \(\overline {x'} = i = \theta_a\). But since \(x'<y \in \B^r(\blam,\bkap)\), this contradicts that \(\Lambda^{\bkap}_+(\omega)\) is \(\theta\)-RoCK.

(3) Say \(z = b\), \(w \neq a\). Then \(a> b = z>w \). Then \(\bar x = \theta'_b = \theta_a = i\) and \(\bar y = \theta'_w = \theta_w \neq i, i+1\). Note then that \(x + 1 < y\). We have \(x = x' -1\) for some \(x' \notin \B^r(\blam,\bkap)\) with \(\overline{x'} = i+1 = \theta_b\). But since \(x'<y \in \B^r(\blam,\bkap)\), this contradicts that \(\Lambda^{\bkap}_+(\omega)\) is \(\theta\)-RoCK.

(4) Say \(z \neq a\), \(w = b\). Then \(z> w = b\). Then \(\bar x = \theta'_z = \theta_z\) and \(\bar y = \theta'_w = \theta'_b = \theta_a = i\). Then \(y = y'-1\) for some \(y' \in \B^r(\blam)\), where \(\overline{y'} = i+1 = \theta_b\). Then \( y' > x \notin \B^r(\blam,\bkap)\), which contradicts that \(\Lambda^{\bkap}_+(\omega)\) is \(\theta\)-RoCK.
 
(5) Say \(z \neq b\), \(w =a\). Then \(z>w = a>b\). Then \(\bar x = \theta'_z = \theta_z \neq i, i+1\) and \(\bar y = \theta'_w = \theta'_a = \theta_b = i+1\).  Note then that \(x+1< y\). Then \(y = y'+1\) for some \(y' \in \B^r(\blam,\bkap)\), where \(\overline{y'} = i = \theta_a\). But \(  y' > x \notin \B^r(\blam,\bkap) \), which contradicts that \(\Lambda^{\bkap}_+(\omega)\) is \(\theta\)-RoCK.
 
Therefore \(\Lambda^{\bkap}_+(\omega')\) is \(\theta'\)-RoCK, completing the proof.

 (iii) Assume that \(\Lambda^{\bkap}_+(\omega')\) is \(\theta\)-RoCK. 
If \(\Psi_i \blam \neq \blam\) for some \(\blam \in \Lambda_+^{\bkap}(\omega)\), then there is some \(x \in \B^r_i(\blam, \bkap)\) such that \(x +1 \notin \B^r_{i+1}(\blam, \bkap)\). Therefore \(x \notin \B^r_i(\Psi_i \blam, \bkap)\) and \(x+1 \in \B^r_{i+1}(\Psi_i \blam, \bkap)\). It follows then, since \(\Lambda^{\bkap}_+(\omega')\) is \(\theta\)-RoCK, that \(a<b\). On the other hand, if \( \Psi_i \blam = \blam\) for all \(\blam \in \Lambda_+^{\bkap}(\omega)\), it follows that \(\Lambda^{\bkap}_+(\omega) = \Psi_i \Lambda_+^{\bkap}(\omega) = \Lambda^{\bkap}_+(\omega')\) is \(\theta\)-RoCK. If \(a>b\), we have then that \(\Lambda^{\bkap}_+(\omega)\) is \((a,b) \cdot \theta\)-RoCK as well (following the argument in the first part).

Therefore we may assume that \(a<b\).
We show that \(\Lambda_+^{\bkap}(\omega)\) is \(\theta'\)-RoCK.
 
Take \(\theta' = (a, b) \cdot \theta\). Then \(\theta'_a = \theta_b = i+1\) and \(\theta'_b = \theta_a = i\). Consider some \(\blam \in \Lambda^{\bkap}_+(\omega)\). Assume by way of contradiction that there exist \(x<y\) with \(y \in \B^r(\blam,\bkap)\) and \(x \notin \B^r(\blam,\bkap)\), where \(\bar x = \theta'_z\) and \(\bar y = \theta'_w\), with \(z>w\). Because \(\Psi_i \blam\) is \((\bkap, \theta)\)-RoCK, it must be that \(\{z,w\} \cap \{a,b\} \neq \varnothing\).
 
 We consider cases:
 
 (a) Say \(z=b\), \(w = a\). Then \(\bar x = \theta'_z = \theta'_b = \theta_a = i\) and \(\bar y = \theta'_w =\theta'_a = \theta_b = i+1\). Then there exists some \(y' \in \B_{i+1}^r(\blam, \bkap)\), \(x' \notin \B_i^r(\blam, \bkap)\) such that \(c=y'-x' >0\) is minimal. If \(y' \neq x'+1\), this implies that \(x'+1 \notin \B^r_{i+1}(\Psi_i\blam, \bkap)\), and \(y'-1 \in \B^r_i(\Psi_i\blam, \bkap)\), with \(y'-1> x'+1\). Since \(\overline{y' -1} = i = \theta_a\) and \(\overline{x+1} = i+1 = \theta_b\), this contradicts that \(\Psi_i \blam\) is \((\bkap, \theta)\)-RoCK. On the other hand, if \(y' = x'+1\), we have \(y' \in \B^r_{i+1}(\blam, \bkap)\) yet \(y' -1 \notin \B_i^r(\blam, \bkap)\), which contradicts the hypotheses of \cref{Psiblock}.
 
 (b) Say \(z=a\), \(w \neq b\). Then \(\bar x = \theta'_a = \theta_b = i+1\) and \(\bar y = \theta'_w = \theta_w \neq i, i+1\). Then \(b>a=z>w\). We have \(x' = x-1 \notin \B^r_i(\Psi_i \blam, \bkap)\) and \(y \in \B^r(\Psi_i \blam, \bkap)\) with \(x'<y\), and \(\overline{x'} = \theta_a\), \(\bar y = \theta_w\). This contradicts that \(\Psi_i \blam\) is \((\bkap, \theta)\)-RoCK.
 
 (c) Say \(z = b\), \(w \neq a\). Then \(b = z>w \). Then \(\bar x = \theta'_b = \theta_a = i\) and \(\bar y = \theta'_w = \theta_w \neq i, i+1\). Then \(x' = x+1 \notin \B^r_{i+1}(\Psi_i \blam, \bkap)\) and \(y \in \B^r(\Psi_i \blam, \bkap)\) with \(x'<y\), and \(\overline{x'} = \theta_b\), \(\bar y = \theta_w\). This contradicts that \(\Psi_i \blam\) is \((\bkap, \theta)\)-RoCK.
 
  (d) Say \(z \neq a\), \(w = b\). Then \(z> w = b\). Then \(\bar x = \theta'_z = \theta_z\) and \(\bar y = \theta'_w = \theta'_b = \theta_a = i\). Then \(y' = y+1 \in \B^r_{i+1}(\Psi_i \blam, \bkap)\) and \(x \notin \B^r(\Psi_i \blam, \bkap)\) with \(x<y'\) and \(\bar x = \theta'_z = \theta_z\) and \(\overline{y'} = i+1 = \theta_b\). This contradicts that \(\Psi_i \blam\) is \((\bkap, \theta)\)-RoCK.
  
   (e) Say \(z \neq b\), \(w =a\). Then \(z>w = a\). Then \(\bar x = \theta'_z = \theta_z \neq i, i+1\) and \(\bar y = \theta'_w = \theta'_a = \theta_b = i+1\).  Then \(y' = y-1 \in \B^r_{i}(\Psi_i \blam, \bkap)\) and  \(x \notin \B^r(\Psi_i \blam, \bkap)\) with \(x < y'\) and \(\bar x = \theta'_z = \theta_z\) and \(\overline{y'} = i = \theta_a\). This contradicts that \(\Psi_i \blam\) is \((\bkap, \theta)\)-RoCK. 
   
  As all cases lead to contradiction, it must be that no such \(x,y\) exist, and therefore \(\Lambda_+^{\bkap}(\omega)\) is \(\theta'\)-RoCK.
   
   \qed
   \\
\end{answer}

The operator \(\Psi_i\) defines a Morita equivalence between cyclotomic KLR algebras associated with these combinatorial blocks, see \cite{websterScopes}. The equivalence relation on blocks generated by the relations \(\{\sim_i \mid i \in \ZZ_e\}\) is called {\em Scopes equivalence}. 
It is explained in \cite[\S4.2.1, Proposition~4.6]{websterScopes} that a block \(\Lambda^{\bkap}_+(\omega)\) is RoCK (in Webster's sense) if and only if it is Scopes equivalent to a Rouquier block. Thus, in view of \cref{transtheta} and the fact that Rouquier blocks are exactly the \((0,1,\dots, e-1)\)-RoCK blocks, we have the following.

\begin{Corollary}\label{RoCKthetaeq}
A block \(\Lambda_+^{\bkap}(\omega)\) is Scopes equivalent to a Rouquier block if and only if \(\Lambda_+^{\bkap}(\omega)\) is \(\theta\)-RoCK for some residue permutation \(\theta\), that is, if and only if it is a RoCK block in the sense of \cref{specialdefs}.
\end{Corollary}

\subsection{Some fundamental results on multicores}
Whether or not a block \(\Lambda^{\bkap}_+(\omega)\) is a core block can be completely determined by considering one of its constituent multipartitions.
As in \cref{specpars}, the sum of an integer \(m\) and a set of beta numbers \(\B(\rho, i) = \B^1(\rho, i)\) denotes the set of beta numbers obtained from \(\B(\rho, i)\) by adding \(m\) to each element.

\begin{Lemma}\label{FayersBase}
Let \(\omega \in \ZZ_{\geq 0}I\) and \(\blam \in \Lambda^{\bkap}_+(\omega)\). Then \(\Lambda^{\bkap}_+(\omega)\) is a core block if and only if there exists some \(i \in [0,e-1]\) and (level one) \(e\)-core partition \(\rho \in \Lambda^i_+\) such that
\begin{align*}
\left \lfloor \frac{j-i}{e} \right \rfloor e + \B(\rho, i) \subseteq \B^r(\blam, \bkap) \subseteq \left \lfloor \frac{j-i + e}{e} \right \rfloor e + \B(\rho, i)
\end{align*}
for all \(j \in [0,e-1]\), \(r \in [1,\ell]\).
\end{Lemma}
\begin{proof}
Follows from \cite[Theorem 3.1]{fay07core}.
\end{proof}

\begin{Proposition}\label{bigmultiprop}
Let \(\omega \in \ZZ_{\geq 0}I\), and let \(\brho \in \Lambda_+^{\bkap}(\omega)\). The following are equivalent:
\begin{enumerate}
\item The multipartition \(\brho\) is a \(\bkap\)-core.
\item We have \(\textup{def}(\brho) = 0\).
\item We have 
\begin{align*}
\B^r(\brho, \bkap) \subseteq 
\B^{s}(\brho, \bkap) \subseteq
e+\B^r(\brho,\bkap) ,
\end{align*}
for all \(r,s \in [1,\ell]\) with \(\overline{\kappa}_r \leq \overline{\kappa}_{s}\).
\item There do not exist \(\bmu, \blam \in \Lambda_+^{\bkap}\) with \(\bmu \subsetneq \brho \subsetneq \blam\) and \(
\cont(\brho/ \bmu) = \cont( \blam/ \brho)
\).
\end{enumerate}
\end{Proposition}
\begin{proof}
The statement (i) \(\iff\) (ii) is given by \cite[Theorem 4.1]{fay06wts}. The statement (i) \(\iff\) (iii) is a slight paraphrase of \cite[Proposition 2.7(2)]{fayerssimcore}. It remains to consider (iv).

((iv) \(\implies\) (i)) Assume \(\brho\) is not a \(\bkap\)-core. Then there exists another \(\bkap\)-multipartition \(\bnu \neq \brho\) such that \(\cont(\bnu) = \cont(\brho)\). Setting \(\bmu = \brho \cap \bnu\), \(\blam = \brho \cup \bnu\), we have \(\bmu \subsetneq \brho \subsetneq \blam\), but
\begin{align*}
\cont(\blam/\brho) &= \cont(\brho \cup \bnu) - \cont(\brho)
=
\cont(\brho) + \cont(\bnu) - \cont(\brho \cap \bnu) - \cont(\brho)\\
&= \cont(\bnu) - \cont(\brho \cap \bnu) =
\cont(\brho) - \cont(\brho \cap \bnu) = \cont(\brho/\bmu).
\end{align*}

((ii) \(\implies\) (iv)) Assume there exist \(\bmu, \blam \in \Lambda_+^{\bkap}\) with \(\bmu \subsetneq \brho \subsetneq \blam\). Recall then that the defects of \(\blam, \bmu, \brho\) are non-negative. A short calculation shows that
 \begin{align*}
 2 \textup{def}(\brho) &=\textup{def}(\blam) + \textup{def}(\bmu) +\sum_{i \in \ZZ_e}  \left(c_i(\brho/\bmu) - c_{i+1}(\brho/\bmu)\right)^2.
 \end{align*}
If the sum \(\sum_{i \in \ZZ_e}  \left(c_i(\brho/\bmu) - c_{i+1}(\brho/\bmu)\right)^2\) above is positive, it follows that \(\textup{def}(\brho)\) is positive. On the other hand, if it is zero, it would follow immediately that \(\textup{def}(\brho/\bmu)\) is positive, and that \(\textup{def}(\brho) = \textup{def}(\bmu) + \textup{def}(\brho/\bmu)\), again guaranteeing that \(\textup{def}(\brho)\) is positive.
\end{proof}

\begin{Corollary}\label{kapcoremulti}
If \(\brho\) is a \(\bkap\)-core, then \(\brho\) is a multicore.
\end{Corollary}
\begin{proof}
Follows immediately from \cref{bigmultiprop}(iii).
\end{proof}

We note that the converse to this corollary is false in general, unlike the level 1 situation in \cref{rem:l1core}.
Any multipartition of positive defect is not a \(\bkap\)-core, but such multipartitions may be multicores. In fact, any core block of positive defect consists of multiple multicores, so that none of them are \(\bkap\)-cores.

\begin{Corollary}\label{coreksep}
If  \(\omega = \cont(\brho)\) for some \(\bkap\)-core \(\brho\), then the pair \((\omega, \beta)\) is \(\bkap\)-separable for any \(\beta \in \ZZ_{\geq 0}I\).
\end{Corollary}
\begin{proof}
Follows immediately from \cref{bigmultiprop}(i),(iv).
\end{proof}

\begin{Lemma}\label{onediff}
If \(\brho\) is a \(\bkap\)-core, then 
\(
|h_{i,j}^r(\brho, \bkap) - h_{i,j}^s(\brho, \bkap)| \leq 1
\)
for all \(i,j \in [0,e-1]\) and \(r, s \in [1,\ell]\).
\end{Lemma}
\begin{proof}
Assume without loss of generality that \(\bar \kappa_r \leq \bar \kappa_s\). Then \(\B^r(\brho,\bkap) \subseteq \B^s(\brho,\bkap) \subseteq e+\B^r(\brho,\bkap)\). Therefore \(M_i^r(\brho,\bkap) \leq M_i^s (\brho,\bkap)\leq M_i^r(\brho,\bkap) +e\) and \(M_j^r(\brho,\bkap) \leq M_j^s(\brho,\bkap) \leq M_j^r(\brho,\bkap) +e\). Thus
\begin{align*}
M_j^s(\brho,\bkap) - M_i^s(\brho,\bkap) \leq M^r_j (\brho,\bkap)- M^r_i(\brho,\bkap) + e, 
\end{align*}
and
\begin{align*}
M_j^r(\brho,\bkap) - M_i^r(\brho,\bkap)  \leq M^s_j (\brho,\bkap)- M^s_i(\brho,\bkap) +e,
\end{align*}
and from this it follows that
\begin{align*}
h_{i,j}^s(\brho,\bkap) \leq h_{i,j}^r(\brho,\bkap) + 1,
\qquad
\textup{and}
\qquad
h_{i,j}^r(\brho,\bkap) \leq h_{i,j}^s(\brho,\bkap) + 1,
\end{align*}
yielding the result.
\end{proof}

\begin{Lemma}\label{RoCKcore}
Let \(\brho\) be a multicore, and \(\theta\) a residue permutation. Then  \(\brho\) is \((\bkap,\theta)\)-RoCK if and only if 
\begin{align}\label{kcohdef}
h^r_{\theta_a, \theta_b}(\brho, \bkap) \geq -1
\end{align}
for all \(r \in [1,\ell]\), \(1 \leq a < b \leq e\).
\end{Lemma}
\begin{proof}
\((\implies)\) Assume condition (\ref{kcohdef}) does not hold. Then there exists \(r \in [1,\ell]\) and \(1 \leq a<b \leq e\) such that \(h^r_{\theta_a, \theta_b}(\brho, \bkap) < -1\). It follows then that \(M^r_{\theta_a}(\brho, \bkap) >M^r_{\theta_b}(\brho, \bkap) +e\). Then, taking \(y = M^r_{\theta_a}(\brho, \bkap) \), \(x = M^r_{\theta_b}(\brho, \bkap)  + e\), we have \(y>x\), \(\overline{x} = \theta_b\), \(\overline{y} = \theta_a\), yet \(b>a\), so \(\brho\) is not \((\bkap,\theta)\)-RoCK.

\((\impliedby)\) Assume condition (\ref{kcohdef}) holds, and that \(y \in \B^r(\brho, \bkap)\), \(x \notin \B^r(\brho, \bkap)\), for some \(x<y\) with \(\overline{x} = \theta_a\), \( \overline{y} = \theta_b\). As \(\brho\) is a multicore, we have \(M^r_{\theta_a}(\brho, \bkap)  +e \leq x <y\leq M^r_{\theta_b}(\brho, \bkap) \). If \(a >b\), then 
\begin{align*}
h_{\theta_b, \theta_a}^r(\brho, \bkap) = 
\left \lfloor
\frac{M^r_{\theta_a}(\brho, \bkap)  - M^r_{\theta_b}(\brho, \bkap) }{e}
\right\rfloor
<
\left \lfloor
\frac{(M^r_{\theta_b}(\brho, \bkap) -e) - M^r_{\theta_b}(\brho, \bkap) }{e}
\right\rfloor
=-1,
\end{align*}
a contradiction of (\ref{kcohdef}). Thus \(a \leq b\), and so \(\brho\) is \((\bkap,\theta)\)-RoCK.
\end{proof}

\begin{Proposition}\label{multicoreblockRoCK}
Let \(\omega \in \ZZ_{\geq 0}I\). Then \(\Lambda_+^{\bkap}(\omega)\) is a core block if and only if there exists a residue permutation \(\theta\) such that every \(\brho \in \Lambda_+^{\bkap}(\omega)\) is a \((\bkap, \theta)\)-RoCK multicore. 
\end{Proposition}
\begin{proof}
The `if' direction is immediate from definitions.
For the `only if' direction, assume \(\Lambda_+^{\bkap}(\omega)\) is a core block. 
By \cref{FayersBase}, there exists a sequence \(\balpha = (\alpha_0, \dots, \alpha_{e-1}) \in \ZZ^e\) such that for all \(r \in [1,\ell]\), \(i,j \in [0,e-1]\), and \(\blam \in \Lambda_+^{\bkap}(\rho)\), we have \(M_i^r(\blam, \bkap) - M_j^r(\blam, \bkap) = \alpha_i - \alpha_j + c_{ij}e\), where \(c_{ij} \in \{-1,0,1\}\). Choose a residue permutation \(\theta = (\theta_1, \dots, \theta_e)\) such that \(\alpha_{\theta_a} \leq \alpha_{\theta_b}\) for all \(a<b\). 
Now consider a multicore \(\brho \in \Lambda_+^{\bkap}(\omega)\). For \(r \in [1,\ell]\), \(1 \leq a < b \leq e\), we have:
\begin{align*}
h^r_{\theta_a, \theta_b}(\brho, \bkap) &=
\left \lfloor
\frac{M_{\theta_b}^r(\brho, \bkap) - M_{\theta_a}^r(\brho, \bkap)   }{e}
\right\rfloor
= \left \lfloor
\frac{
\alpha_{\theta_b} - \alpha_{\theta_a} + c_{\theta_b, \theta_a}e
}{e}
\right \rfloor
\geq
\left \lfloor
\frac{0+c_{\theta_b, \theta_a}e}{e}
\right \rfloor
= c_{\theta_b, \theta_a} \geq -1.
\end{align*}
Hence the multicore \(\brho\) is \((\bkap, \theta)\)-RoCK by \cref{RoCKcore}. 
\end{proof}

\cref{kapcoremulti,multicoreblockRoCK} imply the following result.
\begin{Corollary}\label{kapcorerock}
If \(\brho\) is a \(\bkap\)-core, then \(\brho\) is a \((\bkap, \theta)\)-RoCK multicore for some residue permutation \(\theta\).
\end{Corollary}

\subsection{Reindexing and root systems}\label{reindsec}

Fix a multicharge \(\bkap\) of level \(\ell\), and a residue permutation \(\theta = (\theta_1, \dots, \theta_e)\). Let \(\brho \in \Lambda^{\bkap}_+\) be a multicore, and recall the notation of \S\ref{cdefs1}.
Then, for \(t \in [1,e-1]\), set
\begin{align*}
\gamma_t^\theta &:= \alpha( \theta_t + 1, \overline{ \theta_{t+1} - \theta_t}) = \alpha_{\overline{\theta_t +1}} +  \alpha_{\overline{\theta_t +2}} + \dots + \alpha_{\overline{\theta_{t+1}}};\\
h^{r,\theta}_t(\brho, \bkap) &:=h_{\theta_t, \theta_{t+1}}^r(\brho, \bkap).
\end{align*}
More generally, for \(1 \leq a\leq b \leq e-1\), set 
\begin{align}
\gamma^\theta_{[a,b]} &:= \gamma^\theta_a + \gamma^\theta_{a+1} + \dots + \gamma_b^\theta \label{gamseq}\\
h^{r,\theta}_{[a,b]}(\brho, \bkap)&:=h_a^{r,\theta}(\brho, \bkap) + h_{a+1}^{r,\theta}(\brho, \bkap) + \dots + h_b^{r,\theta} (\brho, \bkap)\label{hrseq}\\
h^{\max, \theta}_{[a,b]}(\brho, \bkap) &:= \max\{h^{r,\theta}_{[a,b]}(\brho, \bkap) \mid r \in [1,\ell]\}\\
h^{\min, \theta}_{[a,b]}(\brho, \bkap)&:= \min\{h^{r,\theta}_{[a,b]}(\brho, \bkap) \mid r \in [1,\ell]\}.
\end{align}
For \(a \leq b\), it is straightforward to check that
\begin{align}
\gamma_{[a,b]}^\theta &= \alpha(\theta_a + 1, \overline{\theta_{b+1} - \theta_a}) + \left \lfloor  \frac{1}{e} \sum_{k = a}^b \overline{ \theta_{k+1} - \theta_k} \right \rfloor \delta \label{gammaconvert}\\
h^{r,\theta}_{[a,b]}(\brho, \bkap) &= h^r_{\theta_a, \theta_{b+1}}(\brho, \bkap) - \left \lfloor  \frac{1}{e} \sum_{k = a}^b \overline{ \theta_{k+1} - \theta_k} \right \rfloor. \label{hconvert}
\end{align}


\begin{Definition}
Let \(\brho \in \Lambda^{\bkap}_+\) be a multicore, and \(\theta\) be a residue permutation.
We define the {\em \(\theta\)-capacity} \(\textup{cap}^\theta_\delta(\brho, \bkap)\) of \(\brho\) as follows:
\begin{align*}
\textup{cap}^\theta_\delta(\brho, \bkap) := \min\{ h_t^{r, \theta}(\brho, \bkap) \mid t \in [1, e-1], r \in [1,\ell] \} +1.
\end{align*}
Note that when \(\brho\) is \((\bkap, \theta)\)-RoCK, we have \(\textup{cap}^\theta_\delta( \brho, \bkap) \geq 0\) by \cref{RoCKcore}. 

For \(\omega \in \ZZ_{\geq 0}I\) such that \(\Lambda_+^{\bkap}(\omega)\) is a core block, we set
\begin{align*}
\textup{cap}^\theta_\delta(\omega, \bkap) := \min \{\textup{cap}^\theta_\delta(\brho, \bkap) \mid \brho \in \Lambda^{\bkap}_+(\omega)\}.
\end{align*}
\end{Definition}

\begin{Lemma}
If \(\brho \in \Lambda^{\bkap}_+\) is a multicore, then there is at most one residue permutation \(\theta\) such that \(\textup{cap}_\delta^\theta( \brho, \bkap)>0\).
\end{Lemma}
\begin{proof}
Assume that \(\theta\), \(\theta'\) are distinct residue permutations such that \(\textup{cap}_\delta^\theta(\brho, \bkap) >0\) and \(\textup{cap}_\delta^{\theta'}(\brho, \bkap) >0\). Then there exists \(i<j\) such that \(a<b\), \(\theta_a =i, \theta_b = j\), and \(a'<b'\), \(\theta'_{a'} = j, \theta'_{b'} = i\). Then, since \(h_t^{r, \theta}, h_t^{r, \theta'} >0\) for all \(r \in [1,\ell], t \in [1,e]\), we have by (\ref{hconvert}) that 
\begin{align*}
	h^r_{i,j} = h^r_{\theta_a, \theta_b} \geq h_{[a,b-1]}^{r,\theta}  >0
	\qquad
	\textup{and}
	\qquad
	h^r_{j,i} = h^r_{\theta'_{a'}, \theta'_{b'}} \geq h_{[a',b'-1]}^{r,\theta'}  >0,
\end{align*}	
reaching a contradiction.
\end{proof}

Recall the `mod \(\delta\)' map \(p:\Z I \to \Z I^\fin \cong \Z I/ \Z \delta\).

\begin{Lemma}\label{posfin}
The elements \(P_+^\theta:=\{p(\gamma^\theta_{[a,b]}) \mid 1 \leq a \leq b \leq e-1\}\) comprise a positive root system in \(\Phi^\fin\), with base \(\Delta^\theta:= \{p(\gamma^\theta_a) \mid 1 \leq a \leq e-1\}\).
\end{Lemma}
\begin{proof}
First, we note that \(P_+^\theta\) is a {\em closed} subset of \(\Phi^\fin\), in the sense that whenever \(\beta_1, \beta_2 \in P_+^\theta\) are such that \(\beta_1 + \beta_2 \in \Phi^\fin\), it follows that \(\beta_1 + \beta_2 \in P_+^\theta\). This is a straightforward check, and full details can be found in the \texttt{arXiv} version of the paper as explained in \S\ref{SS:ArxivVersion}.

\begin{answer}
Details: Assume \(p(\gamma^\theta_{[a,b]}) + p(\gamma^\theta_{[c,d]}) \in \Phi^\fin\), for some \(1 \leq a \leq b \leq e-1\) and \(1 \leq c \leq d \leq e-1\). Without loss of generality, we may reduce our consideration to three cases:

{\em Case 1.} Assume \(\theta_a < \theta_{b+1}\) and \(\theta_c < \theta_{d+1}\). Then by (\ref{gammaconvert}) we have that 
\begin{align*}
p(\gamma^\theta_{[a,b]}) = \alpha_{\theta_a + 1} + \dots + \alpha_{\theta_{b+1}}
\qquad
\textup{and}
\qquad
p(\gamma^\theta_{[c,d]}) = \alpha_{\theta_c + 1} + \dots + \alpha_{\theta_{d+1}}
\end{align*}
Since the sum of these is a root in \(\Phi^\fin\), it is necessary that either \(b+1 = c\) or \(d+1 = a\). Without loss of generality we assume that \(b+1 = c\). Then we have \(a \leq b < b+1 = c \leq d\), and \(\theta_a < \theta_{b+1} = \theta_c < \theta_{d+1}\), so it follows that
\begin{align*}
p(\gamma^\theta_{[a,b]}) + p(\gamma^\theta_{[c,d]}) =  \alpha_{\theta_a + 1} + \dots+ \alpha_{\theta_{d+1}} = p(\gamma^\theta_{[a,d]}) \in P^\theta_+,
\end{align*}
as desired.

{\em Case 2.} Assume \(\theta_a > \theta_{b+1}\) and \(\theta_c > \theta_{d+1}\). Then by (\ref{gammaconvert}) we have that 
\begin{align*}
p(\gamma^\theta_{[a,b]}) = - \alpha_{\theta_{b+1} + 1} - \dots - \alpha_{\theta_a}
\qquad
\textup{and}
\qquad
p(\gamma^\theta_{[a,b]}) = - \alpha_{\theta_{d+1} + 1} - \dots - \alpha_{\theta_c}
\end{align*}
Since the sum of these is a root in \(\Phi^\fin\), it is necessary that either \(b+1 = c\) or \(d+1 = a\). Without loss of generality we assume that \(b+1 = c\). Then we have \(a \leq b < b+1 = c \leq d\), and \(\theta_a > \theta_{b+1} = \theta_c > \theta_{d+1}\), so it follows that
\begin{align*}
p(\gamma^\theta_{[a,b]}) + p(\gamma^\theta_{[c,d]}) =  
-\alpha_{\theta_{d+1} + 1} - \dots - \alpha_{\theta_a}
 = p(\gamma^\theta_{[a,d]}) \in P^\theta_+,
\end{align*}
as desired.

{\em Case 3.} Assume \(\theta_a < \theta_{b+1}\) and \(\theta_c > \theta_{d+1}\). Then by (\ref{gammaconvert}) we have that 
\begin{align*}
p(\gamma^\theta_{[a,b]}) = \alpha_{\theta_a + 1} + \dots + \alpha_{\theta_{b+1}}
\qquad
\textup{and}
\qquad
p(\gamma^\theta_{[a,b]}) = - \alpha_{\theta_{d+1} + 1} - \dots - \alpha_{\theta_c}
\end{align*}
Since the sum of these is a root in \(\Phi^\fin\), it is necessary that one of the following situations hold:
\begin{itemize}
\item \(b+1 = c\) and \(\theta_{d+1}> \theta_a\), or;
\item \(b+1 = c\) and \(\theta_{d+1}< \theta_a\), or;
\item \(d+1 =a\) and \(\theta_{c} < \theta_{b+1}\), or;
\item \(d+1 =a\) and \(\theta_{c} > \theta_{b+1}\).
\end{itemize}
In all cases, the arguments are similar. Let us examine only the first. Assume that \(b+1 = c\) and \(\theta_{d+1}> \theta_a\). Then we have \(a \leq b < b+1 = c \leq d\), and \(\theta_a < \theta_{d+1}\). Then we have:
\begin{align*}
p(\gamma^\theta_{[a,b]}) + p(\gamma^\theta_{[c,d]}) =  \alpha_{\theta_a + 1} + \dots+ \alpha_{\theta_{d+1}} = p(\gamma^\theta_{[a,d]}) \in P^\theta_+,
\end{align*}
as desired.
\end{answer}

Now, if \(\Phi_+^\theta \cap -\Phi_+^\theta \neq \varnothing\), there must exist 
\(1 \leq a \leq b \leq e-1\), \(1 \leq c \leq d \leq e-1\), such that \(p(\gamma^\theta_{[a,b]})+p(\gamma^\theta_{[c,d]})= 0\). For this to occur, we must have \(\gamma^\theta_{[a,b]}+ \gamma^\theta_{[c,d]}= m\delta\) for some \(m >0\). Then, by (\ref{gammaconvert}), we must have that
\begin{align*}
\alpha(\theta_a+1, \overline{\theta_{b+1}-\theta_a}) + \alpha(\theta_c+1, \overline{\theta_{d+1}-\theta_c}) = \delta.
\end{align*}
which implies that \(c = b+1\), \(d=a-1\). As \(a\leq b\) and \(c \leq d\), this is a contradiction. Thus \(\Phi_+^\theta \cap -\Phi_+^\theta = \varnothing\). 

Since \(\Phi_+^\theta\) is closed as a subset of \(\Phi^\fin\), with \(\Phi_+^\theta \cap -\Phi_+^\theta = \varnothing\), it follows that \(\Phi_+^\theta\) may be extended to a positive root system for \(\Phi^\fin\) (see for instance \cite[\S4]{BH}). But by inspection we have \(|\Phi_+^\theta| = e(e-1)/2 = |\Phi^\fin|/2\), so it must be that \(P_+^\theta\) already constitutes a positive root system in \(\Phi^\fin\). 

Finally, note that \(|\Delta^\theta| = e-1\), and by (\ref{gamseq}), 
we have \(
p(\gamma^\theta_{[a,b]}) = p(\gamma^\theta_a) + p(\gamma^\theta_{a+1}) + \dots +p(\gamma^\theta_b)
\) for all \(1 \leq a \leq b \leq e-1\), so \(\Delta^\theta\) constitutes a base for \(P_+^\theta\).
\end{proof}

\begin{Lemma}
Every element of \(\ZZ I\) can be written uniquely as
\begin{align*}
k \delta + m_1 \gamma_1^\theta +  \dots + m_{e-1} \gamma_{e-1}^\theta,
\end{align*}
for some \(k, m_1, \dots, m_{e-1} \in \ZZ\).
\end{Lemma}
\begin{proof}
Let \(\beta \in \ZZ I\). 
As \(\Delta^\theta = \{p(\gamma^\theta_a) \mid 1 \leq a \leq e-1\}\) is a base for \(P_+^\theta \subseteq \Phi^\fin\) by \cref{posfin}, it follows that \(\Delta^\theta\) is a \(\ZZ\)-basis for \(\Z I^\fin = p(\Z I)\).
Then \(p( \beta) = m_1 p( \gamma^\theta_1 )+ \dots + m_{e-1} p( \gamma^\theta_{e-1})  \) for some \(m_1, \dots, m_{e-1} \in \ZZ\). Then \(\beta = k \delta + m_1 \gamma^\theta_1 +  \dots + m_{e-1} \gamma^\theta_{e-1}\) for some \(k \in \ZZ\), establishing existence.

For uniqueness, assume that
\begin{align*}
k \delta + m_1 \gamma^\theta_1 +  \dots + m_{e-1} \gamma^\theta_{e-1} = k' \delta + m_1' \gamma^\theta_1 +  \dots + m_{e-1}' \gamma^\theta_{e-1}
\end{align*}
for some \(k, k', m_1, m_1', \dots, m_{e-1}, m_{e-1}' \in \ZZ\). Then 
\begin{align*}
m_1 p( \gamma^\theta_1) + \dots + m_{e-1}p( \gamma^\theta_{e-1}) =  m_1' p(\gamma^\theta_1)+ \dots + m_{e-1}' p(\gamma^\theta_{e-1}),
\end{align*}
which implies that \(m_a = m_a'\) for all \(a \in [1,e-1]\). Thus \(k \delta = k' \delta\), which implies that \(k = k'\), establishing uniqueness.
\end{proof}

\subsection{Removable and addable ribbons}\label{subsec:remaddribs}
Continuing on with our fixed choice of level \(\ell\), \(\ell\)-multicharge \(\bkap\), multicore \(\brho \in \Lambda^{\bkap}_+\), and residue permutation \(\theta\),
we now define subsets of the root lattice \(\ZZ I\) related to the removable and addable ribbons for \(\brho\):
\begin{align*}
\remrib^\theta &:= \{k \delta + \gamma^\theta_{[a,b]} \mid 1 \leq a \leq b \leq e-1, k < h^{\max,\theta}_{[a,b]}(\brho, \bkap) \}\\
\addrib^\theta &:=  \{k \delta + \gamma^\theta_{[a,b]} \mid 1 \leq a \leq b \leq e-1, k > h^{\min,\theta}_{[a,b]}(\brho, \bkap)\}\\
&\hspace{10mm}
\cup \{k \delta \mid k >0\} \cup \{k \delta - \gamma^\theta_{[a,b]} \mid 1 \leq a \leq b \leq e-1, k  > 0\}.
\end{align*}

\begin{Lemma}\label{reminR}
Let \(\brho \in \Lambda_+^{\bkap}\) be a \((\bkap,\theta)\)-RoCK multicore.
\begin{enumerate}
\item
If \(\xi\) is a removable ribbon for \(\brho\), then \(\cont(\xi) \in \remrib^\theta\). 
\item 
If \(\xi\) is an addable ribbon for \(\brho\), then \(\cont(\xi) \in \addrib^\theta\). 
\end{enumerate}
\end{Lemma}

\begin{proof}
As we are working with fixed \(\brho, \bkap\), we will omit the notation \((\brho, \bkap)\) throughout the proof, for the sake of space.

(i) 
Let \(\xi\) be a removable ribbon in the \(r\)th component of \(\brho\). Then, as explained in \S\ref{remribsexp}, \(\xi\) corresponds to some \(x<y\), \(\bar x = \theta_a\), \(\bar y = \theta_{b+1}\), with \(x \notin \B^r_{\theta_a}\), \(y \in \B^r_{\theta_{b+1}}\), for some \(a \leq b\), since \(\brho\) is \((\bkap,\theta)\)-RoCK. Since \(\brho\) is a multicore
we have \(M^r_{\theta_a} + e \leq x < y \leq M^r_{\theta_{b+1}}\).

We have
\begin{align*}
\cont(\xi) &= \alpha(x+1, y-x) = \alpha(x+1, \overline{y-x}) + \left\lfloor  \frac{y-x}{e} \right\rfloor \delta\\
&= \alpha(\theta_a + 1, \overline{ \theta_{b+1} - \theta_a}) + \left\lfloor  \frac{y-x}{e} \right\rfloor \delta\\
&= \gamma^\theta_{[a,b]} + \left( 
 \left\lfloor  \frac{y-x}{e} \right\rfloor - 
  \left \lfloor  \frac{1}{e} \sum_{k = a}^b \overline{ \theta_{k+1} - \theta_k} \right \rfloor 
\right)
\delta, \stepcounter{equation}\tag{\theequation}\label{gammadeltaeq}
\end{align*}
where the last equality follows from (\ref{gammaconvert}).

We also have
\begin{align*}
 \left\lfloor  \frac{y-x}{e} \right\rfloor
  \leq 
  \left \lfloor
  \frac{M^r_{\theta_{b+1}} - M^r_{\theta_a} - e}{e}  \right \rfloor = h^r_{\theta_a, \theta_{b+1}} -1
  =
  h^{r,\theta}_{[a,b]} +  \left \lfloor  \frac{1}{e} \sum_{k = a}^b \overline{ \theta_{k+1} - \theta_k} \right \rfloor  -1,
\end{align*}
where the last equality follows from (\ref{hconvert}).

Therefore we have \(
\cont(\xi) = k \delta + \gamma^\theta_{[a,b]},
\)
where \(k \leq h^{r, \theta}_{[a,b]} - 1 < h^{\max, \theta}_{[a,b]}\), so \(\cont(\xi) \in \remrib^\theta\), as desired.

(ii) 
Now let \(\xi\) be an addable ribbon in the \(r\)th component of \(\brho\).
Then, as explained in \S\ref{addribsexp}, \(\xi\) corresponds to some \(x<y\), \(\bar x = \theta_a\), \(\bar y = \theta_{b+1}\), with \(x \in \B^r_{\theta_a}\), \(y \notin \B^r_{\theta_{b+1}}\), for some \(a,b\). Then we have \(x \leq M^r_{\theta_a}\) and \(M^r_{\theta_{b+1}} + e \leq y\) since \(\brho\) is a multicore. We consider three cases:

{\em (1) \(b = a-1\)}. Then \(\cont(\xi) = \left( \frac{y-x}{e}\right) \delta \in \addrib^\theta\).

{\em (2) \(b < a-1\)}. Then 
\begin{align*}
\cont(\xi) &= \alpha(\theta_a + 1, \overline{\theta_{b+1} - \theta_a}) + \left \lfloor \frac{y-x}{e} \right \rfloor \delta\\
&= \delta - \alpha(\theta_{b+1} + 1, \overline{\theta_a - \theta_{b+1}}) + \left \lfloor \frac{y-x}{e} \right \rfloor \delta\\
&= \left(  \left \lfloor  \frac{1}{e} \sum_{k = b+1}^{a-1} \overline{ \theta_{k+1} - \theta_k} \right \rfloor + \left \lfloor \frac{y-x}{e} \right \rfloor + 1 \right) \delta - \gamma^\theta_{[b+1,a-1]} \in \addrib^\theta. \stepcounter{equation}\tag{\theequation}\label{gammadeltaeq2}
\end{align*}

{\em (3) \(b > a-1\)}. Then, as in part (i), we have
\begin{align*}
\cont(\xi) =\gamma^\theta_{[a,b]} + \left( 
 \left\lfloor  \frac{y-x}{e} \right\rfloor - 
  \left \lfloor  \frac{1}{e} \sum_{k = a}^b \overline{ \theta_{k+1} - \theta_k} \right \rfloor  \right) \delta. 
\end{align*}
We also have
\begin{align*}
\left \lfloor
\frac{y-x}{e}
\right \rfloor
 \geq \left \lfloor
 \frac{M^r_{\theta_{b+1}} - M^r_{\theta_a} + e}{e}
 \right \rfloor
 =h^r_{\theta_a, \theta_{b+1}} + 1
 =
 h^{r,\theta}_{[a,b]} +  \left \lfloor  \frac{1}{e} \sum_{k = a}^b \overline{ \theta_{k+1} - \theta_k} \right \rfloor + 1.
\end{align*}
So it follows that \(\cont(\xi) = k\delta + \gamma^\theta_{[a,b]}\), where \(k \geq h^{r,\theta}_{[a,b]} + 1 > h^{\min, \theta}_{[a,b]}\), and thus \(\cont(\xi) \in \addrib^\theta\).
\end{proof}

\section{Cuspidal tilings}\label{cusptilesec}

\subsection{Convex preorders}\label{subsec:convexpreorders}
A {\em convex preorder} on \(\Phi_+\) is a binary relation \(\succeq\) on \(\Phi_+\) which, for all \(\beta, \gamma, \nu \in \Phi_+\) satisfies the following:
\begin{enumerate}
\item \(\beta \succeq \beta\) (reflexivity);
\item \(\beta \succeq \gamma\) and \(\gamma \succeq \nu\) imply \(\beta \succeq \nu\) (transitivity);
\item \(\beta \succeq \gamma\) or \(\gamma \succeq \beta\) (totality);
\item \(\beta \succeq \gamma\) and \(\beta + \gamma \in \Phi_+\) imply \(\beta \succeq \beta  + \gamma \succeq \gamma\) (convexity);
\item \(\beta \succeq \gamma\) and \(\gamma \succeq \beta\) if and only if \(\beta = \gamma\) or \(\beta, \gamma \in \Phi_+^\im\) (imaginary equivalence).
\end{enumerate}
We write \(\beta \succ \gamma\) if \(\beta \succeq \gamma\) and \(\gamma \not \succeq \beta\). Then (iii) and (v) together imply that \(\succ\) restricts to a total order on \(\Psi = \Phi_+^\re \sqcup \{ \delta\}\). We also write \(\beta \approx \gamma\) if \(\beta \succeq \gamma\) and \(\gamma \succeq \beta\), so (v) and (\ref{allimag}) imply that \(\beta \approx \gamma\), \(\beta \neq \gamma\) if and only if \(\beta = m \delta\), \(\gamma = m'\delta\), for some \(m \neq m' \in \mathbb{Z}_{>0}\). 

Given a convex preorder \(\succeq\), we write
\begin{align*}
\Phi_{\succ \delta} := \{ \beta \in \Phi_+ \mid \beta \succ \delta\};
\qquad
\Phi_{\prec \delta} := \{ \beta \in \Phi_+ \mid \beta \prec \delta\}.
\end{align*}
It follows from these definitions that \(\Phi_+^\re = \Phi_{\succ \delta} \sqcup \Phi_{\prec \delta}\).

Convex preorders are known to exist for any choice of \(e\), and a classification and general method of construction of convex preorders is given in \cite{Ito}. The proposition below, lightly paraphrased from \cite[Example 2.14(ii)]{affineMVpoly} and \cite[Example 3.5]{McN17}, provides a way to generate convex preorders on \(\Phi_+\).

\begin{Proposition}\label{genconvex} 
Let \((V, \geq)\) be a totally ordered \(\Q\)-vector space. Let \(h:\ZZ I \to V\) be a \(\ZZ\)-linear map such that \(\beta \mapsto  \frac{h(\beta)}{\textup{ht}(\beta)} \) is injective on \(\Psi \subseteq \ZZ I\). Then the relation 
\begin{align*}
\beta \succeq \gamma \iff \frac{h(\beta)}{\textup{ht}(\beta)} \geq \frac{h(\gamma)}{\textup{ht}(\gamma)}
\end{align*}
defines a convex preorder on \(\Phi_+\).
\end{Proposition}

\begin{Proposition}\label{fintoaff}
Let \(P_+ \subseteq \Phi^\fin\) be a positive root system for \(\Phi^\fin\), and let \(P_- = -P_+\). There exists a convex preorder \(\succeq\) on \(\Phi_+\) such that
\begin{align}\label{exconvgen}
\Phi_{\succ \delta} = \{ \beta \in \Phi_+ \mid p(\beta) \in P_+\}
\qquad
\textup{and}
\qquad
\Phi_{\prec \delta} = \{ \beta \in \Phi_+ \mid p(\beta) \in P_-\}.
\end{align}
\end{Proposition}
\begin{proof}
We may orient the root system \(\Phi^\fin\) in \(\RR^{e-1}\) such that the hyperplane \((0, x_2, \dots, x_{e-1})\) separates \(P_+\) (assumed to be positive in the first component) from \(P_-\) (assumed to be negative in the first component). Extend this embedding of \(I^\fin\) in \(\RR^{e-1}\) to a \(\ZZ\)-linear map \(h:\ZZ I \to \RR^{e-1}\) with \(\delta \mapsto 0\). Take the lexicographic total order on \(\RR^{e-1}\), and define a convex preorder \(\succeq\) as in \cref{genconvex}. Note then that the first component of \(h(\beta)\) is positive when \(p(\beta) \in P_+\), negative when \(p(\beta) \in P_-\), and zero when \(\beta= m\delta\) for \(m \in \ZZ_{>0}\), which establishes (\ref{exconvgen}).
\end{proof}

\subsection{Special convex preorders}
Fix a choice of \(\brho \in \Lambda^{\bkap}_+\) and residue permutation \(\theta\). Recall the accompanying data defined in \S\ref{reindsec}. 

\begin{Definition}\label{realizethetadef}
Let \(\theta\) be a residue permutation. We say that a convex preorder \(\succeq\) on \(\Phi_+\) {\em realizes} \(\theta\) provided that the following hold.
\begin{align}
\Phi_{\succ \delta} &= \{ k \delta + \gamma^\theta_{[a,b]} \mid 1 \leq a \leq b \leq e-1, k \in \ZZ\} \cap \Phi_+\label{gdel}\\
\Phi_{\prec \delta} &= \{ k \delta - \gamma^\theta_{[a,b]}  \mid 1 \leq a \leq b \leq e-1, k \in \ZZ\} \cap \Phi_+\label{ldel}
\end{align}
\end{Definition}

\begin{Lemma}\label{realex}
For any residue permutation \(\theta\), there exists a convex preorder \(\succeq\) on \(\Phi_+\) which realizes \(\theta\).
\end{Lemma}
\begin{proof}
Follows from \cref{posfin,fintoaff}.
\end{proof}

\begin{Lemma}\label{realizelem}
Every convex preorder \(\succeq\) realizes a unique residue permutation \(\theta\).
\end{Lemma}
\begin{proof}
For all \(0 \leq i < j \leq e-1\), note that, by convexity, exactly one of \( \alpha(i+1, j-i)\) or \( \delta - \alpha(i+1, j-i)\) is in \(\Phi_{\succ \delta}\), and the other must be in \(\Phi_{\prec \delta}\). Thus we may define a connected, irreflexive order \(<^*\) on \([0,e-1]\) by setting
\begin{align*}
i <^* j \textup{ if } \alpha(i+1, j-i) \in \Phi_{\succ \delta}
\qquad
\textup{and}
\qquad
i >^* j \textup{ if } \delta - \alpha(i+1, j-i) \in \Phi_{\succ \delta}.
 \end{align*}
 for all \(0 \leq i < j \leq e-1\). 
 
Now we must check that \(<^*\) is transitive.
This is a straightforward check, and full details can be found in the \texttt{arXiv} version of the paper as explained in \S\ref{SS:ArxivVersion}. 
 
\begin{answer}
Details: Assume that \(i<^* j <^*k\).
We consider eight possible cases, based on the relative magnitude of \(i,j,k\).

\begin{enumerate}
\item If \(i < j<k\), then \( \alpha(i+1,j-i),  \alpha(j+1, k-j) \in \Phi_{\succ \delta}\), so by convexity, we have that
 \begin{align*}
 \alpha(i+1,j-i) + \alpha(j+1, k-j) = \alpha(i+1, k-i) \in \Phi_{\succ \delta},
 \end{align*}
 and thus \(i<^* k\).
 \item If \(i < k < j\), then \( \alpha(i+1,j-i), \delta- \alpha(k+1, j-k) \in \Phi_{\succ \delta}\), so by convexity we have that 
 \begin{align*}
  \alpha(i+1,j-i) + \delta- \alpha(k+1, j-k) = \delta + \alpha(i+1, k-i)\in \Phi_{\succ \delta},
 \end{align*}
 and thus by convexity again \( \alpha(i+1, k-i)\in \Phi_{\succ \delta} \), so \(i<^*k\).
 \item If \(k < i < j\), then \( \alpha(i+1,j-i),   
  \delta- \alpha(k+1, j-k)
 \in \Phi_{\succ \delta}\), so by convexity we have that 
 \begin{align*}
  \alpha(i+1,j-i) + \delta - \alpha(k+1, j-k) = \delta - \alpha(k+1, i-k)\in \Phi_{\succ \delta},
 \end{align*}
 and thus \(i<^*k\).
 \item If \(j<i<k\), then \(\delta - \alpha(j+1,i-j), \alpha(j+1, k-j) \in \Phi_{\succ \delta}\), so by convexity we have that
 \begin{align*}
 \delta - \alpha(j+1,i-j) +  \alpha(j+1, k-j) = \delta + \alpha(i+1, k-i) \in \Phi_{\succ \delta},
 \end{align*}
 and thus by convexity again \( \alpha(i+1, k-i)\in \Phi_{\succ \delta} \), so \(i<^*k\).
 \item If \(j<k<i\), then \(\delta - \alpha(j+1,i-j),  \alpha(j+1, k-j) \in \Phi_{\succ \delta}\), so by convexity we have that
 \begin{align*}
 \delta - \alpha(j+1,i-j) +   \alpha(j+1, k-j) = \delta  - \alpha(k+1, i-k)  \in \Phi_{\succ \delta},
 \end{align*}
 and thus \(i<^* k\).
 \item If \(k<j<i\), then \(\delta - \alpha(j+1,i-j),  \delta- \alpha(k+1, j-k)  \in \Phi_{\succ \delta}\), so by convexity we have that
 \begin{align*}
 \delta - \alpha(j+1,i-j) +   \delta- \alpha(k+1, j-k) = 2\delta - \alpha(k+1, i-k) \in \Phi_{\succ \delta},
 \end{align*}
 and therefore \(\delta - \alpha(k+1, i-k)\), so \(i<^* k\).
 \item If \(i = k < j\), then \( \alpha(i+1,j-i), \delta- \alpha(i+1, j-i) \in \Phi_{\succ \delta}\), and thus by convexity \(\delta \in \Phi_{\succ \delta}\), a contradiction, so this case cannot occur.
 \item 
 If \(j< i = k\), then \(  \delta - \alpha(j+1, i-j),  \alpha(j+1, i-j) \in \Phi_{\succ \delta}\), and thus by convexity \(\delta \in \Phi_{\succ \delta}\), a contradiction, so this case cannot occur.
\end{enumerate}
Thus, in every possible case, we have \(i^* < k^*\), establishing transitivity.
\end{answer}

As \(<^*\) is connected, irreflexive, and transitive, it is a strict total order, and thus defines a residue permutation \(\theta = (\theta_1, \dots, \theta_e)\), where \(\theta_1 <^* \cdots <^* \theta_e\). 

Let \(1 \leq a \leq b \leq e-1\). Then \(\theta_a <^* \theta_{b+1}\). If \(\theta_a < \theta_{b+1}\), we have that \(\alpha(\theta_a + 1, \overline{\theta_{b+1} - \theta_a}) \in \Phi_{\succ \delta}\). On the other hand, if \(\theta_{b+1} < \theta_a\), we have that \(\delta - \alpha(\theta_{b+1} + 1, \overline{\theta_a - \theta_{b+1}}) =\alpha(\theta_a + 1, \overline{\theta_{b+1} - \theta_a}) \in \Phi_{\succ \delta}\), so in any case we have that \(\alpha(\theta_a + 1, \overline{\theta_{b+1} - \theta_a}) \in \Phi_{\succ \delta}\).
Then by convexity and (\ref{gammaconvert}), we have that
\begin{align*}
\gamma^\theta_{[a,b]} = \alpha(\theta_a + 1, \overline{\theta_{b+1} - \theta_a}) + k \delta \in \Phi_{\succ \delta}
\end{align*}
 for some \(k \in \ZZ_{\geq 0}\). Therefore \(\succeq\) realizes \(\theta\). 
 
 For uniqueness, assume that \(\succeq\) is a convex preorder on \(\Phi_+\) which realizes distinct residue permutations \(\theta\) and \(\theta'\). Then without loss of generality, there is some \(i,j \in [0,e-1]\), \(i<j\), such that \(\theta_a = i\), \(\theta_{b+1} = j\), for some \(a\leq b\), and \(\theta'_{b'+1} = i\), \(\theta'_{a'} = j\), for some \(a' \leq b'\). Then, since \(\succeq\) realizes \(\theta\) and \(\theta'\), we have by 
 (\ref{gammaconvert}) that
 \begin{align*}
 \alpha(i + 1, \overline{j-i}) + k\delta \in \Phi_{\succ \delta}
 \qquad
 \textup{and}
 \qquad
 \alpha(j+1, \overline{i-j}) + k'\delta \in \Phi_{\succ \delta}
 \end{align*}
 for some \(k,k' \in \ZZ_{\geq 0}\). But then the sum of these roots, a multiple of \(\delta\), is in \(\Phi_{\succ \delta}\) as well by convexity, a contradiction. This completes the proof.
\end{proof}

\begin{Remark}
There are many different convex preorders realise the same residue permutation.
Indeed, there are infinitely many convex preorders on \(\Phi_+\) if \(e>2\), but only \(e!\) many residue permutations.
\end{Remark}

\subsection{Cuspidal ribbon tableaux}\label{subsec:cuspribtabs}
Recall the definition of tilings and tableaux from \S\ref{tiletabsec}. Following \cite{muthtiling} we define the following special tilings for multipartitions. 

\begin{Definition}\label{cuspdef}
Let \(\succeq\) be a convex preorder on \(\Phi_+\), and let \(\btau \in \Lambda^\ell(\omega)\) for some \(\omega \in \Phi_+\).
We say that the skew diagram \(\btau\) is {\em cuspidal} provided that, for every two-tile skew tableau \((\bnu_1, \bnu_2)\) of \(\btau\), the following conditions hold:
\begin{enumerate}
\item There exist \(\beta_1, \dots, \beta_k \in \Phi_+\) such that \(\cont(\bnu_1) = \beta_1 + \dots + \beta_k\) and \(\omega \succ \beta_i\) for all \(i \in [1,k]\), and;
\item There exist \(\zeta_1, \dots, \zeta_m \in \Phi_+\) such that \(\cont(\bnu_2) = \zeta_1 + \dots + \zeta_m\) and \(\zeta_i \succ \omega\) for all \(i \in [1,m]\).
\end{enumerate}
We say that \(\btau\) is \emph{semicuspidal} if the above conditions hold after weakening the \(\omega \succ \beta_i\) and \(\zeta_i \succ \omega\) conditions to \(\omega \succeq \beta_i\) and \(\zeta_i \succeq \omega\), respectively.
\end{Definition}

It is shown in \cite[Theorem~6.13]{muthtiling} that every cuspidal diagram is necessarily a ribbon. 
We say that a skew tableau \((\Gamma, {\tt t})\) of \(\btau \in \Lambda^\ell\) is a {\em cuspidal Kostant} tableau provided that each tile in \(\Gamma\) is a cuspidal ribbon, and \(\cont({\tt t}(a)) \succeq \cont({\tt t}(b))\) for all \(1 \leq a \leq b \leq [1, |\Gamma|]\). We say a tiling \(\Gamma\) is a {\em cuspidal Kostant} tiling for \(\btau\) provided there exists a tile ordering \({\tt t}\) such that \((\Gamma, {\tt t})\) is a cuspidal Kostant tableau.

\begin{Theorem}\label{tilethm}\cite[Theorem 6.14]{muthtiling}
Let \(\succeq\) be a convex preorder on \(\Phi_+\), and let \(\btau \in \Lambda^{\ell}\). Then:
\begin{enumerate}
\item There exists a unique cuspidal Kostant tiling \(\Gamma_{\btau}\) for \(\btau\).
\item Let \(\xi\) be a \({\tt SE}\)-removable ribbon in \(\btau\) such that \(\cont(\xi) \in \Psi\) and \(\cont(\xi)\) is \(\succeq\)-minimal. Then \(\xi \in \Gamma_{\btau}\) and \(\cont(\zeta) \succeq \cont(\xi) \) for all \(\zeta \in \Gamma_{\btau}\). Thus the unique cuspidal Kostant tiling \(\Gamma_{\btau}\) may be constructed by iteratively removing \(\succeq\)-minimal \({\tt SE}\)-removable ribbons.
\item Let \(\xi\) be a \({\tt NW}\)-removable ribbon in \(\btau\) such that such that \(\cont(\xi) \in \Psi\) and \(\cont(\xi)\) is \(\succeq\)-maximal. Then \(\xi \in \Gamma_{\btau}\) and \(\cont(\xi) \succeq \cont(\zeta) \) for all \(\zeta \in \Gamma_{\btau}\).  Thus the unique cuspidal Kostant tiling \(\Gamma_{\btau}\) may be constructed by iteratively removing \(\succeq\)-maximal \({\tt NW}\)-removable ribbons.
\end{enumerate}
\end{Theorem}

\subsection{Core blocks and RoCK blocks via cuspidal tilings}\label{subsec:cusptilingsofblocks}
In this section we explain that core blocks and RoCK blocks can be described in terms of cuspidal tiling data for their constituent multipartitions.

For \(\omega \in \ZZ_{\geq 0}I\), we define
\begin{align*}
\Lambda^{\bkap}_+(\omega)_{\succ \delta} &:= \{ \blam \in \Lambda^{\bkap}_+(\omega) \mid \cont(\zeta) \succ \delta \textup{ for all }\zeta \in \Gamma_{\blam}\},\\
\Lambda^{\bkap}_+(\omega)_{\succeq \delta} &:= \{ \blam \in \Lambda^{\bkap}_+(\omega) \mid \cont(\zeta) \succeq \delta\textup{ for all }\zeta \in \Gamma_{\blam}\},\\
\Lambda^{\bkap}_+(\omega)_{\approx \delta} &:= \{ \blam \in \Lambda^{\bkap}_+(\omega) \mid \cont(\zeta) = \delta\textup{ for all }\zeta \in \Gamma_{\blam}\},
\end{align*}
and similarly define corresponding subsets of \(\Lambda_{+/+}^{\bkap}(\omega)\), \(\Lambda_{+/\brho}^{\bkap}(\omega)\) as well.

\begin{Proposition}\label{RoCKsemipar}
Let \(\omega \in \ZZ_{\geq 0}I\) and \(\blam \in \Lambda^{\bkap}_+(\omega)\). Let \(\theta\) be a residue permutation, and \(\succeq\) be a convex preorder on \(\Phi_+\) which realizes \(\theta\). Then \(\blam\) is a \((\bkap, \theta)\)-RoCK multipartition if and only if \(\blam \in \Lambda_+^{\bkap}(\omega)_{\succeq \delta}\).
\end{Proposition}
\begin{proof}
\((\implies)\) Assume \(\blam\) is \((\bkap, \theta)\)-RoCK. Let \(\xi\) be a removable ribbon for \(\blam\). Then \(\xi\) corresponds to some \(x<y\) with \(y \in \B^r(\blam, \bkap)\), \(x \notin \B^r(\blam, \bkap)\), with \(\overline{x} = \theta_a\) and \(\overline{y} = \theta_{b+1}\), for some \(a,b+1 \in [1,e]\). Moreover, since \(\blam\) is \((\bkap, \theta)\)-RoCK, we must have  \(a \leq b+1\). If \(a=b+1\), then \(\cont(\xi)\) is a multiple of \(\delta\). Assume \(a\leq b\). Then, as in 
(\ref{gammadeltaeq}) we have
\(
\cont(\xi) = \gamma^\theta_{[a,b]} + m\delta,
\)
for some \(m \in \ZZ\), which implies that \(\cont(\xi) \in \Phi_{\succ \delta}\). Thus, in any case, we have \(\cont(\xi) \in \Phi_{\succeq \delta}\). It follows then from \cref{tilethm}(ii) that \(\blam \in \Lambda_+^{\bkap}(\omega)_{\succeq \delta}\).

\((\impliedby)\) Assume \(\blam\) is not \((\bkap, \theta)\)-RoCK. Then there exists some \(x<y\) with \(y \in \B^r(\blam, \bkap)\), \(x \notin \B^r(\blam, \bkap)\), with \(\overline{x} = \theta_a\) and \(\overline{y} = \theta_{b+1}\), such that \(a > b+1\) for some \(a,b+1 \in [1,e]\). This data corresponds to a removable ribbon \(\xi\) in \(\blam\) such that (as in (\ref{gammadeltaeq2}))
\(\cont(\xi) = m\delta - \gamma^\theta_{[b+1,a-1]}\), which implies that \(\cont(\xi) \in \Phi_{\prec \delta}\). It follows then from \cref{tilethm}(ii) that \(\blam \notin \Lambda_+^{\bkap}(\omega)_{\succeq \delta}\).
\end{proof}

\begin{Corollary}\label{thetarockblocktile}
Let \(\omega \in \ZZ_{\geq 0}I\), let \(\theta\) be a residue permutation, and \(\succeq\) be a convex preorder on \(\Phi_+\) which realizes \(\theta\). Then \(\Lambda_+^{\bkap}(\omega)\) is a \(\theta\)-RoCK block if and only if \(\Lambda_+^{\bkap}(\omega) = \Lambda_+^{\bkap}(\omega)_{\succeq \delta}\). 
\end{Corollary}

Together with \cref{RoCKthetaeq,realex,realizelem} this immediately implies

\begin{Corollary}
Let \(\omega \in \ZZ_{\geq 0}I\). Then \(\Lambda_+^{\bkap}(\omega)\) is a RoCK block if and only if \(\Lambda_+^{\bkap}(\omega) = \Lambda_+^{\bkap}(\omega)_{\succeq \delta}\) for some convex preorder \(\succeq\) on \(\Phi_+\).
\end{Corollary}

\begin{Proposition}\label{RoCKmultitiling}
Let \(\omega \in \ZZ_{\geq 0}I\) and \(\brho \in \Lambda^{\bkap}_+(\omega)\). Let \(\theta\) be a residue permutation, and \(\succeq\) be a convex preorder on \(\Phi_+\) which realizes \(\theta\). Then \(\brho\) is a \((\bkap, \theta)\)-RoCK multicore if and only if \(\brho \in \Lambda_+^{\bkap}(\omega)_{\succ \delta}\).
\end{Proposition}
\begin{proof}
\((\implies)\) Assume \(\brho\) is a \((\bkap, \theta)\)-RoCK multicore. Then \(\cont(\xi) \succ \delta\) for every removable ribbon \(\xi\) in \(\brho\), by \cref{reminR} and (\ref{gdel}). It follows then from \cref{tilethm}(ii) that \(\brho \in \Lambda_+^{\bkap}(\omega)_{\succ \delta}\).

\((\impliedby)\) Assume \(\brho\) is not a \((\bkap, \theta)\)-RoCK multicore. 
If \(\brho\) is not \((\bkap, \theta)\)-RoCK, then \(\brho \notin \Lambda_+^{\bkap}(\omega)_{\succ\delta} \subseteq \Lambda_+^{\bkap}(\omega)_{\succeq\delta}\) by \cref{RoCKsemipar}. If \(\brho\) is not a multicore, then \(\brho\) has a removable \(e\)-ribbon of content \(\delta\). Then by \cref{tilethm}(ii) we again have that \(\blam \notin \Lambda_+^{\bkap}(\omega)_{\succ \delta}\). 
\end{proof}

\begin{Corollary}\label{coretilingprop}
Let \(\omega \in \ZZ_{\geq 0}I\). Then \(\Lambda_+^{\bkap}(\omega)\) is a core block if and only if \(\Lambda_+^{\bkap}(\omega) = \Lambda_+^{\bkap}(\omega)_{\succ \delta}\) for some convex preorder \(\succeq\) on \(\Phi_+\).
\end{Corollary}
\begin{proof}
This follows from \cref{multicoreblockRoCK,RoCKmultitiling}, as well as \cref{realex,realizelem}.
\end{proof}

\cref{thetarockblocktile,coretilingprop} combine to give Theorem~\hyperlink{thm:C}{C} in the introduction.

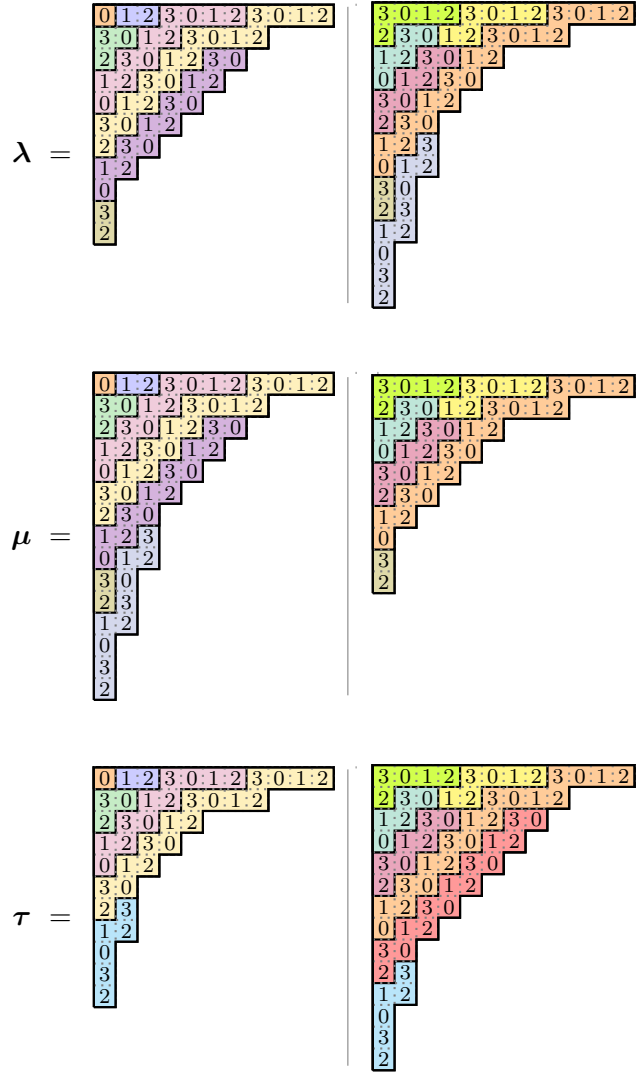
\begin{figure}[h]
\begin{align*}
	{}
		\bla\ = \
		\hackcenter{
			\begin{tikzpicture}[scale=0.29]
				\draw[thick,fill=olive!30]  (-1,1)--(0,1)--(0,4)--(1,4)--(1,5)--(2,5)--(2,6)--(3,6)--(3,7)--(4,7)--(4,8)--(5,8)--(5,9)--(6,9)--(6,10)--(7,10)--(7,11)--(10,11)--(10,12)--(-1,12)--(-1,1);
				\draw[thick,fill=blue!10!violet!30]  (-1,3)--(0,3)--(0,4)--(1,4)--(1,5)--(2,5)--(2,6)--(3,6)--(3,7)--(4,7)--(4,8)--(5,8)--(5,9)--(6,9)--(6,10)--(7,10)--(7,11)--(10,11)--(10,12)--(-1,12)--(-1,3);
				\draw[thick,fill=yellow!70!orange!30]  (-1,5)--(0,5)--(0,6)--(1,6)--(1,7)--(2,7)--(2,8)--(3,8)--(3,9)--(4,9)--(4,10)--(7,10)--(7,11)--(10,11)--(10,12)--(-1,12)--(-1,5);
				\draw[thick,fill=red!70!blue!20]  (-1,7)--(0,7)--(0,8)--(1,8)--(1,9)--(2,9)--(2,10)--(3,10)--(3,11)--(6,11)--(6,12)--(-1,12)--(-1,7);
				\draw[thick,fill=green!50!gray!30]  (-1,9)--(0,9)--(0,10)--(1,10)--(1,11)--(2,11)--(2,12)--(-1,12)--(-1,9);
				\draw[thick,fill=blue!20]  (-1,11)--(2,11)--(2,12)--(-1,12)--(-1,11);
				\draw[thick,fill=orange!40]  (-1,11)--(0,11)--(0,12)--(-1,12)--(-1,11);
				\draw[thick,  gray, dotted] (-1,2)--(0,2);
				\draw[thick,  gray, dotted] (-1,3)--(0,3);
				\draw[thick,  gray, dotted] (-1,4)--(0,4);
				\draw[thick,  gray, dotted] (-1,5)--(1,5);
				\draw[thick,  gray, dotted] (-1,6)--(2,6);
				\draw[thick,  gray, dotted] (-1,7)--(4,7);
				\draw[thick,  gray, dotted] (-1,8)--(5,8);
				\draw[thick,  gray, dotted] (-1,9)--(6,9);
				\draw[thick,  gray, dotted] (-1,10)--(7,10);
				\draw[thick,  gray, dotted] (-1,11)--(8,11);
				\draw[thick,  gray, dotted] (-1,12)--(10,12);
				\draw[thick,  gray, dotted] (0,12)--(0,1);
				\draw[thick,  gray, dotted] (1,12)--(1,4);
				\draw[thick,  gray, dotted] (2,12)--(2,5);
				\draw[thick,  gray, dotted] (3,12)--(3,6);
				\draw[thick,  gray, dotted] (4,12)--(4,7);
				\draw[thick,  gray, dotted] (5,12)--(5,8);
				\draw[thick,  gray, dotted] (6,12)--(6,9);
				\draw[thick, gray,  dotted] (7,12)--(7,10);
				\draw[thick,  gray, dotted] (8,12)--(8,11);
				\draw[thick,  gray, dotted] (9,12)--(9,11);
				\draw[thick,  gray, dotted] (10,12)--(10,11);
				\draw[thick, gray,  dotted] (11,12)--(11,12);
				\draw[thick,  gray, dotted] (12,12)--(12,12);
				\node at (-0.5,11.5){$\scriptstyle 0$};
				\node at (0.5,11.5){$\scriptstyle 1$};
				\node at (1.5,11.5){$\scriptstyle 2$};
				\node at (2.5,11.5){$\scriptstyle 3$};
				\node at (3.5,11.5){$\scriptstyle 0$};
				\node at (4.5,11.5){$\scriptstyle 1$};
				\node at (5.5,11.5){$\scriptstyle 2$};
				\node at (6.5,11.5){$\scriptstyle 3$};
				\node at (7.5,11.5){$\scriptstyle 0$};
				\node at (8.5,11.5){$\scriptstyle 1$};
				\node at (9.5,11.5){$\scriptstyle 2$};
				\node at (-0.5,10.5){$\scriptstyle 3$};
				\node at (0.5,10.5){$\scriptstyle 0$};
				\node at (1.5,10.5){$\scriptstyle 1$};
				\node at (2.5,10.5){$\scriptstyle 2$};
				\node at (3.5,10.5){$\scriptstyle 3$};
				\node at (4.5,10.5){$\scriptstyle 0$};
				\node at (5.5,10.5){$\scriptstyle 1$};
				\node at (6.5,10.5){$\scriptstyle 2$};
				\node at (-0.5,9.5){$\scriptstyle 2$};
				\node at (0.5,9.5){$\scriptstyle 3$};
				\node at (1.5,9.5){$\scriptstyle 0$};
				\node at (2.5,9.5){$\scriptstyle 1$};
				\node at (3.5,9.5){$\scriptstyle 2$};
				\node at (4.5,9.5){$\scriptstyle 3$};
				\node at (5.5,9.5){$\scriptstyle 0$};
				\node at (-0.5,8.5){$\scriptstyle 1$};
				\node at (0.5,8.5){$\scriptstyle 2$};
				\node at (1.5,8.5){$\scriptstyle 3$};
				\node at (2.5,8.5){$\scriptstyle 0$};
				\node at (3.5,8.5){$\scriptstyle 1$};
				\node at (4.5,8.5){$\scriptstyle 2$};
				\node at (-0.5,7.5){$\scriptstyle 0$};
				\node at (0.5,7.5){$\scriptstyle 1$};
				\node at (1.5,7.5){$\scriptstyle 2$};
				\node at (2.5,7.5){$\scriptstyle 3$};
				\node at (3.5,7.5){$\scriptstyle 0$};
				\node at (-0.5,6.5){$\scriptstyle 3$};
				\node at (0.5,6.5){$\scriptstyle 0$};
				\node at (1.5,6.5){$\scriptstyle 1$};
				\node at (2.5,6.5){$\scriptstyle 2$};
				\node at (-0.5,5.5){$\scriptstyle 2$};
				\node at (0.5,5.5){$\scriptstyle 3$};
				\node at (1.5,5.5){$\scriptstyle 0$};
				\node at (-0.5,4.5){$\scriptstyle 1$};
				\node at (0.5,4.5){$\scriptstyle 2$};
				\node at (-0.5,3.5){$\scriptstyle 0$};
				\node at (-0.5,2.5){$\scriptstyle 3$};
				\node at (-0.5,1.5){$\scriptstyle 2$};
				\draw[thick]  (-1,1)--(0,1)--(0,4)--(1,4)--(1,5)--(2,5)--(2,6)--(3,6)--(3,7)--(4,7)--(4,8)--(5,8)--(5,9)--(6,9)--(6,10)--(7,10)--(7,11)--(10,11)--(10,12)--(-1,12)--(-1,1);
				%
				%
												\phantom{
					\draw[thick,fill=blue]  (-1,-2)--(-1,1)--(1,1)--(1,-2)--(-1,-2);
				}
			\end{tikzpicture}
		}
		\hspace{-0.68cm}
		{}
		\hackcenter{
			\begin{tikzpicture}[scale=0.29]
				\node[white] at (-0.5,0){{}};
				\node[white] at (0.5,0){{}};
				\draw[thin, gray,fill=gray!30]  (0,0)--(0,13.6);
				\draw[thin, white]  (0,13.6)--(0,14);
			\end{tikzpicture}
		}
		{}
		\hackcenter{
			\begin{tikzpicture}[scale=0.29]
				\draw[thick,fill=cyan!25!blue!15] (-1,-3)--(0,-3)--(0,0)--(1,0)--(1,3)--(2,3)--(2,6)--(3,6)--(3,7)--(4,7)--(4,8)--(5,8)--(5,9)--(6,9)--(7,9)--(8,9)--(8,10)--(11,10)--(11,11)--(-1,11)--(-1,-3);
				\draw[thick,fill=olive!30] (-1,1)--(0,1)--(0,4)--(1,4)--(1,5)--(2,5)--(2,6)--(3,6)--(3,7)--(4,7)--(4,8)--(5,8)--(5,9)--(6,9)--(7,9)--(8,9)--(8,10)--(11,10)--(11,11)--(-1,11)--(-1,-1);
				\draw[thick,fill=orange!40] (-1,3)--(0,3)--(0,4)--(1,4)--(1,5)--(2,5)--(2,6)--(3,6)--(3,7)--(4,7)--(4,8)--(5,8)--(5,9)--(6,9)--(7,9)--(8,9)--(8,10)--(11,10)--(11,11)--(-1,11)--(-1,-3);
				\draw[thick,fill=purple!35] (-1,5)--(0,5)--(0,6)--(1,6)--(1,7)--(2,7)--(2,8)--(3,8)--(3,9)--(4,9)--(4,10)--(7,10)--(7,11)--(-1,11)--(-1,5);
				\draw[thick,fill=yellow!60] (-1,7)--(0,7)--(0,8)--(1,8)--(1,9)--(4,9)--(4,10)--(7,10)--(7,11)--(-1,11)--(-1,7);
				\draw[thick,fill=blue!40!green!25] (-1,7)--(0,7)--(0,8)--(1,8)--(1,9)--(2,9)--(2,10)--(3,10)--(3,11)--(-1,11)--(-1,7);
				\draw[thick,fill=lime!70] (-1,9)--(0,9)--(0,10)--(3,10)--(3,11)--(-1,11)--(-1,9);
				\draw[thick, gray, dotted] (-1,-2)--(0,-2);
				\draw[thick, gray, dotted] (-1,-1)--(0,-1);
				\draw[thick, gray, dotted] (-1,0)--(0,0);
				\draw[thick, gray, dotted] (-1,1)--(1,1);
				\draw[thick,  gray, dotted] (-1,2)--(1,2);
				\draw[thick,  gray, dotted] (-1,3)--(1,3);
				\draw[thick,  gray, dotted] (-1,4)--(2,4);
				\draw[thick,  gray, dotted] (-1,5)--(2,5);
				\draw[thick,  gray, dotted] (-1,6)--(2,6);
				\draw[thick,  gray, dotted] (-1,7)--(4,7);
				\draw[thick,  gray, dotted] (-1,8)--(5,8);
				\draw[thick,  gray, dotted] (-1,9)--(6,9);
				\draw[thick,  gray, dotted] (-1,10)--(7,10);
				\draw[thick,  gray, dotted] (-1,11)--(8,11);
				\draw[thick,  gray, dotted] (0,11)--(0,0);
				\draw[thick,  gray, dotted] (1,11)--(1,0);
				\draw[thick,  gray, dotted] (2,11)--(2,3);
				\draw[thick,  gray, dotted] (3,11)--(3,6);
				\draw[thick,  gray, dotted] (4,11)--(4,7);
				\draw[thick,  gray, dotted] (5,11)--(5,8);
				\draw[thick,  gray, dotted] (6,11)--(6,9);
				\draw[thick, gray,  dotted] (7,11)--(7,9);
				\draw[thick,  gray, dotted] (8,11)--(8,9);
				\draw[thick,  gray, dotted] (9,11)--(9,10);
				\draw[thick,  gray, dotted] (10,11)--(10,10);
				\draw[thick, gray,  dotted] (11,11)--(11,10);
				\node at (-0.5,10.5){$\scriptstyle 3$};
				\node at (0.5,10.5){$\scriptstyle 0$};
				\node at (1.5,10.5){$\scriptstyle 1$};
				\node at (2.5,10.5){$\scriptstyle 2$};
				\node at (3.5,10.5){$\scriptstyle 3$};
				\node at (4.5,10.5){$\scriptstyle 0$};
				\node at (5.5,10.5){$\scriptstyle 1$};
				\node at (6.5,10.5){$\scriptstyle 2$};
				\node at (7.5,10.5){$\scriptstyle 3$};
				\node at (8.5,10.5){$\scriptstyle 0$};
				\node at (9.5,10.5){$\scriptstyle 1$};
				\node at (10.5,10.5){$\scriptstyle 2$};
				\node at (-0.5,9.5){$\scriptstyle 2$};
				\node at (0.5,9.5){$\scriptstyle 3$};
				\node at (1.5,9.5){$\scriptstyle 0$};
				\node at (2.5,9.5){$\scriptstyle 1$};
				\node at (3.5,9.5){$\scriptstyle 2$};
				\node at (4.5,9.5){$\scriptstyle 3$};
				\node at (5.5,9.5){$\scriptstyle 0$};
				\node at (6.5,9.5){$\scriptstyle 1$};
				\node at (7.5,9.5){$\scriptstyle 2$};
				\node at (-0.5,8.5){$\scriptstyle 1$};
				\node at (0.5,8.5){$\scriptstyle 2$};
				\node at (1.5,8.5){$\scriptstyle 3$};
				\node at (2.5,8.5){$\scriptstyle 0$};
				\node at (3.5,8.5){$\scriptstyle 1$};
				\node at (4.5,8.5){$\scriptstyle 2$};
				\node at (-0.5,7.5){$\scriptstyle 0$};
				\node at (0.5,7.5){$\scriptstyle 1$};
				\node at (1.5,7.5){$\scriptstyle 2$};
				\node at (2.5,7.5){$\scriptstyle 3$};
				\node at (3.5,7.5){$\scriptstyle 0$};
				\node at (-0.5,6.5){$\scriptstyle 3$};
				\node at (0.5,6.5){$\scriptstyle 0$};
				\node at (1.5,6.5){$\scriptstyle 1$};
				\node at (2.5,6.5){$\scriptstyle 2$};
				\node at (-0.5,5.5){$\scriptstyle 2$};
				\node at (0.5,5.5){$\scriptstyle 3$};
				\node at (1.5,5.5){$\scriptstyle 0$};
				\node at (-0.5,4.5){$\scriptstyle 1$};
				\node at (0.5,4.5){$\scriptstyle 2$};
				\node at (1.5,4.5){$\scriptstyle 3$};
				\node at (-0.5,3.5){$\scriptstyle 0$};
				\node at (0.5,3.5){$\scriptstyle 1$};
				\node at (1.5,3.5){$\scriptstyle 2$};
				\node at (-0.5,2.5){$\scriptstyle 3$};
				\node at (0.5,2.5){$\scriptstyle 0$};
				\node at (-0.5,1.5){$\scriptstyle 2$};
				\node at (0.5,1.5){$\scriptstyle 3$};
				\node at (-0.5,0.5){$\scriptstyle 1$};
				\node at (0.5,0.5){$\scriptstyle 2$};
				\node at (-0.5,-0.5){$\scriptstyle 0$};
				\node at (-0.5,-1.5){$\scriptstyle 3$};
				\node at (-0.5,-2.5){$\scriptstyle 2$};
				\draw[thick]  (-1,-3)--(0,-3)--(0,0)--(1,0)--(1,3)--(2,3)--(2,6)--(3,6)--(3,7)--(4,7)--(4,8)--(5,8)--(5,9)--(6,9)--(7,9)--(8,9)--(8,10)--(11,10)--(11,11)--(-1,11)--(-1,-3);
			\end{tikzpicture}
		}
		\\
		\\
		{}
		\bmu\ =\
		\hackcenter{
			\begin{tikzpicture}[scale=0.29]
				\draw[thick,fill=cyan!25!blue!15] (-1,-3)--(0,-3)--(0,0)--(1,0)--(1,3)--(2,3)--(2,6)--(3,6)--(3,7)--(4,7)--(4,8)--(5,8)--(5,9)--(6,9)--(6,10)--(7,10)--(7,11)--(10,11)--(10,12)--(-1,12)--(-1,-3);
				\draw[thick,fill=olive!30]  (-1,1)--(0,1)--(0,4)--(1,4)--(1,5)--(2,5)--(2,6)--(3,6)--(3,7)--(4,7)--(4,8)--(5,8)--(5,9)--(6,9)--(6,10)--(7,10)--(7,11)--(10,11)--(10,12)--(-1,12)--(-1,1);
				\draw[thick,fill=blue!10!violet!30]  (-1,3)--(0,3)--(0,4)--(1,4)--(1,5)--(2,5)--(2,6)--(3,6)--(3,7)--(4,7)--(4,8)--(5,8)--(5,9)--(6,9)--(6,10)--(7,10)--(7,11)--(10,11)--(10,12)--(-1,12)--(-1,3);
				\draw[thick,fill=yellow!70!orange!30]  (-1,5)--(0,5)--(0,6)--(1,6)--(1,7)--(2,7)--(2,8)--(3,8)--(3,9)--(4,9)--(4,10)--(7,10)--(7,11)--(10,11)--(10,12)--(-1,12)--(-1,5);
				\draw[thick,fill=red!70!blue!20]  (-1,7)--(0,7)--(0,8)--(1,8)--(1,9)--(2,9)--(2,10)--(3,10)--(3,11)--(6,11)--(6,12)--(-1,12)--(-1,7);
				\draw[thick,fill=green!50!gray!30]  (-1,9)--(0,9)--(0,10)--(1,10)--(1,11)--(2,11)--(2,12)--(-1,12)--(-1,9);
				\draw[thick,fill=blue!20]  (-1,11)--(2,11)--(2,12)--(-1,12)--(-1,11);
				\draw[thick,fill=orange!40]  (-1,11)--(0,11)--(0,12)--(-1,12)--(-1,11);
				\draw[thick,  gray, dotted] (-1,-2)--(0,-2);
				\draw[thick,  gray, dotted] (-1,-1)--(0,-1);
				\draw[thick,  gray, dotted] (-1,0)--(0,0);
				\draw[thick,  gray, dotted] (-1,1)--(1,1);
				\draw[thick,  gray, dotted] (-1,2)--(1,2);
				\draw[thick,  gray, dotted] (-1,3)--(1,3);
				\draw[thick,  gray, dotted] (-1,4)--(2,4);
				\draw[thick,  gray, dotted] (-1,5)--(2,5);
				\draw[thick,  gray, dotted] (-1,6)--(2,6);
				\draw[thick,  gray, dotted] (-1,7)--(4,7);
				\draw[thick,  gray, dotted] (-1,8)--(5,8);
				\draw[thick,  gray, dotted] (-1,9)--(6,9);
				\draw[thick,  gray, dotted] (-1,10)--(7,10);
				\draw[thick,  gray, dotted] (-1,11)--(8,11);
				\draw[thick,  gray, dotted] (-1,12)--(10,12);
				\draw[thick,  gray, dotted] (0,12)--(0,-2);
				\draw[thick,  gray, dotted] (1,12)--(1,1);
				\draw[thick,  gray, dotted] (2,12)--(2,5);
				\draw[thick,  gray, dotted] (3,12)--(3,6);
				\draw[thick,  gray, dotted] (4,12)--(4,7);
				\draw[thick,  gray, dotted] (5,12)--(5,8);
				\draw[thick,  gray, dotted] (6,12)--(6,9);
				\draw[thick, gray,  dotted] (7,12)--(7,10);
				\draw[thick,  gray, dotted] (8,12)--(8,11);
				\draw[thick,  gray, dotted] (9,12)--(9,11);
				\draw[thick,  gray, dotted] (10,12)--(10,11);
				\draw[thick, gray,  dotted] (11,12)--(11,12);
				\draw[thick,  gray, dotted] (12,12)--(12,12);
				\node at (-0.5,11.5){$\scriptstyle 0$};
				\node at (0.5,11.5){$\scriptstyle 1$};
				\node at (1.5,11.5){$\scriptstyle 2$};
				\node at (2.5,11.5){$\scriptstyle 3$};
				\node at (3.5,11.5){$\scriptstyle 0$};
				\node at (4.5,11.5){$\scriptstyle 1$};
				\node at (5.5,11.5){$\scriptstyle 2$};
				\node at (6.5,11.5){$\scriptstyle 3$};
				\node at (7.5,11.5){$\scriptstyle 0$};
				\node at (8.5,11.5){$\scriptstyle 1$};
				\node at (9.5,11.5){$\scriptstyle 2$};
				\node at (-0.5,10.5){$\scriptstyle 3$};
				\node at (0.5,10.5){$\scriptstyle 0$};
				\node at (1.5,10.5){$\scriptstyle 1$};
				\node at (2.5,10.5){$\scriptstyle 2$};
				\node at (3.5,10.5){$\scriptstyle 3$};
				\node at (4.5,10.5){$\scriptstyle 0$};
				\node at (5.5,10.5){$\scriptstyle 1$};
				\node at (6.5,10.5){$\scriptstyle 2$};
				\node at (-0.5,9.5){$\scriptstyle 2$};
				\node at (0.5,9.5){$\scriptstyle 3$};
				\node at (1.5,9.5){$\scriptstyle 0$};
				\node at (2.5,9.5){$\scriptstyle 1$};
				\node at (3.5,9.5){$\scriptstyle 2$};
				\node at (4.5,9.5){$\scriptstyle 3$};
				\node at (5.5,9.5){$\scriptstyle 0$};
				\node at (-0.5,8.5){$\scriptstyle 1$};
				\node at (0.5,8.5){$\scriptstyle 2$};
				\node at (1.5,8.5){$\scriptstyle 3$};
				\node at (2.5,8.5){$\scriptstyle 0$};
				\node at (3.5,8.5){$\scriptstyle 1$};
				\node at (4.5,8.5){$\scriptstyle 2$};
				\node at (-0.5,7.5){$\scriptstyle 0$};
				\node at (0.5,7.5){$\scriptstyle 1$};
				\node at (1.5,7.5){$\scriptstyle 2$};
				\node at (2.5,7.5){$\scriptstyle 3$};
				\node at (3.5,7.5){$\scriptstyle 0$};
				\node at (-0.5,6.5){$\scriptstyle 3$};
				\node at (0.5,6.5){$\scriptstyle 0$};
				\node at (1.5,6.5){$\scriptstyle 1$};
				\node at (2.5,6.5){$\scriptstyle 2$};
				\node at (-0.5,5.5){$\scriptstyle 2$};
				\node at (0.5,5.5){$\scriptstyle 3$};
				\node at (1.5,5.5){$\scriptstyle 0$};
				\node at (-0.5,4.5){$\scriptstyle 1$};
				\node at (0.5,4.5){$\scriptstyle 2$};
				\node at (1.5,4.5){$\scriptstyle 3$};
				\node at (-0.5,3.5){$\scriptstyle 0$};
				\node at (0.5,3.5){$\scriptstyle 1$};
				\node at (1.5,3.5){$\scriptstyle 2$};
				\node at (-0.5,2.5){$\scriptstyle 3$};
				\node at (0.5,2.5){$\scriptstyle 0$};
				\node at (-0.5,1.5){$\scriptstyle 2$};
				\node at (0.5,1.5){$\scriptstyle 3$};
				\node at (-0.5,0.5){$\scriptstyle 1$};
				\node at (0.5,0.5){$\scriptstyle 2$};
				\node at (-0.5,-0.5){$\scriptstyle 0$};
				\node at (-0.5,-1.5){$\scriptstyle 3$};
				\node at (-0.5,-2.5){$\scriptstyle 2$};
				\draw[thick]  (-1,-3)--(0,-3)--(0,0)--(1,0)--(1,3)--(2,3)--(2,6)--(3,6)--(3,7)--(4,7)--(4,8)--(5,8)--(5,9)--(6,9)--(6,10)--(7,10)--(7,11)--(10,11)--(10,12)--(-1,12)--(-1,-3);
				%
			\end{tikzpicture}
		}
		{}
		\hspace{-0.68cm}
		\hackcenter{
			\begin{tikzpicture}[scale=0.29]
				\node[white] at (-0.5,0){{}};
				\node[white] at (0.5,0){{}};
				\draw[thin, gray,fill=gray!30]  (0,0)--(0,14.6);
				\draw[thin, white]  (0,14.6)--(0,15);
			\end{tikzpicture}
		}
		{}
		\hackcenter{
			\begin{tikzpicture}[scale=0.29]
				\draw[thick,fill=olive!30] (-1,1)--(0,1)--(0,4)--(1,4)--(1,5)--(2,5)--(2,6)--(3,6)--(3,7)--(4,7)--(4,8)--(5,8)--(5,9)--(6,9)--(7,9)--(8,9)--(8,10)--(11,10)--(11,11)--(-1,11)--(-1,1);
				\draw[thick,fill=orange!40] (-1,3)--(0,3)--(0,4)--(1,4)--(1,5)--(2,5)--(2,6)--(3,6)--(3,7)--(4,7)--(4,8)--(5,8)--(5,9)--(6,9)--(7,9)--(8,9)--(8,10)--(11,10)--(11,11)--(-1,11)--(-1,3);
				\draw[thick,fill=purple!35] (-1,5)--(0,5)--(0,6)--(1,6)--(1,7)--(2,7)--(2,8)--(3,8)--(3,9)--(4,9)--(4,10)--(7,10)--(7,11)--(-1,11)--(-1,5);
				\draw[thick,fill=yellow!60] (-1,7)--(0,7)--(0,8)--(1,8)--(1,9)--(4,9)--(4,10)--(7,10)--(7,11)--(-1,11)--(-1,7);
				\draw[thick,fill=blue!40!green!25] (-1,7)--(0,7)--(0,8)--(1,8)--(1,9)--(2,9)--(2,10)--(3,10)--(3,11)--(-1,11)--(-1,7);
				\draw[thick,fill=lime!70] (-1,9)--(0,9)--(0,10)--(3,10)--(3,11)--(-1,11)--(-1,9);
				\draw[thick, gray, dotted] (-1,1)--(0,1);
				\draw[thick,  gray, dotted] (-1,2)--(0,2);
				\draw[thick,  gray, dotted] (-1,3)--(0,3);
				\draw[thick,  gray, dotted] (-1,4)--(1,4);
				\draw[thick,  gray, dotted] (-1,5)--(2,5);
				\draw[thick,  gray, dotted] (-1,6)--(2,6);
				\draw[thick,  gray, dotted] (-1,7)--(4,7);
				\draw[thick,  gray, dotted] (-1,8)--(5,8);
				\draw[thick,  gray, dotted] (-1,9)--(6,9);
				\draw[thick,  gray, dotted] (-1,10)--(7,10);
				\draw[thick,  gray, dotted] (-1,11)--(8,11);
				\draw[thick,  gray, dotted] (0,11)--(0,1);
				\draw[thick,  gray, dotted] (1,11)--(1,5);
				\draw[thick,  gray, dotted] (2,11)--(2,6);
				\draw[thick,  gray, dotted] (3,11)--(3,6);
				\draw[thick,  gray, dotted] (4,11)--(4,7);
				\draw[thick,  gray, dotted] (5,11)--(5,8);
				\draw[thick,  gray, dotted] (6,11)--(6,9);
				\draw[thick, gray,  dotted] (7,11)--(7,9);
				\draw[thick,  gray, dotted] (8,11)--(8,9);
				\draw[thick,  gray, dotted] (9,11)--(9,10);
				\draw[thick,  gray, dotted] (10,11)--(10,10);
				\draw[thick, gray,  dotted] (11,11)--(11,10);
				\node at (-0.5,10.5){$\scriptstyle 3$};
				\node at (0.5,10.5){$\scriptstyle 0$};
				\node at (1.5,10.5){$\scriptstyle 1$};
				\node at (2.5,10.5){$\scriptstyle 2$};
				\node at (3.5,10.5){$\scriptstyle 3$};
				\node at (4.5,10.5){$\scriptstyle 0$};
				\node at (5.5,10.5){$\scriptstyle 1$};
				\node at (6.5,10.5){$\scriptstyle 2$};
				\node at (7.5,10.5){$\scriptstyle 3$};
				\node at (8.5,10.5){$\scriptstyle 0$};
				\node at (9.5,10.5){$\scriptstyle 1$};
				\node at (10.5,10.5){$\scriptstyle 2$};
				\node at (-0.5,9.5){$\scriptstyle 2$};
				\node at (0.5,9.5){$\scriptstyle 3$};
				\node at (1.5,9.5){$\scriptstyle 0$};
				\node at (2.5,9.5){$\scriptstyle 1$};
				\node at (3.5,9.5){$\scriptstyle 2$};
				\node at (4.5,9.5){$\scriptstyle 3$};
				\node at (5.5,9.5){$\scriptstyle 0$};
				\node at (6.5,9.5){$\scriptstyle 1$};
				\node at (7.5,9.5){$\scriptstyle 2$};
				\node at (-0.5,8.5){$\scriptstyle 1$};
				\node at (0.5,8.5){$\scriptstyle 2$};
				\node at (1.5,8.5){$\scriptstyle 3$};
				\node at (2.5,8.5){$\scriptstyle 0$};
				\node at (3.5,8.5){$\scriptstyle 1$};
				\node at (4.5,8.5){$\scriptstyle 2$};
				\node at (-0.5,7.5){$\scriptstyle 0$};
				\node at (0.5,7.5){$\scriptstyle 1$};
				\node at (1.5,7.5){$\scriptstyle 2$};
				\node at (2.5,7.5){$\scriptstyle 3$};
				\node at (3.5,7.5){$\scriptstyle 0$};
				\node at (-0.5,6.5){$\scriptstyle 3$};
				\node at (0.5,6.5){$\scriptstyle 0$};
				\node at (1.5,6.5){$\scriptstyle 1$};
				\node at (2.5,6.5){$\scriptstyle 2$};
				\node at (-0.5,5.5){$\scriptstyle 2$};
				\node at (0.5,5.5){$\scriptstyle 3$};
				\node at (1.5,5.5){$\scriptstyle 0$};
				\node at (-0.5,4.5){$\scriptstyle 1$};
				\node at (0.5,4.5){$\scriptstyle 2$};
				\node at (-0.5,3.5){$\scriptstyle 0$};
				\node at (-0.5,2.5){$\scriptstyle 3$};
				\node at (-0.5,1.5){$\scriptstyle 2$};
				\draw[thick]  (-1,1)--(0,1)--(0,4)--(1,4)--(1,5)--(2,5)--(2,6)--(3,6)--(3,7)--(4,7)--(4,8)--(5,8)--(5,9)--(6,9)--(7,9)--(8,9)--(8,10)--(11,10)--(11,11)--(-1,11)--(-1,1);
								\phantom{
					\draw[thick,fill=blue]  (-1,-4)--(-1,1)--(1,1)--(1,-2)--(-1,-4);
				}
			\end{tikzpicture}
			}
			\\
			\\
		{}	
		\btau\ = \
		\hackcenter{
			\begin{tikzpicture}[scale=0.29]
				\draw[thick,fill=cyan!25]  (-1,1)--(0,1)--(0,4)--(1,4)--(1,6)--(1,7)--(2,7)--(2,8)--(3,8)--(3,9)--(4,9)--(4,10)--(7,10)--(7,11)--(10,11)--(10,12)--(-1,12)--(-1,1);
				\draw[thick,fill=yellow!70!orange!30]  (-1,5)--(0,5)--(0,6)--(1,6)--(1,7)--(2,7)--(2,8)--(3,8)--(3,9)--(4,9)--(4,10)--(7,10)--(7,11)--(10,11)--(10,12)--(-1,12)--(-1,5);
				\draw[thick,fill=red!70!blue!20]  (-1,7)--(0,7)--(0,8)--(1,8)--(1,9)--(2,9)--(2,10)--(3,10)--(3,11)--(6,11)--(6,12)--(-1,12)--(-1,7);
				\draw[thick,fill=green!50!gray!30]  (-1,9)--(0,9)--(0,10)--(1,10)--(1,11)--(2,11)--(2,12)--(-1,12)--(-1,9);
				\draw[thick,fill=blue!20]  (-1,11)--(2,11)--(2,12)--(-1,12)--(-1,11);
				\draw[thick,fill=orange!40]  (-1,11)--(0,11)--(0,12)--(-1,12)--(-1,11);
				\draw[thick,  gray, dotted] (-1,2)--(0,2);
				\draw[thick,  gray, dotted] (-1,3)--(0,3);
				\draw[thick,  gray, dotted] (-1,4)--(0,4);
				\draw[thick,  gray, dotted] (-1,5)--(1,5);
				\draw[thick,  gray, dotted] (-1,6)--(1,6);
				\draw[thick,  gray, dotted] (-1,7)--(2,7);
				\draw[thick,  gray, dotted] (-1,8)--(3,8);
				\draw[thick,  gray, dotted] (-1,9)--(4,9);
				\draw[thick,  gray, dotted] (-1,10)--(7,10);
				\draw[thick,  gray, dotted] (-1,11)--(8,11);
				\draw[thick,  gray, dotted] (-1,12)--(10,12);
				\draw[thick,  gray, dotted] (0,12)--(0,1);
				\draw[thick,  gray, dotted] (1,12)--(1,4);
				\draw[thick,  gray, dotted] (2,12)--(2,7);
				\draw[thick,  gray, dotted] (3,12)--(3,8);
				\draw[thick,  gray, dotted] (4,12)--(4,9);
				\draw[thick,  gray, dotted] (5,12)--(5,10);
				\draw[thick,  gray, dotted] (6,12)--(6,10);
				\draw[thick, gray,  dotted] (7,12)--(7,10);
				\draw[thick,  gray, dotted] (8,12)--(8,11);
				\draw[thick,  gray, dotted] (9,12)--(9,11);
				\draw[thick,  gray, dotted] (10,12)--(10,11);
				\draw[thick, gray,  dotted] (11,12)--(11,12);
				\draw[thick,  gray, dotted] (12,12)--(12,12);
				\node at (-0.5,11.5){$\scriptstyle 0$};
				\node at (0.5,11.5){$\scriptstyle 1$};
				\node at (1.5,11.5){$\scriptstyle 2$};
				\node at (2.5,11.5){$\scriptstyle 3$};
				\node at (3.5,11.5){$\scriptstyle 0$};
				\node at (4.5,11.5){$\scriptstyle 1$};
				\node at (5.5,11.5){$\scriptstyle 2$};
				\node at (6.5,11.5){$\scriptstyle 3$};
				\node at (7.5,11.5){$\scriptstyle 0$};
				\node at (8.5,11.5){$\scriptstyle 1$};
				\node at (9.5,11.5){$\scriptstyle 2$};
				\node at (-0.5,10.5){$\scriptstyle 3$};
				\node at (0.5,10.5){$\scriptstyle 0$};
				\node at (1.5,10.5){$\scriptstyle 1$};
				\node at (2.5,10.5){$\scriptstyle 2$};
				\node at (3.5,10.5){$\scriptstyle 3$};
				\node at (4.5,10.5){$\scriptstyle 0$};
				\node at (5.5,10.5){$\scriptstyle 1$};
				\node at (6.5,10.5){$\scriptstyle 2$};
				\node at (-0.5,9.5){$\scriptstyle 2$};
				\node at (0.5,9.5){$\scriptstyle 3$};
				\node at (1.5,9.5){$\scriptstyle 0$};
				\node at (2.5,9.5){$\scriptstyle 1$};
				\node at (3.5,9.5){$\scriptstyle 2$};
				\node at (-0.5,8.5){$\scriptstyle 1$};
				\node at (0.5,8.5){$\scriptstyle 2$};
				\node at (1.5,8.5){$\scriptstyle 3$};
				\node at (2.5,8.5){$\scriptstyle 0$};
				\node at (-0.5,7.5){$\scriptstyle 0$};
				\node at (0.5,7.5){$\scriptstyle 1$};
				\node at (1.5,7.5){$\scriptstyle 2$};
				\node at (-0.5,6.5){$\scriptstyle 3$};
				\node at (0.5,6.5){$\scriptstyle 0$};
				\node at (-0.5,5.5){$\scriptstyle 2$};
				\node at (0.5,5.5){$\scriptstyle 3$};
				\node at (-0.5,4.5){$\scriptstyle 1$};
				\node at (0.5,4.5){$\scriptstyle 2$};
				\node at (-0.5,3.5){$\scriptstyle 0$};
				\node at (-0.5,2.5){$\scriptstyle 3$};
				\node at (-0.5,1.5){$\scriptstyle 2$};
				\draw[thick]  (-1,1)--(0,1)--(0,4)--(1,4)--(1,6)--(1,7)--(2,7)--(2,8)--(3,8)--(3,9)--(4,9)--(4,10)--(7,10)--(7,11)--(10,11)--(10,12)--(-1,12)--(-1,1);
				%
				%
								\phantom{
					\draw[thick,fill=blue]  (-1,-2)--(-1,1)--(1,1)--(1,-2)--(-1,-2);
				}
			\end{tikzpicture}
		}
	\hspace{-0.68cm}
		\hackcenter{
			\begin{tikzpicture}[scale=0.29]
				\node[white] at (-0.5,0){{}};
				\node[white] at (0.5,0){{}};
				\draw[thin, gray,fill=gray!30]  (0,0)--(0,13.6);
				\draw[thin, white]  (0,13.6)--(0,14);
			\end{tikzpicture}
		}
		\hackcenter{
			\begin{tikzpicture}[scale=0.29]
				\draw[thick,fill=cyan!25] (-1,-3)--(0,-3)--(0,0)--(1,0)--(1,3)--(2,3)--(2,4)--(3,4)--(3,5)--(4,5)--(4,6)--(5,6)--(5,7)--(6,7)--(6,8)--(7,8)--(7,9)--(8,9)--(8,10)--(11,10)--(11,11)--(-1,11)--(-1,-3);
				\draw[thick,fill=red!40] (-1,1)--(0,1)--(0,2)--(1,2)--(1,3)--(2,3)--(2,4)--(3,4)--(3,5)--(4,5)--(4,6)--(5,6)--(5,7)--(6,7)--(6,8)--(7,8)--(7,9)--(8,9)--(8,10)--(11,10)--(11,11)--(-1,11)--(-1,-3);
				\draw[thick,fill=orange!40] (-1,3)--(0,3)--(0,4)--(1,4)--(1,5)--(2,5)--(2,6)--(3,6)--(3,7)--(4,7)--(4,8)--(5,8)--(5,9)--(6,9)--(7,9)--(8,9)--(8,10)--(11,10)--(11,11)--(-1,11)--(-1,-3);
				\draw[thick,fill=purple!35] (-1,5)--(0,5)--(0,6)--(1,6)--(1,7)--(2,7)--(2,8)--(3,8)--(3,9)--(4,9)--(4,10)--(7,10)--(7,11)--(-1,11)--(-1,5);
				\draw[thick,fill=yellow!60] (-1,7)--(0,7)--(0,8)--(1,8)--(1,9)--(4,9)--(4,10)--(7,10)--(7,11)--(-1,11)--(-1,7);
				\draw[thick,fill=blue!40!green!25] (-1,7)--(0,7)--(0,8)--(1,8)--(1,9)--(2,9)--(2,10)--(3,10)--(3,11)--(-1,11)--(-1,7);
				\draw[thick,fill=lime!70] (-1,9)--(0,9)--(0,10)--(3,10)--(3,11)--(-1,11)--(-1,9);
				\draw[thick, gray, dotted] (-1,-2)--(0,-2);
				\draw[thick, gray, dotted] (-1,-1)--(0,-1);
				\draw[thick, gray, dotted] (-1,0)--(0,0);
				\draw[thick, gray, dotted] (-1,1)--(1,1);
				\draw[thick,  gray, dotted] (-1,2)--(1,2);
				\draw[thick,  gray, dotted] (-1,3)--(1,3);
				\draw[thick,  gray, dotted] (-1,4)--(2,4);
				\draw[thick,  gray, dotted] (-1,5)--(3,5);
				\draw[thick,  gray, dotted] (-1,6)--(4,6);
				\draw[thick,  gray, dotted] (-1,7)--(5,7);
				\draw[thick,  gray, dotted] (-1,8)--(6,8);
				\draw[thick,  gray, dotted] (-1,9)--(7,9);
				\draw[thick,  gray, dotted] (-1,10)--(7,10);
				\draw[thick,  gray, dotted] (-1,11)--(8,11);
				\draw[thick,  gray, dotted] (0,11)--(0,0);
				\draw[thick,  gray, dotted] (1,11)--(1,0);
				\draw[thick,  gray, dotted] (2,11)--(2,3);
				\draw[thick,  gray, dotted] (3,11)--(3,5);
				\draw[thick,  gray, dotted] (4,11)--(4,6);
				\draw[thick,  gray, dotted] (5,11)--(5,7);
				\draw[thick,  gray, dotted] (6,11)--(6,8);
				\draw[thick, gray,  dotted] (7,11)--(7,9);
				\draw[thick,  gray, dotted] (8,11)--(8,9);
				\draw[thick,  gray, dotted] (9,11)--(9,10);
				\draw[thick,  gray, dotted] (10,11)--(10,10);
				\draw[thick, gray,  dotted] (11,11)--(11,10);
				\node at (-0.5,10.5){$\scriptstyle 3$};
				\node at (0.5,10.5){$\scriptstyle 0$};
				\node at (1.5,10.5){$\scriptstyle 1$};
				\node at (2.5,10.5){$\scriptstyle 2$};
				\node at (3.5,10.5){$\scriptstyle 3$};
				\node at (4.5,10.5){$\scriptstyle 0$};
				\node at (5.5,10.5){$\scriptstyle 1$};
				\node at (6.5,10.5){$\scriptstyle 2$};
				\node at (7.5,10.5){$\scriptstyle 3$};
				\node at (8.5,10.5){$\scriptstyle 0$};
				\node at (9.5,10.5){$\scriptstyle 1$};
				\node at (10.5,10.5){$\scriptstyle 2$};
				\node at (-0.5,9.5){$\scriptstyle 2$};
				\node at (0.5,9.5){$\scriptstyle 3$};
				\node at (1.5,9.5){$\scriptstyle 0$};
				\node at (2.5,9.5){$\scriptstyle 1$};
				\node at (3.5,9.5){$\scriptstyle 2$};
				\node at (4.5,9.5){$\scriptstyle 3$};
				\node at (5.5,9.5){$\scriptstyle 0$};
				\node at (6.5,9.5){$\scriptstyle 1$};
				\node at (7.5,9.5){$\scriptstyle 2$};
				\node at (-0.5,8.5){$\scriptstyle 1$};
				\node at (0.5,8.5){$\scriptstyle 2$};
				\node at (1.5,8.5){$\scriptstyle 3$};
				\node at (2.5,8.5){$\scriptstyle 0$};
				\node at (3.5,8.5){$\scriptstyle 1$};
				\node at (4.5,8.5){$\scriptstyle 2$};
				\node at (5.5,8.5){$\scriptstyle 3$};
				\node at (6.5,8.5){$\scriptstyle 0$};
				\node at (-0.5,7.5){$\scriptstyle 0$};
				\node at (0.5,7.5){$\scriptstyle 1$};
				\node at (1.5,7.5){$\scriptstyle 2$};
				\node at (2.5,7.5){$\scriptstyle 3$};
				\node at (3.5,7.5){$\scriptstyle 0$};
				\node at (4.5,7.5){$\scriptstyle 1$};
				\node at (5.5,7.5){$\scriptstyle 2$};
				\node at (-0.5,6.5){$\scriptstyle 3$};
				\node at (0.5,6.5){$\scriptstyle 0$};
				\node at (1.5,6.5){$\scriptstyle 1$};
				\node at (2.5,6.5){$\scriptstyle 2$};
				\node at (3.5,6.5){$\scriptstyle 3$};
				\node at (4.5,6.5){$\scriptstyle 0$};
				\node at (-0.5,5.5){$\scriptstyle 2$};
				\node at (0.5,5.5){$\scriptstyle 3$};
				\node at (1.5,5.5){$\scriptstyle 0$};
				\node at (2.5,5.5){$\scriptstyle 1$};
				\node at (3.5,5.5){$\scriptstyle 2$};
				\node at (-0.5,4.5){$\scriptstyle 1$};
				\node at (0.5,4.5){$\scriptstyle 2$};
				\node at (1.5,4.5){$\scriptstyle 3$};
				\node at (2.5,4.5){$\scriptstyle 0$};
				\node at (-0.5,3.5){$\scriptstyle 0$};
				\node at (0.5,3.5){$\scriptstyle 1$};
				\node at (1.5,3.5){$\scriptstyle 2$};
				\node at (-0.5,2.5){$\scriptstyle 3$};
				\node at (0.5,2.5){$\scriptstyle 0$};
				\node at (-0.5,1.5){$\scriptstyle 2$};
				\node at (0.5,1.5){$\scriptstyle 3$};
				\node at (-0.5,0.5){$\scriptstyle 1$};
				\node at (0.5,0.5){$\scriptstyle 2$};
				\node at (-0.5,-0.5){$\scriptstyle 0$};
				\node at (-0.5,-1.5){$\scriptstyle 3$};
				\node at (-0.5,-2.5){$\scriptstyle 2$};
				\draw[thick]  (-1,-3)--(0,-3)--(0,0)--(1,0)--(1,3)--(2,3)--(2,4)--(3,4)--(3,5)--(4,5)--(4,6)--(5,6)--(5,7)--(6,7)--(6,8)--(7,8)--(7,9)--(8,9)--(8,10)--(11,10)--(11,11)--(-1,11)--(-1,-3);
			\end{tikzpicture}
		}
\end{align*}	
\caption{
The three 2-multicores \( \blam, \bmu, \btau\) in the core block \(\Lambda^{\bkap}_+( 26 \alpha_0 + 23 \alpha_1 + 30 \alpha_2 + 24 \alpha_3)\) for \(e=4\) and multicharge \(\bkap=(0,3)\), considered in \cref{3coretileex}. The cuspidal Kostant tilings associated to the convex preorder \(\succeq\) defined in that example are highlighted. 
}
\label{corebigtilings}      
\end{figure}

\begin{figure}[h]
\begin{align*}
{}
\bzeta\ =\ 
\hackcenter{
\begin{tikzpicture}[scale=0.29]
\draw[thick,fill=cyan!25] (0,0)--(1,0)--(1,3)--(2,3)--(2,6)--(3,6)--(3,7)--(4,7)--(4,8)--(5,8)--(5,9)--(6,9)--(6,10)--(7,10)--(7,11)--(10,11)--(10,12)--(13,12)--(13,13)--(0,13)--(0,0);
\draw[thick, fill=yellow!40!orange!50] (0,4)--(1,4)--(1,5)--(2,5)--(2,6)--(3,6)--(3,7)--(4,7)--(4,8)--(5,8)--(5,9)--(6,9)--(6,10)--(7,10)--(7,11)--(10,11)--(10,12)--(13,12)--(13,13)--(0,13)--(0,2);
\draw[thick, fill=brown!35!green!30] (0,6)--(1,6)--(1,7)--(2,7)--(2,8)--(3,8)--(3,9)--(4,9)--(4,10)--(7,10)--(7,11)--(10,11)--(10,12)--(13,12)--(13,13)--(0,13)--(0,2);
\draw[thick, fill=purple!25] (0,6)--(1,6)--(1,7)--(2,7)--(2,8)--(3,8)--(3,9)--(4,9)--(4,10)--(5,10)--(5,11)--(6,11)--(6,12)--(9,12)--(9,13)--(0,13)--(0,6);
\draw[thick, fill=yellow!60] (0,8)--(1,8)--(1,9)--(2,9)--(2,10)--(3,10)--(3,11)--(4,11)--(6,11)--(6,12)--(9,12)--(9,13)--(0,13)--(0,8);
\draw[thick, fill=blue!40!green!25] (0,8)--(1,8)--(1,9)--(2,9)--(2,10)--(3,10)--(3,11)--(4,11)--(4,12)--(5,12)--(5,13)--(0,13)--(0,8);
\draw[thick, fill=red!30] (0,8)--(1,8)--(1,11)--(2,11)--(2,12)--(5,12)--(5,13)--(0,13)--(0,8);
\draw[thick, fill=cyan!35] (0,12)--(1,12)--(1,13)--(0,13)--(0,12);
\draw[thick,red, fill=lightgray!30] (2,5)--(4,5)--(4,6)--(5,6)--(5,7)--(3,7)--(3,6)--(2,6)--(2,5);
\draw[thick,red, fill=lightgray!30](4,7)--(6,7)--(6,8)--(7,8)--(7,9)--(5,9)--(5,8)--(4,8)--(4,7);
\draw[thick,red, fill=lightgray!30] (6,9)--(8,9)--(8,10)--(9,10)--(9,11)--(7,11)--(7,10)--(6,10)--(6,9);
\draw[thick,red, fill=lightgray!30] (7,8)--(9,8)--(9,9)--(10,9)--(10,10)--(8,10)--(8,9)--(7,9)--(7,8);
\draw[thick,red, fill=lightgray!30] (9,10)--(11,10)--(11,11)--(12,11)--(12,12)--(10,12)--(10,11)--(9,11)--(9,10);
\draw[thick, gray, dotted] (0,1)--(1,1);
\draw[thick,  gray, dotted] (0,2)--(1,2);
\draw[thick,  gray, dotted] (0,3)--(1,3);
\draw[thick,  gray, dotted] (0,4)--(2,4);
\draw[thick,  gray, dotted] (0,5)--(2,5);
\draw[thick,  gray, dotted] (0,6)--(4,6);
\draw[thick,  gray, dotted] (0,7)--(5,7);
\draw[thick,  gray, dotted] (0,8)--(6,8);
\draw[thick,  gray, dotted] (0,9)--(9,9);
\draw[thick,  gray, dotted] (0,10)--(10,10);
\draw[thick,  gray, dotted] (0,11)--(11,11);
\draw[thick,  gray, dotted] (0,12)--(12,12);
\draw[thick,  gray, dotted] (1,13)--(1,0);
\draw[thick,  gray, dotted] (2,13)--(2,3);
\draw[thick,  gray, dotted] (3,13)--(3,5);
\draw[thick,  gray, dotted] (4,13)--(4,6);
\draw[thick,  gray, dotted] (5,13)--(5,7);
\draw[thick,  gray, dotted] (6,13)--(6,8);
\draw[thick, gray,  dotted] (7,13)--(7,8);
\draw[thick,  gray, dotted] (8,13)--(8,8);
\draw[thick,  gray, dotted] (9,13)--(9,8);
\draw[thick,  gray, dotted] (10,13)--(10,9);
\draw[thick, gray,  dotted] (11,13)--(11,11);
\draw[thick,  gray, dotted] (12,13)--(12,12);
\node at (0.5,12.5){$\scriptstyle 2$};
\node at (1.5,12.5){$\scriptstyle 3$};
\node at (2.5,12.5){$\scriptstyle 0$};
\node at (3.5,12.5){$\scriptstyle 1$};
\node at (4.5,12.5){$\scriptstyle 2$};
\node at (5.5,12.5){$\scriptstyle 3$};
\node at (6.5,12.5){$\scriptstyle 0$};
\node at (7.5,12.5){$\scriptstyle 1$};
\node at (8.5,12.5){$\scriptstyle 2$};
\node at (9.5,12.5){$\scriptstyle 3$};
\node at (10.5,12.5){$\scriptstyle 0$};
\node at (11.5,12.5){$\scriptstyle 1$};
\node at (12.5,12.5){$\scriptstyle 2$};
\node at (0.5,11.5){$\scriptstyle 1$};
\node at (1.5,11.5){$\scriptstyle 2$};
\node at (2.5,11.5){$\scriptstyle 3$};
\node at (3.5,11.5){$\scriptstyle 0$};
\node at (4.5,11.5){$\scriptstyle 1$};
\node at (5.5,11.5){$\scriptstyle 2$};
\node at (6.5,11.5){$\scriptstyle 3$};
\node at (7.5,11.5){$\scriptstyle 0$};
\node at (8.5,11.5){$\scriptstyle 1$};
\node at (9.5,11.5){$\scriptstyle 2$};
\node at (10.5,11.5){$\scriptstyle 3$};
\node at (11.5,11.5){$\scriptstyle 0$};
\node at (0.5,10.5){$\scriptstyle 0$};
\node at (1.5,10.5){$\scriptstyle 1$};
\node at (2.5,10.5){$\scriptstyle 2$};
\node at (3.5,10.5){$\scriptstyle 3$};
\node at (4.5,10.5){$\scriptstyle 0$};
\node at (5.5,10.5){$\scriptstyle 1$};
\node at (6.5,10.5){$\scriptstyle 2$};
\node at (7.5,10.5){$\scriptstyle 3$};
\node at (8.5,10.5){$\scriptstyle 0$};
\node at (9.5,10.5){$\scriptstyle 1$};
\node at (10.5,10.5){$\scriptstyle 2$};
\node at (0.5,9.5){$\scriptstyle 3$};
\node at (1.5,9.5){$\scriptstyle 0$};
\node at (2.5,9.5){$\scriptstyle 1$};
\node at (3.5,9.5){$\scriptstyle 2$};
\node at (4.5,9.5){$\scriptstyle 3$};
\node at (5.5,9.5){$\scriptstyle 0$};
\node at (6.5,9.5){$\scriptstyle 1$};
\node at (7.5,9.5){$\scriptstyle 2$};
\node at (8.5,9.5){$\scriptstyle 3$};
\node at (9.5,9.5){$\scriptstyle 0$};
\node at (0.5,8.5){$\scriptstyle 2$};
\node at (1.5,8.5){$\scriptstyle 3$};
\node at (2.5,8.5){$\scriptstyle 0$};
\node at (3.5,8.5){$\scriptstyle 1$};
\node at (4.5,8.5){$\scriptstyle 2$};
\node at (5.5,8.5){$\scriptstyle 3$};
\node at (6.5,8.5){$\scriptstyle 0$};
\node at (7.5,8.5){$\scriptstyle 1$};
\node at (8.5,8.5){$\scriptstyle 2$};
\node at (0.5,7.5){$\scriptstyle 1$};
\node at (1.5,7.5){$\scriptstyle 2$};
\node at (2.5,7.5){$\scriptstyle 3$};
\node at (3.5,7.5){$\scriptstyle 0$};
\node at (4.5,7.5){$\scriptstyle 1$};
\node at (5.5,7.5){$\scriptstyle 2$};
\node at (0.5,6.5){$\scriptstyle 0$};
\node at (1.5,6.5){$\scriptstyle 1$};
\node at (2.5,6.5){$\scriptstyle 2$};
\node at (3.5,6.5){$\scriptstyle 3$};
\node at (4.5,6.5){$\scriptstyle 0$};
\node at (0.5,5.5){$\scriptstyle 3$};
\node at (1.5,5.5){$\scriptstyle 0$};
\node at (2.5,5.5){$\scriptstyle 1$};
\node at (3.5,5.5){$\scriptstyle 2$};
\node at (0.5,4.5){$\scriptstyle 2$};
\node at (1.5,4.5){$\scriptstyle 3$};
\node at (0.5,3.5){$\scriptstyle 1$};
\node at (1.5,3.5){$\scriptstyle 2$};
\node at (0.5,2.5){$\scriptstyle 0$};
\node at (0.5,1.5){$\scriptstyle 3$};
\node at (0.5,0.5){$\scriptstyle 2$};
\draw[ thick]  (0,0)--(1,0)--(1,3)--(2,3)--(2,5)--(4,5)--(4,6)--(5,6)--(5,7)--(6,7)--(6,8)--(9,8)--(9,9)--(10,9)--(10,10)--(11,10)--(11,11)--(12,11)--(12,12)--(13,12)--(13,13)--(0,13)--(0,0);
\draw[ thick,black](2,5)--(4,5)--(4,6)--(5,6)--(5,7)--(3,7)--(3,6)--(2,6)--(2,5);
\draw[  thick,black] (4,7)--(6,7)--(6,8)--(7,8)--(7,9)--(5,9)--(5,8)--(4,8)--(4,7);
\draw[  thick,black] (6,9)--(8,9)--(8,10)--(9,10)--(9,11)--(7,11)--(7,10)--(6,10)--(6,9);
\draw[  thick,black] (7,8)--(9,8)--(9,9)--(10,9)--(10,10)--(8,10)--(8,9)--(7,9)--(7,8);
\draw[  thick,black] (9,10)--(11,10)--(11,11)--(12,11)--(12,12)--(10,12)--(10,11)--(9,11)--(9,10);
%
\end{tikzpicture}
}
{}
\hackcenter{
	\begin{tikzpicture}[scale=0.29]
		\node[white] at (-0.5,0){{}};
		\node[white] at (0.5,0){{}};
		\draw[thin, gray,fill=gray!30]  (0,0)--(0,12.6);
		\draw[thin, white]  (0,12.6)--(0,13);
	\end{tikzpicture}
}
{}
\hackcenter{
\begin{tikzpicture}[scale=0.29]
\draw[thick,fill=cyan!25] (0,0)--(1,0)--(1,3)--(2,3)--(2,6)--(3,6)--(3,7)--(4,7)--(4,8)--(5,8)--(5,9)--(6,9)--(7,9)--(8,9)--(8,10)--(11,10)--(11,11)--(0,11)--(0,0);
\draw[thick,fill=yellow!70!orange!30] (0,4)--(1,4)--(1,5)--(2,5)--(2,6)--(3,6)--(3,7)--(4,7)--(4,8)--(5,8)--(5,9)--(6,9)--(7,9)--(8,9)--(8,10)--(11,10)--(11,11)--(0,11)--(0,0);
\draw[thick,fill=red!70!blue!20] (0,6)--(1,6)--(1,7)--(2,7)--(2,8)--(3,8)--(3,9)--(4,9)--(4,10)--(7,10)--(7,11)--(0,11)--(0,6);
\draw[thick,fill=green!50!gray!30] (0,8)--(1,8)--(1,9)--(2,9)--(2,10)--(3,10)--(3,11)--(0,11)--(0,8);
\draw[thick,fill=blue!20] (0,10)--(3,10)--(3,11)--(0,11)--(0,10);
\draw[thick,fill=orange!40] (0,10)--(1,10)--(1,11)--(0,11)--(0,10);
\draw[thick,red, fill=lightgray!30] (11,10)--(15,10)--(15,11)--(11,11)--(11,10);
\draw[thick,red, fill=lightgray!30]   (8,9)--(12,9)--(12,10)--(8,10)--(8,9);
\draw[thick,red, fill=lightgray!30]  (4,7)--(6,7)--(6,8)--(7,8)--(7,9)--(5,9)--(5,8)--(4,8)--(4,7);
\draw[thick,red, fill=lightgray!30]   (2,4)--(3,4)--(3,5)--(4,5)--(4,7)--(3,7)--(3,6)--(2,6)--(2,4);
\draw[thick,red, fill=lightgray!30]   (1,1)--(2,1)--(2,2)--(3,2)--(3,4)--(2,4)--(2,3)--(1,3)--(1,1);
\draw[thick, gray, dotted] (0,1)--(1,1);
\draw[thick,  gray, dotted] (0,2)--(2,2);
\draw[thick,  gray, dotted] (0,3)--(3,3);
\draw[thick,  gray, dotted] (0,4)--(3,4);
\draw[thick,  gray, dotted] (0,5)--(3,5);
\draw[thick,  gray, dotted] (0,6)--(4,6);
\draw[thick,  gray, dotted] (0,7)--(5,7);
\draw[thick,  gray, dotted] (0,8)--(6,8);
\draw[thick,  gray, dotted] (0,9)--(7,9);
\draw[thick,  gray, dotted] (0,10)--(12,10);
\draw[thick,  gray, dotted] (0,11)--(15,11);
\draw[thick,  gray, dotted] (1,11)--(1,0);
\draw[thick,  gray, dotted] (2,11)--(2,2);
\draw[thick,  gray, dotted] (3,11)--(3,5);
\draw[thick,  gray, dotted] (4,11)--(4,7);
\draw[thick,  gray, dotted] (5,11)--(5,7);
\draw[thick,  gray, dotted] (6,11)--(6,8);
\draw[thick, gray,  dotted] (7,11)--(7,8);
\draw[thick,  gray, dotted] (8,11)--(8,9);
\draw[thick,  gray, dotted] (9,11)--(9,9);
\draw[thick,  gray, dotted] (10,11)--(10,9);
\draw[thick, gray,  dotted] (11,11)--(11,9);
\draw[thick, gray,  dotted] (12,11)--(12,9);
\draw[thick, gray,  dotted] (13,11)--(13,10);
\draw[thick, gray,  dotted] (14,11)--(14,10);
\node at (0.5,10.5){$\scriptstyle 0$};
\node at (1.5,10.5){$\scriptstyle 1$};
\node at (2.5,10.5){$\scriptstyle 2$};
\node at (3.5,10.5){$\scriptstyle 3$};
\node at (4.5,10.5){$\scriptstyle 0$};
\node at (5.5,10.5){$\scriptstyle 1$};
\node at (6.5,10.5){$\scriptstyle 2$};
\node at (7.5,10.5){$\scriptstyle 3$};
\node at (8.5,10.5){$\scriptstyle 0$};
\node at (9.5,10.5){$\scriptstyle 1$};
\node at (10.5,10.5){$\scriptstyle 2$};
\node at (11.5,10.5){$\scriptstyle 3$};
\node at (12.5,10.5){$\scriptstyle 0$};
\node at (13.5,10.5){$\scriptstyle 1$};
\node at (14.5,10.5){$\scriptstyle 2$};
\node at (0.5,9.5){$\scriptstyle 3$};
\node at (1.5,9.5){$\scriptstyle 0$};
\node at (2.5,9.5){$\scriptstyle 1$};
\node at (3.5,9.5){$\scriptstyle 2$};
\node at (4.5,9.5){$\scriptstyle 3$};
\node at (5.5,9.5){$\scriptstyle 0$};
\node at (6.5,9.5){$\scriptstyle 1$};
\node at (7.5,9.5){$\scriptstyle 2$};
\node at (8.5,9.5){$\scriptstyle 3$};
\node at (9.5,9.5){$\scriptstyle 0$};
\node at (10.5,9.5){$\scriptstyle 1$};
\node at (11.5,9.5){$\scriptstyle 2$};
\node at (0.5,8.5){$\scriptstyle 2$};
\node at (1.5,8.5){$\scriptstyle 3$};
\node at (2.5,8.5){$\scriptstyle 0$};
\node at (3.5,8.5){$\scriptstyle 1$};
\node at (4.5,8.5){$\scriptstyle 2$};
\node at (5.5,8.5){$\scriptstyle 3$};
\node at (6.5,8.5){$\scriptstyle 0$};
\node at (0.5,7.5){$\scriptstyle 1$};
\node at (1.5,7.5){$\scriptstyle 2$};
\node at (2.5,7.5){$\scriptstyle 3$};
\node at (3.5,7.5){$\scriptstyle 0$};
\node at (4.5,7.5){$\scriptstyle 1$};
\node at (5.5,7.5){$\scriptstyle 2$};
\node at (0.5,6.5){$\scriptstyle 0$};
\node at (1.5,6.5){$\scriptstyle 1$};
\node at (2.5,6.5){$\scriptstyle 2$};
\node at (3.5,6.5){$\scriptstyle 3$};
\node at (0.5,5.5){$\scriptstyle 3$};
\node at (1.5,5.5){$\scriptstyle 0$};
\node at (2.5,5.5){$\scriptstyle 1$};
\node at (3.5,5.5){$\scriptstyle 2$};
\node at (0.5,4.5){$\scriptstyle 2$};
\node at (1.5,4.5){$\scriptstyle 3$};
\node at (2.5,4.5){$\scriptstyle 0$};
\node at (0.5,3.5){$\scriptstyle 1$};
\node at (1.5,3.5){$\scriptstyle 2$};
\node at (2.5,3.5){$\scriptstyle 3$};
\node at (0.5,2.5){$\scriptstyle 0$};
\node at (1.5,2.5){$\scriptstyle 1$};
\node at (2.5,2.5){$\scriptstyle 2$};
\node at (0.5,1.5){$\scriptstyle 3$};
\node at (1.5,1.5){$\scriptstyle 0$};
\node at (0.5,0.5){$\scriptstyle 2$};
\draw[ thick,black] (11,10)--(15,10)--(15,11)--(11,11)--(11,10);
\draw[ thick,black]  (8,9)--(12,9)--(12,10)--(8,10)--(8,9);
\draw[ thick,black] (4,7)--(6,7)--(6,8)--(7,8)--(7,9)--(5,9)--(5,8)--(4,8)--(4,7);
\draw[ thick,black]  (2,4)--(3,4)--(3,5)--(4,5)--(4,7)--(3,7)--(3,6)--(2,6)--(2,4);
\draw[ thick,black]  (1,1)--(2,1)--(2,2)--(3,2)--(3,4)--(2,4)--(2,3)--(1,3)--(1,1);
\draw[thick]  (0,0)--(1,0)--(1,1)--(2,1)--(2,2)--(3,2)--(3,5)--(4,5)--(4,7)--(6,7)--(6,8)--(7,8)--(7,9)--(12,9)--(12,10)--(15,10)--(15,11)--(0,11)--(0,0);
\phantom{
\draw[fill=blue]
(0,0)--(0,-2)--(1,-2)--(1,0)--(0,0);
}
\end{tikzpicture}
}
\end{align*}
\caption{
The cuspidal Kostant tiling for the \((\bkap, \theta)\)-RoCK multipartition \(\bzeta\) considered in \cref{ExBigRoCKnot,3coretileex}.
}
\label{exwithdeltaribsfig}      
\end{figure}

\begin{Example}\label{3coretileex}
Let \(e=4\), \(\bkap = (0,3)\) and \(\omega = 26 \alpha_0 + 23 \alpha_1 + 30 \alpha_2 + 24 \alpha_3\). Consider the core block \(\Lambda^{\bkap}_+(\omega)\). The three members: 
\begin{align*}
\blam &= ((11,8,7,6,5,4,3,2,1^3)\mid(12,9,6,5,4,3^3,2^3,1^3));\\
\bmu &= ((11,8,7,6,5,4,3^3,2^3,1^3)\mid(12,9,6,5,4,3,2,1^3));\\
\btau&= ((11,8,5,4,3,2^3,1^3)\mid(12,9,8,7,6,5,4,3,2^3,1^3)),
\end{align*}
 of this block are displayed in Figure~\ref{corebigtilings}. The \(\bkap\)-beta numbers for \(\blam\) are:
\begin{align*}
{}
\B^1(\blam, \bkap) &=
\hackcenter{
\begin{tikzpicture}[scale=0.46]
\draw[thick, white] (0,0.8);
\draw[thick, lightgray, dotted] (14.5,0)--(15.5,0);
\draw[thick, black, dotted] (-18.5,0)--(-17.5,0);
\draw[thick, lightgray ] (-17.5,0)--(14.5,0);
\blackdot(10,0);
\blackdot(6,0);
\blackdot(4,0);
\blackdot(2,0);
\blackdot(0,0);
\blackdot(-2,0);
\blackdot(-4,0);
\blackdot(-6,0);
\blackdot(-8,0);
\blackdot(-9,0);
\blackdot(-10,0);
\blackdot(-12,0);
\blackdot(-13,0);
\blackdot(-14,0);
\blackdot(-15,0);
\blackdot(-16,0);
\blackdot(-17,0);
\node[below] at (-17,-0.2){$\scriptstyle \textup{-}17$};
\node[below] at (-16,-0.2){$\scriptstyle \textup{-}16$};
\node[below] at (-15,-0.2){$\scriptstyle \textup{-}15$};
\node[below] at (-14,-0.2){$\scriptstyle \textup{-}14$};
\node[below] at (-13,-0.2){$\scriptstyle \textup{-}13$};
\node[below] at (-12,-0.2){$\scriptstyle \textup{-}12$};
\node[below] at (-11,-0.2){$\scriptstyle \textup{-}11$};
\node[below] at (-10,-0.2){$\scriptstyle \textup{-}10$};
\node[below] at (-9,-0.2){$\scriptstyle \textup{-}9$};
\node[below] at (-8,-0.2){$\scriptstyle \textup{-}8$};
\node[below] at (-7,-0.2){$\scriptstyle \textup{-}7$};
\node[below] at (-6,-0.2){$\scriptstyle \textup{-}6$};
\node[below] at (-5,-0.2){$\scriptstyle \textup{-}5$};
\node[below] at (-4,-0.2){$\scriptstyle \textup{-}4$};
\node[below] at (-3,-0.2){$\scriptstyle \textup{-}3$};
\node[below] at (-2,-0.2){$\scriptstyle \textup{-}2$};
\node[below] at (-1,-0.2){$\scriptstyle \textup{-}1$};
\node[below] at (0,-0.2){$\scriptstyle 0$};
\node[below] at (1,-0.2){$\scriptstyle 1$};
\node[below] at (2,-0.2){$\scriptstyle 2$};
\node[below] at (3,-0.2){$\scriptstyle 3$};
\node[below] at (4,-0.2){$\scriptstyle 4$};
\node[below] at (5,-0.2){$\scriptstyle 5$};
\node[below] at (6,-0.2){$\scriptstyle 6$};
\node[below] at (7,-0.2){$\scriptstyle 7$};
\node[below] at (8,-0.2){$\scriptstyle 8$};
\node[below] at (9,-0.2){$\scriptstyle 9$};
\node[below] at (10,-0.2){$\scriptstyle 10$};
\node[below] at (11,-0.2){$\scriptstyle 11$};
\node[below] at (12,-0.2){$\scriptstyle 12$};
\node[below] at (13,-0.2){$\scriptstyle 13$};
\node[below] at (14,-0.2){$\scriptstyle 14$};
\end{tikzpicture}
}\\
{}
\B^2(\blam, \bkap) &=
\hackcenter{
\begin{tikzpicture}[scale=0.46]
\draw[thick, white] (0,0.8);
\draw[thick, lightgray, dotted] (14.5,0)--(15.5,0);
\draw[thick, black, dotted] (-18.5,0)--(-17.5,0);
\draw[thick, lightgray ] (-17.5,0)--(14.5,0);
\blackdot(14,0);
\blackdot(10,0);
\blackdot(6,0);
\blackdot(4,0);
\blackdot(2,0);
\blackdot(0,0);
\blackdot(-1,0);
\blackdot(-2,0);
\blackdot(-4,0);
\blackdot(-5,0);
\blackdot(-6,0);
\blackdot(-8,0);
\blackdot(-9,0);
\blackdot(-10,0);
\blackdot(-12,0);
\blackdot(-13,0);
\blackdot(-14,0);
\blackdot(-15,0);
\blackdot(-16,0);
\blackdot(-17,0);
\node[below] at (-17,-0.2){$\scriptstyle \textup{-}17$};
\node[below] at (-16,-0.2){$\scriptstyle \textup{-}16$};
\node[below] at (-15,-0.2){$\scriptstyle \textup{-}15$};
\node[below] at (-14,-0.2){$\scriptstyle \textup{-}14$};
\node[below] at (-13,-0.2){$\scriptstyle \textup{-}13$};
\node[below] at (-12,-0.2){$\scriptstyle \textup{-}12$};
\node[below] at (-11,-0.2){$\scriptstyle \textup{-}11$};
\node[below] at (-10,-0.2){$\scriptstyle \textup{-}10$};
\node[below] at (-9,-0.2){$\scriptstyle \textup{-}9$};
\node[below] at (-8,-0.2){$\scriptstyle \textup{-}8$};
\node[below] at (-7,-0.2){$\scriptstyle \textup{-}7$};
\node[below] at (-6,-0.2){$\scriptstyle \textup{-}6$};
\node[below] at (-5,-0.2){$\scriptstyle \textup{-}5$};
\node[below] at (-4,-0.2){$\scriptstyle \textup{-}4$};
\node[below] at (-3,-0.2){$\scriptstyle \textup{-}3$};
\node[below] at (-2,-0.2){$\scriptstyle \textup{-}2$};
\node[below] at (-1,-0.2){$\scriptstyle \textup{-}1$};
\node[below] at (0,-0.2){$\scriptstyle 0$};
\node[below] at (1,-0.2){$\scriptstyle 1$};
\node[below] at (2,-0.2){$\scriptstyle 2$};
\node[below] at (3,-0.2){$\scriptstyle 3$};
\node[below] at (4,-0.2){$\scriptstyle 4$};
\node[below] at (5,-0.2){$\scriptstyle 5$};
\node[below] at (6,-0.2){$\scriptstyle 6$};
\node[below] at (7,-0.2){$\scriptstyle 7$};
\node[below] at (8,-0.2){$\scriptstyle 8$};
\node[below] at (9,-0.2){$\scriptstyle 9$};
\node[below] at (10,-0.2){$\scriptstyle 10$};
\node[below] at (11,-0.2){$\scriptstyle 11$};
\node[below] at (12,-0.2){$\scriptstyle 12$};
\node[below] at (13,-0.2){$\scriptstyle 13$};
\node[below] at (14,-0.2){$\scriptstyle 14$};
\end{tikzpicture}
}
\end{align*}
The \(\bkap\)-beta numbers for \(\bmu\) are:
\begin{align*}
{}
\B^1(\bmu, \bkap) &=
\hackcenter{
\begin{tikzpicture}[scale=0.46]
\draw[thick, white] (0,0.8);
\draw[thick, lightgray, dotted] (14.5,0)--(15.5,0);
\draw[thick, black, dotted] (-18.5,0)--(-17.5,0);
\draw[thick, lightgray ] (-17.5,0)--(14.5,0);
\blackdot(10,0);
\blackdot(6,0);
\blackdot(4,0);
\blackdot(2,0);
\blackdot(0,0);
\blackdot(-2,0);
\blackdot(-4,0);
\blackdot(-5,0);
\blackdot(-6,0);
\blackdot(-8,0);
\blackdot(-9,0);
\blackdot(-10,0);
\blackdot(-12,0);
\blackdot(-13,0);
\blackdot(-14,0);
\blackdot(-16,0);
\blackdot(-17,0);
\node[below] at (-17,-0.2){$\scriptstyle \textup{-}17$};
\node[below] at (-16,-0.2){$\scriptstyle \textup{-}16$};
\node[below] at (-15,-0.2){$\scriptstyle \textup{-}15$};
\node[below] at (-14,-0.2){$\scriptstyle \textup{-}14$};
\node[below] at (-13,-0.2){$\scriptstyle \textup{-}13$};
\node[below] at (-12,-0.2){$\scriptstyle \textup{-}12$};
\node[below] at (-11,-0.2){$\scriptstyle \textup{-}11$};
\node[below] at (-10,-0.2){$\scriptstyle \textup{-}10$};
\node[below] at (-9,-0.2){$\scriptstyle \textup{-}9$};
\node[below] at (-8,-0.2){$\scriptstyle \textup{-}8$};
\node[below] at (-7,-0.2){$\scriptstyle \textup{-}7$};
\node[below] at (-6,-0.2){$\scriptstyle \textup{-}6$};
\node[below] at (-5,-0.2){$\scriptstyle \textup{-}5$};
\node[below] at (-4,-0.2){$\scriptstyle \textup{-}4$};
\node[below] at (-3,-0.2){$\scriptstyle \textup{-}3$};
\node[below] at (-2,-0.2){$\scriptstyle \textup{-}2$};
\node[below] at (-1,-0.2){$\scriptstyle \textup{-}1$};
\node[below] at (0,-0.2){$\scriptstyle 0$};
\node[below] at (1,-0.2){$\scriptstyle 1$};
\node[below] at (2,-0.2){$\scriptstyle 2$};
\node[below] at (3,-0.2){$\scriptstyle 3$};
\node[below] at (4,-0.2){$\scriptstyle 4$};
\node[below] at (5,-0.2){$\scriptstyle 5$};
\node[below] at (6,-0.2){$\scriptstyle 6$};
\node[below] at (7,-0.2){$\scriptstyle 7$};
\node[below] at (8,-0.2){$\scriptstyle 8$};
\node[below] at (9,-0.2){$\scriptstyle 9$};
\node[below] at (10,-0.2){$\scriptstyle 10$};
\node[below] at (11,-0.2){$\scriptstyle 11$};
\node[below] at (12,-0.2){$\scriptstyle 12$};
\node[below] at (13,-0.2){$\scriptstyle 13$};
\node[below] at (14,-0.2){$\scriptstyle 14$};
\end{tikzpicture}
}\\
{}
\B^2(\bmu, \bkap) &=
\hackcenter{
\begin{tikzpicture}[scale=0.46]
\draw[thick, white] (0,0.8);
\draw[thick, lightgray, dotted] (14.5,0)--(15.5,0);
\draw[thick, black, dotted] (-18.5,0)--(-17.5,0);
\draw[thick, lightgray ] (-17.5,0)--(14.5,0);
\blackdot(14,0);
\blackdot(10,0);
\blackdot(6,0);
\blackdot(4,0);
\blackdot(2,0);
\blackdot(0,0);
\blackdot(-2,0);
\blackdot(-4,0);
\blackdot(-5,0);
\blackdot(-6,0);
\blackdot(-8,0);
\blackdot(-9,0);
\blackdot(-10,0);
\blackdot(-11,0);
\blackdot(-12,0);
\blackdot(-13,0);
\blackdot(-14,0);
\blackdot(-15,0);
\blackdot(-16,0);
\blackdot(-17,0);
\node[below] at (-17,-0.2){$\scriptstyle \textup{-}17$};
\node[below] at (-16,-0.2){$\scriptstyle \textup{-}16$};
\node[below] at (-15,-0.2){$\scriptstyle \textup{-}15$};
\node[below] at (-14,-0.2){$\scriptstyle \textup{-}14$};
\node[below] at (-13,-0.2){$\scriptstyle \textup{-}13$};
\node[below] at (-12,-0.2){$\scriptstyle \textup{-}12$};
\node[below] at (-11,-0.2){$\scriptstyle \textup{-}11$};
\node[below] at (-10,-0.2){$\scriptstyle \textup{-}10$};
\node[below] at (-9,-0.2){$\scriptstyle \textup{-}9$};
\node[below] at (-8,-0.2){$\scriptstyle \textup{-}8$};
\node[below] at (-7,-0.2){$\scriptstyle \textup{-}7$};
\node[below] at (-6,-0.2){$\scriptstyle \textup{-}6$};
\node[below] at (-5,-0.2){$\scriptstyle \textup{-}5$};
\node[below] at (-4,-0.2){$\scriptstyle \textup{-}4$};
\node[below] at (-3,-0.2){$\scriptstyle \textup{-}3$};
\node[below] at (-2,-0.2){$\scriptstyle \textup{-}2$};
\node[below] at (-1,-0.2){$\scriptstyle \textup{-}1$};
\node[below] at (0,-0.2){$\scriptstyle 0$};
\node[below] at (1,-0.2){$\scriptstyle 1$};
\node[below] at (2,-0.2){$\scriptstyle 2$};
\node[below] at (3,-0.2){$\scriptstyle 3$};
\node[below] at (4,-0.2){$\scriptstyle 4$};
\node[below] at (5,-0.2){$\scriptstyle 5$};
\node[below] at (6,-0.2){$\scriptstyle 6$};
\node[below] at (7,-0.2){$\scriptstyle 7$};
\node[below] at (8,-0.2){$\scriptstyle 8$};
\node[below] at (9,-0.2){$\scriptstyle 9$};
\node[below] at (10,-0.2){$\scriptstyle 10$};
\node[below] at (11,-0.2){$\scriptstyle 11$};
\node[below] at (12,-0.2){$\scriptstyle 12$};
\node[below] at (13,-0.2){$\scriptstyle 13$};
\node[below] at (14,-0.2){$\scriptstyle 14$};
\end{tikzpicture}
}
\end{align*}
The \(\bkap\)-beta numbers for \(\btau\) are:
\begin{align*}
{}
\B^1(\btau, \bkap) &=
\hackcenter{
\begin{tikzpicture}[scale=0.46]
\draw[thick, white] (0,0.8);
\draw[thick, lightgray, dotted] (14.5,0)--(15.5,0);
\draw[thick, black, dotted] (-18.5,0)--(-17.5,0);
\draw[thick, lightgray ] (-17.5,0)--(14.5,0);
\blackdot(10,0);
\blackdot(6,0);
\blackdot(2,0);
\blackdot(0,0);
\blackdot(-2,0);
\blackdot(-4,0);
\blackdot(-5,0);
\blackdot(-6,0);
\blackdot(-8,0);
\blackdot(-9,0);
\blackdot(-10,0);
\blackdot(-12,0);
\blackdot(-13,0);
\blackdot(-14,0);
\blackdot(-15,0);
\blackdot(-16,0);
\blackdot(-17,0);
\node[below] at (-17,-0.2){$\scriptstyle \textup{-}17$};
\node[below] at (-16,-0.2){$\scriptstyle \textup{-}16$};
\node[below] at (-15,-0.2){$\scriptstyle \textup{-}15$};
\node[below] at (-14,-0.2){$\scriptstyle \textup{-}14$};
\node[below] at (-13,-0.2){$\scriptstyle \textup{-}13$};
\node[below] at (-12,-0.2){$\scriptstyle \textup{-}12$};
\node[below] at (-11,-0.2){$\scriptstyle \textup{-}11$};
\node[below] at (-10,-0.2){$\scriptstyle \textup{-}10$};
\node[below] at (-9,-0.2){$\scriptstyle \textup{-}9$};
\node[below] at (-8,-0.2){$\scriptstyle \textup{-}8$};
\node[below] at (-7,-0.2){$\scriptstyle \textup{-}7$};
\node[below] at (-6,-0.2){$\scriptstyle \textup{-}6$};
\node[below] at (-5,-0.2){$\scriptstyle \textup{-}5$};
\node[below] at (-4,-0.2){$\scriptstyle \textup{-}4$};
\node[below] at (-3,-0.2){$\scriptstyle \textup{-}3$};
\node[below] at (-2,-0.2){$\scriptstyle \textup{-}2$};
\node[below] at (-1,-0.2){$\scriptstyle \textup{-}1$};
\node[below] at (0,-0.2){$\scriptstyle 0$};
\node[below] at (1,-0.2){$\scriptstyle 1$};
\node[below] at (2,-0.2){$\scriptstyle 2$};
\node[below] at (3,-0.2){$\scriptstyle 3$};
\node[below] at (4,-0.2){$\scriptstyle 4$};
\node[below] at (5,-0.2){$\scriptstyle 5$};
\node[below] at (6,-0.2){$\scriptstyle 6$};
\node[below] at (7,-0.2){$\scriptstyle 7$};
\node[below] at (8,-0.2){$\scriptstyle 8$};
\node[below] at (9,-0.2){$\scriptstyle 9$};
\node[below] at (10,-0.2){$\scriptstyle 10$};
\node[below] at (11,-0.2){$\scriptstyle 11$};
\node[below] at (12,-0.2){$\scriptstyle 12$};
\node[below] at (13,-0.2){$\scriptstyle 13$};
\node[below] at (14,-0.2){$\scriptstyle 14$};
\end{tikzpicture}
}\\
{}
\B^2(\btau, \bkap) &=
\hackcenter{
\begin{tikzpicture}[scale=0.46]
\draw[thick, white] (0,0.8);
\draw[thick, lightgray, dotted] (14.5,0)--(15.5,0);
\draw[thick, black, dotted] (-18.5,0)--(-17.5,0);
\draw[thick, lightgray ] (-17.5,0)--(14.5,0);
\blackdot(14,0);
\blackdot(10,0);
\blackdot(8,0);
\blackdot(6,0);
\blackdot(4,0);
\blackdot(2,0);
\blackdot(0,0);
\blackdot(-2,0);
\blackdot(-4,0);
\blackdot(-5,0);
\blackdot(-6,0);
\blackdot(-8,0);
\blackdot(-9,0);
\blackdot(-10,0);
\blackdot(-12,0);
\blackdot(-13,0);
\blackdot(-14,0);
\blackdot(-15,0);
\blackdot(-16,0);
\blackdot(-17,0);
\node[below] at (-17,-0.2){$\scriptstyle \textup{-}17$};
\node[below] at (-16,-0.2){$\scriptstyle \textup{-}16$};
\node[below] at (-15,-0.2){$\scriptstyle \textup{-}15$};
\node[below] at (-14,-0.2){$\scriptstyle \textup{-}14$};
\node[below] at (-13,-0.2){$\scriptstyle \textup{-}13$};
\node[below] at (-12,-0.2){$\scriptstyle \textup{-}12$};
\node[below] at (-11,-0.2){$\scriptstyle \textup{-}11$};
\node[below] at (-10,-0.2){$\scriptstyle \textup{-}10$};
\node[below] at (-9,-0.2){$\scriptstyle \textup{-}9$};
\node[below] at (-8,-0.2){$\scriptstyle \textup{-}8$};
\node[below] at (-7,-0.2){$\scriptstyle \textup{-}7$};
\node[below] at (-6,-0.2){$\scriptstyle \textup{-}6$};
\node[below] at (-5,-0.2){$\scriptstyle \textup{-}5$};
\node[below] at (-4,-0.2){$\scriptstyle \textup{-}4$};
\node[below] at (-3,-0.2){$\scriptstyle \textup{-}3$};
\node[below] at (-2,-0.2){$\scriptstyle \textup{-}2$};
\node[below] at (-1,-0.2){$\scriptstyle \textup{-}1$};
\node[below] at (0,-0.2){$\scriptstyle 0$};
\node[below] at (1,-0.2){$\scriptstyle 1$};
\node[below] at (2,-0.2){$\scriptstyle 2$};
\node[below] at (3,-0.2){$\scriptstyle 3$};
\node[below] at (4,-0.2){$\scriptstyle 4$};
\node[below] at (5,-0.2){$\scriptstyle 5$};
\node[below] at (6,-0.2){$\scriptstyle 6$};
\node[below] at (7,-0.2){$\scriptstyle 7$};
\node[below] at (8,-0.2){$\scriptstyle 8$};
\node[below] at (9,-0.2){$\scriptstyle 9$};
\node[below] at (10,-0.2){$\scriptstyle 10$};
\node[below] at (11,-0.2){$\scriptstyle 11$};
\node[below] at (12,-0.2){$\scriptstyle 12$};
\node[below] at (13,-0.2){$\scriptstyle 13$};
\node[below] at (14,-0.2){$\scriptstyle 14$};
\end{tikzpicture}
}
\end{align*}

Taking the residue permutation \(\theta = (1,3,0,2)\), it is straightforward to check from the \(\bkap\)-beta numbers that \(\blam, \bmu, \btau\) are \((\bkap, \theta)\)-RoCK multicores, and hence \(\Lambda^{\bkap}(\omega)\) is a \(\theta\)-RoCK core block. We have
\begin{align*}
P^\theta_+ = \{\alpha_2 + \alpha_3, -\alpha_1, \alpha_2, -\alpha_1 - \alpha_2 - \alpha_3,  - \alpha_3, \alpha_1 + \alpha_2\},
\qquad
\Delta^\theta = \{\alpha_2 + \alpha_3, -\alpha_1 - \alpha_2 - \alpha_3, \alpha_1 + \alpha_2\}.
\end{align*}
Following \cref{genconvex}, we may define a convex preorder \(\succeq\) on \(\Phi_+\) by letting \(h: \Z I \to \Q^2\) be the \(\ZZ\)-linear map given by setting
\begin{align*}
h(\alpha_0) = (1,0),
\qquad
h(\alpha_1) = (-1,-1),
\qquad
h(\alpha_2) = (2,1),
\qquad
h(\alpha_3) = (-2,0),
\end{align*}
and taking the usual total lexicographic order on \(\QQ^2\):
\begin{align*}
(x,y) \geq (x',y') \qquad \iff \qquad x>x' \textup{ or } x=x' \textup{ and }y\geq y',
\end{align*}
for all \((x,y), (x',y') \in \QQ^2\). For \(\beta, \gamma \in \Phi_+\), we then set 
\begin{align*}
\beta \succeq \gamma \qquad \iff \qquad
\frac{h(\beta)}{\height(\beta)} \geq \frac{h(\gamma)}{\height(\gamma)}.
\end{align*}
By construction, the convex preorder \(\succeq\) realizes \(\theta\). The associated cuspidal Kostant tilings for each multipartition in \(\Lambda^{\bkap}_+(\omega)\) are shown in Figure~\ref{corebigtilings}. Note in particular that for all ribbons \(\xi\) in these tilings, we have
\(
p(\cont(\xi)) \in P_+^\theta
\), so \(\cont(\xi) \succ \delta\), and thus \(\Lambda^{\bkap}_+(\omega) = \Lambda^{\bkap}_+(\omega)_{\succ \delta}\), 
 in accordance with \cref{RoCKmultitiling}.
 
Taking \(\bkap = (2,0)\), in Figure~\ref{exwithdeltaribsfig}, the cuspidal Kostant tiling for the \((\bkap,\theta)\)-RoCK multipartition 
 \begin{align*}
 \bzeta &= ((13,12,11,10,9,6,5,4,2^2,1^3)\mid(15, 12,7,6,4^2,3^3,2,1))
 \end{align*}
 considered in \cref{ExBigRoCKnot} is shown. Note that for all ribbons \(\xi\) in this tiling, we have
\(
p(\cont(\xi)) \in P_+^\theta \sqcup \{0\}
\), and so so \(\cont(\xi) \succeq \delta\),
 in accordance with \cref{RoCKsemipar}.
\end{Example}

\subsection{RoCK blocks from core blocks}

\begin{Proposition}\label{combprop2}
Let \(\omega, \eta \in \ZZ_{\geq 0}I\), \(\blam \in \Lambda^{\bkap}_+\), and \(\brho \in \Lambda^{\bkap}_+(\omega)\) be a \((\bkap, \theta)\)-RoCK multicore. Let  \(\succeq\) be a convex preorder on \(\Phi_+\) which realizes \(\theta\).
\begin{enumerate} 
\item If \(\height(\eta) \leq e \cdot \textup{cap}^\theta_\delta(\brho, \bkap)\) and  \(\blam/\brho \in \Lambda^{\bkap}_{+/\brho}(\eta)\), then \(\blam/\brho \in \Lambda^{\bkap}_{+/\brho}(\eta)_{\preceq \delta}\).
\item If \(d \leq \textup{cap}^\theta_\delta(\brho, \bkap)\) and  \(\blam/\brho \in \Lambda^{\bkap}_{+/\brho}(d \delta)\), then \(\blam \in \Lambda^{\bkap}_+(\omega + d \delta)_{\succeq \delta}\) and \(\blam/\brho \in \Lambda^{\bkap}_{+/\brho}(d \delta)_{\approx \delta}\).
\end{enumerate}
\end{Proposition}
\begin{proof}
Let \(\xi\) be an addable ribbon for \(\brho\), and assume that \(|\xi| \leq e \cdot \textup{cap}^\theta_\delta(\brho, \bkap)\). Then \(\cont(\xi) \in \addrib^\theta\) by \cref{reminR}.
Note that as \(\xi\) is nonempty, we must have \(\textup{cap}^\theta_\delta(\brho, \bkap) > 0\), so \(h^{r,\theta}_t \geq 0\) for all \(t\). Thus \(\textup{cap}^\theta_\delta(\brho,\bkap) \leq h^{\min,\theta}_{[a,b]}+1\) for all \(1 \leq a \leq b \leq e-1\).
If \(\cont(\xi) = k \delta + \gamma^\theta_{[a,b]}\), then \(k \geq h^{\min,\theta}_{[a,b]}+1\), so that \(k \geq \textup{cap}^\theta_\delta(\brho,\bkap)\), but then \( \height(\xi) = e k + \height(\gamma^\theta_{[a,b]}) > e \cdot \textup{cap}^\theta_\delta(\brho, \bkap)\).
Therefore \(\cont(\xi) \in \addrib^\theta\) must be of the form \(k\delta\) or \(k \delta - \gamma^\theta_{[a,b]}\) for some  \(1 \leq a \leq b \leq e-1\), \(k \in \ZZ_{>0}\). Thus \(\cont(\xi) \in \Phi_{\preceq \delta}\) by (\ref{ldel}).

Now, assume that \(\height(\beta) \leq e \cdot \textup{cap}^\theta_\delta(\brho, \bkap)\) and  \(\blam/\brho \in \Lambda^{\bkap}_{+/\brho}(\beta)\). Letting \(\xi\) be a \(\succeq\)-maximal \({\tt NW}\)-removable ribbon in \(\blam/\brho\) with \(\cont(\xi) \in \Psi\), it must also be that \(\xi\) is an addable ribbon for \(\brho\), and so by the above paragraph, \(\cont(\xi) \in \Phi_{\preceq \delta}\), since \(|\xi| \leq |\blam/\brho| = \height(\beta) \leq e \cdot \textup{cap}^\theta_\delta(\brho, \bkap)\). Thus by \cref{tilethm}(iii), we have that \(\blam/\brho \in \Lambda^{\bkap}_{+/\brho}(\beta)_{\preceq \delta}\), establishing (i).

For (ii), note that \(\blam/\brho \in \Lambda^{\bkap}_{+/\brho}(d\delta)_{\preceq \delta}\) by (i). Now, since \(d \delta = \sum_{\zeta \in \Gamma_{\blam/\brho}} \cont(\zeta)\), we have by \cite[Lemma 3.5(ii)]{muthtiling} that \(\cont(\zeta) = \delta\) for all \(\zeta \in \Gamma_{\blam / \brho}\), so it follows that \(\blam / \brho \in \Lambda^{\bkap}_{+/\brho}(\beta)_{\approx \delta}\). Moreover, we have that \(\brho \in \Lambda^{\bkap}_+(\omega)_{\succ \delta}\) by \cref{RoCKmultitiling}, so it follows that \(\Gamma_{\brho} \cup \Gamma_{\blam / \brho}\) is a cuspidal Kostant tiling for \(\blam\), and thus \(\blam \in \Lambda_+^{\bkap}(\omega + d\delta)_{\succeq \delta}\).  
\end{proof}

\begin{Lemma}\label{sepofcoreblock}
If \(\Lambda_+^{\bkap}(\omega)\) is a \(\theta\)-RoCK core block with \(\textup{ht}(\beta) \leq e \cdot \textup{cap}_\delta^\theta(\omega, \bkap)\), then \((\omega, \beta)\) is \(\bkap\)-separable.
\end{Lemma}
\begin{proof}
By way of contradiction assume that \(\brho \in \Lambda^{\bkap}_+(\omega)\) and that \(\bmu, \bnu\) such as in \cref{kapsepdef} exist. Let \(\succeq\) be a convex preorder which realizes \(\theta\). 
By \cref{RoCKmultitiling}, we have that \(\brho \in \Lambda^{\bkap}_+(\omega)_{\succ \delta}\). Let
\begin{align*}
\Xi = \{(\brho/\bmu) \cap \gamma \neq \varnothing \mid \gamma \in \Gamma_{\brho}\}.
\end{align*}
We have then that \(\brho/\bmu = \bigsqcup_{\xi \in \Xi} \xi\). For each \(\xi = (\brho/\bmu) \cap \gamma \in \Xi\), note that \(\xi\) is a {\tt SE}-removable ribbon in the cuspidal ribbon \(\gamma\). Hence by cuspidality of \(\gamma\), we have that \(\cont(\xi)\) may be written as a sum of positive roots \(\succ \cont(\gamma) \succ \delta\). Therefore \(\beta' = \cont(\brho/\bmu)\) may be written as a sum of positive roots \( \succ \delta\), and thus \(0 \neq p(\beta') \in \ZZ_{\geq 0} P^\theta_+\). 

On the other hand, by \cref{combprop2}(i) we have that \(\bnu/\brho \in \Lambda^{\bkap}_{+/\brho}(\beta')_{\preceq \delta}\). Letting
\begin{align*}
\Xi' = \{(\bnu/\brho) \cap \gamma \neq 0 \mid \gamma \in \Gamma_{\bnu/\brho}\},
\end{align*}
we have that \(\bnu /\brho = \bigsqcup_{\xi \in \Xi} \xi\), and (similarly to the previous paragraph) each \(\cont(\xi)\) is a sum of positive roots \(\preceq \delta\) by the cuspidality of \(\xi\). Hence \(\beta'\) may be written as a sum of positive roots  \(\preceq \delta\), and thus \(p(\beta') \in \ZZ_{\geq 0} P^\theta_-\). 

From the previous two paragraphs, we have \(0 \neq p(\beta') \in  \ZZ_{\geq 0} P^\theta_+ \cap  \ZZ_{\geq 0} P^\theta_- = \{0\}\), which gives the desired contradiction.
\end{proof}

\begin{Theorem}\label{coreplus}
Let \(\omega \in \Phi_+\), and let \(\theta\) be a residue permutation.
A block \(\Lambda^{\bkap}_+(\omega)\) is \(\theta\)-RoCK if and only if there exists a \(\theta\)-RoCK core block \(\Lambda^{\bkap}_+(\eta)\) and \(d \leq \textup{cap}^\theta_\delta(\eta, \bkap)\) such that \(\omega = \eta + d \delta\).
\end{Theorem}

\begin{proof}
Let \(\succeq\) be a convex preorder on \(\Phi_+\) which realizes \(\theta\).

\((\implies)\) Assume \(\Lambda^{\bkap}_+(\omega)\) is \(\theta\)-RoCK. Let \(X\) be the set of multicores realized by removing all \(e\)-ribbons from the multipartitions in \(\Lambda^{\bkap}_+(\omega)\). Let \(\blam\) be an element of \(X\) of minimal size, and set \(\eta:=\cont(\blam)\). Then \(\omega = \eta + d \delta\) for some \(d \geq 0\). For any \(\bmu \in 
\Lambda^{\bkap}_+(\eta)\), note that we may append a row ribbon of length \(de\) to the top row of \(\bmu\) to yield a multipartition \(\bmu' \in \Lambda^{\bkap}_+(\omega)\). Hence \(\bmu \in X\) as we may remove \(d\) row \(e\)-ribbons \(\xi_1, \dots, \xi_d\) of content \(\delta\) from \(\bmu'\) to form \(\bmu\), and the existence of any additional removable \(e\)-ribbons in \(\bmu\) would contradict the minimality of \(\blam\). By \cref{thetarockblocktile}, \(\bmu' \in \Lambda^{\bkap}_+(\omega)_{\succeq \delta}\), so it follows from \cref{tilethm}(ii) that \(\xi_1, \dots, \xi_d \in \Gamma_{\bmu'}\), and that \(\bmu\) has cuspidal Kostant tiling \(\Gamma_{\bmu} = \Gamma_{\bmu'}\backslash \{\xi_1, \dots, \xi_d\}\). Then, since \(\bmu\) has no removable ribbons of content \(\delta\), it follows that \(\bmu \in \Lambda^{\bkap}_+(\eta)_{\succ \delta}\). Hence by \cref{RoCKmultitiling,coretilingprop}, we have that \(\Lambda^{\bkap}_+(\eta)\) is a \(\theta\)-RoCK core block.

Now by way of contradiction assume that \(d> \textup{cap}^\theta_\delta(\eta, \bkap)\). Then for some \(\blam \in \Lambda^{\bkap}_+(\eta, \bkap), r,t\) we have 
\begin{align*}
\textup{cap}^\theta_\delta(\blam, \bkap) = h^{r, \theta}_t(\blam, \bkap) +1 =  \left \lfloor
\frac{M^r_{\theta_{t+1}}(\blam, \bkap) - M^r_{\theta_t}(\blam, \bkap) + e}{e} \right\rfloor 
\end{align*}
which implies that \(M^r_{\theta_{t+1}}(\blam, \bkap) < M^r_{\theta_{t}}(\blam, \bkap) + e \cdot \textup{cap}^\theta_\delta(\eta, \bkap)\). Now, let \(\btau \in \Lambda^{\bkap}(\omega)\) be the multipartition whose beta numbers are given by deleting \(M_{\theta_t}^r(\blam, \bkap)\) from \(\B^r(\blam, \bkap)\) and replacing it with \(M_{\theta_t}^r + de\). This has the effect of adding a ribbon of content \(d\delta\) to \(\blam\). Then we have \(M_{\theta_t}^r(\blam, \bkap) + de \in \B^r(\btau, \bkap)\) and
\begin{align*}
M_{\theta_t}^r(\blam, \bkap) + de -  \overline{ \theta_t - \theta_{t+1}} \geq M_{\theta_t}^r(\blam, \bkap) + (\textup{cap}^\theta_\delta(\eta, \bkap) + 1 )e -  \overline{ \theta_t - \theta_{t+1}} > M^r_{\theta_{t+1}}(\blam, \bkap) = M^r_{\theta_{t+1}}(\btau,\bkap),
\end{align*}
so \(M_{\theta_t}^r(\blam, \bkap) + de -  \overline{ \theta_t - \theta_{t+1}}  \notin \B^r(\btau, \bkap)\). Since \(\overline{M_{\theta_t}^r(\blam, \bkap) + de } = \theta_t\) and \(\overline{M_{\theta_t}^r(\blam, \bkap) + de -  \overline{ \theta_t - \theta_{t+1}}} = \theta_{t+1}\), this contradicts \cref{specialdefs}(iii) and the fact that \(\Lambda^{\bkap}_+(\omega)\) is \(\theta\)-RoCK. Hence \(d\leq \textup{cap}^\theta_\delta(\eta, \bkap)\).

\((\impliedby)\) Now assume that \(\Lambda_+^{\bkap}(\eta)\) is a \(\theta\)-RoCK core block, and  \(d \leq \textup{cap}^\theta_\delta(\eta, \bkap)\). Let \(\omega = \eta + d \delta\). By way of contradiction, assume there is some \(\bmu \in \Lambda^{\bkap}_+(\omega)\) which is not \((\bkap, \theta)\)-RoCK. 

Let \(\xi\) be a \(\succeq\)-minimal removable ribbon for \(\bmu\). Then by \cref{RoCKsemipar,tilethm}(ii),  \(\cont(\xi) \prec \delta\), so \(\cont(\xi) = \beta \in \Phi_+^\re\), and \(\height(\beta)\) is not a multiple of \(e\).
First assume that \(\height(\beta)>de\). Decompose \(\xi\) into ribbons \(\xi = \xi' \sqcup \xi''\), where \(\xi''\) consists of the northeasternmost \(de\) nodes of \(\xi\). Then one of \(\xi'\) or \(\xi''\) must be a removable ribbon in \(\bmu\). However, by cuspidality we have 
\begin{align*}
 \cont(\xi'') = d\delta \succ \cont(\xi) = \beta \succ \cont(\xi') = \beta - d \delta.
\end{align*}
Thus by \(\succeq\)-minimality of \(\xi\) in \(\bmu\), \(\xi'\) cannot be a removable ribbon in \(\bmu\), so it must be that \(\xi''\) is a removable ribbon in \(\bmu\). But then \(\bmu\backslash\xi'' \in \Lambda^{\bkap}_+(\eta)\), and \(\xi'\) is a removable ribbon in \(\bmu \backslash \xi''\). But \(\cont(\xi') = \beta - d \delta \prec \delta\), so by \cref{RoCKmultitiling,tilethm}(ii), this contradicts that \(\bmu \backslash \xi''\) is a \((\bkap, \theta)\)-RoCK multicore.

Now we assume \(\height(\beta)< de\). Then \(\bmu \backslash \xi \in \Lambda^{\bkap}_+(\eta + (d\delta - \beta))\). Note by convexity and \cref{RoCKsemipar} that \( \Phi_+^\re \ni d \delta - \beta \succ \delta\). Then by \cref{addribbonlemma} there exists a multipartition \(\blam \in \Lambda^{\bkap}_+(\eta)\) with an addable ribbon of content \(d\delta - \beta\). But since \(\height(d\delta - \beta) < e \cdot \textup{cap}^\theta_{\delta}(\eta, \bkap) \leq e \cdot \textup{cap}^\theta_{\delta}(\blam, \bkap) \), this contradicts \cref{combprop2}(i). 
\end{proof}

\begin{Remark}\label{paramRoCK}
The preceding theorem shows that RoCK blocks of multicharge \(\bkap\) may be parameterized:
\begin{align*}
\{\Lambda_+^{\bkap}(\eta + d\delta) \mid \Lambda_+^{\bkap}(\eta) \textup{ a \(\theta\)-RoCK core block}, d \leq \textup{cap}_\delta^\theta(\eta, \bkap) \textup{ for some residue permutation } \theta\}.
\end{align*}
For any  \(\omega \in \mathbb{Z}_{\geq 0}I\), there is at most one pair $(\eta,d)$ such that $\omega = \eta + d\delta$ and $\Lambda_+^{\bkap}(\eta)$ is a core block. It follows that each RoCK block in the set thus corresponds to a unique $\theta$-RoCK core block $\Lambda_+^{\bkap}(\eta)$.
\end{Remark}

\section{Imaginary semicuspidal diagrams and words}\label{combimagsec}
In this section we give a detailed investigation of the combinatorics of imaginary skew diagrams and their associated words.
Throughout, we fix a residue permutation \(\theta\), a \((\bkap, \theta)\)-RoCK multicore \(\brho \in \Lambda_+^{\bkap}(\omega)\), and \( d \leq \textup{cap}_\delta^\theta(\brho, \bkap)\).
We will define a higher-level analogue of \(e\)-quotients via beta numbers, and use these to determine skew diagrams that will play a key role in our main results later.
For each \( t \in [0,e-1], r \in [1,\ell]\), let \(\nu^{(t,r)}\) be a partition, and name the \(e\ell\)-multipartition
\[\bnu = (\nu^{(0,1)}\mid \dots\mid \nu^{(e-1,1)}\mid
\nu^{(0,2)}\mid \dots\mid \nu^{(e-1,2)}\mid
\dots
\mid
\nu^{(0,\ell)} \mid \cdots \mid \nu^{(e-1,\ell)}).\]
We write \(\Lambda^{(e, \ell)}_+(d)\) for the set of all such \(e \ell\)-multipartitions with
\begin{align*}
|\bnu| = \sum_{\substack{t \in [0,e-1] \\ r \in [1,\ell]}} |\nu^{(t,r)}| = d.
\end{align*}
Given \(\bnu \in \Lambda^{(e, \ell)}_+(d)\) we define a new skew multipartition \(\bnu^\theta_{\brho} \in \Lambda^{\bkap}_{+/\brho}( d \delta)\) via \(\bkap\)-beta numbers as follows. For \(\blam \in \Lambda^{\bkap}_+\), \(i \in [0,e-1]\), \(r \in [1,\ell]\), \(q \in \mathbb{Z}_{>0}\), set \(m^r_i(\blam, \bkap)_q \in \mathbb{Z}_{>0}\) to be the \(q\)th largest element in \(\B^r_i(\blam, \bkap)\). 
Then, define \(\bnu^\theta_{\brho} \in \Lambda^{\bkap}_{+/\brho}(d \delta)\) such that
\begin{align}\label{mmove}
m^r_{\theta_{e-t}}(\brho \sqcup \bnu^\theta_{\brho}, \bkap)_q = m^r_{\theta_{e-t}}(\brho, \bkap)_q + e\nu^{(t,r)}_q
\end{align}
for all \(t \in [0,e-1]\), \(r \in [1,\ell]\), \(q \in \mathbb{Z}_{>0}\). 
For \(u \in \bnu\), set
\begin{align*}
u^{\theta}_{\brho} := (\textup{rect}_u)^\theta_{\brho}/ (\textup{rect}_u \backslash \{u\})^\theta_{\brho} \subseteq \bnu^\theta_{\brho}.
\end{align*}
It follows from definitions that \(u^\theta_{\brho}\) is an \(e\)-ribbon for every \(u \in \bnu\).

\begin{Remark}
From the abacus point of view, \(\brho \sqcup \bnu_{\brho}^\theta\) is achieved by moving the \(q\)th lowest bead on the \(\theta_{e-t}\)-runner in the \(r\)th component of \(\brho\) down \(\nu_q^{(t,r)}\) spots, so \( \bnu\) is just a (reordered) higher-level analogue of the \(e\)-quotient for \(\brho \sqcup \bnu_{\brho}^\theta\) (see for instance \cite[\S2]{Lyle22}).
\end{Remark}

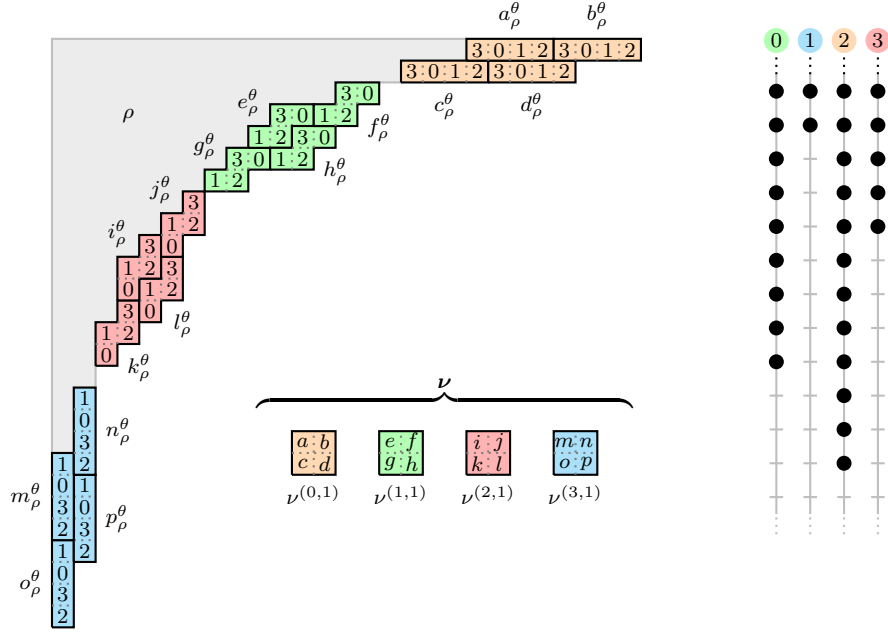
\begin{figure}[h]
\begin{align*}
{}
\hackcenter{
\begin{tikzpicture}[scale=0.29]
\draw[thick,lightgray,fill=lightgray!30] (0,8)--(1,8)--(1,11)--(2,11)--(2,14)--(3,14)--(3,17)--(4,17)--(4,18)--(5,18)--(5,19)--(6,19)--(6,20)--(7,20)--(7,21)--(8,21)--(8,22)--(9,22)--(9,23)--(10,23)--(10,24)--(13,24)--(13,25)--(16,25)--(16,26)--(19,26)--(19,27)--(0,27)--(0,8);
\draw[thick,fill=cyan!30] (0+0,0+0)--(1+0,0+0)--(1+0,4+0)--(0+0,4+0)--(0+0,0+0);
\node at (0.5+0,0.5+0){$\scriptstyle 2$};
\node at (0.5+0,1.5+0){$\scriptstyle 3$};
\node at (0.5+0,2.5+0){$\scriptstyle 0$};
\node at (0.5+0,3.5+0){$\scriptstyle 1$};
\draw[thick, gray, dotted] (0+0,1+0)--(1+0,1+0);
\draw[thick, gray, dotted] (0+0,2+0)--(1+0,2+0);
\draw[thick, gray, dotted] (0+0,3+0)--(1+0,3+0);
\draw[thick,fill=cyan!30] (0+1,0+3)--(1+1,0+3)--(1+1,4+3)--(0+1,4+3)--(0+1,0+3);
\node at (0.5+1,0.5+3){$\scriptstyle 2$};
\node at (0.5+1,1.5+3){$\scriptstyle 3$};
\node at (0.5+1,2.5+3){$\scriptstyle 0$};
\node at (0.5+1,3.5+3){$\scriptstyle 1$};
\draw[thick, gray, dotted] (0+1,1+3)--(1+1,1+3);
\draw[thick, gray, dotted] (0+1,2+3)--(1+1,2+3);
\draw[thick, gray, dotted] (0+1,3+3)--(1+1,3+3);
\draw[thick,fill=cyan!30] (0+0,0+4)--(1+0,0+4)--(1+0,4+4)--(0+0,4+4)--(0+0,0+4);
\node at (0.5+0,0.5+4){$\scriptstyle 2$};
\node at (0.5+0,1.5+4){$\scriptstyle 3$};
\node at (0.5+0,2.5+4){$\scriptstyle 0$};
\node at (0.5+0,3.5+4){$\scriptstyle 1$};
\draw[thick, gray, dotted] (0+0,1+4)--(1+0,1+4);
\draw[thick, gray, dotted] (0+0,2+4)--(1+0,2+4);
\draw[thick, gray, dotted] (0+0,3+4)--(1+0,3+4);
\draw[thick,fill=cyan!30] (0+1,0+7)--(1+1,0+7)--(1+1,4+7)--(0+1,4+7)--(0+1,0+7);
\node at (0.5+1,0.5+7){$\scriptstyle 2$};
\node at (0.5+1,1.5+7){$\scriptstyle 3$};
\node at (0.5+1,2.5+7){$\scriptstyle 0$};
\node at (0.5+1,3.5+7){$\scriptstyle 1$};
\draw[thick, gray, dotted] (0+1,1+7)--(1+1,1+7);
\draw[thick, gray, dotted] (0+1,2+7)--(1+1,2+7);
\draw[thick, gray, dotted] (0+1,3+7)--(1+1,3+7);
\draw[thick,fill=red!30] (2+0,12+0)--(3+0,12+0)--(3+0,13+0)--(4+0,13+0)--(4+0,15+0)--(3+0,15+0)--(3+0,14+0)--(2+0,14+0)--(2+0,12+0);
\node at (2.5+0,12.5+0){$\scriptstyle 0$};
\node at (2.5+0,13.5+0){$\scriptstyle 1$};
\node at (3.5+0,13.5+0){$\scriptstyle 2$};
\node at (3.5+0,14.5+0){$\scriptstyle 3$};
\draw[thick, gray, dotted] (2+0,13+0)--(3+0,13+0);
\draw[thick, gray, dotted] (3+0,13+0)--(3+0,14+0);
\draw[thick, gray, dotted] (3+0,14+0)--(4+0,14+0);
\draw[thick,fill=red!30] (2+2,12+2)--(3+2,12+2)--(3+2,13+2)--(4+2,13+2)--(4+2,15+2)--(3+2,15+2)--(3+2,14+2)--(2+2,14+2)--(2+2,12+2);
\node at (2.5+2,12.5+2){$\scriptstyle 0$};
\node at (2.5+2,13.5+2){$\scriptstyle 1$};
\node at (3.5+2,13.5+2){$\scriptstyle 2$};
\node at (3.5+2,14.5+2){$\scriptstyle 3$};
\draw[thick, gray, dotted] (2+2,13+2)--(3+2,13+2);
\draw[thick, gray, dotted] (3+2,13+2)--(3+2,14+2);
\draw[thick, gray, dotted] (3+2,14+2)--(4+2,14+2);
\draw[thick,fill=red!30] (2+1,12+3)--(3+1,12+3)--(3+1,13+3)--(4+1,13+3)--(4+1,15+3)--(3+1,15+3)--(3+1,14+3)--(2+1,14+3)--(2+1,12+3);
\node at (2.5+1,12.5+3){$\scriptstyle 0$};
\node at (2.5+1,13.5+3){$\scriptstyle 1$};
\node at (3.5+1,13.5+3){$\scriptstyle 2$};
\node at (3.5+1,14.5+3){$\scriptstyle 3$};
\draw[thick, gray, dotted] (2+1,13+3)--(3+1,13+3);
\draw[thick, gray, dotted] (3+1,13+3)--(3+1,14+3);
\draw[thick, gray, dotted] (3+1,14+3)--(4+1,14+3);
\draw[thick,fill=red!30] (2+3,12+5)--(3+3,12+5)--(3+3,13+5)--(4+3,13+5)--(4+3,15+5)--(3+3,15+5)--(3+3,14+5)--(2+3,14+5)--(2+3,12+5);
\node at (2.5+3,12.5+5){$\scriptstyle 0$};
\node at (2.5+3,13.5+5){$\scriptstyle 1$};
\node at (3.5+3,13.5+5){$\scriptstyle 2$};
\node at (3.5+3,14.5+5){$\scriptstyle 3$};
\draw[thick, gray, dotted] (2+3,13+5)--(3+3,13+5);
\draw[thick, gray, dotted] (3+3,13+5)--(3+3,14+5);
\draw[thick, gray, dotted] (3+3,14+5)--(4+3,14+5);
\draw[thick,fill=green!30] (7+0,20+0)--(9+0,20+0)--(9+0,21+0)--(10+0,21+0)--(10+0,22+0)--(8+0,22+0)--(8+0,21+0)--(7+0,21+0)--(7+0,20+0);
\node at (7.5+0,20.5+0){$\scriptstyle 1$};
\node at (8.5+0,20.5+0){$\scriptstyle 2$};
\node at (8.5+0,21.5+0){$\scriptstyle 3$};
\node at (9.5+0,21.5+0){$\scriptstyle 0$};
\draw[thick, gray, dotted] (8+0,20+0)--(8+0,21+0);
\draw[thick, gray, dotted] (8+0,21+0)--(9+0,21+0);
\draw[thick, gray, dotted] (9+0,21+0)--(9+0,22+0);
\draw[thick,fill=green!30] (7+2,20+2)--(9+2,20+2)--(9+2,21+2)--(10+2,21+2)--(10+2,22+2)--(8+2,22+2)--(8+2,21+2)--(7+2,21+2)--(7+2,20+2);
\node at (7.5+2,20.5+2){$\scriptstyle 1$};
\node at (8.5+2,20.5+2){$\scriptstyle 2$};
\node at (8.5+2,21.5+2){$\scriptstyle 3$};
\node at (9.5+2,21.5+2){$\scriptstyle 0$};
\draw[thick, gray, dotted] (8+2,20+2)--(8+2,21+2);
\draw[thick, gray, dotted] (8+2,21+2)--(9+2,21+2);
\draw[thick, gray, dotted] (9+2,21+2)--(9+2,22+2);
\draw[thick,fill=green!30] (7+3,20+1)--(9+3,20+1)--(9+3,21+1)--(10+3,21+1)--(10+3,22+1)--(8+3,22+1)--(8+3,21+1)--(7+3,21+1)--(7+3,20+1);
\node at (7.5+3,20.5+1){$\scriptstyle 1$};
\node at (8.5+3,20.5+1){$\scriptstyle 2$};
\node at (8.5+3,21.5+1){$\scriptstyle 3$};
\node at (9.5+3,21.5+1){$\scriptstyle 0$};
\draw[thick, gray, dotted] (8+3,20+1)--(8+3,21+1);
\draw[thick, gray, dotted] (8+3,21+1)--(9+3,21+1);
\draw[thick, gray, dotted] (9+3,21+1)--(9+3,22+1);
\draw[thick,fill=green!30] (7+5,20+3)--(9+5,20+3)--(9+5,21+3)--(10+5,21+3)--(10+5,22+3)--(8+5,22+3)--(8+5,21+3)--(7+5,21+3)--(7+5,20+3);
\node at (7.5+5,20.5+3){$\scriptstyle 1$};
\node at (8.5+5,20.5+3){$\scriptstyle 2$};
\node at (8.5+5,21.5+3){$\scriptstyle 3$};
\node at (9.5+5,21.5+3){$\scriptstyle 0$};
\draw[thick, gray, dotted] (8+5,20+3)--(8+5,21+3);
\draw[thick, gray, dotted] (8+5,21+3)--(9+5,21+3);
\draw[thick, gray, dotted] (9+5,21+3)--(9+5,22+3);
\draw[thick,fill=orange!30] (16+0,25+0)--(20+0,25+0)--(20+0,26+0)--(16+0,26+0)--(16+0,25+0);
\node at (16.5+0,25.5+0){$\scriptstyle 3$};
\node at (17.5+0,25.5+0){$\scriptstyle 0$};
\node at (18.5+0,25.5+0){$\scriptstyle 1$};
\node at (19.5+0,25.5+0){$\scriptstyle 2$};
\draw[thick, gray, dotted] (17+0,25+0)--(17+0,26+0);
\draw[thick, gray, dotted] (18+0,25+0)--(18+0,26+0);
\draw[thick, gray, dotted] (19+0,25+0)--(19+0,26+0);
\draw[thick,fill=orange!30] (16+4,25+0)--(20+4,25+0)--(20+4,26+0)--(16+4,26+0)--(16+4,25+0);
\node at (16.5+4,25.5+0){$\scriptstyle 3$};
\node at (17.5+4,25.5+0){$\scriptstyle 0$};
\node at (18.5+4,25.5+0){$\scriptstyle 1$};
\node at (19.5+4,25.5+0){$\scriptstyle 2$};
\draw[thick, gray, dotted] (17+4,25+0)--(17+4,26+0);
\draw[thick, gray, dotted] (18+4,25+0)--(18+4,26+0);
\draw[thick, gray, dotted] (19+4,25+0)--(19+4,26+0);
\draw[thick,fill=orange!30] (16+3,25+1)--(20+3,25+1)--(20+3,26+1)--(16+3,26+1)--(16+3,25+1);
\node at (16.5+3,25.5+1){$\scriptstyle 3$};
\node at (17.5+3,25.5+1){$\scriptstyle 0$};
\node at (18.5+3,25.5+1){$\scriptstyle 1$};
\node at (19.5+3,25.5+1){$\scriptstyle 2$};
\draw[thick, gray, dotted] (17+3,25+1)--(17+3,26+1);
\draw[thick, gray, dotted] (18+3,25+1)--(18+3,26+1);
\draw[thick, gray, dotted] (19+3,25+1)--(19+3,26+1);
\draw[thick,fill=orange!30] (16+7,25+1)--(20+7,25+1)--(20+7,26+1)--(16+7,26+1)--(16+7,25+1);
\node at (16.5+7,25.5+1){$\scriptstyle 3$};
\node at (17.5+7,25.5+1){$\scriptstyle 0$};
\node at (18.5+7,25.5+1){$\scriptstyle 1$};
\node at (19.5+7,25.5+1){$\scriptstyle 2$};
\draw[thick, gray, dotted] (17+7,25+1)--(17+7,26+1);
\draw[thick, gray, dotted] (18+7,25+1)--(18+7,26+1);
\draw[thick, gray, dotted] (19+7,25+1)--(19+7,26+1);
\node at (3.5,23.5){$\scriptstyle \rho$};
\node[left] at (0,2){$\scriptstyle o^\theta_\rho$};
\node[left] at (0,6){$\scriptstyle m^\theta_\rho$};
\node[right] at (2,5){$\scriptstyle p^\theta_\rho$};
\node[right] at (2,9){$\scriptstyle n^\theta_\rho$};
\node[below] at (4,13){$\scriptstyle k^\theta_\rho$};
\node[below] at (6,15){$\scriptstyle l^\theta_\rho$};
\node[above] at (3,17){$\scriptstyle i^\theta_\rho$};
\node[above] at (5,19){$\scriptstyle j^\theta_\rho$};
\node[above] at (7,21){$\scriptstyle g^\theta_\rho$};
\node[above] at (9,23){$\scriptstyle e^\theta_\rho$};
\node[below] at (13,22){$\scriptstyle h^\theta_\rho$};
\node[below] at (15,24){$\scriptstyle f^\theta_\rho$};
\node[below] at (18,25){$\scriptstyle c^\theta_\rho$};
\node[below] at (22,25){$\scriptstyle d^\theta_\rho$};
\node[above] at (21,27){$\scriptstyle a^\theta_\rho$};
\node[above] at (25,27){$\scriptstyle b^\theta_\rho$};
\draw[thick,,fill=orange!30] (10+1,7)--(12+1,7)--(12+1,9)--(10+1,9)--(10+1,7);
\draw[thick,,fill=green!30] (14+1,7)--(16+1,7)--(16+1,9)--(14+1,9)--(14+1,7);
\draw[thick,,fill=red!30] (18+1,7)--(20+1,7)--(20+1,9)--(18+1,9)--(18+1,7);
\draw[thick,,fill=cyan!30] (22+1,7)--(24+1,7)--(24+1,9)--(22+1,9)--(22+1,7);
\draw[thick, gray, dotted] (12,7)--(12,9);
\draw[thick, gray, dotted] (11,8)--(13,8);
\draw[thick, gray, dotted] (16,7)--(16,9);
\draw[thick, gray, dotted] (15,8)--(17,8);
\draw[thick, gray, dotted] (20,7)--(20,9);
\draw[thick, gray, dotted] (19,8)--(21,8);
\draw[thick, gray, dotted] (24,7)--(24,9);
\draw[thick, gray, dotted] (23,8)--(25,8);
%
\node at (11.5,8.5){$\scriptstyle a$};
\node at (12.5,8.5){$\scriptstyle b$};
\node at (11.5,7.5){$\scriptstyle c$};
\node at (12.5,7.5){$\scriptstyle d$};
\node at (15.5,8.5){$\scriptstyle e$};
\node at (16.5,8.5){$\scriptstyle f$};
\node at (15.5,7.5){$\scriptstyle g$};
\node at (16.5,7.5){$\scriptstyle h$};
\node at (19.5,8.5){$\scriptstyle i$};
\node at (20.5,8.5){$\scriptstyle j$};
\node at (19.5,7.5){$\scriptstyle k$};
\node at (20.5,7.5){$\scriptstyle l$};
\node at (23.5,8.5){$\scriptstyle m$};
\node at (24.5,8.5){$\scriptstyle n$};
\node at (23.5,7.5){$\scriptstyle o$};
\node at (24.5,7.5){$\scriptstyle p$};
%
\node[below] at (12,7){$\scriptstyle \nu^{(0,1)}$};
\node[below] at (16,7){$\scriptstyle \nu^{(1,1)}$};
\node[below] at (20,7){$\scriptstyle \nu^{(2,1)}$};
\node[below] at (24,7){$\scriptstyle \nu^{(3,1)}$};
\node[above] at (18,9.5){$\scriptstyle  \overbrace{\hspace{50mm}}$};
\node[above] at (18,10.5){$\scriptstyle \bnu$};
\end{tikzpicture}
\qquad
\qquad
\begin{tikzpicture}[scale=0.45]
	\draw[thick, lightgray] (0,0.5)--(0,-12.5);
	\draw[thick, lightgray] (1,0.5)--(1,-12.5);
	\draw[thick, lightgray] (2,0.5)--(2,-12.5);
	\draw[thick, lightgray] (3,0.5)--(3,-12.5);
	\draw[thick, lightgray, dotted] (0,-12.6)--(0,-13.2);
	\draw[thick, lightgray, dotted] (1,-12.6)--(1,-13.2);
	\draw[thick, lightgray, dotted] (2,-12.6)--(2,-13.2);
	\draw[thick, lightgray, dotted] (3,-12.6)--(3,-13.2);
	\draw[thick, black, dotted] (0,0.5)--(0,1);
	\draw[thick, black, dotted] (1,0.5)--(1,1);
	\draw[thick, black, dotted] (2,0.5)--(2,1);
	\draw[thick, black, dotted] (3,0.5)--(3,1);
	\filldraw[green!30] (0,1.5) circle (10pt);
	\filldraw[cyan!30] (1,1.5) circle (10pt);
	\filldraw[orange!30] (2,1.5) circle (10pt);
	\filldraw[red!30] (3,1.5) circle (10pt);
	\node[above] at (0,1){$\scriptstyle 0$};
	\node[above] at (1,1){$\scriptstyle 1$};
	\node[above] at (2,1){$\scriptstyle 2$};
	\node[above] at (3,1){$\scriptstyle 3$};
	\blackdot(0,0);
	\blackdot(0,-1);
	\blackdot(0,-2);
	\blackdot(0,-3);
	\blackdot(0,-4);
	\blackdot(0,-5);
	\blackdot(0,-6);
	\blackdot(0,-7);
	\blackdot(0,-8);
	\draw[thick, lightgray] (-0.2,-9)--(0.2,-9);
	\draw[thick, lightgray] (-0.2,-10)--(0.2,-10);
	\draw[thick, lightgray] (-0.2,-11)--(0.2,-11);
	\draw[thick, lightgray] (-0.2,-12)--(0.2,-12);
	\blackdot(1,0);
	\blackdot(1,-1);
	\draw[thick, lightgray] (0.8,-2)--(1.2,-2);
	\draw[thick, lightgray] (0.8,-3)--(1.2,-3);
	\draw[thick, lightgray] (0.8,-4)--(1.2,-4);
	\draw[thick, lightgray] (0.8,-5)--(1.2,-5);
	\draw[thick, lightgray] (0.8,-6)--(1.2,-6);
	\draw[thick, lightgray] (0.8,-7)--(1.2,-7);
	\draw[thick, lightgray] (0.8,-8)--(1.2,-8);
	\draw[thick, lightgray] (0.8,-9)--(1.2,-9);
	\draw[thick, lightgray] (0.8,-10)--(1.2,-10);
	\draw[thick, lightgray] (0.8,-11)--(1.2,-11);
	\draw[thick, lightgray] (0.8,-12)--(1.2,-12);
	\blackdot(2,0);
	\blackdot(2,-1);
	\blackdot(2,-2);
	\blackdot(2,-3);
	\blackdot(2,-4);
	\blackdot(2,-5);
	\blackdot(2,-6);
	\blackdot(2,-7);
	\blackdot(2,-8);
	\blackdot(2,-9);
	\blackdot(2,-10);
	\blackdot(2,-11);
	\draw[thick, lightgray] (1.8,-12)--(2.2,-12);
	\blackdot(3,0);
	\blackdot(3,-1);
	\blackdot(3,-2);
	\blackdot(3,-3);
	\blackdot(3,-4);
	\draw[thick, lightgray] (2.8,-5)--(3.2,-5);
	\draw[thick, lightgray] (2.8,-6)--(3.2,-6);
	\draw[thick, lightgray] (2.8,-7)--(3.2,-7);
	\draw[thick, lightgray] (2.8,-8)--(3.2,-8);
	\draw[thick, lightgray] (2.8,-9)--(3.2,-9);
	\draw[thick, lightgray] (2.8,-10)--(3.2,-10);
	\draw[thick, lightgray] (2.8,-11)--(3.2,-11);
	\draw[thick, lightgray] (2.8,-12)--(3.2,-12);
	\phantom{
		\draw (0,-13)--(0,-16);
	}
\end{tikzpicture}
}
\end{align*}
\caption{
The arrangement of \(e\)-ribbons in \(\bnu^{\theta}_{\rho}\) corresponding to some nodes in \(\bnu\), for the case \(e=4\), \(\bkap = \kappa_1 = 0\), \(\brho = \rho^{(1)} = (19,16,13,10,9,8,7,6,5,4,3^3,2^3,1^3)\), \(\theta = (1,3,0,2)\), considered in \cref{imagnodes}, together with the abacus diagram for $\rho$ on the right.
}
\label{imagnodesarr}      
\end{figure}


\begin{Example}\label{imagnodes}
Take \(e=4\), \(\ell = 1\), \(\bkap = \kappa_1 = 0\), and consider the \(\bkap\)-core
\begin{align*}
\brho = \rho = (19,16,13,10,9,8,7,6,5,4,3^3,2^3,1^3) \in \Lambda^{\bkap}_+.
\end{align*}
Letting \(\theta = (1,3,0,2)\), we have that \(\rho\) is \((\bkap, \theta)\)-RoCK, with capacity \(\textup{cap}_\delta^\theta(\rho) = 4\). In Figure~\ref{imagnodesarr} we depict \(\rho\), together with some of the \(e\)-ribbons \(u^\theta_{\rho}\) associated with nodes \(u \in \bnu \in \Lambda^{(e,\ell)}\).
Note that moving a bead down on a particular coloured runner on the abacus diagram corresponds to adding the associated coloured skew diagram to $\rho$. Refer to Figure~\ref{twoexpic} for examples of skew diagrams corresponding to specific $e$-quotients $\bnu$ that can be read off from the abacus displays given therein.
\end{Example}

\begin{Lemma}\label{standkost}
Let \(\theta\) be a residue permutation, \(\brho \in \Lambda_+^{\bkap}(\omega)\) a \((\bkap, \theta)\)-RoCK multicore, and \( d \leq \textup{cap}_\delta^\theta(\brho, \bkap)\).
The unique cuspidal Kostant tiling for \(\bnu^\theta_{\brho}\) is given by
\begin{align*}
\Gamma_{\bnu^\theta_{\brho}} = \{u^\theta_{\brho} \mid u \in \bnu\},
\end{align*}
and any standard tableau \({\tt T} \in \Std(\bnu)\) defines a skew tableau \((\Gamma_{\bnu^\theta_{\brho}}, {\tt T}^\theta_{\brho})\) for \(\bnu^\theta_{\brho}\), where \({\tt T}^\theta_{\brho}(k) := ({\tt T}(k))^\theta_{\brho}\).
\end{Lemma}
\begin{proof}
Let \({\tt T} \in \Std(\bnu)\). Then, by the construction of \(\bnu^\theta_{\brho}\), there is a skew tableau \({\tt t}\) for \(\bnu^\theta_{\brho}\), with tiles \(\Gamma_{\tt t} := \{{\tt t}(1), \dots, {\tt t}(|\bnu|)\}\), given by
\begin{align}\label{half1}
{\tt t}(k) =(\textup{sh}^{\downarrow}_{k}{\tt T})^\theta_{\brho}/(\textup{sh}^{\downarrow}_{k-1}{\tt T})^\theta_{\brho}.
\end{align}
Note that the tiles of \(\Gamma_{\tt t}\) are ribbons of content \(\delta\). Thus by \cref{tilethm}(ii) and \cref{combprop2}(ii), it follows that \(\Gamma_{\tt t} = \Gamma_{\bnu^\theta_{\brho}}\) is the unique cuspidal Kostant tiling of \(\bnu^\theta_{\brho}\).

We now show that \({\tt t}(k) = ({\tt T}(k))^\theta_{\brho}\) for all \(1 \leq k \leq |\bnu|\). We go by induction on \(|\bnu|\), the base case \(|\bnu| = 0\) being trivial. Assume that \(|\bnu| >0\). 
For all \(1 \leq k < |\bnu|\), we note that the induction assumption applied to \(\textup{sh}^\downarrow_{|\bnu|-1}{\tt T}\) implies that \({\tt t}(k) =({\tt T}(k))^\theta_{\brho}\), so it remains to show, letting \(u = {\tt T}(|\bnu|)\), that \({\tt t}(|\bnu|) =u^\theta_{\brho}\).

If \(\bnu = \textup{rect}_{u}\), then we have \({\tt t}(|\bnu|) = u^\theta_{\brho}\) by definition, so assume now that \(\bnu \neq \textup{rect}_{u}\). As \(u = {\tt T}(|\bnu|)\), it follows that \(u\) is removable in \(\bnu\), and since \(\bnu \neq \textup{rect}_{u}\), there must be another removable node \(v = {\tt T}(m) \in \bnu\), where \(m < |\bnu|\), and \(v \neq u\). Let \({\tt X} \in \Std(\bnu)\) be defined by taking:
\begin{align*}
{\tt X}(k) =
\begin{cases}
{\tt T}(k) & \textup{if } 1 \leq k < m;\\
{\tt T}(k+1) & \textup{if } m \leq k \leq |\bnu| -1;\\
v & \textup{if } k=|\bnu|.
\end{cases} 
\end{align*}
In essence \({\tt X}\) arises from \({\tt T}\) by moving \(v\) to the end of the tableau (so that \( {\tt X}(|\bnu| -1) = u\)). We thus may define a skew tableau \((\Gamma_{\bnu^\theta_{\brho}}, {\tt x})\) for \(\bnu^\theta_{\brho}\), where \({\tt x}\) is defined from \({\tt X}\) as in (\ref{half1}).

Again, by the induction assumption we have that \({\tt x}(k) = ({\tt X}(k))^\theta_{\brho}\) for all \(1 \leq k < |\bnu|\). 
Therefore it follows that
\begin{align*}
\Gamma_{\bnu^\theta_{\brho}} = \{{\tt t}(|\bnu|)\} \sqcup \{ w^\theta_{\brho} \mid w \in \bnu \backslash \{u\}\} = \{{\tt x}(|\bnu|)\} \sqcup \{ w^\theta_{\brho} \mid w \in \bnu \backslash \{v\}\},
\end{align*}
and thus
\begin{align*}
{\tt t}(|\bnu|) \in \{{\tt x}(|\bnu|), u^\theta_{\brho}\}.
\end{align*}
If \({\tt t}(|\bnu|) = {\tt x}(|\bnu|)\), this would imply that \((\textup{sh}^\downarrow_{|\bnu|-1}{\tt T})^\theta_{\brho} = (\textup{sh}^\downarrow_{|\bnu|-1}{\tt X})^\theta_{\brho} \) even though \(\textup{sh}^\downarrow_{|\bnu|-1}{\tt T} \neq \textup{sh}^\downarrow_{|\bnu|-1}{\tt X} \), which is easily seen to contradict the definition of \(\bnu^\theta_{\brho}\). Therefore \({\tt t}(|\bnu|) =u^\theta_{\brho}\), as desired.
\end{proof}

\begin{Proposition}\label{resirrlist}
Let \(\theta\) be a residue permutation, \(\brho \in \Lambda_+^{\bkap}(\omega)\) a \((\bkap, \theta)\)-RoCK multicore, and \( d \leq \textup{cap}_\delta^\theta(\brho, \bkap)\).
We have an irredundant list
\begin{align*}
\Lambda^{\bkap}_{+/\brho}(d\delta) = \{\bnu_{\brho}^\theta \mid \bnu \in \Lambda_+^{(e,\ell)}(d)\}.
\end{align*}
\end{Proposition}
\begin{proof}
By \cref{combprop2}(ii), any \(\blam \in \Lambda^{\bkap}_+(\omega + d\delta)\) may be constructed from \(\brho\) by progressively adding \(d\) ribbons of content \(\delta\). In terms of \(\bkap\)-beta numbers, each such addition is equivalent to replacing some \(m\) in the \(\bkap\)-beta numbers with the value \(m+e\). Then \(\brho \sqcup \bnu^\theta_{\brho}\) is the result of a process which moves the \(q\)th largest element of \(\B^r_{\theta_{e-t}}(\brho, \bkap)\) up \(e \cdot \nu_q^{(t,r)}\) times. It is straightforward to see then that 
\begin{align*}
\{\brho \sqcup \bnu^\theta_{\brho} \mid \bnu \in \Lambda_+^{(e,\ell)}(d)\}
\end{align*}
gives an irredundant list of all possible results of \(d\) such moves, proving the proposition.
\end{proof}

\subsection{Ribbon diagrams}\label{ribbonshapesec}
Recall that an addable ribbon \(\xi\) for \(\brho\) corresponds to some integers \(x<y\), with \(x \in \B^r(\brho, \bkap)\), \(y \notin \B^r(\brho, \bkap)\). More specifically, it is straightforward to check that
\begin{itemize}
\item \(\xi\) is contained within a single {\em row} if \(z \notin \B^r(\brho, \bkap)\) for all \(x<z<y\), and;
\item  \(\xi\) is contained within a single {\em column} if \(z \in \B^r(\brho, \bkap)\) for all \(x<z<y\).
\end{itemize}
From this, it follows that in general, if 
\begin{align*}
x = z_{k+1} < z_k <  \dots < z_2 < z_1 = y,
\end{align*}
where
\begin{align*}
\{z \mid x<z<y \} \cap \B^r_i(\brho, \bkap) = \{z_2, \dots, z_k\},
\end{align*}
then the ribbon \(\xi\) consists of \(k\) rows, where the \(j\)th row from the top has length \(z_j - z_{j+1}\), and, reading from left to right, consists of the residues
\begin{align*}
\overline{z_{j+1} +1}, \overline{z_{j+1} +2}, \dots, \overline{z}_j.
\end{align*}

\subsubsection{Minuscule ribbons}\label{minribs} Let \(\theta = (\theta_1, \dots, \theta_e)\) be a residue permutation, and for all \(a \in [0,e-1]\), let \((\phi_a^1, \dots, \phi_a^{a+1})\) be the subsequence of \((\overline{\theta}_{e-a},  \overline{\theta_{e-a} - 1}, \overline{\theta_{e-a} - 2}, \dots, \overline{\theta_{e-a} - e+1})\) consisting only of elements \(\theta_b\), where \(b\geq e-a\). In other words, to find \((\phi_a^1, \dots, \phi_a^{a+1})\), one travels clockwise for one full cycle around the Dynkin diagram in Figure~\ref{fig:dynkin}, recording only the residues that correspond to \(\theta_b\) with \(b \geq e-a\), beginning with \(\phi_a^1 = \bar{\theta}_{e-a}\).

For \(a \in [0,e-1]\), we say that a ribbon \(\xi \in \Lambda(\delta)\) is a {\em \((\theta, a)\)-ribbon} provided that \(\xi\) consists of \(a+1\) rows, with the \(j\)th row consisting of the residues
\begin{align*}
\overline{\phi_a^{j+1} + 1}, \overline{\phi_a^{j+1} + 2}, \dots, \overline{\phi_a^j},
\end{align*}
for all \(j \in [1,a+1]\).


\begin{Example}\label{thetaadeltaribbons}
For the trivial residue permutation $\theta=(0,1,\dots,e-1)$, the \((\theta,a)\)-ribbons are hook partitions with increasing leg lengths as $a$ increases.
More generally, \((\theta,a)\)-ribbons need not be hooks.
For example, given \(e=4\) and \(\theta = (1,3,0,2)\), the \((\theta,a)\)-ribbons look as follows.
\begin{align*}
{}
(\theta,0):
{}
\hackcenter{
\begin{tikzpicture}[scale=0.29]
\draw[thick,fill=orange!30] (16+7,25+1)--(20+7,25+1)--(20+7,26+1)--(16+7,26+1)--(16+7,25+1);
\node at (16.5+7,25.5+1){$\scriptstyle 3$};
\node at (17.5+7,25.5+1){$\scriptstyle 0$};
\node at (18.5+7,25.5+1){$\scriptstyle 1$};
\node at (19.5+7,25.5+1){$\scriptstyle 2$};
\draw[thick, gray, dotted] (17+7,25+1)--(17+7,26+1);
\draw[thick, gray, dotted] (18+7,25+1)--(18+7,26+1);
\draw[thick, gray, dotted] (19+7,25+1)--(19+7,26+1);
\end{tikzpicture}
}
\qquad
(\theta,1):
{}
\hackcenter{
\begin{tikzpicture}[scale=0.29]
\draw[thick,fill=green!30] (7+0,20+0)--(9+0,20+0)--(9+0,21+0)--(10+0,21+0)--(10+0,22+0)--(8+0,22+0)--(8+0,21+0)--(7+0,21+0)--(7+0,20+0);
\node at (7.5+0,20.5+0){$\scriptstyle 1$};
\node at (8.5+0,20.5+0){$\scriptstyle 2$};
\node at (8.5+0,21.5+0){$\scriptstyle 3$};
\node at (9.5+0,21.5+0){$\scriptstyle 0$};
\draw[thick, gray, dotted] (8+0,20+0)--(8+0,21+0);
\draw[thick, gray, dotted] (8+0,21+0)--(9+0,21+0);
\draw[thick, gray, dotted] (9+0,21+0)--(9+0,22+0);
\end{tikzpicture}
}
{}
\qquad
(\theta,2):
\hackcenter{
\begin{tikzpicture}[scale=0.29]
\draw[thick,fill=red!30] (2+0,12+0)--(3+0,12+0)--(3+0,13+0)--(4+0,13+0)--(4+0,15+0)--(3+0,15+0)--(3+0,14+0)--(2+0,14+0)--(2+0,12+0);
\node at (2.5+0,12.5+0){$\scriptstyle 0$};
\node at (2.5+0,13.5+0){$\scriptstyle 1$};
\node at (3.5+0,13.5+0){$\scriptstyle 2$};
\node at (3.5+0,14.5+0){$\scriptstyle 3$};
\draw[thick, gray, dotted] (2+0,13+0)--(3+0,13+0);
\draw[thick, gray, dotted] (3+0,13+0)--(3+0,14+0);
\draw[thick, gray, dotted] (3+0,14+0)--(4+0,14+0);
\end{tikzpicture}
}
\qquad
(\theta,3): 
\hackcenter{
\begin{tikzpicture}[scale=0.29]
\draw[thick,fill=cyan!30] (0+0,0+0)--(1+0,0+0)--(1+0,4+0)--(0+0,4+0)--(0+0,0+0);
\node at (0.5+0,0.5+0){$\scriptstyle 2$};
\node at (0.5+0,1.5+0){$\scriptstyle 3$};
\node at (0.5+0,2.5+0){$\scriptstyle 0$};
\node at (0.5+0,3.5+0){$\scriptstyle 1$};
\draw[thick, gray, dotted] (0+0,1+0)--(1+0,1+0);
\draw[thick, gray, dotted] (0+0,2+0)--(1+0,2+0);
\draw[thick, gray, dotted] (0+0,3+0)--(1+0,3+0);
\end{tikzpicture}
}
\end{align*}
As a visual aid, the statements of \cref{allaboutshapes} below may be compared against these \((\theta,a)\)-ribbons and Figure~\ref{imagnodesarr}.
\end{Example}

\subsection{Imaginary semicuspidal skew diagrams}
In this section we relate the shape of the skew multipartition \(\bnu^\theta_{\brho}\) to the shape of \(\bnu\).

\begin{Proposition}\label{allaboutshapes}
Let \(\theta\) be a residue permutation, \(\brho \in \Lambda_+^{\bkap}(\omega)\) a \((\bkap, \theta)\)-RoCK multicore, \( d \leq \textup{cap}_\delta^\theta(\brho, \bkap)\), and \(\bnu \in \Lambda^{(e,\ell)}_+(d)\).
\begin{enumerate}
\item Each \((\nu^{(a,r)})^\theta_{\brho}\) is a connected skew diagram, and \((\nu^{(a,r)})^\theta_{\brho} \NEarrow (\nu^{(a',r')})^\theta_{\brho}\) whenever \(r> r'\), or \(r=r'\) and \(a>a'\). 
\item If \(u \in \nu^{(a,r)}\), then \(u^\theta_{\brho} \subseteq (\nu^{(a,r)})^\theta_{\brho}\) and \(u^\theta_{\brho}\) is a \((\theta,a)\)-ribbon.
\item If \(u, {\tt E}u \in \nu^{(a,r)}\), and \(v\) is the northeasternmost node in \(u^\theta_{\brho}\), then \({\tt E}v\) is the southwesternmost node in \(({\tt E}u)^\theta_{\brho}\).
\item If \(u, {\tt S}u \in \nu^{(a,r)}\), and \(v\) is the southwesternmost node in \(u^\theta_{\brho}\), then \({\tt S}v\) is the northeasternmost node in \(({\tt S}u)^\theta_{\brho}\).
\item If \(u, {\tt SE}u \in \nu^{(a,r)}\), then \(({\tt SE}u)^\theta_{\brho} = {\tt SE}(u^\theta_{\brho})\).
\end{enumerate}
\end{Proposition}
\begin{proof}
Let \(u \in \nu^{(a,r)}\). Then
\begin{align*}
u^\theta_{\brho} = (\textup{rect}_u)^\theta_{\brho}/ (\textup{rect}_u \backslash \{u\})^\theta_{\brho} = ( \brho \sqcup (\textup{rect}_u)^\theta_{\brho})/ (\brho \sqcup (\textup{rect}_u \backslash \{u\})^\theta_{\brho})
\end{align*}
Assume that \(u\) is in the \(q\)th row and \(p\)th column of \(\textup{rect}_u\). Considering (\ref{mmove}), we have that
\begin{align*}
m^r_{\theta_{e-a}}(\brho \sqcup (\textup{rect}_u)^\theta_{\brho}, \bkap)_q \in \B^r(\brho \sqcup (\textup{rect}_u)^\theta_{\brho}, \bkap),
\qquad
\textup{and}
\qquad
m^r_{\theta_{e-a}}(\brho \sqcup (\textup{rect}_u)^\theta_{\brho}, \bkap)_q-e \notin \B^r(\brho \sqcup (\textup{rect}_u)^\theta_{\brho}, \bkap),
\end{align*}
and the \(\bkap\)-beta numbers for \(\brho \sqcup (\textup{rect}_u \backslash \{u\})^\theta_{\brho}\) are achieved by replacing the former number with the latter. 

Assume that \(z \in\mathbb{Z}\) is such that
\begin{align*}
m^r_{\theta_{e-a}}(\brho \sqcup (\textup{rect}_u)^\theta_{\brho}, \bkap)_q-e < z < m^r_{\theta_{e-a}}(\brho \sqcup (\textup{rect}_u)^\theta_{\brho}, \bkap)_q.
\end{align*}
If \(\overline{z} = \theta_b\), it follows that
\begin{align*}
z = m^r_{\theta_{e-a}}(\brho \sqcup (\textup{rect}_u)^\theta_{\brho}, \bkap)_q-e + \overline{\theta_b - \theta_{e-a}}.
\end{align*}
Let \(e-a < b\). Now, note that 
\begin{align*}
\frac{M^r_{\theta_{b}} - M^r_{\theta_{e-a}} - \overline{\theta_b - \theta_{e-a}}}{e} = \left \lfloor  \frac{M^r_{\theta_{b}} - M^r_{\theta_{e-a}}}{e}    \right \rfloor = h^r_{\theta_{e-a}, \theta_b} \geq h^{r,\theta}_{[e-a,b-1]} \geq  h_{e-a}^{r,\theta}(\brho, \bkap)   \geq \textup{cap}^\theta_\delta(\brho,\bkap) -1 \geq pq - 1,
\end{align*}
so 
\begin{align*}
M^r_{\theta_{b}} - M^r_{\theta_{e-a}} - \overline{\theta_b - \theta_{e-a}} \geq epq - e.
\end{align*}
Therefore,
\begin{align*}
z &= m^r_{\theta_{e-a}}(\brho \sqcup (\textup{rect}_u)^\theta_{\brho}, \bkap)_q -e + \overline{\theta_b - \theta_{e-a}}=m^r_{\theta_{e-a}}(\brho, \bkap)_q + e (\textup{rect}_u)^{(a,r)}_q - e + \overline{\theta_b - \theta_{e-a}}\\
&=m^r_{\theta_{e-a}}(\brho, \bkap)_q + e p - e + \overline{\theta_b - \theta_{e-a}}
\leq M^r_{\theta_{e-a}}(\brho, \bkap) + ep - e + \overline{\theta_b - \theta_{e-a}}\\
&\leq M^r_{\theta_b}(\brho, \bkap) -epq + ep 
\leq M^r_{\theta_b}(\brho, \bkap).
\end{align*}
Since \(\B^r_{\theta_b}(\brho \sqcup (\textup{rect}_u)^\theta_{\brho}, \bkap) =\B^r_{\theta_b}(\brho, \bkap) \) and \(\brho\) is a multicore, it follows that \(z \in \B^r_{\theta_b}(\brho \sqcup (\textup{rect}_u)^\theta_{\brho}, \bkap)\).

On the other hand, assume that \(b< e-a\). Then as before, 
\begin{align*}
M_{\theta_{e-a}}^r - M_{\theta_b}^r - \overline{\theta_{e-a} - \theta_b} \geq epq - e,
\end{align*}
so
\begin{align*}
m^r_{\theta_{e-a}}(\brho, \bkap)_q - m^r_{\theta_b}(\brho, \bkap)_q - \overline{\theta_{e-a} - \theta_b} \geq epq - e.
\end{align*}
Therefore
\begin{align*}
z  &= m^r_{\theta_{e-a}}(\brho \sqcup (\textup{rect}_u)^\theta_{\brho}, \bkap)_q -e + \overline{\theta_b - \theta_{e-a}}
=m^r_{\theta_{e-a}}(\brho, \bkap)_q + e (\textup{rect}_u)^{(a,r)}_q - e + \overline{\theta_b - \theta_{e-a}}\\
&=m^r_{\theta_{e-a}}(\brho, \bkap)_q + e p - e + \overline{\theta_b - \theta_{e-a}}
\geq (m^r_{\theta_b}(\brho, \bkap)_q + \overline{\theta_{e-a} - \theta_b} + epq - e)+ e p - e + \overline{\theta_b - \theta_{e-a}}\\
&=m^r_{\theta_b}(\brho, \bkap)_q + epq + ep -e
= (M^r_{\theta_b} - (q-1)e)+ epq + ep -e
> M^r_{\theta_b}. 
\end{align*}
Since \(\B^r_{\theta_b}(\brho \sqcup (\textup{rect}_u)^\theta_{\brho}, \bkap) =\B^r_{\theta_b}(\brho, \bkap) \), it follows then that \(z \notin \B^r_{\theta_b}(\brho \sqcup (\textup{rect}_u)^\theta_{\brho}, \bkap)\).

Therefore we have that \(z \in \B^r_{\theta_b}(\brho \sqcup (\textup{rect}_u)^\theta_{\brho}, \bkap)\) is such that
\begin{align*}
m^r_{\theta_{e-a}}(\brho \sqcup (\textup{rect}_u)^\theta_{\brho}, \bkap)_q-e < z < m^r_{\theta_{e-a}}(\brho \sqcup (\textup{rect}_u)^\theta_{\brho}, \bkap)_q.
\end{align*}
if and only if \(\overline{z} = \theta_b\), where \(e-a<b\). It follows then from \S\ref{minribs} that \(u^\theta_{\brho}\) is a \((\theta,a)\)-ribbon, proving (ii).

(iii) Suppose \(u, {\tt E}u\) are in the \(q\)th row of \(\nu^{(a,r)}\).
We have
\begin{align*}
u^\theta_{\brho} \sqcup ({\tt E}u)^\theta_{\brho}
=
(\textup{rect}_{{\tt E}u})^\theta_{\brho} / (\textup{rect}_{{\tt E}u}\backslash \{u, {\tt E} u\})^\theta_{\brho}
=
(\brho \sqcup (\textup{rect}_{{\tt E}u})^\theta_{\brho}) / (\brho \sqcup (\textup{rect}_{{\tt E}u}\backslash \{u, {\tt E} u\})^\theta_{\brho}).
\end{align*}
Considering (\ref{mmove}), we have that
\begin{align*}
m^r_{\theta_{e-a}}(\brho \sqcup (\textup{rect}_{{\tt E}u})^\theta_{\brho}, \bkap)_q \in \B^r(\brho \sqcup (\textup{rect}_{{\tt E}u})^\theta_{\brho}, \bkap),
\quad
\textup{and}
\quad
m^r_{\theta_{e-a}}(\brho \sqcup (\textup{rect}_{{\tt E}u})^\theta_{\brho}, \bkap)_q-2e \notin \B^r(\brho \sqcup (\textup{rect}_{{\tt E}u})^\theta_{\brho}, \bkap),
\end{align*}
and the \(\bkap\)-beta numbers \(\B^r(\brho \sqcup (\textup{rect}_{{\tt E}u} \backslash \{u, {\tt E}u\})^\theta_{\brho}, \bkap)\) are achieved from \(\B^r(\brho \sqcup (\textup{rect}_{{\tt E}u})^\theta_{\brho}, \bkap)\) by replacing the former number with the latter. 

It follows then that \(u^\theta_{\brho} \sqcup({\tt E}u)^\theta_{\brho} = (\xi_1, \dots, \xi_{2e})\) is a single ribbon (with nodes \(\xi_1, \dots, \xi_{2e}\) labeled from southwesternmost to northeasternmost) which by (ii) is the union of the two \((\theta,a)\)-ribbons \(u^\theta_{\brho}\) and \(({\tt E}u)^\theta_{\brho}\). Therefore 
\begin{align*}
\{u^\theta_{\brho}, ({\tt E}u)^\theta_{\brho}\} = \{(\xi_1, \dots, \xi_e), (\xi_{e+1}, \dots, \xi_{2e}) \}.
\end{align*}
We note moreover that 
\begin{align*}
m^r_{\theta_{e-a}}(\brho \sqcup (\textup{rect}_{{\tt E}u})^\theta_{\brho}, \bkap)_q-e \notin \B^r(\brho \sqcup (\textup{rect}_{{\tt E}u})^\theta_{\brho}, \bkap),
\end{align*}
and thus by \S\ref{ribbonshapesec}, we have that 
\begin{align}\label{e1}
\xi_{e+1} = {\tt E} \xi_{e}.
\end{align}
Therefore \((\xi_{e+1}, \dots, \xi_{2e})\) is removable in \(\xi\), and  \((\xi_{1}, \dots, \xi_e)\) is not. As \({\tt E}u\) is removable in \(\textup{rect}_{{\tt E}u}\), it follows from \cref{standkost} that \(({\tt E}u)^\theta_{\brho}\) is removable in \((\textup{rect}_{{\tt E}u})^\theta_{\brho}\), and so it must be that
\begin{align*}
u^\theta_{\brho} = (\xi_1, \dots, \xi_e), 
\qquad
\textup{and}
\qquad
({\tt E}u)^\theta_{\brho} = (\xi_{e+1}, \dots, \xi_{2e}),
\end{align*}
which in view of (\ref{e1}), establishes (iii). The proof of (iv) is similar. Full details are included in the \texttt{arXiv} version of the paper, see \S\ref{SS:ArxivVersion}. 
\begin{answer}
(iv) Suppose that \({\tt S}u\) is in the \(q\)th row of \(\nu^{(a,r)}\).
We have
\begin{align*}
u^\theta_{\brho} \sqcup ({\tt S}u)^\theta_{\brho} = 
(\textup{rect}_{{\tt S}u})^\theta_{\brho} / (\textup{rect}_{{\tt S}u}\backslash \{u, {\tt S}u\})^\theta_{\brho}
=
(\brho \sqcup (\textup{rect}_{{\tt S}u})^\theta_{\brho}) / (\brho \sqcup (\textup{rect}_{{\tt S}u}\backslash \{u, {\tt S}u\})^\theta_{\brho}).
\end{align*}
Considering (\ref{mmove}), we have that
\begin{align*}
m^r_{\theta_{e-a}}(\brho \sqcup (\textup{rect}_{{\tt S}u})^\theta_{\brho}, \bkap)_{q-1} \in \B^r(\brho \sqcup (\textup{rect}_{{\tt S}u})^\theta_{\brho}, \bkap),
\quad
\textup{and}
\quad
m^r_{\theta_{e-a}}(\brho \sqcup (\textup{rect}_{{\tt S}u})^\theta_{\brho}, \bkap)_{q-1}-2e \notin \B^r(\brho \sqcup (\textup{rect}_{{\tt S}u})^\theta_{\brho}, \bkap),
\end{align*}
and the \(\bkap\)-beta numbers \(\B^r(\brho \sqcup (\textup{rect}_{{\tt S}u} \backslash \{u, {\tt S}u\})^\theta_{\brho}, \bkap)\) are achieved from \(\B^r(\brho \sqcup(\textup{rect}_{{\tt S}u})^\theta_{\brho}, \bkap)\) by replacing the former number with the latter. 

It follows then that \(u^\theta_{\brho} \sqcup ({\tt S}u)^\theta_{\brho} = (\xi_1, \dots, \xi_{2e})\) is a single ribbon which by (ii) is the union of the two \((\theta,a)\)-ribbons \(u^\theta_{\brho}\) and \(({\tt S}u)^\theta_{\brho}\). Therefore 
\begin{align*}
\{u^\theta_{\brho}, ({\tt S}u)^\theta_{\brho}\} = \{(\xi_1, \dots, \xi_e), (\xi_{e+1}, \dots, \xi_{2e}) \}.
\end{align*}
We note moreover that 
\begin{align*}
m^r_{\theta_{e-a}}(\brho \sqcup (\textup{rect}_{{\tt S}u})^\theta_{\brho}, \bkap)_{q-1}-e  = m^r_{\theta_{e-a}}(\brho \sqcup (\textup{rect}_{{\tt S}u})^\theta_{\brho}, \bkap)_{q}\in \B^r(\brho \sqcup (\textup{rect}_{{\tt S}u})^\theta_{\brho}, \bkap),
\end{align*}
and thus by \S\ref{ribbonshapesec}, we have that 
\begin{align}\label{n1}
\xi_{e} = {\tt S}\xi_{e+1}.
\end{align} Therefore \((\xi_{1}, \dots, \xi_{e})\) is removable in \(\xi\), and  \((\xi_{e+1}, \dots, \xi_{2e})\) is not. As \({\tt S}u\) is removable in \(\textup{rect}_{{\tt S}u}\), it follows from \cref{standkost} that \(({\tt S}u)^\theta_{\brho}\) is removable in \((\textup{rect}_{{\tt S}u})^\theta_{\brho}\), and so it must be that
\begin{align*}
u^\theta_{\brho} =  (\xi_{e+1}, \dots, \xi_{2e}), 
\qquad
\textup{and}
\qquad
({\tt S}u)^\theta_{\brho} =(\xi_1, \dots, \xi_e),
\end{align*}
which, in view of (\ref{n1}), establishes (iv). 
\end{answer}

Noting that \(u^\theta_{\brho}\), \(({\tt SE}u)^\theta_{\brho}\) are both ribbons of the same shape by (ii), we have that (v) immediately follows from (iii) and (iv). As for (i), the fact that \((\nu^{(a,r)})^\theta_{\brho}\) is connected follows from (iii), (iv). If \(r > r'\), then \((\nu^{(a,r)})^\theta_{\brho} \NEarrow (\nu^{(a',r')})^\theta_{\brho}\) follows from the fact that \((\nu^{(a,r)})^\theta_{\brho} \subseteq ( \bnu^\theta_{\brho})^{(r)}\) and \((\nu^{(a',r')})^\theta_{\brho} \subseteq ( \bnu^\theta_{\brho})^{(r')}\). Assume that \(r = r'\) and \(a> a'\), and that \(\nu^{(a,r)}, \nu^{(a',r)} \neq \varnothing\). Then \(e-a < e-a'\).
The diagrams \((\nu^{(a,r)})^\theta_{\brho}, (\nu^{(a',r)})^\theta_{\brho}\) are necessarily disconnected by \cite[Proposition 7.14]{muthtiling} and \cref{combprop2}(ii), since \((\nu^{(a,r)})^\theta_{\brho}\) is tiled by \((\theta,a)\)-ribbons and \((\nu^{(a',r)})^\theta_{\brho}\) is tiled by \((\theta,a')\)-ribbons by (ii), so either 
\begin{align}\label{NEor}
(\nu^{(a,r)})^\theta_{\brho} \NEarrow (\nu^{(a',r)})^\theta_{\brho}
\qquad
\textup{or}
\qquad
(\nu^{(a',r)})^\theta_{\brho} \NEarrow (\nu^{(a,r)})^\theta_{\brho}.
\end{align}
Let \(u\) be a removable node in the \(q\)th row of \(\nu^{(a,r)}\), and let \(v\) be a removable node in the \(q'\)th row of \(\nu^{(a',r)}\).
Then, letting \(\bmu = \textup{rect}_u \sqcup \textup{rect}_{v}\),  each of \(u^\theta_{\brho},v^\theta_{\brho}\) is a removable ribbon in \(\brho \sqcup \bmu^\theta_{\brho}\) by \cref{standkost}. 
We have that
\begin{align*}
M_{\theta_{e-a}}^r(\brho, \bkap) + e \nu_q^{(a,r)} \in \B^r_{\theta_{e-a}}(\brho \sqcup \bmu^\theta_{\brho}, \bkap),
\qquad
\textup{and}
\qquad
M_{\theta_{e-a}}^r(\brho, \bkap) + e (\nu_q^{(a,r)}-1) \notin \B^r_{\theta_{e-a}}(\brho \sqcup \bmu^\theta_{\brho}, \bkap),
\end{align*}
and the \(\bkap\)-beta numbers \(\B^r((\brho \sqcup \bmu^\theta_{\brho})\backslash u^\theta_{\brho}, \bkap)\) are achieved from \(\B^r(\brho \sqcup \bmu^\theta_{\brho}, \bkap)\) by replacing the former number with the latter. 
Similarly, we have that
\begin{align*}
M_{\theta_{e-a'}}^r(\brho, \bkap) + e \nu_{q'}^{(a',r)} \in \B^r_{\theta_{e-a'}}(\brho \sqcup \bmu^\theta_{\brho}, \bkap),
\qquad
\textup{and}
\qquad
M_{\theta_{e-a'}}^r(\brho, \bkap) + e (\nu_{q'}^{(a',r)}-1) \notin \B^r_{\theta_{e-a'}}(\brho \sqcup \bmu^\theta_{\brho}, \bkap),
\end{align*}
and the \(\bkap\)-beta numbers \(\B^r((\brho \sqcup  \bmu^\theta_{\brho})\backslash v^\theta_{\brho}, \bkap)\) are achieved from \(\B^r(\brho \sqcup \bmu^\theta_{\brho}, \bkap)\) by replacing the former number with the latter. But note now that 
\begin{align*}
M^r_{\theta_{e-a}}(\brho, \bkap) + e\nu^{(a,r)}_q
&\leq
M^r_{\theta_{e-a}}(\brho, \bkap) + e \cdot \textup{cap}^\theta_\delta(\brho, \bkap)
\leq M^r_{\theta_{e-a}}(\brho, \bkap) + e (h_{\theta_{e-a}, \theta_{e-a'}}^r(\brho, \bkap) + 1)\\
&< M^r_{\theta_{e-a'}}(\brho, \bkap) + e
\leq M^r_{\theta_{e-a'}}(\brho, \bkap) + e \nu_{q'}^{(a',r)},
\end{align*}
which implies, in view of the construction of \(\bkap\)-beta numbers, that the bottom row of \(u^\theta_{\brho}\) is in a lower row of \(\brho \sqcup \bmu^\theta_{\brho}\) than the bottom row of \(v^\theta_{\brho}\). Therefore it cannot be that \(v^\theta_{\brho} \NEarrow u^\theta_{\brho}\). As \(u^\theta_{\brho} \subseteq (\nu^{(a,r)})^\theta_{\brho}\), and \(v^\theta_{\brho} \subseteq (\nu^{(a',r)})^\theta_{\brho}\), the result follows from (\ref{NEor}).
\end{proof}

\begin{Corollary}\label{thetaaribcusp}
Let \(\xi \in \N\) be a ribbon of content \(\delta\). Then \(\xi\) is cuspidal if and only if \(\xi\) is a \((\theta,a)\)-ribbon for some \(a \in [0,e-1]\).
\end{Corollary}

\begin{Remark}
	Suppose that one adds $d$ addable $\delta$-ribbons 
	to $\brho\in\Lambda_+^{\bkap}(\omega)$. Observe from \cref{allaboutshapes}(ii) that if $d\leqslant\textup{cap}^\theta_\delta(\omega, \bkap)$, then each of the $\delta$-ribbons is precisely one of the distinct \((\theta,a)\)-ribbons defined above. However, if $d>\textup{cap}^\theta_\delta(\omega, \bkap)$, then the addable $\delta$-ribbons need not be \((\theta,a)\)-ribbons.
\end{Remark}

\subsection{Gelfand--Graev Words}\label{subsec:ggwords}
Let \(\bb_{(\theta, a)}\) be the leading row word for a \((\theta,a)\)-ribbon, so:
\begin{align*}
\bb_{(\theta,a)} :=(\overline{ \phi_a^2 + 1} \dots  \overline{ \phi_a} )( \overline{ \phi_a^3 + 1}\dots  \overline{ \phi_a^2} ) \dots 
( \overline{ \phi_a^{a+1} + 1} \dots \overline{ \phi_a^{a}} )( \overline{ \phi_a + 1} \dots  \overline{ \phi_a^{a+1}} ) \in I^\delta.
\end{align*}
For a word \(\bi = i_1 \dots i_m \in I^m\), \(n \in \mathbb{Z}_{>0}\), and a partition \(\lambda = (\lambda_1, \lambda_2, \ldots ) \vdash n\), we will write
\begin{align*}
\bi^{(n)} &:= i_1^n \dots i_m^n \in I^{nm};
\qquad
\textup{and}
\qquad
\bi^{\lambda} := \bi^{(\lambda_1)} \bi^{(\lambda_2)} \ldots \in I^{nm}.
\end{align*}
Then, for \(\bnu = (\nu^{(t,r)})_{t \in [0,e-1], r \in [1,\ell]} \in \Lambda^{(e,\ell)}_+(d)\), we define the {\em Gelfand--Graev word}:
\begin{align*}
\bb^{\bnu}_{\theta} :=
(\bb_{(\theta, 0)}^{\nu^{(0,1)}} \dots \bb_{(\theta, e-1)}^{\nu^{(e-1,1)}} ) 
(\bb_{(\theta, 0)}^{\nu^{(0,2)}} \dots \bb_{(\theta, e-1)}^{\nu^{(e-1,2)}} )
\dots
 (\bb_{(\theta, 0)}^{\nu^{(0,\ell)}} \dots \bb_{(\theta, e-1)}^{\nu^{(e-1,\ell)}} ) \in I^{d\delta}.
\end{align*}
We note that Gelfand--Graev words were introduced in \cite{km17} to construct Gelfand--Graev representations analogous to those introduced in \cite{bdk01}.

\begin{Lemma}\label{reswords}
Let \(\ell = 1\), \(\bkap = 0\), and \(\brho = \rho\) be a \(\bkap\)-core. Let \(\bnu, \bmu \in \Lambda^{(e,1)}_+(d)\), with \(\nu^{(0,1)} = \mu^{(0,1)} = \varnothing\).
Then \(\bb_\theta^{\bnu}\) is a word in \(\bnu^\theta_{\brho}\), and if \(\bb_\theta^{\bnu}\) is a word in \(\bmu^\theta_{\brho}\) then \(\bmu \trianglerighteq \bnu\).
\end{Lemma}

\begin{proof}
Let \(\bnu, \bmu \in \Lambda^{(e,1)}_+(d)\), with \(\nu^{(0,1)} = \mu^{(0,1)} = \varnothing\). Assume \(\bb_\theta^{\bnu}\) is a word in \(\bmu^\theta_{\brho}\).
Let \(k = \nu^{(a,1)}_m\) be the length of the lowest nonempty row in \(\bnu\). Let \(\bnu'\) be the multipartition achieved by deleting this row. Then we have that \(\bb_\theta^{\bnu} = \bb_\theta^{\bnu'}
\bb_{(\theta,a)}^{(k)}
\).
Then there exists a skew tableau \((\balpha', \balpha'')\) for \(\bmu_{\brho}^\theta\), where \(\bb_\theta^{\bnu'}\) is a word in \(\balpha'\) and \(\bb_{(\theta,a)}^{(k)}\) is a word in \(\balpha''\). As 
\(\bb^{\bnu'}_\theta \in I^{(d-k)\delta}\), \(\bb_{(\theta,a)}^{(k)} \in I^{k\delta}\), and \(\bmu^\theta_{\brho} \in \Lambda_{+/+}(k\delta)_{\approx \delta}\)
by \cref{combprop2}, it follows that each of \(\balpha', \balpha''\) is a union of tiles in \(\Gamma_{\bmu^\theta_{\brho}}\), and hence
there is some \(\bmu' \subset \bmu\) such that \((\bmu')^\theta_{\brho} = \balpha'\) and \((\bmu/\bmu')^\theta_{\brho} = \balpha''\). 
Thus \(\bb_\theta^{\bnu'}\) is a word in \((\bmu')^\theta_{\brho}\) and
\(\bb_{(\theta,a)}^{(k)}\) is a word in \((\bmu/\bmu')^\theta_{\brho}\). We may assume by induction then that \(\bmu'  \trianglerighteq \bnu'\).

By \cref{allaboutshapes}(i) we have that
\begin{align*}
((\bmu/\bmu')^{(e-1)})^\theta_{\brho} \; \NEarrow \cdots \NEarrow ((\bmu/\bmu')^{(a)})^\theta_{\brho}\; \NEarrow \cdots \NEarrow ((\bmu/\bmu')^{(1)})^\theta_{\brho}.
\end{align*}
Note that we have \(\bar{\phi}_a = \bar{\theta}_{e-a}\), so in the word \(\bb_{(\theta,a)}^{(k)}\), all (\(\overline{\theta}_{e-a}\))s appear to the left of all (\(\overline{\theta_{e-a} +1}\))s. 
Therefore 
since \(\bb_{(\theta,a)}^{(k)}\) is a word in \((\bmu/\bmu')^\theta_{\brho}\), there can be no nodes with residue \(\overline{\theta}_{e-a}\) directly south of a node with residue \(\overline{\theta_{e-a} +1}\) in \((\bmu/\bmu')^\theta_{\brho}\). But when \(a+1 \leq x \leq e-1\), a \((\theta,x)\)-ribbon has the node with residue \(\overline{\theta}_{e-a}\) south of the node with residue \(\overline{\theta_{e-a} + 1}\). Therefore by \cref{allaboutshapes}(ii) we have \((\bmu/\bmu')^{(a+1)} = \dots = (\bmu/\bmu')^{(e-1)} = \varnothing\). Next, note that if there are nodes \(u, {\tt S}u \in (\bmu/\bmu')^{(a)}\), then by \cref{allaboutshapes}(iv) we have a node in \(\bmu/\bmu'\) with residue \(\overline{\theta}_{e-a}\) south of a node with residue \(\overline{\theta_{e-a} + 1}\). Therefore \((\bmu/\bmu')^{(a)}\) may consist only of disconnected row segments. 

We note the following facts:
(i) \(|\bnu| = |\bmu|\);
(ii) \(|\bmu'|  = |\bnu'|\);
(iii) \(\bmu'  \trianglerighteq \bnu'\);
(iv) \(\bnu\) is equal to the union of \(\bnu'\) with a row segment in the \(a\)th component;
(v) \((\bmu/\bmu')^{(a+1)} = \dots = (\bmu/\bmu')^{(e-1)} = \varnothing\); and
(vi) \((\bmu/\bmu')^{(a)}\) consists only of disconnected row segments.
Combining these, we see that \(\bmu  \trianglerighteq \bnu\).

Finally, note that by \cref{allaboutshapes} the diagram \((\bnu/\bnu')^\theta_{\brho}\) is a ribbon tiled by \((\theta,a)\)-ribbons (each tile of which consists of more than one row), in such a way that the northeasternmost node (of residue \(\overline{\theta_{e-a}}\)) in each tile is west of the southwesternmost node (of residue \(\overline{\theta_{e-a} +1}\)) in the next tile to the northeast.
Since the row-leading standard tableau for each tile has content sequence \(\bb^{(1)}_{(\theta,a)}\), it follows that one may interleave the row-leading standard tableaux for each tile to form a standard tableau for  \((\bnu'/\bnu)^\theta_{\brho}\) with content sequence \(\bb^{(k)}_{(\theta,a)}\).
Therefore, proceeding by an inductive claim on \((\bnu')^\theta_{\brho}\), we have that  \(\bb_\theta^{\bnu} = \bb_\theta^{\bnu'}
\bb_{(\theta,a)}^{(k)}
\) is a word in \(\bnu^\theta_{\brho}\).
\end{proof}

\section{KLR algebras}\label{cycKLRsec}

We orient the edges in the Dynkin diagram for \({\tt A}_{e-1}^{(1)}\) (Figure~\ref{fig:dynkin}), setting \(\alpha_{i} \to \alpha_{\overline{i+1}}\) for \(i \in \Z_e\).

\subsection{KLR algebra}\label{subsec:KLR}
Fix a field \(\k\) of characteristic \(p \geq 0\). Let \(\beta \in \ZZ_{\geq 0}I\), and set \(\height(\beta) = m\). As in \cite{kl09, Rouq}, the KLR (or {\em quiver Hecke}) algebra (of type \({\tt A}_{e-1}^{(1)}\)) is the unital \(\Z\)-graded \(\k\)-algebra \(R_\beta\) generated by
\begin{align*}
\{1_\bi \mid \bi \in I^\beta\} \cup \{y_1, \dots, y_m\} \cup \{\psi_1, \dots, \psi_{m-1}\}, 
\end{align*}
subject to the relations:
\begin{align*}
1_{\bi} 1_{\bj} &= \delta_{\bi, \bj} 1_{\bi};
\qquad
\sum_{\bi \in I^\beta}1_{\bi} = 1;
\qquad
y_r 1_{\bi} = 1_{\bi} y_r;
\qquad
y_r y_s = y_s y_r;
\qquad
\psi_r 1_{\bi} = 1_{s_r \bi} \psi_r;
\end{align*}
\begin{align*}
\psi_r y_s = y_s \psi_r \;(\textup{if } s \neq r, r+1);
\qquad
\psi_r \psi_s = \psi_s \psi_r \;(\textup{if } |r-s|>1);
\end{align*}
\begin{align*}
\psi_r y_{r+1} 1_{\bi} = (y_r \psi_r + \delta_{i_r, i_{r+1}})1_{\bi};
\qquad
y_{r+1} \psi_r 1_{\bi} = (\psi_r y_r + \delta_{i_r, i_{r+1}})1_{\bi};
\end{align*}
\begin{align*}
\psi_r^2 1_\bi
=
\begin{cases}
0 & \textup{if }i_r = i_{r+1};\\
(y_{r} - y_{r+1})1_{\bi} & \textup{if } i_r \rightarrow i_{r+1} ;\\
(y_{r+1} - y_r)1_{\bi} & \textup{if }i_r \leftarrow i_{r+1} ;\\
-(y_{r+1} - y_r)^21_{\bi} & \textup{if } i_r \leftrightarrows i_{r+1};\\
1_\bi & \textup{otherwise};
\end{cases}
\end{align*}
\begin{align*} 
(\psi_{r+1} \psi_r \psi_{r+1} - \psi_r \psi_{r+1} \psi_r )1_\bi &= 
\begin{cases}
1_\bi & \textup{if } i_{r+2}  = i_r \rightarrow i_{r+1};\\
-1_\bi & \textup{if }  i_{r+2} = i_r \leftarrow i_{r+1};\\
(-y_r + 2y_{r+1} - y_{r+2})1_{\bi} & \textup{if }  i_{r+2} = i_r \leftrightarrows i_{r+1};\\
0  &\textup{otherwise}.
\end{cases}
\end{align*}
The \(\Z\)-grading is given by:
\begin{align*}
\deg_\Z(1_\bi) = 0;
\qquad
\deg_\Z(y_r) = 2;
\qquad
\deg_\Z(\psi_r 1_\bi) = 
\begin{cases}
-2 & \textup{if } i_r = i_{r+1};\\
1 & \textup{if }i_{r} \to i_{r+1} \textup{ or }i_{r} \leftarrow i_{r+1};\\
2 & \textup{if }i_{r} \leftrightarrows i_{r+1};\\
0 & \textup{otherwise}.
\end{cases}
\end{align*}

There exists a unique anti-isomorphism \(j_\beta\) of \(R_\beta\) which fixes the KLR generators.
For a left \(R_\beta\)-module, \(M\), we let \(^{j_\beta}M\) denote the right-module isomorphic to \(M\) as a vector space, with \(R_\beta\)-action twisted by \(j_\beta\).

To each \(\sigma \in \mathfrak{S}_m\), we may associate a choice of distinguished reduced expression \(\sigma = s_{j_1} \dots s_{j_k}\) for some \(j_1, \dots, j_k \in [1,m-1]\), and define an associated element \(\psi_\sigma := \psi_{j_1} \dots \psi_{j_k} \in R_\beta\). For a standard tableau \({\tt t}\) we will write \(\psi^{\tt t}:=\psi_{w^{\tt t}}\), where \(w^{\tt t}\) is as defined in \S\ref{subsubsec:Snaction}.

There is an embedding
\[
\iota_{\omega,\beta}: R_\omega \otimes R_\beta \longhookrightarrow 1_{\omega, \beta} R_{\omega + \beta} 1_{\omega, \beta} \subseteq R_{\omega + \beta},
\]
where
\[
1_{\omega, \beta} := \sum_{\bi \in I^\omega, \bj \in I^\beta} 1_{\bi \bj}.
\]

\subsubsection{Special choices of reduced expressions} For \(r,b \in \ZZ_{\geq 0}\), we now define a special choice of reduced expressions for elements of \(\mathfrak{S}_{r+b}\), as follows. First, fix reduced expression choices for \(\mathfrak{S}_r, \mathfrak{S}_b\). Then fix reduced espressions for \({}^{r,b}\mathscr{D}\), the set of minimal right coset representatives for \((\mathfrak{S}_r \times \mathfrak{S}_b) \backslash \mathfrak{S}_{r+b}\). Now, for every \(\sigma \in \mathfrak{S}_{r,b}\), we may write \(\sigma = (\sigma_2, \sigma_3)\sigma_1\), where \(\ell(\sigma) = \ell(\sigma_1) + \ell(\sigma_2) + \ell(\sigma_3)\), and 
\begin{align*}
\sigma_1 = s_{j_1} \dots s_{j_{k}} \in {}^{r,b}\mathscr{D}, \qquad \sigma_2 = s_{j'_1} \dots s_{j'_{k'}} \in \mathfrak{S}_r, \qquad \sigma_3 = s_{j''_1} \dots s_{j''_{k''}} \in \mathfrak{S}_b,
\end{align*}
with pre-chosen fixed reduced expressions as shown. Then fix the following associated reduced expression for \(\sigma\):
\begin{align*}
\sigma = (s_{j'_1} \dots s_{j'_{k'}})(s_{j''_1+r} \dots s_{j''_{k''}+r})(s_{j_1} \dots s_{j_{k}} ).
\end{align*}
We say that reduced expressions for \(\mathfrak{S}_r, \mathfrak{S}_b, \mathfrak{S}_{r+b}\) chosen in this fashion are {\em \((r,b)\)-adapted}.

Assume that \((r,b)\)-adapted reduced expressions have been chosen, and let \(\omega, \beta \in \ZZ_{\geq 0}I\) with \(\height(\omega) = r\), \(\height(\beta) = b\). It follows then that
\begin{align}\label{rbpsi}
\psi^{{\tt T}{\tt t}} = \iota_{\omega,\beta} (\psi^{\tt T} \otimes \psi^{\tt t}) \psi^{{\tt T}^{\brho} {\tt t}^{\blam/\brho}} \in R_{\omega + \beta}.
\end{align}
for all \(\brho \in \Lambda_+^{\bkap}(\omega)\), \(\blam/\brho \in \Lambda^{\bkap}_{+/\brho}(\beta)\), \({\tt T} \in \Std(\brho)\), \({\tt t} \in \Std(\blam/\brho)\),
recalling that ${\tt T}^{\brho}$ and ${\tt t}^{\blam/\brho}$ are the row-leading tableaux of ${\brho}$ and ${\blam/\brho}$, respectively.

\begin{Theorem}\cite[Theorem 2.5]{kl09}, \cite[Theorem 3.7]{Rouq}\label{klrbasis}
For \(\beta \in \ZZ_{\geq0}I\) of height \(m\), the following is a \(\k\)-basis for \(R_\beta\):
\begin{align*}
\{\psi_\sigma y_1^{c_1} \dots y_m^{c_m} 1_\bi \mid \sigma \in \mathfrak{S}_m, c_1, \dots, c_m \in \ZZ_{\geq 0}, \bi \in I^\beta\}.
\end{align*}
\end{Theorem}

\subsection{Cyclotomic KLR algebras}\label{subsec:cycKLR}
Let \(\Lambda = (L_0, L_1, \dots, L_{e-1}) \in \ZZ_{\geq 0}^{e}\). We define the {\em cyclotomic KLR algebra} \(R^\Lambda_\beta\) to be the quotient of \(R_\beta\) by the two-sided ideal generated by the elements
\begin{align}\label{cycdef}
\{ y_1^{L_{i_1}}1_\bi \mid \bi \in I^\beta\}.
\end{align}

\subsection{Skew Specht modules}\label{Spechtsec} 
All \(R_\beta\)-modules that we work with in this paper are graded, and all module homomorphisms will be (not necessarily degree 0) graded maps of graded modules. As we are not primarily concerned with grading shifts or graded decomposition numbers in this paper, we use the notation \([M:L] \in \ZZ_{\geq 0}\) to indicate the ungraded multiplicity of a simple module \(L\) as a factor of a module \(M\). We write \(M \cong N\) to indicate that two \(R_\beta\)-modules are isomorphic, and \(M \approx N\) to indicate that they are isomorphic up to some grading shift.

Let \(\beta \in \ZZ_{\geq 0}I\), and let \(\btau \in \Lambda^\ell(\beta)\) be a skew diagram. We define the {\em (row) skew Specht module} \(\zS^{\btau}\) to be the graded \(R_\beta\)-module generated by the vector \(v^{\btau}\) in degree zero, and subject to the following relations.
\begin{enumerate}
\item \(1_{\bi} v^{\btau} = \delta_{\bi, \bi^{\btau}} v^{\btau}\) for all \(\bi \in I^\beta\);
\item \(y_r v^{\btau} = 0 \) for all \(r \in [1,\textup{ht}(\beta)]\);
\item \(\psi_r v^{\btau} = 0 \) for all \(r \in [1, \textup{ht}(\beta)-1]\) such that \({\tt E}{\tt t}^{\btau}(r) = {\tt t}^{\btau}(r+1)\);
\item \(g^{u} v^{\btau} = 0\) for all \(u \in \btau\) such that \({\tt S}u \in \btau\).
\end{enumerate}
The element \(g^u \in R_\beta\) above is the {\em Garnir element}, see \cite[\S5]{KMR}, \cite[\S4]{muthskew}. The description of this element is rather technical, and not needed for the purposes of the paper, so we refer the reader to the above papers for the definition. The next lemma follows immediately from the defining relations of skew Specht modules.

\begin{Lemma}\label{simSpechts} Recalling \cref{similaritysec}, 
let \(\btau \in \Lambda^{\ell}(\omega)\) and \(\btau' \in \Lambda^{\ell'}(\omega)\) be such that \(\btau \sim \btau'\). Then we have an isomorphism of the associated skew Specht modules, \(\zS^{\btau} \cong \zS^{\btau'}\).
\end{Lemma}

Now we make a remark about grading shifts.
Assume \(\btau \sim \blam/\bmu \in \Lambda^{\bkap}_{+/+}\) for some multicharge \(\bkap\). 
Then we set
\begin{align*}
S^{\blam/\bmu} := \zS^{\btau} \langle \textup{deg}_{\blam/\bmu}({\tt t}^{\blam / \bmu})\rangle,
\end{align*} 
where \(\deg_{\blam/\bmu}\) is a combinatorial degree function defined in \cite{bkw11} for Specht modules, and \cite[\S2.6]{muthskew} for skew Specht modules. Then the generator \(v^{\blam/\bmu}\) of \(S^{\blam/\bmu}\) is in degree \(\textup{deg}_{\blam/\bmu}({\tt t}^{\blam / \bmu})\).

\begin{Remark}
Note that \(S^{\blam /\bmu}\) and \(\zS^{\btau}\) are isomorphic up to grading shift, and the choice of a skew multipartition which realizes \(\btau\) defines a choice of grading shift for \(\zS^{\btau}\). We remark that \(S^{\blam/\bmu}\), defined for a {\em skew multipartition} \(\blam/\bmu\), is exactly the skew Specht module defined in \cite{muthskew}, while our degree-shifted \(\zS^{\btau}\) allows us to define skew Specht modules for any {\em skew diagram} \(\btau\) without reference to an associated skew multipartition, which we make use of in \cref{cuspsyssec,sec:simples}. For any skew diagram \(\btau\), there always exists some multipartition \(\blam/\bmu \in \Lambda^{\bkap}_{+/+}\) such that \(\btau \sim \blam/\bmu\), so standard results on skew Specht modules (like the following basis result) may freely be applied to \(\zS^{\btau}\), modulo some grading shift.
\end{Remark}

\begin{Proposition}\label{Spechtbasis}
\cite[Corollary 6.24]{KMR}
\cite[Theorem 5.1]{muthskew}
Let \(\beta \in \ZZ_{\geq 0}I\), \(\blam/\bmu \in \Lambda^{\bkap}_{+/+}(\beta)\). Then
\(S^{\blam/\bmu}\) has a homogeneous \(\k\)-basis 
\begin{align*}
\{v^{\tt t} := \psi^{\tt t} v^{\blam/\bmu}  = 1_{\bi^{\tt t}}  \psi^{\tt t} v^{\blam/\bmu}  \mid {\tt t} \in \Std(\blam/\bmu)\},
\quad
\text{where}\
\deg_{\ZZ}(v^{\tt t}) = \deg_{\blam/\bmu}({\tt t}).
\end{align*}
\end{Proposition}

Recall the notion of \(\bkap\)-separability from \cref{kapsepdef}. The following decomposition lemma will be important.

\begin{Lemma}\label{decompSkew}
Let \(\omega, \beta \in \ZZ_{\geq 0}I\), with \(\height(\omega) = r, \height(\beta) = b\). Assume that \((\omega, \beta)\) is \(\bkap\)-separated, and let \(\brho \in \Lambda_+^{\bkap}(\omega)\), \(\blam/\brho \in \Lambda^{\bkap}_{+/\brho}(\beta)\). Assume that \((r,b)\)-adapted reduced expressions have been used to define \(\psi_\sigma\), for \(\sigma \in \mathfrak{S}_r, \mathfrak{S}_b, \mathfrak{S}_{r+b}\). 
Then there is an isomorphism of \(\Z\)-graded left (\(R_\omega \otimes R_\beta\))-modules given by:
\begin{align*}
G^{\blam/\brho}: S^{\brho} \boxtimes S^{\blam/\brho}  \bijection  1_{\omega, \beta} S^{\blam},
\qquad
v^{\tt T} \boxtimes v^{\tt t} \longmapsto v^{\tt Tt},
\end{align*}
where $\tt T\in\Std(\brho)$, $\tt t\in\Std(\bla/\brho)$ and $\tt T \tt t\in\Std(\bla)$.
\end{Lemma}
\begin{proof}
By \cite[Theorem 5.13]{muthskew}, \(1_{\omega, \beta}S^{\blam}\) has a multiplicity-free filtration by all (\(R_\omega \otimes R_\beta\))-modules of the form \(S^{\bnu} \boxtimes S^{\blam/\bnu}\), where \(\bnu \in \Lambda_+^{\bkap}(\omega)\). Then \(S^{\brho} \boxtimes S^{\blam/\brho}\) appears once in this filtration. If \(\blam \supseteq \brho' \in \Lambda^{\bkap}_+(\omega)\) is another multipartition with \(\brho' \neq \brho\), we would have \(\cont(\brho / (\brho \cap \brho')) = \cont((\brho \cup \brho')/\brho)\), a contradiction of the fact that \((\omega, \beta)\) is \(\bkap\)-separable. Hence there are no other layers in the filtration, which yields the isomorphism.
The explicit description of the map follows from \cite[\S5.2]{muthskew}, \cref{Spechtbasis}, and (\ref{rbpsi}).
\end{proof}

\subsubsection{Cellular algebras}\label{celldef} Following \cite{GL98, hm10}, let \(A\) be a \(\Z\)-graded \(\k\)-algebra that is free of finite rank as a \(\k\)-module. A {\em graded cell datum} for \(A\) is an ordered tuple \((\mathcal{P}, T, \textup{inv}, C, \deg_\Z)\), where \((\mathcal{P}, \trianglerighteq)\) is a poset, \(T(\lambda)\) is a finite set for \(\lambda \in \mathcal{P}\), \(\textup{inv}:A \xrightarrow{\sim} A^\op\) is a \(\k\)-algebra isomorphism,
\begin{align*}
C: \coprod_{\lambda \in \mathcal{P}} T(\lambda) \times T(\lambda) \longrightarrow A; \qquad \deg_\Z: \coprod_{\lambda \in \mathcal{P}} T(\lambda) \longrightarrow \Z
\end{align*}
are maps (with \(C\) injective) such that:
\begin{enumerate}
\item[(C1)] The set \(\{ c^{\lambda}_{{\tt s}, {\tt t}} := C({\tt s}, {\tt t})\mid {\tt s}, {\tt t} \in T(\lambda) \textup{ for } \lambda \in \mathcal{P}\}\) is a homogeneous \(\k\)-basis for \(A\).
\item[(C2)] For all \(\lambda \in \mathcal{P}\), \({\tt s}, {\tt t} \in T(\lambda)\), and \(a \in A\), there  exist scalars \(l^\lambda_{{\tt x}, {\tt s}}(a)\), such that
\begin{align*}
ac^\lambda_{{\tt s}, {\tt t}}  = \sum_{{\tt x} \in T(\lambda)} l^\lambda_{{\tt x}, {\tt s}}(a) c^\lambda_{{\tt x}, {\tt t}} \quad (\textup{mod } A^{\triangleright \lambda}),
\end{align*}
for all \({\tt t} \in T(\lambda)\), where
\(
A^{\triangleright \lambda} = \k \{ c^\mu_{{\tt a}, {\tt b}} \mid \mu \triangleright \lambda, {\tt a}, {\tt b} \in T(\mu)\}\) is a two-sided ideal in \(A\).
\item[(C3)] For all \(\lambda \in \mathcal{P}\), \({\tt s}, {\tt t} \in T(\lambda)\), we have \(\textup{inv}(c^\lambda_{{\tt s}, {\tt t}}) = c^\lambda_{{\tt t},{\tt s}}\quad (\textup{mod } A^{\triangleright \lambda})\).
\item[(CG)] For all \(\lambda \in \mathcal{P}\), \({\tt s}, {\tt t} \in T(\lambda)\), we have \(\deg_\Z(c^\lambda_{{\tt s}, {\tt t}}) = \deg_\Z({\tt s}) + \deg_\Z({\tt t})\).
\end{enumerate}
If such a graded cell datum exists for \(A\), we say that \(A\) is a {\em graded cellular algebra}, with {\em graded cellular basis} \(\{c_{{\tt s}, {\tt t}}^\lambda\}\).

\begin{Remark}
Our (C3) is slightly weaker than the requirement in \cite{GL98,hm10} that \(\textup{inv}(c^\lambda_{{\tt s}, {\tt t}}) = c^\lambda_{{\tt t},{\tt s}}\) on the nose. We make use of this weakened axiom, as in  \cite[\S2.2]{GG11}, where it is explained that the results of \cite{GL98} nonetheless remain valid in this setting.
\end{Remark}

\begin{Theorem}\label{cellHM}
\cite[Theorem~5.4 and Corollary~5.11]{hm10}
The cyclotomic KLR algebra \(R^\Lambda_\beta\) is a graded cellular algebra, with graded cell datum given by:
\begin{align*}
\mathcal{P} &= (\Lambda_+^{\bkap}(\beta), \trianglerighteq^D);\qquad
T(\blam) = \Std(\blam); \qquad
\textup{inv}= j_\beta,
\end{align*}
\begin{align*}
C({\tt S}, {\tt T}) = c_{{\tt S}, {\tt T}}^{\blam} :=  1_{\bi^{\tt S}} \psi_{{\tt S}} c^{\blam} j_\beta(\psi_{{\tt T}}) 1_{\bi^{\tt T}};
\qquad
\deg_\Z({\tt S}) = \deg_{\blam}({\tt S}),
\end{align*}
for certain explicit monomials \(  c^{\blam} \in \k[y_1, \dots, y_{\textup{ht}(\beta)}] \in R^{\Lambda}_{\beta}\).
The simple \(R^\Lambda_\beta\)-modules, denoted \(D^{\bmu}\), with respect to this cellular structure are indexed by \emph{Kleshchev multipartitions}, and we denote the set of all such multipartitions by \(\Lambda^{\bkap}_+(\omega + \beta)'\).
\end{Theorem}

Again, we will not need the specific definition of the \( c^{\blam}\)s, so we refer the reader to \cite{hm10} for a complete description.
Importantly, the {\em cell modules} corresponding to the cellular algebra structure are exactly the graded Specht modules of \S\ref{Spechtsec}, which yields the following pertinent fact.

\begin{Lemma}\label{transSpecht}
For all \(\blam \in \Lambda_+^{\bkap}(\nu)\), there is an isomorphism of \((R^{\Lambda}_\nu, R^\Lambda_\nu)\)-bimodules:
\begin{align*}
F^{\blam}: (R^{\Lambda}_\nu)^{\trianglerighteq^D\blam}/ (R^{\Lambda}_\nu)^{\triangleright^D\blam}
\bijection
S^{\blam} \otimes {}^{j_\nu}S^{\blam},
\qquad
c^{\blam}_{{\tt S}, {\tt T}} \longmapsto v^{\tt S} \otimes v^{\tt T}.
\end{align*}
\end{Lemma}

\subsection{Core cyclotomic Hecke algebras}
\begin{Proposition}\label{coresimple}
Let \(\omega \in \ZZ_{\geq 0}I\), and let \(\brho \in \Lambda^{\bkap}_+(\omega)\) be a \(\bkap\)-core.  The cyclotomic KLR algebra \(R^\Lambda_{\omega}\) is a simple algebra, and \(R^\Lambda_{\omega} \cong \End_{\k}(S^{\brho})\) as \(\ZZ\)-graded \(\k\)-algebras.
\end{Proposition}
\begin{proof}
As \(\brho\) is a \(\bkap\)-core, \(\Lambda_+^{\bkap}(\omega) = \{\brho\}\), and thus \(R^\Lambda_{\omega}\) is a cellular algebra with one cell by \cref{cellHM}. Thus, as noted in the proof of \cite[Proposition 4.9]{evseevrock}, we have that \(R^\Lambda_{\omega}\) is a split simple algebra, and the cell module \(S^{\brho}\) is the unique simple \(R^\Lambda_{\omega}\)-module, completing the proof. 
\end{proof}

\section{Skew cyclotomic KLR algebras}\label{skewcycsec1}

In this section we introduce $\omega$-skew cyclotomic KLR algebras, arising as quotients of affine KLR algebras, and explore their properties, in particular their connection with RoCK blocks.
	
\subsection{Truncations and cyclotomic quotient maps}

We continue with \(\omega, \beta \in \ZZ_{\geq 0}I\) and a fixed multicharge \(\bkap\) of level \(\ell\).

Define the map \(\pi^\Lambda_{\omega, \beta}: R_\omega \otimes R_\beta \to 1_{\omega, \beta} R_{\omega + \beta}^\Lambda 1_{\omega, \beta}\) as the composition of the following embedding and projection maps
\begin{align*}
\pi^\Lambda_{\omega, \beta}:R_\omega \otimes R_\beta \xrightarrow{ \iota_{\omega, \beta}} 1_{\omega, \beta} R_{\omega + \beta} 1_{\omega, \beta} \xrightarrow{\pi^\Lambda_{\omega + \beta}} 1_{\omega, \beta} R_{\omega + \beta}^\Lambda 1_{\omega, \beta},
\end{align*}
and define the algebra
\begin{align*}
R_{\omega, \beta}^{\Lambda} := (R_\omega \otimes R_\beta) / \ker(\pi^\Lambda_{\omega + \beta} \circ \iota_{\omega, \beta})
\end{align*}
to be the homomorphic image of this composition, with \(\bar{\pi}^{\Lambda}_{\omega, \beta}: R^{\Lambda}_{\omega, \beta} \hookrightarrow 1_{\omega, \beta} R^{\Lambda}_{\omega + \beta} 1_{\omega, \beta}\) being the canonical inclusion map. 
By consideration of the cyclotomic relation (\ref{cycdef}), we have that
\begin{align*}
R_\omega \xrightarrow{l_{\omega, \beta}} R_\omega \otimes R_\beta \xrightarrow{\pi^\Lambda_{\omega, \beta} }1_{\omega, \beta} R_{\omega + \beta}^\Lambda 1_{\omega, \beta}
\end{align*}
factors through the cyclotomic quotient \(R^\Lambda_\omega\),
where $l_{\omega, \beta}$ is the natural \emph{left} embedding of $R_\omega$ in $R_\omega \otimes R_\beta$,
yielding maps
\begin{align*}
R_\omega \xrightarrow{\pi^\Lambda_{\omega}} R_\omega^\Lambda \xrightarrow{\eta^\Lambda_\omega}1_{\omega, \beta} R_{\omega + \beta}^\Lambda 1_{\omega, \beta}
\end{align*}
such that \(\pi^\Lambda_{\omega, \beta} \circ l_{\omega, \beta} = \eta^\Lambda_\omega  \circ \pi^\Lambda_\omega\). 
We also have maps
\begin{align}\label{skewmaps}
R_\beta \xrightarrow{r_{\omega, \beta}} R_\omega \otimes R_\beta \xrightarrow{\pi^\Lambda_{\omega, \beta}} 1_{\omega, \beta} R_{\omega + \beta}^\Lambda 1_{\omega, \beta},
\end{align}
where $r_{\omega, \beta}$ is the natural \emph{right} embedding of $R_\beta$ in $R_\omega \otimes R_\beta$.

\begin{Definition}
The \emph{\(\omega\)-skew cyclotomic KLR algebra} is defined to be the quotient
\begin{align*}
	R_{\beta}^{\Lambda/\omega} := R_\beta / \ker( \pi^\Lambda_{\omega, \beta} \circ r_{\omega, \beta}).
\end{align*}
\end{Definition}

Note that this is the homomorphic image of $R_{\beta}$ under the composition of the maps in~(\ref{skewmaps}).

\subsection{Properties of truncations and cyclotomic quotient maps}\label{subsec:truncations}

Note that the composition \(\pi^\Lambda_{\omega, \beta} \circ r_{\omega, \beta}\) of the maps in~(\ref{skewmaps}) factors through \(R^{\Lambda / \omega}_{\beta}\), yielding maps 

\begin{align*}
R_\beta \xrightarrow{ \pi^{\Lambda / \omega}_\beta} R^{\Lambda / \omega}_\beta \xrightarrow{ \eta^{\Lambda / \omega}_\beta} 1_{\omega, \beta} R_{\omega + \beta}^\Lambda 1_{\omega, \beta},
\end{align*}
where \(\pi^{\Lambda/ \omega}_\beta\) is the canonical quotient map, and \(\eta^{\Lambda / \omega}_\beta\) is an injection such that \(\pi^\Lambda_{\omega, \beta} \circ r_{\omega, \beta} = \eta^{\Lambda / \omega}_\beta \circ \pi^{\Lambda/ \omega}_\beta\). Putting everything together, we have the following commutative diagram of \(\Z\)-graded \(\k\)-algebras:

\begin{equation}\label{comdiag1}
\begin{tikzcd}
	& R_\omega && R_\omega^\Lambda \\
	&& 1_{\omega, \beta}R_{\omega + \beta}1_{\omega, \beta} \\
	R_\omega \otimes R_\beta &&&& 1_{\omega,\beta} R^\Lambda_{\omega + \beta} 1_{\omega, \beta} \\
	& R_\beta && R_\beta^{\Lambda / \omega}
	\arrow["\pi^\Lambda_\omega", from=1-2, to=1-4]
	\arrow["\pi^\Lambda_{\omega + \beta}", from=2-3, to=3-5]
	\arrow["\eta_{\omega}^\Lambda", from=1-4, to=3-5]
	\arrow["\iota_{\omega, \beta}", from=3-1, to=2-3]
	\arrow["\pi^\Lambda_{\omega, \beta}"', from=3-1, to=3-5]
	\arrow["\pi^{\Lambda / \omega}_\beta"', from=4-2, to=4-4]
	\arrow["\eta^{\Lambda / \omega}_\beta"', from=4-4, to=3-5]
	\arrow["r_{\omega, \beta}", from=4-2, to=3-1]
	\arrow["l_{\omega, \beta}"', from=1-2, to=3-1]
\end{tikzcd}
\end{equation}

Note that since the subalgebras \(\eta^{\Lambda}_\omega(R^\Lambda_\omega)\) and
\(\eta_\beta^{\Lambda / \omega}(R^{\Lambda / \omega}_\beta)\) are mutually centralizing in \(1_{\omega, \beta} R_{\omega + \beta}^\Lambda 1_{\omega, \beta}\), we have 
an algebra map
\begin{align*}
\eta^{\Lambda}_{\omega, \beta}: R^\Lambda_\omega \otimes R^{\Lambda / \omega}_{\beta} \longrightarrow 1_{\omega,\beta} R^\Lambda_{\omega + \beta} 1_{\omega, \beta},
\qquad
x_1 \otimes x_2 \longmapsto 1_{\omega, \beta}(x_1 \otimes x_2)1_{\omega, \beta} = \eta_\omega^\Lambda(x_1) \eta_\beta^{\Lambda / \omega}(x_2),
\end{align*} 
and a quotient map
\begin{align*}
q^{\Lambda}_{\omega, \beta} : R_\omega \otimes R_\beta \longrightarrow R^\Lambda_\omega \otimes R^{\Lambda / \omega}_\beta
\end{align*}
such that the following diagram commutes:
\begin{equation}\label{comdiag2}
\begin{tikzcd}
	&  R_\omega \otimes R_\beta\\
	 R^\Lambda_\omega \otimes R^{\Lambda / \omega}_\beta &&  1_{\omega,\beta} R^\Lambda_{\omega + \beta} 1_{\omega, \beta}
	\arrow["q^{\Lambda}_{\omega, \beta}"', from=1-2, to=2-1]
	\arrow["\eta^{\Lambda}_{\omega, \beta}"', from=2-1, to=2-3]
	\arrow["\pi^\Lambda_{\omega, \beta}", from=1-2, to=2-3]
\end{tikzcd}
\end{equation}

We now investigate properties of the maps in (\ref{comdiag1}) and (\ref{comdiag2}).

\begin{Lemma}\label{idemcutlem}
Assume that \(\omega, \beta \in \ZZ_{\geq 0}I\) are such that the pair \((\omega, \beta)\) is \(\bkap\)-separable. Then 
the map \(\pi_{\omega, \beta}^\Lambda\) is a surjection.
\end{Lemma}
\begin{proof}
Write \(r = \height(\omega)\) and \(b = \textup{ht}(\beta)\).
For any \(\sigma \in \mathfrak{S}_m\), we may write \(\sigma = w X u\), where \(w = (w_1, w_2), u = (u_1, u_2) \in \mathfrak{S}_r \times \mathfrak{S}_b \subseteq \mathfrak{S}_m\), and where \(X\) is a block transposition of the form:
\begin{align*}
X = (c, r+1) (c+1, r+2) \dots (r, 2r+ 1-c)
\end{align*}
for some \(c \in [1,r+1]\), noting that such \(X\) are representatives for the double cosets \(\mathfrak{S}_r \times \mathfrak{S}_b \backslash \mathfrak{S}_m / \mathfrak{S}_r \times \mathfrak{S}_b\). In choosing distinguished elements \(\psi_\sigma \in R_{\omega + \beta}\) for \(\sigma \in \mathfrak{S}_m\), we may assume that we have chosen such that \(\psi_\sigma = \iota_{\omega,\beta}(\psi_{w_1} \otimes \psi_{w_2}) \psi_X \iota_{\omega,\beta}(\psi_{u_1} \otimes \psi_{u_2})\) in accordance with this decomposition.
Then by \cref{klrbasis}, we have that \(1_{\omega , \beta} R^\Lambda_{\omega + \beta} 1_{\omega , \beta}\) is spanned by elements of the form
\begin{align}\label{spanR}
1_{\omega, \beta} \iota_{\omega,\beta}(\psi_{w_1} \otimes \psi_{w_2}) 1_{\bi \bk \bj \bm}\psi_X 1_{\bi \bj \bk \bm} \iota_{\omega,\beta}(\psi_{u_1} \otimes \psi_{u_2}) y_1^{f_1} \dots y_m^{f_m} 1_{\omega, \beta},
\end{align}
where \(u_1, u_2, w_1,w_2, X\) are as above, \(f_1, \dots, f_m \in \ZZ_{\geq 0}\), \(\bi \in I^{c-1}\), \(\bj, \bk \in I^{r-c+1}\), \(\bm \in I^{b-c+1}\), with 
 \(\bi \bj, \bi \bk \in I^{\omega}\), \(\bk \bm, \bj \bm \in I^{\beta}\). 
It is most convenient to view elements of the form (\ref{spanR}) using the diagrammatic presentation of \(R_{\omega + \beta}\) (see \cite{kl09}) where they look like the following.
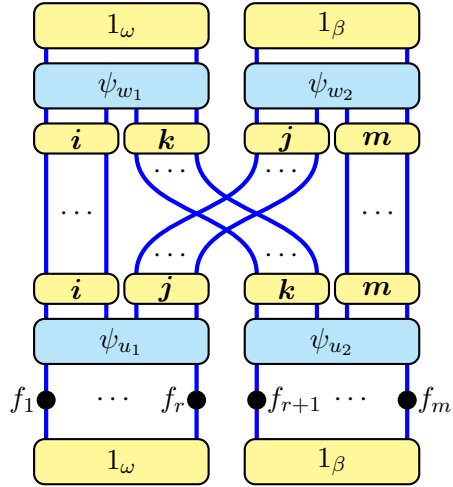
\begin{align*}
\hackcenter{
{}
}
\hackcenter{
\begin{tikzpicture}[scale=.8]
  \draw[ultra thick,blue] (-0.5,0) .. controls ++(0,1) and ++(0,-1) .. (1.5,2);
    \draw[ultra thick,blue] (-1.5,0) .. controls ++(0,1) and ++(0,-1) .. (0.5,2);
      \draw[ultra thick,blue] (0.5,0) .. controls ++(0,1) and ++(0,-1) .. (-1.5,2);
    \draw[ultra thick,blue] (1.5,0) .. controls ++(0,1) and ++(0,-1) .. (-0.5,2);
     \draw[ultra thick,blue] (-3,-3)--(-3,0);
  \draw[ultra thick,blue] (-0.5,-3)--(-0.5,0);
    \draw[ultra thick,blue] (-2,-1)--(-2,0);
       \draw[ultra thick,blue] (-1.5,-1)--(-1.5,0);
   \draw[ultra thick,blue] (3,-3)--(3,0);
  \draw[ultra thick,blue] (0.5,-3)--(0.5,0);
      \draw[ultra thick,blue] (2,-1)--(2,0);
       \draw[ultra thick,blue] (1.5,-1)--(1.5,0);
           \draw[ultra thick,blue] (-3,2)--(-3,4);
           \draw[ultra thick,blue] (-2,2)--(-2,3);
                 \draw[ultra thick,blue] (-1.5,2)--(-1.5,3);
           \draw[ultra thick,blue] (-0.5,2)--(-0.5,4);
             \draw[ultra thick,blue] (3,2)--(3,4);
           \draw[ultra thick,blue] (2,2)--(2,3);
                 \draw[ultra thick,blue] (1.5,2)--(1.5,3);
           \draw[ultra thick,blue] (0.5,2)--(0.5,4);
     \draw[ultra thick,blue] (-3,0)--(-3,2);
      \draw[ultra thick,blue] (-2,0)--(-2,2);
      \draw[ultra thick,blue] (2,0)--(2,2);
     \draw[ultra thick,blue] (3,0)--(3,2);
     \draw[thick, rounded corners, fill=yellow!50] (-3.2, -0.5) rectangle (-1.8, 0) {};
       \draw[thick, rounded corners, fill=yellow!50] (-1.7, -0.5) rectangle (-0.3, 0) {};
            \draw[thick, rounded corners, fill=yellow!50] (3.2, -0.5) rectangle (1.8, 0) {};
       \draw[thick, rounded corners, fill=yellow!50] (1.7, -0.5) rectangle (0.3, 0) {};
      \node[] at (-2.5, -0.25) {$\bi$};
      \node[] at (-1, -0.25) {$\bj$};
       \node[] at (2.5, -0.25) {$\bm$};
      \node[] at (1, -0.25) {$\bk$};
     \draw[thick, rounded corners, fill=yellow!50] (-3.2, -0.5+2.5) rectangle (-1.8, 0+2.5) {};
       \draw[thick, rounded corners, fill=yellow!50] (-1.7, -0.5+2.5) rectangle (-0.3, 0+2.5) {};
            \draw[thick, rounded corners, fill=yellow!50] (3.2, -0.5+2.5) rectangle (1.8, 0+2.5) {};
       \draw[thick, rounded corners, fill=yellow!50] (1.7, -0.5+2.5) rectangle (0.3, 0+2.5) {};
      \node[] at (-2.5, -0.25+2.5) {$\bi$};
      \node[] at (-1, -0.25+2.5) {$\bk$};
       \node[] at (2.5, -0.25+2.5) {$\bm$};
      \node[] at (1, -0.25+2.5) {$\bj$};
     \draw[thick, rounded corners, fill=cyan!25] (-3.2, -1.5) rectangle (-0.3, -0.75) {};
               \node[] at (-1.75, -1.125) {$\psi_{u_1}$};
         \draw[thick, rounded corners, fill=cyan!25] (3.2, -1.5) rectangle (0.3, -0.75) {};
           \node[] at (1.75, -1.125) {$\psi_{u_2}$};
     \draw[thick, rounded corners, fill=cyan!25] (-3.2, -1.5+4.25) rectangle (-0.3, -0.75+4.25) {};
      \node[] at (-1.75, -1.125+4.25) {$\psi_{w_1}$};
         \draw[thick, rounded corners, fill=cyan!25] (3.2, -1.5+4.25) rectangle (0.3, -0.75+4.25) {};
            \node[] at (1.75, -1.125+4.25) {$\psi_{w_2}$};
              \draw[thick, fill=black]  (-3,-2.1) circle (4pt);
               \node[left] at (-3,-2.1) {$f_1$};
                \draw[thick, fill=black]  (-0.5,-2.1) circle (4pt);
               \node[left] at (-0.5,-2.1) {$f_{r}$};
                \draw[thick, fill=black]  (0.5,-2.1) circle (4pt);
               \node[right] at (0.5,-2.1) {$f_{r+1}$};
                 \draw[thick, fill=black]  (3,-2.1) circle (4pt);
               \node[right] at  (3,-2.1) {$f_{m}$};
                   \draw[thick, rounded corners, fill=yellow!50] (-3.2, -1.5-2) rectangle (-0.3, -0.75-2) {};
               \node[] at (-1.75, -1.125-2) {$1_{\omega}$};
         \draw[thick, rounded corners, fill=yellow!50] (3.2, -1.5-2) rectangle (0.3, -0.75-2) {};
           \node[] at (1.75, -1.125-2) {$1_{\beta}$};
                   \draw[thick, rounded corners, fill=yellow!50] (-3.2, -1.5+5.25) rectangle (-0.3, -0.75+5.25) {};
               \node[] at (-1.75, -1.125+5.25) {$1_{\omega}$};
         \draw[thick, rounded corners, fill=yellow!50] (3.2, -1.5+5.25) rectangle (0.3, -0.75+5.25) {};
           \node[] at (1.75, -1.125+5.25) {$1_{\beta}$};
            \node[] at (-2.45,1) {$\cdots$};
             \node[] at (2.55,1) {$\cdots$};
             \node[] at (-0.9,0.3) {$\cdots$};
             \node[] at (0.95,0.3) {$\cdots$};
             \node[] at (-0.9,1.7) {$\cdots$};
             \node[] at (0.95,1.7) {$\cdots$};
              \node[] at (-1.85,-2.1) {$\cdots$};
              \node[] at (2.1,-2.1) {$\cdots$};
\end{tikzpicture}}
\end{align*}
Consider an element \(z\) of the form (\ref{spanR}). Then if \(0 \neq z \in 1_{\omega, \beta} R^\Lambda_{\omega + \beta} 1_{\omega, \beta}\), it must be that \(1_{\bi \bj \bk \bm}\) is nonzero in \(R^\Lambda_{\omega+ \beta}\).
By \cref{cellHM}, \(1_{\bi \bj \bk \bm}\) is nonzero only if there exists \(\blam \in \Lambda_+^{\bkap}(\omega + \beta)\) and \({\tt T} \in \Std(\blam)\) with \(\bi^{\tt T} = \bi \bj \bk \bm\). Setting
\begin{align*}
\bmu = \textup{sh}({\tt T} \downarrow_{c-1}), \qquad \brho = \textup{sh}({\tt T} \downarrow_{r}), \qquad \bnu = \textup{sh}({\tt T} \downarrow_{2 r-c + 1}),
\end{align*}
we have that \(\bmu \subseteq \brho \subseteq \bnu\) and \(\brho \in \Lambda^{\bkap}_+(\omega)\) with \(\cont(\brho/\bmu) = \cont(\bnu/\brho) = \beta'\) for some \(\beta' \subseteq \beta\). By the \(\bkap\)-separability assumption, it follows that \(\beta' = 0\) and thus \(\bmu = \brho = \bnu\), so \(X \in \mathfrak{S}_m\) is trivial. Therefore \(z \in \pi^{\Lambda}_{\omega, \beta}(R_\omega \otimes R_\beta) \), completing the proof.
\end{proof}

\begin{Corollary}\label{kcoresurj}
If \(\brho\) is a \(\bkap\)-core and \(\omega = \cont(\brho)\), then 
the map \(\pi_{\omega, \beta}^\Lambda\) is a surjection.
\end{Corollary}
\begin{proof}
Follows immediately from \cref{idemcutlem,coreksep}.
\end{proof}

\begin{Corollary}
If \(\Lambda_+^{\bkap}(\omega)\) is a \(\theta\)-RoCK core block with \(\textup{ht}(\beta) \leq e \cdot \textup{cap}_\delta^\theta(\omega, \bkap)\), then the map \(\pi_{\omega, \beta}^\Lambda\) is a surjection.
\end{Corollary}
\begin{proof}
Follows immediately from \cref{idemcutlem,sepofcoreblock}.
\end{proof}

\begin{Lemma}\label{tenssep}
If \(\brho\) is a \(\bkap\)-core and \(\omega = \cont(\brho)\), the map \(\eta^\Lambda_{\omega, \beta}\) is an isomorphism. Moreover, for any homogeneous primitive idempotent \(\varepsilon \in R^{\Lambda}_\omega\), we have a graded algebra isomorphism:
\begin{align*} 
R_\beta^{\Lambda/\omega} \bijection (\eta^\Lambda_{\omega, \beta}(\varepsilon \otimes 1_\beta)) R_{\omega + \beta}^\Lambda (\eta^\Lambda_{\omega, \beta}(\varepsilon \otimes 1_\beta)).
\end{align*}
\end{Lemma}
\begin{proof}
By \cref{kcoresurj}, we have that \(\pi^\Lambda_{\omega, \beta}\) is surjective, and thus consideration of (\ref{comdiag2}) shows that \(\eta^\Lambda_{\omega, \beta}\) is surjective. We note that \(\eta^\Lambda_\omega\) is nonzero, and thus necessarily an injection by \cref{coresimple}, and  \(\eta^\Lambda_\omega(R^\Lambda_\omega) \cong R^\Lambda_\omega \cong \END(S^{\brho})\). Therefore, by \cite[Proposition 4.10]{evseevrock}, we have \(R^\Lambda_{\omega, \beta} \cong R^\Lambda_{\omega} \otimes C_{R^\Lambda_{\omega, \beta}}( \eta_\omega^\Lambda(R^\Lambda_{\omega}) )\), where \(C_{R^\Lambda_{\omega, \beta}}( \eta_\omega^\Lambda(R^\Lambda_{\omega}) )\) denotes the centralizer of \(\eta_\omega^\Lambda(R^\Lambda_{\omega})\) in \(R^\Lambda_{\omega, \beta}\).
Noting that \(R_{\beta}^{\Lambda/\omega} \cong \eta^{\Lambda/\omega}_\beta(R_\beta^{\Lambda / \omega})  \subseteq  C_{R^\Lambda_{\omega, \beta}}(  \eta_\omega^\Lambda(R^\Lambda_{\omega})) \subseteq R^{\Lambda}_{\omega, \beta}\), there exists an injection \(R^\Lambda_{\omega} \otimes R^{\Lambda/ \omega}_{\beta} \to R^\Lambda_{\omega, \beta}\). Therefore the surjection \(\eta^\Lambda_{\omega, \beta}\) must be an isomorphism, proving the first claim. The second claim follows from \cite[Proposition 4.10]{evseevrock} as well.
\end{proof}

\begin{Conjecture}
If \(\Lambda_+^{\bkap}(\omega)\) is a \(\theta\)-RoCK core block with \(\textup{ht}(\beta) \leq e \cdot \textup{cap}_\delta^\theta(\omega, \bkap)\), then the map \(\eta_{\omega, \beta}^\Lambda\) is an isomorphism.
\end{Conjecture}

We note that it follows from \cref{idemcutlem,sepofcoreblock} that \(\eta^{\Lambda}_{\omega, \beta}\) is surjective, and it is straightforward to see that the restrictions of \(\eta_{\omega, \beta}^{\Lambda}\) to \(R^{\Lambda}_\omega \otimes \k\) and \(\k \otimes R^{\Lambda/\omega}_\beta\) are injective as well, but injectivity of \(\eta^{\Lambda}_{\omega, \beta}\) in general remains open.

\begin{Theorem}\label{wbcutthm}
If \(\omega, \beta \in \ZZ_{\geq 0}I\) are such that the pair \((\omega, \beta)\) is \(\bkap\)-separable, then:
\begin{enumerate}
\item  We have an isomorphism \( \bar{\pi}^{\Lambda}_{\omega, \beta}: R^{\Lambda}_{\omega, \beta} \xrightarrow{\sim} 1_{\omega, \beta} R^{\Lambda}_{\omega +\beta}1_{\omega, \beta} \) as graded algebras.
\item \(R^{\Lambda}_{\omega, \beta}\) is a graded cellular algebra, with graded cell datum given by:
\begin{align*}
\mathcal{P} &= (\{(\bmu, \bxi) \mid \bmu \in \Lambda_+^{\bkap}(\omega), \bxi \in \Lambda_{+/ \bmu}(\beta)\}, \triangleright^D);\qquad
T(\bmu, \bxi) = \Std(\bmu) \times \Std(\bxi); \qquad
\textup{inv}= j_\beta,
\end{align*}
\begin{align*}
C(({\tt S}, {\tt s}), ({\tt T},{\tt t})) = ( \bar{\pi}^{\Lambda}_{\omega, \beta})^{-1}(c_{{\tt Ss}, {\tt Tt}}^{\bmu \sqcup \bxi});
\qquad
\deg_\Z({\tt S},{\tt s}) = \deg_{\bmu \sqcup \bxi}({\tt Ss}),
\end{align*}
where the the ordering \(\triangleright^D\) on pairs \((\bmu, \bxi)\) is induced by the usual dominance order applied to the multipartitions \(\bmu \sqcup \bxi \in \Lambda_+^{\bkap}(\omega + \beta) \).
\item The cell modules for this datum are isomorphic to tensor products of Specht (\(R_\omega \otimes R_\beta\))-modules:
\begin{align*}
\{ S^{\bmu} \boxtimes S^{\bxi}\mid \bmu \in \Lambda_+^{\bkap}(\omega), \bxi \in \Lambda^{\bkap}_{+/ \bmu}(\beta)\}.
\end{align*}
\item There is an exact functor of \(\ZZ\)-graded categories
\begin{align*}
\mathcal{U}: R^{\Lambda}_{\omega + \beta}\textup{-mod} \longrightarrow R^{\Lambda}_{\omega, \beta}\textup{-mod},
\qquad
M \longmapsto 1_{\omega, \beta}M,
\end{align*}
where we identify the action of \(R^{\Lambda}_{\omega, \beta}\) with the action of \(1_{\omega, \beta} R^{\Lambda}_{\omega +\beta}1_{\omega, \beta}\) action via \({\bar{\pi}^{\Lambda}_{\omega, \beta}}\), such that
\begin{align*}
\mathcal{U}S^{\blam} \cong 
\begin{cases}
S^{\bmu} \boxtimes S^{\blam/\bmu} & \textup{if there exists } \blam \supseteq \bmu \in \Lambda_+^{\bkap}(\omega);\\
0 & \textup{otherwise,}
\end{cases}
\end{align*}
and moreover
\begin{align*}
\{ \mathcal{U} D^{\bmu \sqcup \bxi} \cong \textup{hd}(S^{\bmu} \boxtimes S^{\bxi} )  \mid \bmu \in \Lambda_+^{\bkap}(\omega), \bxi \in \Lambda_{+/\bmu}^{\bkap}(\beta), \bmu \sqcup \bxi \in \Lambda^{\bkap}_+(\omega + \beta)', 1_{\omega, \beta}D^{\bmu \sqcup \bxi} \neq 0\}
\end{align*}
 is an irredundant set of simple \(R^{\Lambda}_{\omega,\beta}\)-modules, up to grading shift.
In other words, \(\mathcal{U}\) sends cell modules to cell modules (or zero), and simple modules to simple modules (or zero).
\end{enumerate}
\end{Theorem}
\begin{proof}
Part (i) follows from \cref{idemcutlem} since \(\pi_{\omega, \beta}^{\Lambda} \) is injective by definition. 

For part (ii), we note that for any \(c^{\blam}_{{\tt X}, {\tt Y}}\) in the cellular basis for \(R^{\Lambda}_{\omega + \beta}\) in \cref{cellHM}, we have
that \(
1_{\omega, \beta} c^{\blam}_{{\tt X}, {\tt Y}} 1_{\omega, \beta} = c^{\bmu \sqcup \bxi}_{{\tt Ss}, {\tt Yy}}
\)
for some \(\bmu \in \Lambda^{\bkap}_{+}(\omega), \bxi \in \Lambda^{\bkap}_{+/\bmu}(\beta)\),  \({\tt S}, {\tt T} \in \Std(\bmu)\), \({\tt s}, {\tt t} \in \Std(\bxi)\) with \(\blam = \bmu \sqcup \bxi\), \({\tt X} = {\tt Ss}\), \({\tt Y}= {\tt Tt}\) if such \(\bmu, \bxi, {\tt S}, {\tt T}, {\tt s}, {\tt t}\) exist, and \(1_{\omega, \beta} c^{\blam}_{{\tt X}, {\tt Y}} 1_{\omega, \beta}  = 0\) otherwise. Moreover, if \(\blam = \bmu \sqcup \bxi\) as above, then such \(\bmu, \bxi\) are unique due to the fact that \((\omega, \beta)\) is \(\bkap\)-separated, by the same argument as in the proof of \cref{decompSkew}.

Thus, retaining only the nonzero basis elements and corresponding cell data in the idempotent truncation by \(1_{\omega, \beta}\) it immediately follows that the graded cellular datum for \(R_{\omega + \beta}\) restricts to a graded cellular datum for the subalgebra \(1_{\omega, \beta} R_{\omega + \beta} 1_{\omega, \beta}\) (see for instance \cite[\S2.3]{Spence}). Translating through the isomorphism \(\bar{\pi}^{\Lambda}_{\omega, \beta}\) yields the graded cellular datum for \(R^{\Lambda}_{\omega, \beta}\) in (ii).

Part (iii) follows immediately from \cref{decompSkew}.

For (iv), note that \(\mathcal{U}\) is the composition of an idempotent truncation functor and an isomorphism of categories:
\begin{align*}
\mathcal{U} : R^{\Lambda}_{\omega + \beta}\textup{-mod} \longrightarrow 1_{\omega, \beta}R^{\Lambda}_{\omega + \beta}1_{\omega, \beta} \textup{-mod} \bijection R^{\Lambda}_{\omega, \beta}\textup{-mod},
\end{align*}
and therefore is an exact functor, with the simple \(R^{\Lambda}_{\omega, \beta}\)-modules given by 
\begin{align*}
\{ \mathcal{U} D^{\blam} \mid \blam \in \Lambda^{\bkap}_+(\omega + \beta)', \mathcal{U}D^{\blam} \neq 0\},
\end{align*}
see \cite[\S6.2]{Green}.
The fact that \(\mathcal{U}(S^{\blam}) \cong S^{\bmu} \boxtimes S^{\blam/\bmu}\) if such \(\Lambda^{\bkap}_+(\omega) \ni \bmu \subseteq \blam\) exists follows from \cref{decompSkew}. If \(\blam \in \Lambda^{\bkap}_+(\omega + \beta)'\), with \(\mathcal{U}D^{\blam} \neq 0\), then it must be that \(\blam = \bmu \sqcup \bxi\) for some (unique) \(\bmu \in \Lambda^{\bkap}_+(\omega)\), \(\bxi \in \Lambda_{+/\bmu}^{\bkap}(\beta)\). Moreover, 
as in \cite[\S6.2, Remark 2]{Green}, \(\textup{hd}(S^{\blam}) = D^{\blam}\) implies that \(\mathcal{U}S^{\blam}\cong S^{\bmu} \boxtimes S^{\bxi}\) has simple head \(\mathcal{U}D^{\blam}\).
\end{proof}

\subsection{Cellularity of the \(\omega\)-skew cyclotomic KLR algebra}\label{cellskewsec}
Throughout this section we assume that \(\omega, \beta \in \ZZ_{\geq 0}I\), \(\height(\omega) = r\), \(\height(\beta) = b\), and that \(\brho \in \Lambda^{\bkap}_+(\omega)\) is a \(\bkap\)-core. We further assume that \((r,b)\)-adapted reduced expressions have been used to define \(\psi_\sigma\), for all \(\sigma \in \mathfrak{S}_r, \mathfrak{S}_b, \mathfrak{S}_{r+b}\).

The Specht module \(S^{\brho}\) is a simple \(R_\omega\)-module by \cref{coresimple}, and thus there exists a word \(\bj\) of \(S^{\brho}\) such that the highest \(\Z\)-degree \(\bj\)-space \(1_{\bj}S^{\brho}_m\) in \(S^{\brho}\) is one-dimensional, and the lowest \(\Z\)-degree \(\bj\)-space  \(1_{\bj}S^{\brho}_{-m}\) in \(S^{\brho}\) is one-dimensional as well (see for instance \cite[Lemma 2.30]{klesh14}). Then by \cref{Spechtbasis} we must have \(1_{\bj}S^{\brho}_m = \k\{v^{\tt T}\}\) and \(1_{\bj}S^{\brho}_{-m} = \k\{v^{\tt \bar T}\}\) for some \({\tt T}, {\tt \bar T} \in \Std(\brho)\).

As \(R^\Lambda_{\omega} = \END(S^{\brho})\) by \cref{coresimple}, there exist homogeneous
elements \(f_1 = 1_{\bj} f_1 1_{\bj}\), \(f_2 = 1_{\bj} f_2 1_{\bj} \in R^\Lambda_{\omega}\), of \(\Z\)-degree \(-2m\), \(2m\) respectively such that
\begin{align*}
f_1 v^{\tt X} = \delta_{{\tt \bar T}, {\tt X}} v^{\tt T}, \qquad 
f_2 v^{\tt X} = \delta_{{\tt T}, {\tt X}} v^{\tt \bar T}.
\end{align*}
Then, setting \(e^{\tt T} := f_1 f_2\), \(e^{\tt \bar T} :=f_2 f_1\), we have 
\begin{align*}
e^{\tt T} v^{\tt X} = \delta_{{\tt T}, {\tt X}} v^{\tt T}, \qquad
e^{\tt \bar T} v^{\tt X} = \delta_{{\tt \bar T}, {\tt X}} v^{\tt \bar T}.
\end{align*}
It follows then that \(e^{\tt T} = 1_{\bj} e^{\tt T} 1_{\bj}\), \(e^{\tt \bar T} = 1_{\bj} e^{\tt \bar T} 1_{\bj}\) are homogeneous primitive idempotents in \(R^\Lambda_{\omega}\).

Now, since \(1_{\bj}S^{\brho}_{m}\) and \(1_{\bj}S^{\brho}_{-m}\) are one-dimensional, we have \((1_\bj R^{\Lambda}_{\omega}1_{\bj} )_{2m} = \k\{f_2\}\) and \((1_\bj R^{\Lambda}_{\omega}1_{\bj})_{-2m} = \k\{f_1\}\). As the involution \(j_\omega\) is graded, with \(j_\omega(1_{\bj}) = 1_{\bj}\), it must be that \(j_\omega(f_1) = c_1 f_1\), and \(j_\omega(f_2) = c_2 f_2\), for some \(c_1, c_2 \in \k\). We thus have that \(j_\omega(e^{\tt T}) = j_\omega(f_1 f_2) = c_1 c_2 f_2 f_1 = c_1 c_2 e^{\tt \bar T}\). As an anti-isomorphism must preserve primitive idempotents, it follows that \(j_\omega(e^{\tt T}) = e^{\tt \bar T}\), \(j_\omega(e^{\tt \bar T}) = e^{\tt T}\), and \(c_1 c_2 = 1\).

Via \cref{tenssep}, we will identify \(R^\Lambda_\omega \otimes R^{\Lambda /\omega}_\beta\) with \(1_{\omega,\beta}R^\Lambda_{\omega + \beta}1_{\omega,\beta}\). Write \(\varepsilon^{\tt T} := e^{\tt T} \otimes 1_\beta \in R^\Lambda_\omega \otimes R^{\Lambda /\omega}_\beta = 1_{\omega,\beta}R^\Lambda_{\omega + \beta}1_{\omega,\beta}\). As \(e^{\tt T}\) is primitive, it follows from \cref{tenssep} that we have an algebra isomorphism
\begin{align}\label{eiso}
\eta^{\tt T}: R_{\beta}^{\Lambda/ \omega} \bijection \k\{e^{\tt T}\} \otimes R_{\beta}^{\Lambda/ \omega} = \varepsilon^{\tt T} R^{\Lambda}_{\omega + \beta} \varepsilon^{\tt T},
\qquad
x \longmapsto e^{\tt T} \otimes x,
\end{align}
for all \(x \in R_{\beta}^{\Lambda/ \omega} \).
Recalling \cref{cellHM}, for \(\blam/\brho \in \Lambda^{\bkap}_{+/\brho}(\beta)\), \({\tt x}, {\tt y} \in \Std(\blam/\brho)\), set
\begin{align*}
 c^{\blam/\brho}_{{\tt x}, {\tt y}}:= (\eta^{\tt T})^{-1}(\varepsilon^{\tt T}C^{\blam}_{{\tt T}{\tt x}, {\tt \bar T}{\tt y}} \varepsilon^{\tt T}) \in R^{\Lambda/ \omega}_\beta.
 \end{align*}

The following theorem is Theorem~\hyperlink{thm:D}{D} in the introduction.

\begin{Theorem}\label{skewcell}
If \(\omega, \beta \in \ZZ_{\geq 0}I\) and \(\brho \in \Lambda^{\bkap}_+(\omega)\) is a \(\bkap\)-core, then:
\begin{enumerate}
\item The \(\omega\)-skew cyclotomic KLR algebra \(R^{\Lambda/\omega}_\beta\) is a graded cellular algebra, with graded cell datum given by:
\begin{align}\label{d1}
\mathcal{P} &= (\Lambda^{\bkap}_{+/\brho}(\beta), \triangleright^D);\qquad
T(\blam/\brho) = \Std(\blam/\brho); \qquad
\textup{inv}= j_\beta;
\end{align}
\begin{align}\label{d2}
C({\tt x}, {\tt y}) =  c^{\blam/\brho}_{{\tt x}, {\tt y}}; \qquad
\deg_\Z({\tt x}) = \deg_{\blam/\brho}({\tt x}).
\end{align}
\item The left cell modules for this datum are the skew Specht \(R_\beta\)-modules:
\begin{align*}
\{S^{\blam/\brho} \mid \blam/\brho \in \Lambda^{\bkap}_{+/\brho}(\beta)\}.
\end{align*}
\item There is an exact functor of \(\ZZ\)-graded categories
\begin{align*}
\mathcal{T} : R^{\Lambda}_{\omega + \beta}\textup{-mod} \longrightarrow R^{\Lambda/\omega}_{\beta}\textup{-mod},
\qquad
M \longmapsto \varepsilon^{\tt T}M,
\end{align*}
where we identify the action of \(\varepsilon^{\tt T} R^{\Lambda}_{\omega + \beta} \varepsilon^{\tt T}\) with the action of \(R_{\beta}^{\Lambda/ \omega}\) via \(\eta^{\tt T}\),
such that \(\mathcal{T}S^{\blam} \cong S^{\blam/\brho}\).
Moreover,
\begin{align*}
\{ \mathcal{T} D^{\blam} \cong \textup{hd}(S^{\blam/\brho})  \mid \blam \in \Lambda^{\bkap}_+(\omega + \beta)', 1_{\omega, \beta}D^{\blam} \neq 0\}
\end{align*}
is an irredundant set of simple \(R^{\Lambda/\omega}_{\beta}\)-modules, up to grading shift.
In other words, \(\mathcal{T}\) sends cell modules to cell modules (or zero), and simple modules to simple modules (or zero).
\end{enumerate}
\end{Theorem}
\begin{proof}

As in the proof of \cref{wbcutthm}(ii), we may establish (i) via idempotent truncation by \(\varepsilon^{\tt T}\) of the cell structure for \(R^\Lambda_{\omega + \beta}\). However, the proof is slightly more delicate due to the fact that \(\varepsilon^{\tt T}\) is not a word idempotent in the KLR algebra, so we include full details in what follows.

Recall the cellular structure and basis for \(R^\Lambda_{\omega + \beta}\) as defined in \cref{cellHM}.
We first establish a number of claims about the \(\varepsilon^{\tt T}\)-truncation of this cellular basis.

{\em Claim (1): 
Let \(\blam \in \Lambda^{\bkap}_+(\omega + \beta)\). The set
\begin{align*}
\{ \varepsilon^{\tt T} C^{\blam}_{{\tt Tx}, {\tt \bar Ty}} \varepsilon^{\tt T} \mid {\tt x, y} \in \Std(\blam/\brho)\}
\end{align*}
is a \(\k\)-basis for \(\varepsilon^{\tt T} ((R^{\Lambda}_{\omega + \beta})^{\trianglerighteq^D \blam}  /(R^{\Lambda}_{\omega + \beta})^{\triangleright^D \blam}) \varepsilon^{\tt T}\).
}

Fix \(\blam \in \Lambda^{\bkap}_+(\omega + \beta)\). Letting \({\tt U, V} \in \Std(\blam)\), we note that 
\(
\varepsilon^{\tt T} C^{\blam}_{{\tt U, V}} \varepsilon^{\tt T} = \varepsilon^{\tt T} 1_{\omega, \beta}C^{\blam}_{{\tt U, V}}1_{\omega, \beta} \varepsilon^{\tt T}  = 0
\)
unless
\(
\cont(\textup{sh}^\downarrow_r({\tt U})) =  \cont(\textup{sh}^\downarrow_r({\tt V})) = \omega, 
\)
which, as \(\brho\) is a \(\bkap\)-core, implies that 
\(
\textup{sh}^\downarrow_r({\tt U}) =  \textup{sh}^\downarrow_r({\tt V}) = \brho
\).
Thus 
\begin{align}\label{makezero1}
\varepsilon^{\tt T} C^{\blam}_{{\tt U, V}} \varepsilon^{\tt T} \neq 0 \iff {\tt U} = {\tt Xx}, {\tt V} = {\tt Y y}
\end{align}
for some \({\tt X, Y} \in \Std(\brho)\), \({\tt x,y} \in \Std(\blam/\brho)\).

Recalling the isomorphisms \(G^{\blam/\bmu}\), \(F^{\blam}\) from \cref{decompSkew,transSpecht} respectively, note that
\begin{align*}
 F^{\blam}(\varepsilon^{\tt T} C^{\blam}_{{\tt Xx}, {\tt Yy}} \varepsilon^{\tt T}) &=  F^{\blam}(1_{\omega, \beta}\varepsilon^{\tt T} C^{\blam}_{{\tt Xx}, {\tt Yy}} \varepsilon^{\tt T}1_{\omega, \beta})
 =
1_{\omega, \beta} \varepsilon^{\tt T} v^{\tt Xx} \otimes 1_{\omega, \beta} j_{\omega + \beta}(\varepsilon^{\tt T}) v^{\tt Yy}
\\
&=
 (e^{\tt T} \otimes 1_\beta)
 1_{\omega, \beta}
 v^{\tt Xx} 
 \otimes 
 (e^{\tt \bar T} \otimes 1_\beta)
 1_{\omega, \beta}
 v^{\tt Yy}
 =G^{\blam/\brho}(e^{\tt T}v^{\tt X} \boxtimes v^{\tt x}) \otimes G^{\blam/\brho}(e^{\tt \bar T}v^{\tt Y} \boxtimes v^{\tt y})\\
 &=\delta_{{\tt T},{\tt X}} \delta_{{\tt \bar T}, {\tt Y}}G^{\blam/\brho}(v^{\tt T} \boxtimes v^{\tt x}) \otimes G^{\blam/\brho}(v^{\tt \bar T} \boxtimes v^{\tt y})
 =\delta_{{\tt T},{\tt X}} \delta_{{\tt \bar T}, {\tt Y}} (1_{\omega, \beta} v^{\tt Tx}) \otimes (1_{\omega, \beta} v^{\tt \bar T y})
 \\
 &=\delta_{{\tt T},{\tt X}} \delta_{{\tt \bar T}, {\tt Y}} v^{\tt Tx} \otimes v^{\tt \bar T y}
 =\delta_{{\tt T},{\tt X}} \delta_{{\tt \bar T}, {\tt Y}}   F^{\blam}( C^{\blam}_{{\tt Tx}, {\tt \bar T y}}),
\end{align*}
so it follows from \cref{transSpecht} that 
\begin{align}\label{notmakezero1}
\varepsilon^{\tt T} C^{\blam}_{{\tt Xx}, {\tt Yy}} \varepsilon^{\tt T} = \delta_{{\tt T},{\tt X}} \delta_{{\tt \bar T}, {\tt Y}} C^{\blam}_{{\tt Tx}, {\tt \bar Ty}}\qquad
\textup{mod } (R^{\Lambda}_{\omega + \beta})^{\triangleright^D \blam}.
\end{align}
Thus by \cref{cellHM}, this establishes Claim (1). 

Now, writing \(l_{{\tt X}, {\tt Y}}(a)\) for the scalar functions corresponding to the cell datum of \cref{cellHM}, define
\begin{align*}
l'_{{\tt u},{\tt v}}(\varepsilon^{\tt T} a \varepsilon^{\tt T}) := l_{{\tt T u}, {\tt T v}}(a\varepsilon^{\tt T })
\end{align*}
for all \(\blam/\brho \in \Lambda^{\bkap}_{+/\brho}(\beta)\), \({\tt u}, {\tt v} \in \Std(\blam/\brho)\), and \(a \in R^{\Lambda}_{\omega + \beta}\). 

{\em Claim (2). For all \(\blam \in \Lambda^{\bkap}_+(\omega + \beta)\), \({\tt x, u} \in \Std(\blam/\brho)\), \(a \in R^\Lambda_{\omega + \beta}\), we have:
\begin{align*}
\varepsilon^{ \tt T} a\varepsilon^{ \tt T}  \cdot \varepsilon^{ \tt T}C^{\blam}_{{\tt T x},{ \tt \bar T y}}\varepsilon^{ \tt T} = 
\sum_{{\tt u} \in \Std(\blam/\brho)} l'_{{\tt u},{\tt x}}(\varepsilon^{\tt T} a \varepsilon^{\tt T}) \varepsilon^{ \tt T}C^{\blam}_{{\tt T u}, {\tt \bar T y}}\varepsilon^{ \tt T}
\qquad \textup{mod } (R^{\Lambda}_{\omega + \beta})^{\triangleright^D \blam}.
\end{align*}
}

The proof of Claim (2) is a simple exercise using (C2) in the definition of a cellular basis in~\cref{celldef}.

\begin{answer}
To see this, note that
\begin{align*}
  \varepsilon^{ \tt T} a\varepsilon^{ \tt T} \cdot \varepsilon^{ \tt T}C^{\blam}_{{\tt T x},{ \tt \bar T y}}\varepsilon^{ \tt T}  &=
 \varepsilon^{ \tt T} (a\varepsilon^{ \tt T} \cdot C^{\blam}_{{\tt T x},{ \tt \bar T y}}) \varepsilon^{ \tt T} \\
&= \varepsilon^{\tt T} \bigg(\sum_{{\tt U} \in \Std(\blam)} l_{{\tt U}, {\tt T x}}(a\varepsilon^{\tt T})C^{\blam}_{{\tt U}, {\tt \bar T y}}\bigg) \varepsilon^{\tt T} 
&\textup{mod } (R^{\Lambda}_{\omega + \beta})^{\triangleright^D \blam}\\
&= \sum_{{\tt U} \in \Std(\blam)} l_{{\tt U}, {\tt T x}}(a\varepsilon^{\tt T}) \varepsilon^{\tt T} C^{\blam}_{{\tt U}, {\tt \bar T y}} \varepsilon^{\tt T} 
&\textup{mod } (R^{\Lambda}_{\omega + \beta})^{\triangleright^D \blam}\\
&= \sum_{{\tt u} \in \Std(\blam/\brho)} l_{{\tt T u}, {\tt T x}}(a\varepsilon^{\tt T}) \varepsilon^{\tt T} C^{\blam}_{{\tt T u}, {\tt \bar T y}} \varepsilon^{\tt T} 
&\textup{mod } (R^{\Lambda}_{\omega + \beta})^{\triangleright^D \blam}\\
&= \sum_{{\tt u} \in \Std(\blam/\brho)} l'_{{\tt u}, {\tt x}}(\varepsilon^{\tt T}a\varepsilon^{\tt T}) \varepsilon^{\tt T} C^{\blam}_{{\tt T u}, {\tt \bar T y}} \varepsilon^{\tt T} 
&\textup{mod } (R^{\Lambda}_{\omega + \beta})^{\triangleright^D \blam}
\end{align*}
which establishes Claim (2).
\end{answer}

Let \(\upsilon\) be the anti-automorphism on \(\varepsilon^{\tt T}R^{\Lambda}_{\omega+ \beta} \varepsilon^{\tt T} \cong \k\{e^{\tt T}\} \otimes R^{\Lambda/\omega}_\beta\) given by \(\upsilon = \textup{id} \otimes j_\beta = (j_\omega \otimes \textup{id}) \circ j_{\omega + \beta}\). 

{\em Claim (3). For all \(\blam \in \Lambda^{\bkap}_+(\omega + \beta)\), \({\tt x, y} \in \Std(\blam/\brho)\), we have
\begin{align*}
 \upsilon( \varepsilon^{\tt T}C^{\blam}_{{\tt T}{\tt x}, {\tt \bar T}{\tt y}}\varepsilon^{\tt T} ) =   \varepsilon^{\tt T} C^{\blam}_{{\tt T}{\tt y}, {\tt \bar T}{\tt x}} \varepsilon^{\tt T} \qquad \textup{mod } (R^{\Lambda}_{\omega + \beta})^{\triangleright^D \blam}.
\end{align*}
}

To see this, note that for \(\varepsilon^{\tt \bar T} x \varepsilon^{\tt \bar T} \in \varepsilon^{\tt \bar T} R^{\Lambda}_{\omega + \beta} \varepsilon^{\tt \bar T} \cong \k\{ e^{\tt \bar T}\} \otimes R^{\Lambda/\omega}_{\beta}\), we have \(\varepsilon^{\tt \bar T} x \varepsilon^{\tt \bar T} = e^{\tt \bar T} \otimes x'\) for some \(x' \in R^{\Lambda/\omega}_{\beta}\), and so it follows that
\begin{align*}
(j_\omega \otimes id)(\varepsilon^{\tt \bar T} x \varepsilon^{\tt \bar T}) 
= (j_\omega \otimes id)(e^{\tt \bar T} \otimes x') 
= e^{\tt T} \otimes x' 
= f_1 \varepsilon^{\tt \bar T} f_2 \otimes x'
=
(f_1 \otimes 1_\beta)(\varepsilon^{\tt \bar T} x \varepsilon^{\tt \bar T})(f_2 \otimes 1_\beta).
\end{align*}
Utilizing this fact, for all \(\blam \in \Lambda^{\bkap}_+(\omega + \beta)\), \({\tt x, y} \in \Std(\blam/\brho)\), we have by \cref{decompSkew,transSpecht,cellHM} that
\begin{align*}
F^{\blam} \circ \upsilon( \varepsilon^{\tt T}C^{\blam}_{{\tt T}{\tt x}, {\tt \bar T}{\tt y}}\varepsilon^{\tt T} )
&= F^{\blam} \circ(j_\omega \otimes \textup{id}) \circ j_{\omega+ \beta} (\varepsilon^{\tt T}C^{\blam}_{{\tt T}{\tt x}, {\tt \bar T}{\tt y}} \varepsilon^{\tt T})
= F^{\blam} \circ (j_\omega \otimes \textup{id})( \varepsilon^{\tt \bar T} C^{\blam}_{{\tt \bar T}{\tt y}, {\tt T}{\tt x}} \varepsilon^{\tt \bar T})\\
&=F^{\blam} ((f_1 \otimes 1_\beta)\varepsilon^{\tt \bar T} C^{\blam}_{{\tt \bar T}{\tt y}, {\tt T}{\tt x}} \varepsilon^{\tt \bar T}(f_2 \otimes 1_\beta))
=(f_1 \otimes 1_\beta) \varepsilon^{\tt \bar T} v^{\tt \bar T y} \otimes j_{\omega + \beta}( \varepsilon^{\tt \bar T}(f_2 \otimes 1_\beta)) v^{\tt T x}\\
&=(f_1 \otimes 1_\beta) \varepsilon^{\tt \bar T} v^{\tt \bar T y} \otimes (c_2f_2 \otimes 1_\beta) \varepsilon^{\tt T} v^{\tt T x}
=(f_1 \otimes 1_\beta) \varepsilon^{\tt \bar T} 1_{\omega, \beta} v^{\tt \bar T y} \otimes (c_2f_2 \otimes 1_\beta) \varepsilon^{\tt T}1_{\omega, \beta}  v^{\tt T x}\\
&=G^{\blam/\brho}(f_1 e^{\tt \bar T}v^{\tt \bar T} \boxtimes v^{\tt y}) \otimes G^{\blam/\brho}( c_2 f_2 e^{\tt T} v^{\tt T} \boxtimes v^{\tt x})
=G^{\blam/\brho}(v^{\tt  T} \boxtimes v^{\tt y}) \otimes G^{\blam/\brho}( c_2 v^{\tt \bar T} \boxtimes v^{\tt x})\\
&= c_2 v^{\tt T y} \otimes  v^{\tt \bar T x} 
= F^{\blam}(c_2 C^{\blam}_{{\tt T}{\tt y}, {\tt \bar T}{\tt x}})
=F^{\blam}(c_2  \varepsilon^{\tt T} C^{\blam}_{{\tt T}{\tt y}, {\tt \bar T}{\tt x}} \varepsilon^{\tt T}).
\end{align*}
Thus it follows from \cref{transSpecht} that
\begin{align*}
 \upsilon( \varepsilon^{\tt T}C^{\blam}_{{\tt T}{\tt x}, {\tt \bar T}{\tt y}}\varepsilon^{\tt T} ) = c_2  \varepsilon^{\tt T} C^{\blam}_{{\tt T}{\tt y}, {\tt \bar T}{\tt x}} \varepsilon^{\tt T} \qquad \textup{mod } (R^{\Lambda}_{\omega + \beta})^{\triangleright^D \blam}.
\end{align*}
Now, to work out the value \(c_2\), consider that when \(\beta = 0\), the above implies
\begin{align*}
c_2 \varepsilon^{\tt T} C^{\blam}_{{\tt T}, {\tt \bar T}} \varepsilon^{\tt T} = \upsilon( \varepsilon^{\tt T} C^{\blam}_{{\tt T}, {\tt \bar T}} \varepsilon^{\tt T}) = \varepsilon^{\tt T} C^{\blam}_{{\tt T}, {\tt \bar T}} \varepsilon^{\tt T}\qquad \textup{mod } (R^{\Lambda}_{\omega + \beta})^{\triangleright^D \blam}
\end{align*}
and hence \(c_2 = 1\), thus establishing Claim (3).

Now, Claims (1)--(3) show that 
\begin{align}\label{ed1}
\mathcal{P}(\Lambda^{\bkap}_{+/\brho}(\beta), \triangleright^{D});
\qquad
T(\blam/\brho) = \Std(\blam/\brho);
\qquad
\textup{inv} = \upsilon;
\end{align}
\begin{align}\label{ed2}
C({\tt x}, {\tt y}) = \varepsilon^{\tt T}C^{\blam}_{{\tt Tx}, {\tt \bar T y}} \varepsilon^{\tt T};
\qquad
\deg_{\ZZ}({\tt x}) = \deg_{\blam/\brho}({\tt x}),
\end{align}
is a graded cell datum for \( \varepsilon^{\tt T} R^{\Lambda}_{\omega + \beta} \varepsilon^{\tt T}\).

Now let \(\blam \in \Lambda^{\bkap}_+(\omega + \beta)\), and recall that \(S^{\blam}\) is isomorphic to the left cell module \(\Delta^{\blam}\) for \(R^\Lambda_{\omega + \beta}\). By (\ref{makezero1}) and (\ref{notmakezero1}), we have that
\(
\varepsilon^{\tt T} v^{\tt X} = 0
\) unless \({\tt X} = {\tt Tx}\) for some \({\tt x} \in \Std(\blam/\brho)\), in which case \(\varepsilon^{\tt T}v^{\tt Tx} = v^{\tt Tx}\). Thus it follows that 
\begin{align*}
\{v^{\tt Tx} \mid {\tt x} \in \Std(\blam/\brho)\}
\end{align*}
is a \(\k\)-basis for \(\varepsilon^{\tt T}S^{\blam}\). Moreover, for \(a \in R^{\Lambda}_{\omega + \beta}\), we have
\begin{align*}
\varepsilon^{\tt T}a \varepsilon^{\tt T} \cdot  v^{\tt Tx}
&=\varepsilon^{\tt T} \cdot \varepsilon^{\tt T}a  \cdot  v^{\tt Tx} 
=\varepsilon^{\tt T} \sum_{{\tt Y} \in \Std(\blam)} l_{{\tt Y}, {\tt Tx}}(\varepsilon^{\tt T}a) v^{\tt Y}
= \sum_{{\tt Y} \in \Std(\blam)} l_{{\tt Y}, {\tt Tx}}(\varepsilon^{\tt T}a) \varepsilon^{\tt T} v^{\tt Y}\\
&= \sum_{{\tt y} \in \Std(\blam/\brho)} l_{{\tt Ty}, {\tt Tx}}(\varepsilon^{\tt T}a) v^{\tt Ty}
=  \sum_{{\tt y} \in \Std(\blam/\brho)} l'_{{\tt y}, {\tt x}}(\varepsilon^{\tt T}a \varepsilon^{\tt T}) v^{\tt Ty},
\end{align*}
and so we have an isomorphism
\begin{align}\label{cell4}
\varepsilon^{\tt T}S^{\blam} \cong\Delta^{\blam/\brho}
\end{align}
with the left cell module for \( \varepsilon^{\tt T} R^{\Lambda}_{\omega + \beta} \varepsilon^{\tt T}\) under the above cell datum. 

Now, translating through the isomorphism \(\eta^{\tt T}\), the fact that (\ref{ed1}), (\ref{ed2}) is a cell datum for \(\k\{e^{\tt T}\} \otimes R^{\Lambda/\omega}_\beta \cong \varepsilon^{\tt T} R^{\Lambda}_{\omega + \beta}  \varepsilon^{\tt T}\) implies that (\ref{d1}), (\ref{d2}) is a cell datum for \(R^{\Lambda/\omega}_\beta\). 
Moreover, as we have an isomorphism
\begin{align*}
\k\{v^{\tt T}\}  \boxtimes S^{\blam/\brho}
=
e^{\tt T}S^{\brho} \boxtimes S^{\blam/\brho}
\xrightarrow{G^{\blam/\brho}} \varepsilon^{\tt T}1_{\omega, \beta} S^{\blam} =\varepsilon^{\tt T}S^{\blam},
\qquad
v^{\tt T} \boxtimes v^{\tt x} \longmapsto v^{\tt T x}
\end{align*}
of \( \k\{e^{\tt T}\} \otimes R^{\Lambda/\omega}_\beta = \varepsilon^{\tt T} R^{\Lambda}_{\omega + \beta}  \varepsilon^{\tt T}\)-modules, we have an isomorphism
\begin{align}\label{cell5}
S^{\blam/\brho} 
\cong    \k\{v^{\tt T}\} \boxtimes S^{\blam/\brho} 
\cong \varepsilon^{\tt T} S^{\blam}
\end{align}
of \(R^{\Lambda/\omega}_\beta\)-modules, and thus by (\ref{cell4}), \(S^{\blam/\bmu}\) is isomorphic to the left cell module \(\Delta^{\blam/\brho}\) for \(R^{\Lambda/\omega}_\beta\) under the cell datum (\ref{d1}), (\ref{d2}).
This completes the proof of (i) and (ii).

For (iii), note that \(\mathcal{T}\) is the composition of an idempotent truncation functor and an isomorphism of categories, by~(\ref{eiso}):
\begin{align*}
\mathcal{T} : R^{\Lambda}_{\omega + \beta}\textup{-mod} \longrightarrow \varepsilon_{\tt T}R^{\Lambda}_{\omega + \beta}\varepsilon_{\tt T} \textup{-mod} \bijection R^{\Lambda/\omega}_{\beta}\textup{-mod},
\end{align*}
and therefore is an exact functor, with the simple \(R^{\Lambda/\omega}_\beta\)-modules given by 
\begin{align*}
\{ \mathcal{T} D^{\blam} \mid \blam \in \Lambda^{\bkap}_+(\omega + \beta)', \mathcal{T}D^{\blam} \neq 0\},
\end{align*}
see \cite[\S6.2]{Green}.
The fact that \(\mathcal{T}S^{\blam} \cong S^{\blam/\brho}\) follows from~(\ref{cell5}). If \(\blam \in \Lambda^{\bkap}_+(\omega + \beta)'\), with \(\mathcal{T}D^{\blam} \neq 0\), then, as in \cite[\S6.2, Remark 2]{Green}, \(\textup{hd}(S^{\blam}) = D^{\blam}\) implies that \(\mathcal{T}S^{\blam}\cong S^{\blam/\bmu} \) has simple head \(\mathcal{T}D^{\blam}\). 
 
  Finally we show that \(\mathcal{T} D^{\blam} \neq 0\) if and only if \(1_{\omega, \beta}D^{\blam}  \neq 0\).
 Clearly \(\mathcal{T} D^{\blam} \neq 0\) implies that \(\varepsilon^{\tt T}D^{\blam}  =\varepsilon^{\tt T}1_{\omega, \beta} D^{\blam}  \neq 0\), giving the `only if' direction. For the `if' direction, assume that \(1_{\omega, \beta}D^{\blam}  \neq 0\). Then we have that \(1_{\omega, \beta}D^{\blam}\) is a nonzero \(1_{\omega, \beta} R^{\Lambda}_{\omega + \beta} 1_{\omega, \beta} \cong (R^\Lambda_\omega \otimes R^{\Lambda/\omega}_\beta)\)-module by \cref{tenssep}. Thus we have a nonzero map \(R^\Lambda_\omega \otimes \k\{1_\beta\} \to \END_\k(1_{\omega, \beta}D^{\blam})\), which is necessarily an injection as \(R^\Lambda_\omega\) is simple, so \(0 \neq (e^{\tt T} \otimes 1_\beta)1_{\omega, \beta} D^{\blam} =\varepsilon^{\tt T}D^{\blam}\), and thus \(\mathcal{T} D^{\blam} \neq 0\).
 \end{proof}

\subsection{Core truncation of RoCK blocks in level one}

Parts (i) and (ii) of Theorem~\hyperlink{thm:E}{E} in the introduction follow from the next result.

\begin{Theorem}\label{Morcut}
Let \(\kappa\) be a charge of level \(\ell = 1\).
Let \(\rho \in \Lambda_+^\kappa(\omega)\) be a \((\kappa, \theta)\)-RoCK \(e\)-core, and \(d \leq \textup{cap}^\theta_\delta(\rho, \kappa)\).
The functor 
\begin{align*}
\mathcal{T}: R^{\Lambda}_{\omega + d \delta}\textup{-mod} \longrightarrow R^{\Lambda/\omega}_{ d \delta}\textup{-mod}
\end{align*}
is a Morita equivalence with \(\mathcal{T}S^{\rho \sqcup \bnu^\theta_{\rho}} \cong S^{\bnu^\theta_{\rho}}\), and 
\begin{align}\label{allsimpsmor}
\{ \mathcal{T}D^{\rho \sqcup \bnu^\theta_{\rho}} \cong \textup{hd}(S^{\bnu^\theta_{\rho}})  \mid\bnu = (\varnothing \mid \nu^{(1,1)}\mid \dots\mid \nu^{(e-1,1)}) \in \Lambda^{(e, 1)}(d)\}.
\end{align}
yields a complete and irredundant set of simple \(R^{\Lambda/\omega}_{d \delta}\)-modules.
\end{Theorem}

\begin{proof}
In this proof, we will say `\(\bi\) is a word in \(M\)' to mean that \(1_\bi M \neq 0\).
As we are in level one, by \cref{resirrlist} we have
\begin{align*}
\Lambda^{\kappa}_+(\omega + d \delta) = \{ \rho \sqcup \bnu^\theta_{\rho} \mid \bnu \in \Lambda^{(e,1)}(d)\}.
\end{align*}
It is well known that level one Kleshchev partitions are exactly the \(e\)-restricted partitions, i.e.~those partitions \(\la=(\la_1,\dots,\la_m)\) such that \(\la_{i+1} - \la_i < e\) for all \(i\in [1,m-1]\) and \(\la_m < e\).
In consideration of \S\ref{ribbonshapesec}, these correspond to \(\bnu \in \Lambda^{(e,1)}\) with empty initial component:
\begin{align*}
\Lambda^{\bkap}_+(\omega + d \delta)' = \{ \rho \sqcup \bnu^\theta_{\rho}  \mid\bnu = (\varnothing \mid \nu^{(1,1)}\mid \dots\mid \nu^{(e-1,1)}) \in \Lambda^{(e, 1)}(d)\}.
\end{align*}
Write \(X = \{\bnu \in \Lambda^{(e,1)}(d) \mid \nu^{(0,1)} = \varnothing\}\). 
Let \({\tt T} \in \Std(\rho)\) be chosen as in \S\ref{cellskewsec}, and let \(\bnu \in X\). It follows from \cref{reswords} that  \(\bi^{\tt T} \bb_\theta^{\bnu}\) is a word in \(S^{\rho \sqcup \bnu^\theta_{\rho}}\).
Assume now that \(\bi^{\tt T} \bb_\theta^{\bnu}\) is a word in \(S^{\rho \sqcup \bmu^\theta_{\rho}}\) for some other \(\bmu \in X\). Then there is some \({\tt U} \in \Std( \rho \sqcup \bmu^\theta_{\brho})\) such that \(\bi^{\tt U} = \bi^{\tt T} \bb_\theta^{\bnu}\). Then 
\(\cont(\textup{sh} ({\tt U}\downarrow_r)) = \omega\), so since \(\rho\) is a \(\kappa\)-core, we must have that \(\textup{sh} ({\tt U}\downarrow_r) = \rho\). Hence we have that \(\bb_\theta^{\bnu}\) is a word in \(S^{\bmu^\theta_{\rho}}\). Then by \cref{reswords} it follows that \(\bmu \trianglerighteq \bnu\).

Therefore, for each \(\bnu \in X\), the word \(\bi^{\tt T} \bb_\theta^{\bnu}\) appears in \(S^{\rho \sqcup \bnu^\theta_{\brho}}\) and does not appear in any \(S^{\rho \sqcup \bmu^\theta_{\brho}}\) for \(\bnu \triangleright \bmu\).
As every simple factor in \(S^{\rho \sqcup \bnu^\theta_{\rho}}\) is of the form \(D^{\rho \sqcup \bmu^\theta_{\rho}}\) for \(\bnu \trianglerighteq \bmu\), and \(D^{\rho \sqcup \bnu^\theta_{\rho}} = \textup{hd}(S^{\rho \sqcup \bnu^\theta_{\rho}})\) it follows that \(\bi^{\tt T} \bb_\theta^{\bnu}\) must be a word in \(D^{\rho \sqcup \bnu^\theta_{\rho}}\).
Hence \(1_{\omega, \beta} D^{\rho \sqcup \bnu^\theta_{\rho}} \neq 0\) for all \(\bnu \in X\), so by \cref{skewcell}(iii), we have that (\ref{allsimpsmor}) is a complete and irredundant set of simple  \(R_\beta^{\Lambda/\omega}\)-modules. Moreover, since no simples are annihilated by the exact idempotent truncation functor \(\mathcal{T}\), we have that \(\mathcal{T}\) is a Morita equivalence.
\end{proof}

\section{Cuspidal systems and KLR modules}\label{cuspsyssec}

\subsection{Cuspidal systems and classification of simple \(R_\theta\)-modules}\label{stratasec}
Following \cite{klesh14, McN17, km17a, TW}, for \(m \in \mathbb{Z}_{>0}\), \(\beta \in \Psi\), we say an \(R_{m\beta}\)-module \(M\) is {\em semicuspidal} provided that for all \(0 \neq \theta_1, \theta_2 \in \ZZ_{\geq 0}I\) with \(\theta_1 + \theta_2 = m\beta\), we have \(\Res_{\theta_1, \theta_2}^{m\beta} M \neq 0\) only if \(\theta_1\) is a sum of positive roots \(\preceq \beta\) and \(\theta_2\) is a sum of positive roots \(\succeq \beta\). We say moreover that \(M\) is {\em cuspidal} if \(m=1\) and the comparisons above are strict.
Cuspidal and semicuspidal modules are key building blocks in the representation theory of \(R_\omega\). 

As explained in \cite{klesh14}, for \(\beta \in \Phi_+^\re\), \(m \in \mathbb{Z}_{>0}\), there is a unique `real' simple semicuspidal \(R_{m \beta}\)-module denoted \(L(\beta^m)\). On the other hand, the `imaginary' simple semicuspidal \(R_{m \delta}\)-modules are more plentiful; they may be indexed by \((e-1)\)-multipartitions of \(m\):
\begin{align}\label{indeximag}
\{ L( \blam ) \mid \blam = (\lambda^{(1)}\mid \dots\mid \lambda^{(e-1)}) \text{ a multipartition of }m\}.
\end{align}
There is some amount of freedom in this choice of labelling, see for instance \cite[\S21]{McN17}. Our choices in this section are made for compatibility with {\em row} Specht modules and \(e\)-restricted partition labels, and differ for instance with the labels chosen in \cite{km17} by reversing the order of components in \(\blam\).

For \(\omega \in \ZZ_{\geq 0}I\), a {\em Kostant partition} of \(\omega\) is a tuple of non-negative integers \(\boldsymbol{K} = (K_\beta)_{\beta \in \Psi}\) such that \(\sum_{\beta \in \Psi} K_\beta \beta = \omega\). If \(\beta_1 \succ \cdots \succ \beta_t\) are the members of \(\Psi\) such that \(K_{\beta_1} \neq 0\), then it is convenient to write \(\boldsymbol{K}\) in the form
\(
\boldsymbol{K} = 
(\beta_1^{K_{\beta_1}} \mid \dots \mid \beta_t^{K_{\beta_t}}).
\)
We write \(\Xi(\omega)\) for the set of all Kostant partitions of \(\omega\).

For \(\omega \in \ZZ_{\geq 0}I\), a {\em root partition} \(\pi = (\boldsymbol{K}, \bnu)\) is the data of a Kostant partition \(\boldsymbol{K} = (K_\beta)_{\beta \in \Psi}\in \Xi(\omega)\), together with 
an \((e-1)\)-multipartition \(\bnu\) of \(K_\delta\). We set \(\Pi(\omega)\) to be the set of all root partitions of \(\omega\).

To each \(\pi \in \Pi(\omega)\), one may associate a certain proper standard module \(\bar\Delta(\pi)\) which is an ordered induction product of simple semicuspidal modules, see for instance \cite[(6.5)]{km17a}. The module \(\bar\Delta(\pi)\) has a self-dual simple head \(L(\pi)\), and 
\(
\{ L(\pi) \mid \pi \in \Pi(\omega)\}
\)
is a complete and irredundant set of simple \(R_\omega\)-modules up to isomorphism and grading shift, as explained in \cite{klesh14, McN17, km17a, TW}.
To be precise, if
\begin{align*}
\pi = ((\beta_1^{K_{\beta_1}} \mid \dots \mid \beta_u^{K_{\beta_u}} \mid \delta^{K_\delta} \mid \beta_{u+1}^{K_{\beta_{u+1}}} \mid \dots \mid \beta_t^{K_{\beta_t}}), \blam),
\end{align*}
then \(\bar\Delta(\pi)\) is (up to grading shift) the induction product of semicuspidal simple modules:
\begin{align*}
\bar\Delta(\pi) := L(\beta_1^{K_{\beta_1}}) \circ \dots \circ L(\beta_u^{K_{\beta_u}}) \circ
L( \blam) \circ L(\beta_{u+1}^{K_{\beta_{u+1}}}) \circ \dots \circ L(\beta_t^{K_{\beta_t}}),
\end{align*}
so to construct the proper standard modules, it suffices to construct the simple semicuspidal modules associated to multiples of positive roots. We remark that semicuspidal modules, and thus the simple modules \(L(\pi)\), depend on the chosen convex preorder \(\succeq\), so we will use the notation \(L(\pi)_{\succeq}\) in situations where the choice of convex preorder is unclear or particularly relevant.

\subsection{Semicuspidal modules and Specht modules}
The notions of semicuspidality of skew diagrams and KLR modules is connected in the following lemma:

\begin{Lemma}\label{spechtshape}
\cite[Proposition~8.3]{muthtiling}
Let \(\succeq\) be a fixed convex preorder on \(\Phi_+\), and let \(\btau\) be a skew diagram. Then the Specht module \(\zS^{\btau}\) is semicuspidal if and only if the skew diagram \(\btau\) is semicuspidal.
\end{Lemma}
In \cite[Proposition~8.4]{muthtiling}, it is shown that for every \(\beta \in \Phi_+^\re\) and \(m\in \ZZ_{>0}\), the unique real simple semicuspidal \(R_{m\beta}\)-module \(L(\beta^{m})\) may be directly constructed as a certain skew Specht module. On the other hand, the simple imaginary \(R_{m \delta}\)-modules do not arise as skew Specht modules. However, it was 
conjectured in \cite[\S8.7.3]{muthtiling} that they may be constructed as certain simple heads of skew Specht modules related to RoCK blocks, and we prove that conjecture now. 

\begin{Theorem}\label{semicuspthm}
Let \(\kappa\) be a charge of level \(\ell = 1\). Let \(\rho \in \Lambda_+^\kappa(\omega)\) be a \((\kappa, \theta)\)-RoCK \(e\)-core, and \(d \leq \textup{cap}^\theta_\delta(\rho, \kappa)\). Let \(\succeq\) be a convex preorder which realizes \(\theta\). Then 
\begin{align}\label{semicusplist}
\{\textup{hd}(S^{\bnu^\theta_{\rho}})  \mid\bnu = (\varnothing \mid \nu^{(1,1)}\mid \dots\mid \nu^{(e-1,1)}) \in \Lambda^{(e, 1)}(d)\}
\end{align}
is a complete and irredundant set of simple imaginary semicuspidal \(R_{d\delta}\)-modules up to degree shift.
\end{Theorem}
\begin{proof}
The Specht modules \(S^{\bnu^\theta_{\rho}}\) in (\ref{semicusplist}) are semicuspidal by \cref{combprop2}(ii) and \cref{spechtshape}, and have non-isomorphic simple heads by \cref{Morcut}.
As the cardinality of (\ref{semicusplist}) matches the expected number of simple imaginary semicuspidal \(R_{d\delta}\)-modules in (\ref{indeximag}), the list (\ref{semicusplist}) must be exhaustive. 
\end{proof}

The above Theorem proves part (iii) of Theorem~\hyperlink{thm:E}{E} in the introduction.

\begin{Remark}\label{imagshapesrem}
\cref{semicuspthm} may be used to construct simple imaginary semicuspidal \(R_{d\delta}\) modules as heads of skew Specht modules for any convex preorder \(\succeq\). Indeed, \(\succeq\) realizes a unique residue permutation \(\theta\), and it is straightforward, using beta numbers, to construct a \((\kappa, \theta)\)-RoCK \(e\)-core \(\rho\) large enough that \(d \leq \textup{cap}^\theta_\delta(\rho, \kappa)\). Then the simple imaginary semicuspidal \(R_{d\delta}\)-modules will be those given in (\ref{semicusplist}).
\end{Remark}

\begin{figure}[h]
\begin{align*}
{}
\hackcenter{
\begin{tikzpicture}[scale=0.29]
\draw[thick,lightgray,fill=lightgray!30] (0,8)--(1,8)--(1,11)--(2,11)--(2,14)--(3,14)--(3,17)--(4,17)--(4,18)--(5,18)--(5,19)--(6,19)--(6,20)--(7,20)--(7,21)--(8,21)--(8,22)--(9,22)--(9,23)--(10,23)--(10,24)--(13,24)--(13,25)--(16,25)--(16,26)--(19,26)--(19,27)--(0,27)--(0,8);
\draw[thick,fill=cyan!30] (0+0,0+4)--(1+0,0+4)--(1+0,4+4)--(0+0,4+4)--(0+0,0+4);
\node at (0.5+0,0.5+4){$\scriptstyle 2$};
\node at (0.5+0,1.5+4){$\scriptstyle 3$};
\node at (0.5+0,2.5+4){$\scriptstyle 0$};
\node at (0.5+0,3.5+4){$\scriptstyle 1$};
\draw[thick, gray, dotted] (0+0,1+4)--(1+0,1+4);
\draw[thick, gray, dotted] (0+0,2+4)--(1+0,2+4);
\draw[thick, gray, dotted] (0+0,3+4)--(1+0,3+4);
\draw[thick,fill=red!30] (2+1,12+3)--(3+1,12+3)--(3+1,13+3)--(4+1,13+3)--(4+1,15+3)--(3+1,15+3)--(3+1,14+3)--(2+1,14+3)--(2+1,12+3);
\node at (2.5+1,12.5+3){$\scriptstyle 0$};
\node at (2.5+1,13.5+3){$\scriptstyle 1$};
\node at (3.5+1,13.5+3){$\scriptstyle 2$};
\node at (3.5+1,14.5+3){$\scriptstyle 3$};
\draw[thick, gray, dotted] (2+1,13+3)--(3+1,13+3);
\draw[thick, gray, dotted] (3+1,13+3)--(3+1,14+3);
\draw[thick, gray, dotted] (3+1,14+3)--(4+1,14+3);
\draw[thick,fill=green!30] (7+0,20+0)--(9+0,20+0)--(9+0,21+0)--(10+0,21+0)--(10+0,22+0)--(8+0,22+0)--(8+0,21+0)--(7+0,21+0)--(7+0,20+0);
\node at (7.5+0,20.5+0){$\scriptstyle 1$};
\node at (8.5+0,20.5+0){$\scriptstyle 2$};
\node at (8.5+0,21.5+0){$\scriptstyle 3$};
\node at (9.5+0,21.5+0){$\scriptstyle 0$};
\draw[thick, gray, dotted] (8+0,20+0)--(8+0,21+0);
\draw[thick, gray, dotted] (8+0,21+0)--(9+0,21+0);
\draw[thick, gray, dotted] (9+0,21+0)--(9+0,22+0);
\draw[thick,fill=green!30] (7+2,20+2)--(9+2,20+2)--(9+2,21+2)--(10+2,21+2)--(10+2,22+2)--(8+2,22+2)--(8+2,21+2)--(7+2,21+2)--(7+2,20+2);
\node at (7.5+2,20.5+2){$\scriptstyle 1$};
\node at (8.5+2,20.5+2){$\scriptstyle 2$};
\node at (8.5+2,21.5+2){$\scriptstyle 3$};
\node at (9.5+2,21.5+2){$\scriptstyle 0$};
\draw[thick, gray, dotted] (8+2,20+2)--(8+2,21+2);
\draw[thick, gray, dotted] (8+2,21+2)--(9+2,21+2);
\draw[thick, gray, dotted] (9+2,21+2)--(9+2,22+2);
%
\node at (3.5,23.5){$\scriptstyle \rho$};
%
%
%
%
%
%
\node at (8.5+1-3,8.5){$\scriptstyle \varnothing$};
\draw[thick,,fill=green!30]
(12+1-3,7)--(13+1-3,7)--(13+1-3,9)--(12+1-3,9)--(12+1-3,7);
\draw[thick,,fill=red!30]
(16+1-3,8)--(17+1-3,8)--(17+1-3,9)--(16+1-3,9)--(16+1-3,8);
\draw[thick,,fill=cyan!30]
(20+1-3,8)--(21+1-3,8)--(21+1-3,9)--(20+1-3,9)--(20+1-3,8);
\draw[thick, gray, dotted] (13-3,8)--(14-3,8);
\draw[thin,gray,fill=gray!30]
(10.5+1-3,7.2)--(10.5+1-3,8.8);
\draw[thin,gray,fill=gray!30]
(14.5+1-3,7.2)--(14.5+1-3,8.8);
\draw[thin,gray,fill=gray!30]
(18.5+1-3,7.2)--(18.5+1-3,8.8);
\node[below] at (10-3,7){$\scriptstyle \nu^{(0,1)}$};
\node[below] at (14-3,7){$\scriptstyle \nu^{(1,1)}$};
\node[below] at (18-3,7){$\scriptstyle \nu^{(2,1)}$};
\node[below] at (22-3,7){$\scriptstyle \nu^{(3,1)}$};
\node[above] at (16-3,9.5){$\scriptstyle  \overbrace{\hspace{50mm}}$};
\node[above] at (16-3,10.5){$\scriptstyle \bnu$};
\end{tikzpicture}
\qquad
\qquad
\begin{tikzpicture}[scale=0.45]
	\draw[thick, lightgray] (0,0.5)--(0,-12.5);
	\draw[thick, lightgray] (1,0.5)--(1,-12.5);
	\draw[thick, lightgray] (2,0.5)--(2,-12.5);
	\draw[thick, lightgray] (3,0.5)--(3,-12.5);
	\draw[thick, lightgray, dotted] (0,-12.6)--(0,-13.2);
	\draw[thick, lightgray, dotted] (1,-12.6)--(1,-13.2);
	\draw[thick, lightgray, dotted] (2,-12.6)--(2,-13.2);
	\draw[thick, lightgray, dotted] (3,-12.6)--(3,-13.2);
	\draw[thick, black, dotted] (0,0.5)--(0,1);
	\draw[thick, black, dotted] (1,0.5)--(1,1);
	\draw[thick, black, dotted] (2,0.5)--(2,1);
	\draw[thick, black, dotted] (3,0.5)--(3,1);
	\filldraw[green!30] (0,1.5) circle (10pt);
	\filldraw[cyan!30] (1,1.5) circle (10pt);
	\filldraw[orange!30] (2,1.5) circle (10pt);
	\filldraw[red!30] (3,1.5) circle (10pt);
	\node[above] at (0,1){$\scriptstyle 0$};
	\node[above] at (1,1){$\scriptstyle 1$};
	\node[above] at (2,1){$\scriptstyle 2$};
	\node[above] at (3,1){$\scriptstyle 3$};
	\blackdot(0,0);
	\blackdot(0,-1);
	\blackdot(0,-2);
	\blackdot(0,-3);
	\blackdot(0,-4);
	\blackdot(0,-5);
	\blackdot(0,-6);
	\draw[thick, lightgray] (-0.2,-7)--(0.2,-7);
	\blackdot(0,-8);
	\blackdot(0,-9);
	\draw[thick, lightgray] (-0.2,-10)--(0.2,-10);
	\draw[thick, lightgray] (-0.2,-11)--(0.2,-11);
	\draw[thick, lightgray] (-0.2,-12)--(0.2,-12);
	\blackdot(1,0);
	\draw[thick, lightgray] (0.8,-1)--(1.2,-1);
	\blackdot(1,-2);
	\draw[thick, lightgray] (0.8,-3)--(1.2,-3);
	\draw[thick, lightgray] (0.8,-4)--(1.2,-4);
	\draw[thick, lightgray] (0.8,-5)--(1.2,-5);
	\draw[thick, lightgray] (0.8,-6)--(1.2,-6);
	\draw[thick, lightgray] (0.8,-7)--(1.2,-7);
	\draw[thick, lightgray] (0.8,-8)--(1.2,-8);
	\draw[thick, lightgray] (0.8,-9)--(1.2,-9);
	\draw[thick, lightgray] (0.8,-10)--(1.2,-10);
	\draw[thick, lightgray] (0.8,-11)--(1.2,-11);
	\draw[thick, lightgray] (0.8,-12)--(1.2,-12);
	\blackdot(2,0);
	\blackdot(2,-1);
	\blackdot(2,-2);
	\blackdot(2,-3);
	\blackdot(2,-4);
	\blackdot(2,-5);
	\blackdot(2,-6);
	\blackdot(2,-7);
	\blackdot(2,-8);
	\blackdot(2,-9);
	\blackdot(2,-10);
	\blackdot(2,-11);
	\draw[thick, lightgray] (1.8,-12)--(2.2,-12);
	\blackdot(3,0);
	\blackdot(3,-1);
	\blackdot(3,-2);
	\blackdot(3,-3);
	\draw[thick, lightgray] (2.8,-4)--(3.2,-4);
	\blackdot(3,-5);
	\draw[thick, lightgray] (2.8,-6)--(3.2,-6);
	\draw[thick, lightgray] (2.8,-7)--(3.2,-7);
	\draw[thick, lightgray] (2.8,-8)--(3.2,-8);
	\draw[thick, lightgray] (2.8,-9)--(3.2,-9);
	\draw[thick, lightgray] (2.8,-10)--(3.2,-10);
	\draw[thick, lightgray] (2.8,-11)--(3.2,-11);
	\draw[thick, lightgray] (2.8,-12)--(3.2,-12);
	\phantom{
		\draw (0,-13)--(0,-13.5);
	}
\end{tikzpicture}
}
\end{align*}
\vspace{5pt}
\begin{align*}
\hackcenter{
\begin{tikzpicture}[scale=0.29]
\draw[thick, white]  (0+0,0+4)--(1+0,0+4)--(1+0,4+4)--(0+0,4+4)--(0+0,0+4);
\draw[thick,lightgray,fill=lightgray!30] (0,8)--(1,8)--(1,11)--(2,11)--(2,14)--(3,14)--(3,17)--(4,17)--(4,18)--(5,18)--(5,19)--(6,19)--(6,20)--(7,20)--(7,21)--(8,21)--(8,22)--(9,22)--(9,23)--(10,23)--(10,24)--(13,24)--(13,25)--(16,25)--(16,26)--(19,26)--(19,27)--(0,27)--(0,8);
\draw[thick,fill=red!30] (2+0,12+0)--(3+0,12+0)--(3+0,13+0)--(4+0,13+0)--(4+0,15+0)--(3+0,15+0)--(3+0,14+0)--(2+0,14+0)--(2+0,12+0);
\node at (2.5+0,12.5+0){$\scriptstyle 0$};
\node at (2.5+0,13.5+0){$\scriptstyle 1$};
\node at (3.5+0,13.5+0){$\scriptstyle 2$};
\node at (3.5+0,14.5+0){$\scriptstyle 3$};
\draw[thick, gray, dotted] (2+0,13+0)--(3+0,13+0);
\draw[thick, gray, dotted] (3+0,13+0)--(3+0,14+0);
\draw[thick, gray, dotted] (3+0,14+0)--(4+0,14+0);
\draw[thick,fill=red!30] (2+2,12+2)--(3+2,12+2)--(3+2,13+2)--(4+2,13+2)--(4+2,15+2)--(3+2,15+2)--(3+2,14+2)--(2+2,14+2)--(2+2,12+2);
\node at (2.5+2,12.5+2){$\scriptstyle 0$};
\node at (2.5+2,13.5+2){$\scriptstyle 1$};
\node at (3.5+2,13.5+2){$\scriptstyle 2$};
\node at (3.5+2,14.5+2){$\scriptstyle 3$};
\draw[thick, gray, dotted] (2+2,13+2)--(3+2,13+2);
\draw[thick, gray, dotted] (3+2,13+2)--(3+2,14+2);
\draw[thick, gray, dotted] (3+2,14+2)--(4+2,14+2);
\draw[thick,fill=red!30] (2+1,12+3)--(3+1,12+3)--(3+1,13+3)--(4+1,13+3)--(4+1,15+3)--(3+1,15+3)--(3+1,14+3)--(2+1,14+3)--(2+1,12+3);
\node at (2.5+1,12.5+3){$\scriptstyle 0$};
\node at (2.5+1,13.5+3){$\scriptstyle 1$};
\node at (3.5+1,13.5+3){$\scriptstyle 2$};
\node at (3.5+1,14.5+3){$\scriptstyle 3$};
\draw[thick, gray, dotted] (2+1,13+3)--(3+1,13+3);
\draw[thick, gray, dotted] (3+1,13+3)--(3+1,14+3);
\draw[thick, gray, dotted] (3+1,14+3)--(4+1,14+3);
\draw[thick,fill=red!30] (2+3,12+5)--(3+3,12+5)--(3+3,13+5)--(4+3,13+5)--(4+3,15+5)--(3+3,15+5)--(3+3,14+5)--(2+3,14+5)--(2+3,12+5);
\node at (2.5+3,12.5+5){$\scriptstyle 0$};
\node at (2.5+3,13.5+5){$\scriptstyle 1$};
\node at (3.5+3,13.5+5){$\scriptstyle 2$};
\node at (3.5+3,14.5+5){$\scriptstyle 3$};
\draw[thick, gray, dotted] (2+3,13+5)--(3+3,13+5);
\draw[thick, gray, dotted] (3+3,13+5)--(3+3,14+5);
\draw[thick, gray, dotted] (3+3,14+5)--(4+3,14+5);
\node at (3.5,23.5){$\scriptstyle \rho$};
%
%
%
\draw[thick,,fill=red!30] (16+1-3,7)--(18+1-3,7)--(18+1-3,9)--(16+1-3,9)--(16+1-3,7);
%
\draw[thick, gray, dotted] (18-3,7)--(18-3,9);
\draw[thick, gray, dotted] (17-3,8)--(19-3,8);
%
%
%
%
\node at (9.5-3,8.5){$\scriptstyle \varnothing$};
\node at (13.5-3,8.5){$\scriptstyle \varnothing$};
\node at (22.5-3,8.5){$\scriptstyle \varnothing$};
\draw[thin,gray,fill=gray!30]
(11.5-3,7.2)--(11.5-3,8.8);
\draw[thin,gray,fill=gray!30]
(15.5-3,7.2)--(15.5-3,8.8);
\draw[thin,gray,fill=gray!30]
(20.5-3,7.2)--(20.5-3,8.8);
\node[below] at (10-3,7){$\scriptstyle \mu^{(0,1)}$};
\node[below] at (14-3,7){$\scriptstyle \mu^{(1,1)}$};
\node[below] at (18-3,7){$\scriptstyle \mu^{(2,1)}$};
\node[below] at (23-3,7){$\scriptstyle \mu^{(3,1)}$};
\node[above] at (16.5-3,9.5){$\scriptstyle  \overbrace{\hspace{50mm}}$};
\node[above] at (16.5-3,10.5){$\scriptstyle \bmu$};
\end{tikzpicture}
\qquad
\qquad
\begin{tikzpicture}[scale=0.45]
	\draw[thick, lightgray] (0,0.5)--(0,-12.5);
	\draw[thick, lightgray] (1,0.5)--(1,-12.5);
	\draw[thick, lightgray] (2,0.5)--(2,-12.5);
	\draw[thick, lightgray] (3,0.5)--(3,-12.5);
	\draw[thick, lightgray, dotted] (0,-12.6)--(0,-13.2);
	\draw[thick, lightgray, dotted] (1,-12.6)--(1,-13.2);
	\draw[thick, lightgray, dotted] (2,-12.6)--(2,-13.2);
	\draw[thick, lightgray, dotted] (3,-12.6)--(3,-13.2);
	\draw[thick, black, dotted] (0,0.5)--(0,1);
	\draw[thick, black, dotted] (1,0.5)--(1,1);
	\draw[thick, black, dotted] (2,0.5)--(2,1);
	\draw[thick, black, dotted] (3,0.5)--(3,1);
	\filldraw[green!30] (0,1.5) circle (10pt);
	\filldraw[cyan!30] (1,1.5) circle (10pt);
	\filldraw[orange!30] (2,1.5) circle (10pt);
	\filldraw[red!30] (3,1.5) circle (10pt);
	\node[above] at (0,1){$\scriptstyle 0$};
	\node[above] at (1,1){$\scriptstyle 1$};
	\node[above] at (2,1){$\scriptstyle 2$};
	\node[above] at (3,1){$\scriptstyle 3$};
	\blackdot(0,0);
	\blackdot(0,-1);
	\blackdot(0,-2);
	\blackdot(0,-3);
	\blackdot(0,-4);
	\blackdot(0,-5);
	\blackdot(0,-6);
	\blackdot(0,-7);
	\blackdot(0,-8);
	\draw[thick, lightgray] (-0.2,-9)--(0.2,-9);
	\draw[thick, lightgray] (-0.2,-10)--(0.2,-10);
	\draw[thick, lightgray] (-0.2,-11)--(0.2,-11);
	\draw[thick, lightgray] (-0.2,-12)--(0.2,-12);
	\blackdot(1,0);
	\blackdot(1,-1);
	\draw[thick, lightgray] (0.8,-2)--(1.2,-2);
	\draw[thick, lightgray] (0.8,-3)--(1.2,-3);
	\draw[thick, lightgray] (0.8,-4)--(1.2,-4);
	\draw[thick, lightgray] (0.8,-5)--(1.2,-5);
	\draw[thick, lightgray] (0.8,-6)--(1.2,-6);
	\draw[thick, lightgray] (0.8,-7)--(1.2,-7);
	\draw[thick, lightgray] (0.8,-8)--(1.2,-8);
	\draw[thick, lightgray] (0.8,-9)--(1.2,-9);
	\draw[thick, lightgray] (0.8,-10)--(1.2,-10);
	\draw[thick, lightgray] (0.8,-11)--(1.2,-11);
	\draw[thick, lightgray] (0.8,-12)--(1.2,-12);
	\blackdot(2,0);
	\blackdot(2,-1);
	\blackdot(2,-2);
	\blackdot(2,-3);
	\blackdot(2,-4);
	\blackdot(2,-5);
	\blackdot(2,-6);
	\blackdot(2,-7);
	\blackdot(2,-8);
	\blackdot(2,-9);
	\blackdot(2,-10);
	\blackdot(2,-11);
	\draw[thick, lightgray] (1.8,-12)--(2.2,-12);
	\blackdot(3,0);
	\blackdot(3,-1);
	\blackdot(3,-2);
	\draw[thick, lightgray] (2.8,-3)--(3.2,-3);
	\draw[thick, lightgray] (2.8,-4)--(3.2,-4);
	\blackdot(3,-5);
	\blackdot(3,-6);
	\draw[thick, lightgray] (2.8,-7)--(3.2,-7);
	\draw[thick, lightgray] (2.8,-8)--(3.2,-8);
	\draw[thick, lightgray] (2.8,-9)--(3.2,-9);
	\draw[thick, lightgray] (2.8,-10)--(3.2,-10);
	\draw[thick, lightgray] (2.8,-11)--(3.2,-11);
	\draw[thick, lightgray] (2.8,-12)--(3.2,-12);
	\phantom{
	\draw (0,-13)--(0,-13.5);
	}
\end{tikzpicture}
}
\end{align*}
\caption{
The skew diagrams associated to the semicuspidal skew Specht modules \(S^{\bnu^{\theta}_{\rho}}\) and \(S^{\bmu^{\theta}_{\rho}}\), for \(\bnu = (\varnothing \mid (1^2)\mid (1)\mid (1))\in \Lambda^{(4,1)}(4)\) and \(\bmu = (\varnothing\mid \varnothing\mid (2^2)\mid \varnothing) \in \Lambda^{(4,1)}(4)\) considered in \cref{twoex}.
On the right, the abacus diagrams corresponding to $\rho\sqcup\bnu^{\theta}_{\rho}$ and $\rho\sqcup\bmu^{\theta}_{\rho}$ are displayed.
}
\label{twoexpic}      
\end{figure}

\begin{Example}\label{twoex}
Take \(e = 4\), and let \(\succeq\) be the convex preorder on \(\Phi_+\) defined in \cref{3coretileex}. Then \(\succeq\) realizes the residue permutation \(\theta = (1,3,0,2)\), and the \(e\)-core partition 
 \(\rho = (19,16,13,10,9,8,7,6,5,4,3^3,2^3,1^3)\)
with charge \(\kappa = 0\), considered in \cref{imagnodes} is a \((\kappa, \theta)\)-RoCK core with capacity \(\textup{cap}_\delta^\theta(\rho) = 4\). Thus by \cref{semicuspthm}, the 51 simple imaginary semicuspidal \(R_{4\delta}\)-modules are given by
\begin{align*}
\{ \textup{hd}(S^{\bnu^\theta_{\brho}}) \mid \bnu = (\varnothing \mid \nu^{(1,1)}\mid \nu^{(2,1)}\mid \nu^{(3,1)}) \in \Lambda^{(4,1)}(4)\}.
\end{align*}
We include in Figure~\ref{twoexpic} two examples -- the skew diagrams \({\bnu^\theta_{\rho}}\) and \({\bmu^\theta_{\rho}}\) associated to the semicuspidal skew Specht modules \(S^{\bnu^\theta_{\rho}}\) and \(S^{\bmu^\theta_{\rho}}\), for \(\bnu = (\varnothing \mid (1^2) \mid (1)\mid (1)\mid )\) and \(\bmu = (\varnothing\mid \varnothing\mid (2^2) \mid \varnothing) \in \Lambda^{(4,1)}(4)\).
Referring to Figure~\ref{imagnodesarr}, observe that
$\bnu^\theta_{\rho}=a^\theta_{\rho}\sqcup e^\theta_{\rho}\sqcup i^\theta_{\rho}\sqcup k^\theta_{\rho}$ and 
$\bmu^\theta_{\rho}= e^\theta_{\rho}\sqcup f^\theta_{\rho}\sqcup g^\theta_{\rho}\sqcup h^\theta_{\rho}$.
\end{Example}

\subsection{Cuspidal systems, core blocks, and RoCK blocks}\label{subsec:cuspsystems}
In this section, we show that core blocks and RoCK blocks may be distinguished via cuspidal system data. For a fixed convex preorder \(\succeq\), recall that the simple \(R_\omega\)-modules may be indexed \(\{L(\pi)_{\succeq} \mid \pi \in \Pi(\omega)\}\), so it follows that the simple \(R_\omega^\Lambda\)-modules (which lift to simple \(R_\omega\)-modules) may be indexed by \(\{L(\pi)_{\succeq} \mid \pi \in \Pi^{\Lambda}_{\succeq}(\omega)\}\) for some subset \(\Pi^{\Lambda}_{\succeq}(\omega) \subseteq \Pi(\omega)\). 

We define
\begin{align*}
\Pi(\omega)_{\succeq \delta} &:= \{\pi = (\boldsymbol{K}, \blam) \in \Pi(\omega) \mid K_\beta = 0 \textup{ for } \delta \succ \beta \in \Psi\};\\
\Pi(\omega)_{\succ \delta} &:= \{\pi = (\boldsymbol{K}, \varnothing) \in \Pi(\omega) \mid K_\beta = 0 \textup{ for } \delta \succeq \beta \in \Psi\}.
\end{align*}
That is, \(\pi = (\boldsymbol{K}, \blam)\) is in \( \Pi(\omega)_{\succeq \delta}\) (resp.~\(\Pi(\omega)_{\succ \delta}\)) provided that the Kostant partition \(\boldsymbol{K}\) only involves roots \(\succeq \delta\) (resp.~\(\succ \delta\)).

Recall that we say the cyclotomic KLR algebra \(R^\Lambda_{\omega}\) is a {\em RoCK block (resp.~core block)} provided that the associated multipartition block \(\Lambda_+^{\bkap}(\omega)\) is a RoCK block (resp.~core block) in the sense of \cref{specialdefs}. As outlined in \cref{subsec:rockblocks} this description coincides with usage in the literature \cite{CK02, Lyle22, websterScopes}.

\begin{Theorem}\label{coreRoCKcuspthm}
Let \(R^\Lambda_\omega\) be a cyclotomic KLR algebra. Then \(R^\Lambda_\omega\) is a RoCK block (resp.~core block) if and only if there exists a convex preorder \(\succeq\) on \(\Phi_+\) such that \(\Pi^\Lambda_{\succeq}(\omega) \subseteq \Pi(\omega)_{\succeq \delta}\) (resp.~\(\Pi^\Lambda_{\succeq}(\omega) \subseteq \Pi(\omega)_{\succ \delta}\)).
\end{Theorem}
\begin{proof}
\((\implies)\) Assume that \(R^\Lambda_\omega\) is a RoCK block, and choose some multicharge \(\bkap\) for \(\Lambda\). Then the simple \(R^\Lambda_\omega\)-modules arise as heads of the Specht modules \(\{S^{\blam} \mid \blam \in \Lambda^{\bkap}_+(\omega)'\}\). By \cref{RoCKthetaeq}, there is some residue permutation \(\theta\) such that \(\Lambda^{\bkap}_+(\omega)\) is \(\theta\)-RoCK. Choose some convex preorder \(\succeq\) for \(\Phi_+\) which realizes \(\theta\). Then by \cref{thetarockblocktile}, we have that \(\Lambda^{\bkap}_+(\omega) = \Lambda^{\bkap}_+(\omega)_{\succeq \delta}\). Then by \cite[Theorem~8.5]{muthtiling} we have that \(\pi \in \Pi(\omega)_{\succeq \delta}\) for all simple factors \(L(\pi)_{\succeq}\) of \(S^{\blam}\), and all \(\blam \in \Lambda_+^{\bkap}(\omega)'\). Thus we have \(\Pi^\Lambda_{\succeq}(\omega) \subseteq \Pi(\omega)_{\succeq \delta}\), as required. 
On the other hand, if \(R^\Lambda_\omega\) is a core block, we may utilize the same argument as above, merely replacing \cref{thetarockblocktile} with \cref{coretilingprop}.

\((\impliedby)\) Assume by contraposition that \(R^\Lambda_\omega\) is {\em not} \(\theta\)-RoCK. Then by \cref{thetarockblocktile}, there exists \(\blam \in \Lambda^{\bkap}_+(\omega) \backslash \Lambda^{\bkap}_+(\omega)_{\succeq \delta}\). It follows then by  \cite[Theorem~8.5]{muthtiling} that \(S^{\blam}\) has a factor \(L(\pi)_{\succeq}\) with \(\pi \notin \Pi(\omega)_{\succeq \delta}\), and hence \(\Pi^\Lambda_{\succeq}(\omega) \not\subseteq \Pi(\omega)_{\succ \delta}\), as required. On the other hand, if \(R^\Lambda_\omega\) is not a core block, we again utilize the same argument as above,  replacing \cref{thetarockblocktile} with \cref{coretilingprop}.
\end{proof}

\begin{Corollary}\label{finalcor}
Let \(R^\Lambda_\omega\) be a cyclotomic KLR algebra. Then \(R^\Lambda_\omega\) is a RoCK block (resp.~core block) if and only if there exists a convex preorder \(\succeq\) on \(\Phi_+\) such that \(\Res_{\omega- \beta, \beta} L = 0\) for all simple \(R^\Lambda_\omega\)-modules \(L\) and \(\delta \succ \beta \in \Phi_+\) (resp.~\(\delta \succeq \beta \in \Phi_+\)).
\end{Corollary}
\begin{proof}
By \cref{coreRoCKcuspthm}, \(R^\Lambda_\omega\) is a RoCK block if and only if \(\Pi_{\succeq}^\Lambda(\omega) \subseteq \Pi(\omega)_{\succeq \delta}\) for some convex preorder \(\succeq\). On the other hand, by \cite[Theorem~4.1(v)]{klesh14}, we have \(\Res_{\omega-\beta, \beta}L(\pi)_{\succeq} = 0\) for every \(\delta \succ \beta \in \Phi_+\) if and only if \(\pi \in \Pi(\omega)_{\succeq \delta}\), which proves the statement for RoCK blocks.
The proof for core blocks is similar.
\end{proof}

\section{Skew Specht modules for cuspidal systems}\label{sec:simples}

In this final section, we use our previous results to render a direct and combinatorial-flavored description of semicuspidal and simple \(R_\omega\)-modules. 
Throughout this section, we fix a convex preorder \(\succeq\) on \(\Phi_+\).

\subsection{Constructing cuspidal ribbons}\label{subsec:cusprib}

\begin{Definition}\label{def:zetaribs}
Let \(\kappa \in \ZZ\), \(\beta = \alpha(a,L) \in \Psi\), where $a\in\mathbb{Z}_e$, \(L \in \ZZ_{>0}\). Let \(u \in \N\) be a node with \(\textup{res}(u) = a\). 
We define \(\zeta(\beta,u)\) to be the ribbon constructed by beginning with the node \(u\), then iteratively adding nodes of residue \(\overline{a+1},\dots, \overline{a+L-1}\) step-by-step, either to the east or to the north of the previous node, as follows. After \(i\) steps of this process, we have constructed a ribbon of content \(\alpha(a,i)\). Then in step \(i+1\) we either:
\begin{itemize}
	\item add the node of residue \(\overline{a+i}\) to the north of the \(\overline{a+i-1}\)-node if \(\alpha(a,i) \succ \beta\), or
	\item add the node of residue \(\overline{a+i}\) to the east of the \(\overline{a+i-1}\)-node if \(\alpha(a,i) \prec \beta\).
\end{itemize}
After the \(L\)th step, we have constructed the ribbon \(\zeta(\beta,u)\) of content \(\beta\), with southwesternmost node \(u\). 
\end{Definition}

\begin{Lemma}\label{firstcuspconstlem}
\cite[Theorem~6.13]{muthtiling}
For \(\beta \in \Psi\), a ribbon \(\xi\) of content \(\beta\) is cuspidal if and only if \(\xi = \zeta(\beta,u)\) for some \(u \in \N\).
\end{Lemma}

\begin{Lemma}\label{secondcusplem}
Let \(\theta\) be a residue permutation that realizes \(\succeq\). If \(\textup{res}(u) = \overline{\theta_{e-i}+1}\) for some \(i \in [0,e-1]\), then \(\zeta(\delta, u)\) is a \((\theta,i)\)-ribbon, and thus consists of \(i+1\) rows.
\end{Lemma}
\begin{proof}
By \cref{firstcuspconstlem}, \(\zeta(\delta,u)\) is cuspidal with a northeasternmost node of residue \(\theta_{e-i}\), and thus by \cref{thetaaribcusp} must be a \((\theta,i)\)-ribbon.
\end{proof}

We now choose a distinguished set of cuspidal ribbons.

\begin{Definition}\(\)
\begin{enumerate}
\item For \(\beta = \alpha(a, L)\in \Phi_+^\re\), we fix some \(u \in \N\) with \(\textup{res}(u) = a\), and set \(\zeta(\beta):= \zeta(\beta,u)\).
\item For \(i \in [0,e-1]\), we fix some \(u \in \N\) such that \(\zeta(\delta,u)\) consists of \(i+1\) rows, and set \(\zeta_i := \zeta(\delta,u)\).
\end{enumerate}
\end{Definition}

\begin{Example}\label{distinctdeltas}
	Let \(e=4\) and \(\succeq\) be the convex preorder on \(\Phi_+\) as described in \cref{3coretileex}. Taking
	\begin{align*}
		\beta = 2\delta + \alpha_0 + \alpha_1 + \alpha_2 = \alpha(0,11) \in \Phi_+^\re,
	\end{align*}
	we construct the ribbon \(\zeta(\beta)\) as follows.
	\begin{align*}
		\hackcenter
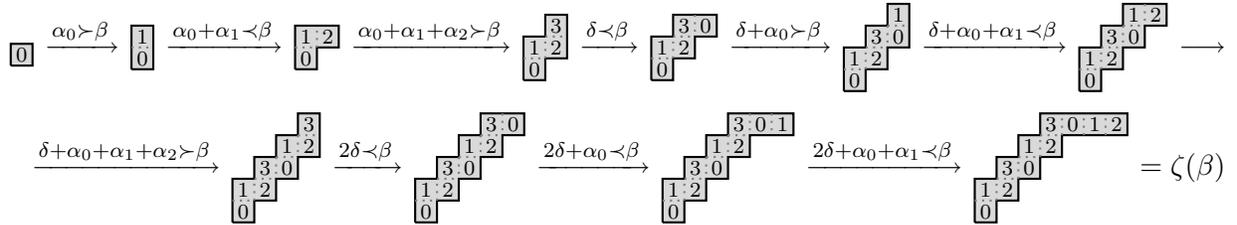
{
			\begin{tikzpicture}[scale=0.29]
				\draw[thick,fill=gray!30]  (0,0)--(0,1)--(1,1)--(1,0)--(0,0);
				\node at (0.5,0.5){$\scriptstyle 0$};
			\end{tikzpicture}
		}
		\xrightarrow{ \alpha_0 \succ \beta}
		\hackcenter{
			\begin{tikzpicture}[scale=0.29]
				\draw[thick,fill=gray!30]  (0,0)--(0,2)--(1,2)--(1,0)--(0,0);
				\node at (0.5,0.5){$\scriptstyle 0$};
				\node at (0.5,1.5){$\scriptstyle 1$};
				\draw[thick,  gray, dotted] (0,1)--(1,1);
			\end{tikzpicture}
		}
		\xrightarrow{ \alpha_0+ \alpha_1 \prec \beta}
		\hackcenter{
			\begin{tikzpicture}[scale=0.29]
				\draw[thick,fill=gray!30]  (0,0)--(0,2)--(2,2)--(2,1)--(1,1)--(1,0)--(0,0);
				\node at (0.5,0.5){$\scriptstyle 0$};
				\node at (0.5,1.5){$\scriptstyle 1$};
				\node at (1.5,1.5){$\scriptstyle 2$};
				\draw[thick,  gray, dotted] (0,1)--(1,1);
				\draw[thick,  gray, dotted] (1,1)--(1,2);
			\end{tikzpicture}
		}
		\xrightarrow{ \alpha_0+ \alpha_1 + \alpha_2 \succ \beta}
		\hackcenter{
			\begin{tikzpicture}[scale=0.29]
				\draw[thick,fill=gray!30]  (0,0)--(0,2)--(1,2)--(1,3)--(2,3)--(2,1)--(1,1)--(1,0)--(0,0);
				\node at (0.5,0.5){$\scriptstyle 0$};
				\node at (0.5,1.5){$\scriptstyle 1$};
				\node at (1.5,1.5){$\scriptstyle 2$};
				\node at (1.5,2.5){$\scriptstyle 3$};
				\draw[thick,  gray, dotted] (0,1)--(1,1);
				\draw[thick,  gray, dotted] (1,1)--(1,2);
				\draw[thick,  gray, dotted] (1,2)--(2,2);
			\end{tikzpicture}
		}
		\xrightarrow{ \delta \prec \beta}
		\hackcenter{
			\begin{tikzpicture}[scale=0.29]
				\draw[thick,fill=gray!30]  (0,0)--(0,2)--(1,2)--(1,3)--(3,3)--(3,2)--(2,2)--(2,1)--(1,1)--(1,0)--(0,0);
				\node at (0.5,0.5){$\scriptstyle 0$};
				\node at (0.5,1.5){$\scriptstyle 1$};
				\node at (1.5,1.5){$\scriptstyle 2$};
				\node at (1.5,2.5){$\scriptstyle 3$};
				\node at (2.5,2.5){$\scriptstyle 0$};
				\draw[thick,  gray, dotted] (0,1)--(1,1);
				\draw[thick,  gray, dotted] (1,1)--(1,2);
				\draw[thick,  gray, dotted] (1,2)--(2,2);
				\draw[thick,  gray, dotted] (2,3)--(2,2);
			\end{tikzpicture}
		}
		\xrightarrow{ \delta+\alpha_0 \succ \beta}
		\hackcenter{
			\begin{tikzpicture}[scale=0.29]
				\draw[thick,fill=gray!30]  (0,0)--(0,2)--(1,2)--(1,3)--(2,3)--(2,4)--(3,4)--(3,2)--(2,2)--(2,1)--(1,1)--(1,0)--(0,0);
				\node at (0.5,0.5){$\scriptstyle 0$};
				\node at (0.5,1.5){$\scriptstyle 1$};
				\node at (1.5,1.5){$\scriptstyle 2$};
				\node at (1.5,2.5){$\scriptstyle 3$};
				\node at (2.5,2.5){$\scriptstyle 0$};
				\node at (2.5,3.5){$\scriptstyle 1$};
				\draw[thick,  gray, dotted] (0,1)--(1,1);
				\draw[thick,  gray, dotted] (1,1)--(1,2);
				\draw[thick,  gray, dotted] (1,2)--(2,2);
				\draw[thick,  gray, dotted] (2,3)--(2,2);
				\draw[thick,  gray, dotted] (2,3)--(3,3);
			\end{tikzpicture}
		}
		\xrightarrow{ \delta+\alpha_0+\alpha_1 \prec \beta}
		\hackcenter{
			\begin{tikzpicture}[scale=0.29]
				\draw[thick,fill=gray!30]  (0,0)--(0,2)--(1,2)--(1,3)--(2,3)--(2,4)--(4,4)--(4,3)--(3,3)--(3,2)--(2,2)--(2,1)--(1,1)--(1,0)--(0,0);
				\node at (0.5,0.5){$\scriptstyle 0$};
				\node at (0.5,1.5){$\scriptstyle 1$};
				\node at (1.5,1.5){$\scriptstyle 2$};
				\node at (1.5,2.5){$\scriptstyle 3$};
				\node at (2.5,2.5){$\scriptstyle 0$};
				\node at (2.5,3.5){$\scriptstyle 1$};
				\node at (3.5,3.5){$\scriptstyle 2$};
				\draw[thick,  gray, dotted] (0,1)--(1,1);
				\draw[thick,  gray, dotted] (1,1)--(1,2);
				\draw[thick,  gray, dotted] (1,2)--(2,2);
				\draw[thick,  gray, dotted] (2,3)--(2,2);
				\draw[thick,  gray, dotted] (2,3)--(3,3);
				\draw[thick,  gray, dotted] (3,4)--(3,3);
			\end{tikzpicture}
		}
		\longrightarrow
		\\
		\xrightarrow{ \delta+\alpha_0+\alpha_1+\alpha_2 \succ \beta}
		\hackcenter{
			\begin{tikzpicture}[scale=0.29]
				\draw[thick,fill=gray!30]  (0,0)--(0,2)--(1,2)--(1,3)--(2,3)--(2,4)--(3,4)--(3,5)--(4,5)--(4,3)--(3,3)--(3,2)--(2,2)--(2,1)--(1,1)--(1,0)--(0,0);
				\node at (0.5,0.5){$\scriptstyle 0$};
				\node at (0.5,1.5){$\scriptstyle 1$};
				\node at (1.5,1.5){$\scriptstyle 2$};
				\node at (1.5,2.5){$\scriptstyle 3$};
				\node at (2.5,2.5){$\scriptstyle 0$};
				\node at (2.5,3.5){$\scriptstyle 1$};
				\node at (3.5,3.5){$\scriptstyle 2$};
				\node at (3.5,4.5){$\scriptstyle 3$};
				\draw[thick,  gray, dotted] (0,1)--(1,1);
				\draw[thick,  gray, dotted] (1,1)--(1,2);
				\draw[thick,  gray, dotted] (1,2)--(2,2);
				\draw[thick,  gray, dotted] (2,3)--(2,2);
				\draw[thick,  gray, dotted] (2,3)--(3,3);
				\draw[thick,  gray, dotted] (3,4)--(3,3);
				\draw[thick,  gray, dotted] (3,4)--(4,4);
			\end{tikzpicture}
		}
		\xrightarrow{ 2\delta\prec \beta}
		\hackcenter{
			\begin{tikzpicture}[scale=0.29]
				\draw[thick,fill=gray!30]  (0,0)--(0,2)--(1,2)--(1,3)--(2,3)--(2,4)--(3,4)--(3,5)--(5,5)--(5,4)--(4,4)--(4,3)--(3,3)--(3,2)--(2,2)--(2,1)--(1,1)--(1,0)--(0,0);
				\node at (0.5,0.5){$\scriptstyle 0$};
				\node at (0.5,1.5){$\scriptstyle 1$};
				\node at (1.5,1.5){$\scriptstyle 2$};
				\node at (1.5,2.5){$\scriptstyle 3$};
				\node at (2.5,2.5){$\scriptstyle 0$};
				\node at (2.5,3.5){$\scriptstyle 1$};
				\node at (3.5,3.5){$\scriptstyle 2$};
				\node at (3.5,4.5){$\scriptstyle 3$};
				\node at (4.5,4.5){$\scriptstyle 0$};
				\draw[thick,  gray, dotted] (0,1)--(1,1);
				\draw[thick,  gray, dotted] (1,1)--(1,2);
				\draw[thick,  gray, dotted] (1,2)--(2,2);
				\draw[thick,  gray, dotted] (2,3)--(2,2);
				\draw[thick,  gray, dotted] (2,3)--(3,3);
				\draw[thick,  gray, dotted] (3,4)--(3,3);
				\draw[thick,  gray, dotted] (3,4)--(4,4);
				\draw[thick,  gray, dotted] (4,5)--(4,4);
			\end{tikzpicture}
		}
		\xrightarrow{ 2\delta+\alpha_0\prec \beta}
		\hackcenter{
			\begin{tikzpicture}[scale=0.29]
				\draw[thick,fill=gray!30]  (0,0)--(0,2)--(1,2)--(1,3)--(2,3)--(2,4)--(3,4)--(3,5)--(6,5)--(6,4)--(4,4)--(4,3)--(3,3)--(3,2)--(2,2)--(2,1)--(1,1)--(1,0)--(0,0);
				\node at (0.5,0.5){$\scriptstyle 0$};
				\node at (0.5,1.5){$\scriptstyle 1$};
				\node at (1.5,1.5){$\scriptstyle 2$};
				\node at (1.5,2.5){$\scriptstyle 3$};
				\node at (2.5,2.5){$\scriptstyle 0$};
				\node at (2.5,3.5){$\scriptstyle 1$};
				\node at (3.5,3.5){$\scriptstyle 2$};
				\node at (3.5,4.5){$\scriptstyle 3$};
				\node at (4.5,4.5){$\scriptstyle 0$};
				\node at (5.5,4.5){$\scriptstyle 1$};
				\draw[thick,  gray, dotted] (0,1)--(1,1);
				\draw[thick,  gray, dotted] (1,1)--(1,2);
				\draw[thick,  gray, dotted] (1,2)--(2,2);
				\draw[thick,  gray, dotted] (2,3)--(2,2);
				\draw[thick,  gray, dotted] (2,3)--(3,3);
				\draw[thick,  gray, dotted] (3,4)--(3,3);
				\draw[thick,  gray, dotted] (3,4)--(4,4);
				\draw[thick,  gray, dotted] (4,5)--(4,4);
				\draw[thick,  gray, dotted] (5,5)--(5,4);
			\end{tikzpicture}
		}
		\xrightarrow{ 2\delta+\alpha_0+\alpha_1\prec \beta}
		\hackcenter{
			\begin{tikzpicture}[scale=0.29]
				\draw[thick,fill=gray!30]  (0,0)--(0,2)--(1,2)--(1,3)--(2,3)--(2,4)--(3,4)--(3,5)--(7,5)--(7,4)--(4,4)--(4,3)--(3,3)--(3,2)--(2,2)--(2,1)--(1,1)--(1,0)--(0,0);
				\node at (0.5,0.5){$\scriptstyle 0$};
				\node at (0.5,1.5){$\scriptstyle 1$};
				\node at (1.5,1.5){$\scriptstyle 2$};
				\node at (1.5,2.5){$\scriptstyle 3$};
				\node at (2.5,2.5){$\scriptstyle 0$};
				\node at (2.5,3.5){$\scriptstyle 1$};
				\node at (3.5,3.5){$\scriptstyle 2$};
				\node at (3.5,4.5){$\scriptstyle 3$};
				\node at (4.5,4.5){$\scriptstyle 0$};
				\node at (5.5,4.5){$\scriptstyle 1$};
				\node at (6.5,4.5){$\scriptstyle 2$};
				\draw[thick,  gray, dotted] (0,1)--(1,1);
				\draw[thick,  gray, dotted] (1,2)--(2,2);
				\draw[thick,  gray, dotted] (2,3)--(3,3);
				\draw[thick,  gray, dotted] (3,4)--(4,4);
				\draw[thick,  gray, dotted] (1,1)--(1,2);
				\draw[thick,  gray, dotted] (2,2)--(2,3);
				\draw[thick,  gray, dotted] (3,3)--(3,4);
				\draw[thick,  gray, dotted] (4,4)--(4,5);
				\draw[thick,  gray, dotted] (5,4)--(5,5);
				\draw[thick,  gray, dotted] (6,4)--(6,5);
			\end{tikzpicture}
		}
		{}
		=
		\zeta(\beta)
	\end{align*}
The corresponding ribbons $\zeta(\beta_i)$ for the indivisible real positive roots
\[
\beta_1=\alpha_2 \hspace{-0.5mm}+\hspace{-0.5mm} \alpha_3 \hspace{-0.5mm}+\hspace{-0.5mm} \alpha_0,
\quad
\beta_2=2 \delta \hspace{-0.5mm} + \hspace{-0.5mm} \alpha_0  \hspace{-0.5mm}+ \hspace{-0.5mm} \alpha_1  \hspace{-0.5mm}+ \hspace{-0.5mm} \alpha_2,
\quad
\beta_3=\delta  \hspace{-0.5mm}+ \hspace{-0.5mm} \alpha_2  \hspace{-0.5mm}+ \hspace{-0.5mm} \alpha_3,
\quad
\beta_4=\delta  \hspace{-0.5mm}+ \hspace{-0.5mm} \alpha_1
\]
in $\Phi_+$, where $\beta_1\succ\beta_2\succ\beta_3\succ\beta_4$, are
\begin{align*}
	\zeta(\beta_1) =
	{}
	\hackcenter{
		\begin{tikzpicture}[scale=0.29]
			\draw[thick,fill=gray!30]  (0,0)--(0,2)--(2,2)--(2,1)--(1,1)--(1,0)--(0,0);
			\node at (0.5,0.5){$\scriptstyle 2$};
			\node at (0.5,1.5){$\scriptstyle 3$};
			\node at (1.5,1.5){$\scriptstyle 0$};
			\draw[thick,  gray, dotted] (0,1)--(1,1);
			\draw[thick,  gray, dotted] (1,1)--(1,2);
		\end{tikzpicture}
	}
	\quad
	\zeta(\beta_2) =
	{}
	\hackcenter{
		\begin{tikzpicture}[scale=0.29]
			\draw[thick,fill=gray!30]  (0,0)--(0,2)--(1,2)--(1,3)--(2,3)--(2,4)--(3,4)--(3,5)--(7,5)--(7,4)--(4,4)--(4,3)--(3,3)--(3,2)--(2,2)--(2,1)--(1,1)--(1,0)--(0,0);
			\node at (0.5,0.5){$\scriptstyle 0$};
			\node at (0.5,1.5){$\scriptstyle 1$};
			\node at (1.5,1.5){$\scriptstyle 2$};
			\node at (1.5,2.5){$\scriptstyle 3$};
			\node at (2.5,2.5){$\scriptstyle 0$};
			\node at (2.5,3.5){$\scriptstyle 1$};
			\node at (3.5,3.5){$\scriptstyle 2$};
			\node at (3.5,4.5){$\scriptstyle 3$};
			\node at (4.5,4.5){$\scriptstyle 0$};
			\node at (5.5,4.5){$\scriptstyle 1$};
			\node at (6.5,4.5){$\scriptstyle 2$};
			\draw[thick,  gray, dotted] (0,1)--(1,1);
			\draw[thick,  gray, dotted] (1,2)--(2,2);
			\draw[thick,  gray, dotted] (2,3)--(3,3);
			\draw[thick,  gray, dotted] (3,4)--(4,4);
			\draw[thick,  gray, dotted] (1,1)--(1,2);
			\draw[thick,  gray, dotted] (2,2)--(2,3);
			\draw[thick,  gray, dotted] (3,3)--(3,4);
			\draw[thick,  gray, dotted] (4,4)--(4,5);
			\draw[thick,  gray, dotted] (5,4)--(5,5);
			\draw[thick,  gray, dotted] (6,4)--(6,5);
		\end{tikzpicture}
	}
	\quad
	\zeta(\beta_3) =
	{}
	\hackcenter{
		\begin{tikzpicture}[scale=0.29]
			\draw[thick,fill=gray!30]  (0,0)--(0,4)--(1,4)--(1,5)--(2,5)--(2,3)--(1,3)--(1,0)--(0,0);
			\node at (0.5,0.5){$\scriptstyle 2$};
			\node at (0.5,1.5){$\scriptstyle 3$};
			\node at (0.5,2.5){$\scriptstyle 0$};
			\node at (0.5,3.5){$\scriptstyle 1$};
			\node at (1.5,3.5){$\scriptstyle 2$};
			\node at (1.5,4.5){$\scriptstyle 3$};
			\draw[thick,  gray, dotted] (0,1)--(1,1);
			\draw[thick,  gray, dotted] (0,2)--(1,2);
			\draw[thick,  gray, dotted] (0,3)--(1,3);
			\draw[thick,  gray, dotted] (1,3)--(1,4);
			\draw[thick,  gray, dotted] (1,4)--(2,4);
		\end{tikzpicture}
	}
	\quad
	\zeta(\beta_4) =
	{}
	\hackcenter{
		\begin{tikzpicture}[scale=0.29]
			\draw[thick,fill=gray!30]  (0,0)--(0,1)--(1,1)--(1,2)--(2,2)--(2,3)--(3,3)--(3,1)--(2,1)--(2,0)--(0,0);
			\node at (0.5,0.5){$\scriptstyle 1$};
			\node at (1.5,0.5){$\scriptstyle 2$};
			\node at (1.5,1.5){$\scriptstyle 3$};
			\node at (2.5,1.5){$\scriptstyle 0$};
			\node at (2.5,2.5){$\scriptstyle 1$};
			\draw[thick,  gray, dotted] (1,1)--(2,1);
			\draw[thick,  gray, dotted] (2,2)--(3,2);
			\draw[thick,  gray, dotted] (1,0)--(1,1);
			\draw[thick,  gray, dotted] (2,1)--(2,2);
		\end{tikzpicture}
	},
\end{align*}
and the $e$ distinct $\delta$-ribbons corresponding to the null root $\delta$ are
\begin{align*}
	{}
	\zeta_0=
	{}
	\hackcenter{
		\begin{tikzpicture}[scale=0.29]
			\draw[thick,fill=orange!30] (7+0,20+0)--(11+0,20+0)--(11+0,21+0)--(7+0,21+0)--(7+0,20+0);
			\node at (7.5+0,20.5+0){$\scriptstyle 3$};
			\node at (8.5+0,20.5+0){$\scriptstyle 0$};
			\node at (9.5+0,20.5+0){$\scriptstyle 1$};
			\node at (10.5+0,20.5+0){$\scriptstyle 2$};
			\draw[thick, gray, dotted] (8+0,20+0)--(8+0,21+0);
			\draw[thick, gray, dotted] (9+0,20+0)--(9+0,21+0);
			\draw[thick, gray, dotted] (10+0,20+0)--(10+0,21+0);
		\end{tikzpicture}
	}
	\quad
	\zeta_1=
	\hackcenter{
		\begin{tikzpicture}[scale=0.29]
			\draw[thick,fill=green!30] (7+0,20+0)--(9+0,20+0)--(9+0,21+0)--(10+0,21+0)--(10+0,22+0)--(8+0,22+0)--(8+0,21+0)--(7+0,21+0)--(7+0,20+0);
			\node at (7.5+0,20.5+0){$\scriptstyle 1$};
			\node at (8.5+0,20.5+0){$\scriptstyle 2$};
			\node at (8.5+0,21.5+0){$\scriptstyle 3$};
			\node at (9.5+0,21.5+0){$\scriptstyle 0$};
			\draw[thick, gray, dotted] (8+0,20+0)--(8+0,21+0);
			\draw[thick, gray, dotted] (8+0,21+0)--(9+0,21+0);
			\draw[thick, gray, dotted] (9+0,21+0)--(9+0,22+0);
		\end{tikzpicture}
	}
	\quad
	\zeta_2=
	\hackcenter{
		\begin{tikzpicture}[scale=0.29]
			\draw[thick,fill=red!30] (2+0,12+0)--(3+0,12+0)--(3+0,13+0)--(4+0,13+0)--(4+0,15+0)--(3+0,15+0)--(3+0,14+0)--(2+0,14+0)--(2+0,12+0);
			\node at (2.5+0,12.5+0){$\scriptstyle 0$};
			\node at (2.5+0,13.5+0){$\scriptstyle 1$};
			\node at (3.5+0,13.5+0){$\scriptstyle 2$};
			\node at (3.5+0,14.5+0){$\scriptstyle 3$};
			\draw[thick, gray, dotted] (2+0,13+0)--(3+0,13+0);
			\draw[thick, gray, dotted] (3+0,13+0)--(3+0,14+0);
			\draw[thick, gray, dotted] (3+0,14+0)--(4+0,14+0);
		\end{tikzpicture}
	}
	\qquad
	\zeta_3=
	\hackcenter{
		\begin{tikzpicture}[scale=0.29]
			\draw[thick,fill=cyan!30] (0+0,0+0)--(1+0,0+0)--(1+0,4+0)--(0+0,4+0)--(0+0,0+0);
			\node at (0.5+0,0.5+0){$\scriptstyle 2$};
			\node at (0.5+0,1.5+0){$\scriptstyle 3$};
			\node at (0.5+0,2.5+0){$\scriptstyle 0$};
			\node at (0.5+0,3.5+0){$\scriptstyle 1$};
			\draw[thick, gray, dotted] (0+0,1+0)--(1+0,1+0);
			\draw[thick, gray, dotted] (0+0,2+0)--(1+0,2+0);
			\draw[thick, gray, dotted] (0+0,3+0)--(1+0,3+0);
		\end{tikzpicture}
	}.
\end{align*}
\end{Example}

The next lemmas follow immediately from \cref{firstcuspconstlem,secondcusplem}.

\begin{Lemma}\label{lemcuspall}
Up to \(\sim\)-similarity, the set \(\{\zeta(\beta) \mid \beta \in \Phi_+^\re \} \sqcup \{\zeta_i \mid i \in [0,e-1]\}\) comprises a complete list of cuspidal ribbons associated to the convex preorder \(\succeq\) on \(\Phi_+\).
\end{Lemma}

\begin{Lemma}\label{lem:zetaistheta}
Let \(\theta\) be the residue permutation realized by \(\succeq\). 
For each \(i \in [0,e-1]\), \(\zeta_i\) is a \((\theta,i)\)-ribbon.
\end{Lemma}

For instance, see \cref{thetaadeltaribbons} and~\cref{distinctdeltas} for a comparison of these labelings. The following is a slight rephrasing of the dilation process defined in \cite[\S7]{muthtiling}.

\begin{Definition}
Let \(\nu\) be a partition of \(d\), and let \(i \in [0,e-1]\). 
We define the skew diagram \(\textup{dil}_i(\nu) \in \Lambda(d\delta)\) by setting
\begin{align*}
\textup{dil}_i(\nu) := \bigsqcup_{(x,y) \in \nu} {\tt W}^{(x-1)(e-i-1)} {\tt S}^{(x-1)(i+1)}  {\tt E}^{(y-1)(e-i)} {\tt N}^{(y-1)i} \zeta_i.
\end{align*}
We refer to \(\textup{dil}_i(\nu)\) as the {\em \(i\)-dilation} of \(\nu\). Note that \(\textup{dil}_i((1)) = \zeta_i\).
\end{Definition}

\begin{Remark}
The definition of \(i\)-dilation above is slightly difficult to parse at a glance, but is easily described via the following heuristic. One associates to every node \(u \in \nu\) a shifted copy \(\xi(u)\) of \(\zeta_i\), which we think of as an \(i\)-dilation of the node \(u\). The ribbons \(\xi(u)\) should be chosen such that 
\begin{itemize}
\item \(\xi((1,1)) = \zeta_i\); 
\item if \(u, {\tt E}u \in \nu\), then the northeasternmost node in \(\xi(u)\) is one step to the west of the southwesternmost node in \(\xi({\tt E}u)\); and
\item if \(u, {\tt S}u \in \nu\), then the southwesternmost node in \(\xi(u)\) is one step to the north of the northeasternmost node in \(\xi({\tt S}u)\).
\end{itemize}
For example, if  $e=4$ and \(\succeq\) is the convex preorder on \(\Phi_+\) as described in \cref{3coretileex}, 
then the \(2\)-dilation of $\nu=(2^2)$ is as follows:
\begin{align*}
\nu = 
	\hackcenter{
		\begin{tikzpicture}[scale=0.29]
			\draw[thick,fill=blue!40!green!25]
			(12+1-3,8)--(13+1-3,8)--(13+1-3,9)--(12+1-3,9)--(12+1-3,8);
			\draw[thick,fill=purple!35]
			(13+1-3,8)--(14+1-3,8)--(14+1-3,9)--(13+1-3,9)--(13+1-3,8);
			\draw[thick,fill=blue!20]
			(12+1-3,7)--(13+1-3,7)--(13+1-3,8)--(12+1-3,8)--(12+1-3,7);
			\draw[thick,fill=yellow!70!orange!30]
			(13+1-3,7)--(14+1-3,7)--(14+1-3,8)--(13+1-3,8)--(13+1-3,7);
		\end{tikzpicture}
	}
	\quad
	\xrightarrow{2\textup{-dilation}}
	\quad
	\hackcenter{
		\begin{tikzpicture}[scale=0.29]
			\draw[thick,fill=blue!20] (2+0,12+0)--(3+0,12+0)--(3+0,13+0)--(4+0,13+0)--(4+0,15+0)--(3+0,15+0)--(3+0,14+0)--(2+0,14+0)--(2+0,12+0);
			\node at (2.5+0,12.5+0){$\scriptstyle 0$};
			\node at (2.5+0,13.5+0){$\scriptstyle 1$};
			\node at (3.5+0,13.5+0){$\scriptstyle 2$};
			\node at (3.5+0,14.5+0){$\scriptstyle 3$};
			\draw[thick, gray, dotted] (2+0,13+0)--(3+0,13+0);
			\draw[thick, gray, dotted] (3+0,13+0)--(3+0,14+0);
			\draw[thick, gray, dotted] (3+0,14+0)--(4+0,14+0);
			\draw[thick,fill=yellow!70!orange!30] (2+2,12+2)--(3+2,12+2)--(3+2,13+2)--(4+2,13+2)--(4+2,15+2)--(3+2,15+2)--(3+2,14+2)--(2+2,14+2)--(2+2,12+2);
			\node at (2.5+2,12.5+2){$\scriptstyle 0$};
			\node at (2.5+2,13.5+2){$\scriptstyle 1$};
			\node at (3.5+2,13.5+2){$\scriptstyle 2$};
			\node at (3.5+2,14.5+2){$\scriptstyle 3$};
			\draw[thick, gray, dotted] (2+2,13+2)--(3+2,13+2);
			\draw[thick, gray, dotted] (3+2,13+2)--(3+2,14+2);
			\draw[thick, gray, dotted] (3+2,14+2)--(4+2,14+2);
			\draw[thick,fill=blue!40!green!25] (2+1,12+3)--(3+1,12+3)--(3+1,13+3)--(4+1,13+3)--(4+1,15+3)--(3+1,15+3)--(3+1,14+3)--(2+1,14+3)--(2+1,12+3);
			\node at (2.5+1,12.5+3){$\scriptstyle 0$};
			\node at (2.5+1,13.5+3){$\scriptstyle 1$};
			\node at (3.5+1,13.5+3){$\scriptstyle 2$};
			\node at (3.5+1,14.5+3){$\scriptstyle 3$};
			\draw[thick, gray, dotted] (2+1,13+3)--(3+1,13+3);
			\draw[thick, gray, dotted] (3+1,13+3)--(3+1,14+3);
			\draw[thick, gray, dotted] (3+1,14+3)--(4+1,14+3);
			\draw[thick,fill=purple!35] (2+3,12+5)--(3+3,12+5)--(3+3,13+5)--(4+3,13+5)--(4+3,15+5)--(3+3,15+5)--(3+3,14+5)--(2+3,14+5)--(2+3,12+5);
			\node at (2.5+3,12.5+5){$\scriptstyle 0$};
			\node at (2.5+3,13.5+5){$\scriptstyle 1$};
			\node at (3.5+3,13.5+5){$\scriptstyle 2$};
			\node at (3.5+3,14.5+5){$\scriptstyle 3$};
			\draw[thick, gray, dotted] (2+3,13+5)--(3+3,13+5);
			\draw[thick, gray, dotted] (3+3,13+5)--(3+3,14+5);
			\draw[thick, gray, dotted] (3+3,14+5)--(4+3,14+5);
		\end{tikzpicture}
	}
	= \textup{dil}_2(\nu),
\end{align*}
where the colored copies of \(\zeta_2\) on the right side are the \(2\)-dilations of the nodes in \(\nu\) of the same color.
\end{Remark}

The following is \cite[Theorem~7.18]{muthtiling}.

\begin{Lemma}\label{semicuspuniqueness}
For all \(\beta \in \Phi_+^\re\) and \(K \in \ZZ_{>0}\), the skew K-diagram \(\zeta(\beta)^K = (\zeta(\beta) \mid \dots \mid \zeta(\beta))\) is (up to \(\sim\)-similarity) the unique semicuspidal skew diagram of content \(K\beta\). For each \(i \in [0,e-1]\) and partition \(\nu\), the skew diagram \(\textup{dil}_i(\nu)\) is a connected semicuspidal skew diagram of content \(|\nu|\delta\).
\end{Lemma}

\begin{Definition}\label{def:zetanu}
Given an \((e-1)\)-multipartition \(\bnu = (\nu^{(1)} \mid \dots \mid \nu^{(e-1)})\) of some \(K\in\ZZ_{>0}\), we set
\begin{align*}
\zeta(\bnu) := ( \textup{dil}_1(\nu^{(1)}) \mid \dots \mid \textup{dil}_{e-1}(\nu^{(e-1)})) \in \Lambda^{{(e-1)}}(K\delta).
\end{align*}
\end{Definition}

\begin{Example}
Let \(e=4\) and \(\succeq\) be the convex preorder on \(\Phi_+\) as described in \cref{3coretileex}.
For $\bnu=((3^2,1)\mid(2^2)\mid(2))$, we obtain the skew diagram $\zeta(\bnu)$ displayed in~\cref{firstexintro}.

\begin{Definition}\label{def:zetapi}
Let \(
\pi = (\boldsymbol{K}, \bnu) = ((\beta_1^{K_{\beta_1}} \mid \dots \mid \beta_u^{K_{\beta_u}} \mid \delta^{K_\delta} \mid \beta_{u+1}^{K_{\beta_{u+1}}} \mid \dots \mid \beta_t^{K_{\beta_t}}), \bnu) \in \Pi(\omega)
\)
be a root partition of \(\beta \in \ZZ_{\geq 0} I\). Writing \(k = e-1 + \sum_{\beta \in \Phi_+^\re} K_\beta\), we define the skew diagram \(\zeta(\pi)\) by setting
 \begin{align*}
 \zeta(\pi) := \left(
\zeta(\beta_1)^{K_{\beta_1}} \mid \cdots \mid \zeta(\beta_u)^{K_{\beta_u}} \mid \zeta(\bnu) \mid  \zeta(\beta_{u+1})^{K_{\beta_{u+1}}} \mid \cdots \mid \zeta(\beta_t)^{K_{\beta_t}}\right) \in \Lambda^{k}(\omega).
\end{align*}
\end{Definition}

\end{Example}

\subsection{Semicuspidal modules via skew Specht modules}\label{subsec:simplesviaskews}
We continue with the fixed choice of convex preorder \(\succeq\) on \(\Phi_+\). We recall the following result from \cite[Proposition~8.4]{muthtiling}.

\begin{Proposition}\label{prop:realsemis}
Let \(\beta \in \Phi_+^\re\) and \(K \in \ZZ_{>0}\). 
The real semicuspidal self-dual simple module \(L(\beta^{K})\) is isomorphic to the skew Specht module \(\zS^{\zeta({\beta})^{K}}\), up to grading shift.
\end{Proposition}

Combining the  \(K=1\) statement of \cref{prop:realsemis} with \cref{firstcuspconstlem,spechtshape} yields that real cuspidal self-dual simple modules are classified by cuspidal ribbons.

In order to work on the imaginary case, let \(\theta\) be the residue permutation realized by \(\succeq\) (see \cref{realizethetadef}).
Let \(d \in \ZZ_{\geq 0}\), and fix
a level one \(\theta\)-RoCK core \(\rho\) and a charge \(\kappa\) such that \(d\leq \textup{cap}^\theta_\delta(\rho, \kappa)\).
Indeed, one may construct such a core \(\rho\) from the abacus perspective by ensuring that there are at least \(d\) more beads on the \(\theta_{i+1}\)-runner than on the \(\theta_i\)-runner.

For an \((e-1)\)-multipartition \(\bnu = (\nu^{(1)} \mid \dots \mid \nu^{(e-1)})\) of \(d\), we define the following \(e\)-multipartition of \(d\):
\begin{align*}
\widehat \bnu := (\varnothing \mid \nu^{(1)} \mid \dots \mid \nu^{(e-1)}).
\end{align*}
Recall then that \(\rho \sqcup \widehat{\bnu}^\theta_{\rho} \in \Lambda^{\kappa}(\omega)\) is 
constructed by moving the \(q\)th largest element of \(\B_{\theta_{e-t}}(\rho,\kappa)\) up \(e \cdot \widehat{\nu}_q^{(t)}\) positions; or, from the abacus perspective, moving the \(q\)th lowest bead on the \(\theta_{e-t}\)-runner down \(\widehat{\nu}_q^{(t)}\) times.

Then we use the following notation to describe decomposition numbers in the RoCK block \(R^{\Lambda_\kappa}_{\omega + d\delta}\).
\begin{align*}
d^{\textup{RoCK}}_{\bnu, \bmu} := [S^{\rho \sqcup \widehat{\bnu}^\theta_{\rho}} : \textup{hd}(S^{\rho \sqcup \widehat{\bmu}^\theta_{\rho}})]
\end{align*}
These numbers are well-studied and have explicit formulae when \(\text{char}(\k) = 0\) or \(\text{char}(\k) > d\), see for instance \cite{jlm}. 
Whereas \(S^{\rho \sqcup \widehat{\bnu}^\theta_{\rho}}\) and \(S^{\rho \sqcup \widehat{\bmu}^\theta_{\rho}}\) depend on \(\theta, \rho\), the decomposition number \(d^{\textup{RoCK}}_{\bnu, \bmu}\) defined above does not.
Indeed, by iteratively applying Scopes equivalences (see \S\ref{subsec:rockblocks}), we have that \( [S^{\rho \sqcup \widehat{\bnu}^\theta_{\rho}} : \textup{hd}(S^{\rho \sqcup \widehat{\bmu}^\theta_{\rho}})] =  [S^{\rho \sqcup \widehat{\bnu}^{\theta_0}_{\rho}} : \textup{hd}(S^{\rho \sqcup \widehat{\bmu}^{\theta_0}_{\rho}})]\), where \(\theta_0 = (0,1,\ldots, e-1)\) is the trivial residue permutation. In this setting \(\bnu, \bmu\) are the so-called `\(e\)-quotients' of the Young diagrams \(\rho \sqcup \widehat{\bnu}^{\theta_0}_{\rho},\rho \sqcup \widehat{\bmu}^{\theta_0}_{\rho} \) in a Rouquier block, and the decomposition numbers \( d^{\textup{RoCK}}_{\bnu, \bmu}  = [S^{\rho \sqcup \widehat{\bnu}^{\theta_0}_{\rho}} : \textup{hd}(S^{\rho \sqcup \widehat{\bmu}^{\theta_0}_{\rho}})]\) are well-studied, do not depend on \(\rho\), and have explicit formulae when \(\text{char}(\k) = 0\) or \(\text{char}(\k) > d\), see for instance \cite{jlm}. 



Recall from \cref{cuspsyssec} the imaginary semicuspidal simple modules may be labeled \(L(\bnu)\) for the set of \((e-1)\)-multipartitions \(\bnu\).
%
%
%
We complete the description of semicuspidal simple modules via skew Specht modules by describing these \(L(\bnu)\) in \cref{thm:Asimp}.
The next result is Theorem~\hyperlink{thm:A}{A} in the introduction.

\begin{Theorem}\label{thm:Asimp}
Let \(\bnu\) be an \((e-1)\)-multipartition of \(d\).
Then the following statements hold.
\begin{enumerate}
\item
The skew Specht module \(\zS^{\zeta(\bnu)}\) is an indecomposable semicuspidal \(R_{d \delta}\)-module, with simple semicuspidal head \(\textup{hd}(\zS^{\zeta(\bnu)})\).

\item
Let \(L(\bnu)\) be the unique self-dual grading shift of \(\textup{hd}(\zS^{\zeta(\bnu)})\).
Then
\[
\{L(\bnu) \mid \bnu \text{ an $(e-1)$-multipartition of }d\}
\]
is a complete and irredundant set of simple imaginary semicuspidal \(R_{d \delta}\)-modules, up to isomorphism and grading shift.
\item
For any \((e-1)\)-multipartition \(\bmu\) of \(d\), we have \([\zS^{\zeta(\bnu)}: L(\bmu)] = d^{\textup{RoCK}}_{\bnu, \bmu}\).
\end{enumerate}
\end{Theorem}

\begin{proof}
It follows from \cref{lem:zetaistheta,allaboutshapes} that for every \((e-1)\)-multipartition \(\bnu\) of \(d\), we have \(\widehat{\bnu}^\theta_{\rho} \sim \zeta(\bnu)\).
Then parts (i) and (ii) are just a recasting of \cref{semicuspthm} with the new notation introduced in \cref{subsec:cusprib}, utilizing the fact that simple modules have a unique shift which makes them self-dual.

For part (iii), we recall that by \cref{Morcut}, there is a Morita equivalence
\(
\mathcal{T}: R^{\Lambda}_{\omega + d \delta}\textup{-mod} \to R^{\Lambda/\omega}_{ d \delta}\textup{-mod}
\)
such that
 \(\mathcal{T}S^{\rho \sqcup \widehat{\bnu}^\theta_{\rho}} \approx \zS^{\zeta(\bnu)}\)
for each \((e-1)\)-multipartition \(\bnu\) of \(d\),
and
\(\mathcal{T}D^{\rho \sqcup \widehat{\bmu}^\theta_{\rho}} \approx \textup{hd}(S^{\zeta(\bmu)})\)
for each \((e-1)\)-multipartition.
It now follows that
\(
[\zS^{\zeta(\bnu)}: L(\bmu)] = [\zS^{\zeta(\bnu)}: \textup{hd}(\zS^{\zeta(\bmu)})] = [\zS^{\rho\sqcup \widehat{\bnu}^\theta_{\rho}}: \textup{hd}(\zS^{\rho\sqcup\widehat{\bmu}^\theta_{\rho}})] = d^{\textup{RoCK}}_{\bnu, \bmu}
\). 
\end{proof}

Recall that $\Xi(\omega)$ is the set of all Kostant partitions of $\omega$.
\begin{Definition}
	There is a right lexicographic total order $\geqslant_R$ on $\Xi(\omega)$, where
	\[
	\boldsymbol{K}>_{\textup{R}}\boldsymbol{K}'
	\quad
	\iff
	\quad
	\text{there exists $\beta\in\Psi$ such that $K_{\beta}>K'_{\beta}$ and $K_{\beta'}=K'_{\beta'}$ for all $\beta\succ\beta'$.}
	\]
	There is a left lexicographic total order $\geqslant_L$ on $\Xi(\omega)$, where
	\[
	\boldsymbol{K}>_{\textup{L}}\boldsymbol{K}'
	\quad
	\iff
	\quad
	\text{there exists $\beta\in\Psi$ such that $K_{\beta}>K'_{\beta}$ and $K_{\beta'}=K'_{\beta'}$ for all $\beta\prec\beta'$.}
	\]
	These right and left orders induce a \emph{bilexicographic} partial order $\geqslant_{\textup{b}}$ on $\Xi(\omega)$, defined by
	\[
	\boldsymbol{K}\geq_{\textup{b}}\boldsymbol{K}'
	\quad
	\iff
	\quad
	\boldsymbol{K}\geqslant_{\textup{R}}\boldsymbol{K}'
	\text{ and }
	\boldsymbol{K}\geqslant_{\textup{L}}\boldsymbol{K}'.
	\]
\end{Definition}

We now refine the bilexicographic order as follows.
\begin{Definition}
	There is a \emph{bilexicographic dominance} partial order $\geqslant_{\textup{bd}}$ on $\Pi(\omega)$, defined by
	\[
	\pi\geq_{\textup{bd}}\pi'
	\quad
	\iff
	\quad
	\boldsymbol{K}>_{\textup{b}}\boldsymbol{K}'
	\quad
	\text{ or }
	\quad
	\boldsymbol{K}=\boldsymbol{K}'
	\text{ and }
	\bnu\trianglerighteq^{D}\bnu',
	\]
	where $\pi = (\boldsymbol{K}, \bnu)$ and $\pi' = (\boldsymbol{K}', \bnu')$.
\end{Definition}

\begin{Definition}
Let \(\boldsymbol{K} = (\beta_1^{K_{\beta_1}} \mid \dots \mid \beta_u^{K_{\beta_u}} \mid \delta^{K_\delta} \mid \beta_{u+1}^{K_{\beta_{u+1}}} \mid \dots \mid \beta_t^{K_{\beta_t}}) \in \Xi(\omega).
\)
We define the parabolic subalgebra
\[
R_{\boldsymbol{K}} := R_{K_{\beta_1}\beta_1} \otimes R_{K_{\beta_2}\beta_2} \otimes \dots \otimes R_{K_{\beta_u}\beta_u} \otimes R_{K_\delta \delta} \otimes R_{K_{\beta_{u+1}}\beta_{u+1}} \otimes \dots \otimes R_{K_{\beta_t}\beta_t},
\]
of \(R_\omega\), 
and let \(\Ind_{\boldsymbol{K}}^\omega,\Res_{\boldsymbol{K}}^\omega\) denote the exact induction and restriction functors, respectively,  between \(R_\omega\textup{-mod} \) and \(R_{\boldsymbol{K}}\textup{-mod} \).
These functors are given by
\[
\Ind_{\boldsymbol{K}}^\omega: M \longmapsto R_\omega 1_{\boldsymbol{K}} \otimes_{R_{\boldsymbol{K}}} M
\quad \text{ and } \quad
\Res_{\boldsymbol{K}}^\omega: M \longmapsto 1_{\boldsymbol{K}} R_\omega \otimes_{R_\omega} M.
\]
\end{Definition}

By \cref{prop:realsemis,thm:Asimp}, we have now constructed every imaginary semicuspidal simple module \(L(\bnu)\), and every real semicuspidal simple module \(L(\beta^K)\), as (the head of) an explicit skew Specht module.  Recall from \cref{stratasec} that from this data we construct a proper standard module \(\bar{\Delta}(\pi)\) and simple module \(L(\pi)\) associated to every root partition \(\pi \in \Pi(\omega)\).
Our next main result -- Theorem~\hyperlink{thm:B}{B} in the introduction -- directly describes these simple KLR-modules as heads of explicit skew Specht modules.

\begin{Theorem}\label{thm:klrsimples}
Let \(\pi = (\boldsymbol{K}, \bnu) \in \Pi(\omega)\).
Then the following statements hold.
\begin{enumerate}
\item The skew Specht module \(\zS^{\zeta(\pi)}\) is indecomposable with simple head. We have
\begin{align*} 
\textup{hd}(\zS^{\zeta(\pi)}) \approx L(\pi),
\end{align*} 
so that \(\{\textup{hd}(\zS^{\zeta(\pi)}) \mid \pi \in \Pi(\omega)\}\) gives a complete and irredundant set of simple \(R_\omega\)-modules up to grading shift.

\item For \(\sigma \in \Pi(\omega)\), we have \([\zS^{\zeta(\pi)} : L(\pi)] = 1\), and \([\zS^{\zeta(\pi)} : L(\sigma)] >0\) only if \(\sigma \leq_{\textup{bd}} \pi\). Moreover, for all root partitions of the form \((\boldsymbol{K}, \bmu) \in \Pi(\omega)\), we have \([\zS^{\zeta(\pi)} : L(\boldsymbol{K},\bmu)] = d^{\textup{RoCK}}_{\bnu, \bmu}\).

\item There exists a surjection \(\zS^{\zeta(\pi)} \twoheadrightarrow \bar{\Delta}(\pi)\), and \(\zS^{\zeta(\pi)}\) has a filtration by proper standard modules of the form \(\bar{\Delta}(\boldsymbol{K}, \bmu)\), where  \((\zS^{\zeta(\pi)} : \bar{\Delta}(\boldsymbol{K}, \bmu)) = d^{\textup{RoCK}}_{\bnu, \bmu}\). 
\end{enumerate}
\end{Theorem}

\begin{proof} \(\)
We begin with (i).
As in \cite{km17a}, we define \(L_\pi\) to be the following module over the parabolic subalgebra \(R_{\boldsymbol{K}}\),
\[
L_\pi := L({\beta_1^{K_{\beta_1}}}) \boxtimes \dots \boxtimes L({\beta_u^{K_{\beta_u}}}) \boxtimes
L(\bnu) \boxtimes
L(\beta_{u+1}^{K_{\beta_{u+1}}}) \boxtimes \dots \boxtimes L(\beta_t^{K_{\beta_t}}),
\]
and then define the proper standard module \(\bar\Delta(\pi)\) and the simple module \(L(\pi)\) to be
\[
\bar\Delta(\pi) := \Ind_{\boldsymbol{K}}^\omega L_\pi
\qquad \text{ and } \qquad
L(\pi) :=  \textup{hd}(\bar\Delta(\pi)).
\]
Now we define the \(R_{\boldsymbol{K}}\)-module
\begin{align*}
\widehat{L}_\pi &:=  L({\beta_1^{K_{\beta_1}}}) \boxtimes \dots \boxtimes L({\beta_u^{K_{\beta_u}}}) \boxtimes
\zS^{\zeta(\bnu)} \boxtimes
L(\beta_{u+1}^{K_{\beta_{u+1}}}) \boxtimes \dots \boxtimes L(\beta_t^{K_{\beta_t}}),
\end{align*}
from which we define the \(R_\omega\)-modules
\begin{align*}
\widehat{\Delta}(\pi) &:= \Ind_{\boldsymbol{K}}^\omega \widehat{L}_\pi \qquad \text{ and } \qquad \widehat{L}(\pi) :=  \textup{hd}(\widehat{\Delta}(\pi)).
\end{align*}
By \cref{prop:realsemis}, \(L({\beta_i^{K_{\beta_i}}}) \approx \zS^{\zeta({\beta_i})^{K_{\beta_i}}}\) for each \(i\), so that \(\widehat{\Delta}(\pi) \approx \zS^{\zeta(\pi)}\),
and the first claim in the theorem amounts to showing that \(L(\pi) \approx \widehat{L}(\pi)\).
By \cref{Morcut}, \(\textup{hd}(\zS^{\zeta(\bnu)}) \approx L(\bnu)\) and all other composition factors of \(\zS^{\zeta(\bnu)}\) are isomorphic to simple modules \(L(\bmu)\) up to grading shift, with \(\bmu\) an \((e-1)\)-multipartition of \(K_\delta\).
It follows that \(\widehat{\Delta}(\pi)\) is filtered by (grading shifts of) proper standard modules \(\bar{\Delta}(\pi)\) and \(\bar{\Delta}(\pi')\) for root partitions \(\pi' = (\boldsymbol{K}, \bmu) \in \Pi(\beta)\).
So in the Grothendieck group for \(R_\omega\textup{-mod}\), we have some \(g_{\pi'} \in \ZZ_{\geq0}\) such that
\[
[\widehat{L}_\pi] = [L_\pi] + \sum_{\pi' \neq \pi} g_{\pi'}[L_{\pi'}],
\]
where the \(\pi'\) elements in the sum are all of the form \(\pi' = (\boldsymbol{K}, \bmu) \in \Pi(\omega)\), and therefore 
\[
[\widehat{\Delta}(\pi)] = [\bar{\Delta}(\pi)] + \sum_{\pi' \neq \pi} g_{\pi'}[\bar{\Delta}(\pi')].
\]
It follows that
\begin{align*}
[\Res_{\boldsymbol{K}}^\omega \widehat{\Delta}(\pi)] &= [\Res_{\boldsymbol{K}}^\omega \bar{\Delta}(\pi)] + \sum_{\pi' \neq \pi} g_{\pi'}[\Res_{\boldsymbol{K}}^\omega \bar{\Delta}(\pi')]\\
&= [L_\pi] + \sum_{\pi' \neq \pi} g_{\pi'}[L_{\pi'}] = [\widehat{L}_\pi]
\end{align*}
by \cite[Corollary~6.9]{km17a}.
By the Mackey Theorem \cite[Proposition~2.18]{kl09},
\[
\Res_{\boldsymbol{K}}^\omega \widehat{\Delta}(\pi) = \Res_{\boldsymbol{K}}^\omega \Ind_{\boldsymbol{K}}^\beta \widehat{L}_\pi
\]
has a filtration with \(\widehat{L}_\pi\) as one of the layers, so it follows that in fact
\[
\Res_{\boldsymbol{K}}^\omega \widehat{\Delta}(\pi) \cong \widehat{L}_\pi.
\]
Now, by construction we have that \(\widehat{L}_\pi\) surjects onto \(L_\pi\), so by exactness of \(\Ind_{\boldsymbol{K}}^\omega\), we have surjections
\[
\widehat{\Delta}(\pi) \longtwoheadrightarrow \bar{\Delta}(\pi) \longtwoheadrightarrow L(\pi),
\]
composing to give a surjection
\(\widehat{\Delta}(\pi) \twoheadrightarrow L(\pi)\).
It follows that \(L(\pi)\) appears in the head of \(\widehat{\Delta}(\pi)\), up to some shift.

On the other hand, suppose that we have some other simple module \(L(\pi')\), and a surjection \(\widehat{\Delta}(\pi) \twoheadrightarrow L(\pi')\).
Since \(\widehat{\Delta}(\pi)\) has a filtration by (indecomposable) proper standard modules \(\bar{\Delta}(\chi)\) for root partitions of the form  \(\chi = (\boldsymbol{K}, \bmu) \in \Pi(\omega)\), and each \(\bar{\Delta}(\chi)\) has simple head \(L(\chi)\), it must be that \(\pi' = \chi\) for some \(\chi = (\boldsymbol{K}, \bmu) \in \Pi(\omega)\).
Then by \cite[Theorem~6.8(v)]{km17a}, \(\Res_{\boldsymbol{K}}^\omega L(\pi') \cong L_{\pi'}\)
Now, by exactness of \(\Res_{\boldsymbol{K}}^\omega\),
\[
\widehat{L}_\pi
\cong
\Res_{\boldsymbol{K}}^\omega \widehat{\Delta}(\pi)
\longtwoheadrightarrow
\Res_{\boldsymbol{K}}^\omega L(\pi')
\cong
L_{\pi'}.
\]
Since \(\widehat{L}_\pi\) has simple head \(L_{\pi}\), it follows that in fact \(\pi' = \pi\).
Therefore, the head of \(\widehat{\Delta}(\pi)\) contains only the simple module \(L(\pi)\) up to some shift, possibly with multiplicity greater than one.

To rule out this multiplicity, we note again that \(\widehat{\Delta}(\pi)\) is filtered by proper standard modules \(\bar{\Delta}(\pi)\) and \(\bar{\Delta}(\pi')\) for root partitions \(\pi' = (\boldsymbol{K}, \bmu) \in \Pi(\omega)\), and any composition factors of \(\bar{\Delta}(\pi')\) besides \(L(\pi')\) are of the form \(L(\chi)\), for \(\chi <{_\textup{b}} \pi'\), by \cite[Theorem~6.8(iv)]{km17a}, and therefore  \(\chi <{_\textup{b}} \pi\),
and so
\[
L(\pi) \approx \textup{hd}(\zS^{\zeta(\pi)})
\]
as \(R_\omega\)-modules, as claimed.

For the proof of (ii), note that since \(L(\pi)\) appears just once in \(\bar{\Delta}(\pi)\), and is not a composition factor of any of the other \(\bar{\Delta}(\pi')\) modules filtering \(\widehat{\Delta}(\pi) \approx \zS^{\zeta(\pi)}\), the multiplicity of \(L(\pi)\) as a composition factor of \(\zS^{\zeta(\pi)}\) is one.
Similarly, we see that all other composition factors of \(\zS^{\zeta(\pi)}\) are those \(L(\sigma)\) with \(\sigma = (\boldsymbol{K}, \bmu)\), for \(\bmu \trianglelefteq^{D} \bnu\), so that \(\sigma \leq_{\textup{bd}} \pi\), as well as other composition factors \(L(\chi)\) of \(\bar{\Delta}(\sigma)\), which satisfy \(\chi <_{\textup{b}} \pi\).

For the proof of (iii), note that
we already explained in the proof of part (i) that there exists a surjection
\(\widehat{\Delta}(\pi) \approx \zS^{\zeta(\pi)} \twoheadrightarrow \bar{\Delta}(\pi)\)
and that \(\zS^{\zeta(\pi)}\) has a filtration by proper standard modules of the form \(\bar{\Delta}(\boldsymbol{K}, \bmu)\).
To prove that the multiplicity of \(\bar{\Delta}(\boldsymbol{K}, \bmu)\) in this filtration of \(\zS^{\zeta(\pi)}\) is equal to 
\(d^{\textup{RoCK}}_{\bnu, \bmu}\),
it suffices to recall that each simple module of the form \(L(\boldsymbol{K}, \bmu)\) only occurs in the heads of the corresponding proper standard modules \(\bar{\Delta}(\boldsymbol{K}, \bmu)\).
Thus we have that
\(
(\zS^{\zeta(\pi)} : \bar{\Delta}(\boldsymbol{K}, \bmu)) = [\zS^{\zeta(\pi)} : L(\boldsymbol{K}, \bmu)] = d^{\textup{RoCK}}_{\bnu, \bmu},
\)
by part (ii), and the proof is complete.\qedhere
\end{proof}

\subsection{A cuspidal `regularization' theorem for Specht modules}\label{subsec:regularize}
As an application of \cref{thm:klrsimples}, we finish with a result which combinatorially identifies factors \(L(\pi)\) in Specht modules via the skew diagram \(\zeta(\pi)\).
First we briefly introduce some other notions of skew diagram tilings which will be useful in this section.
\begin{Definition}
Let \(\btau\) be a skew diagram with \(\cont(\btau) = \omega\). We say that a skew tableau \((\Gamma, {\tt t})\) for \(\btau\) is:
\begin{enumerate}
\item {\em semicuspidal} provided that each tile \(\gamma \in \Gamma\) is a semicuspidal diagram of content \(K \beta\) for some \(K \in \ZZ_{>0}\) and \(\beta \in \Psi\);
\item {\em Kostant} provided that for \(i =1, \ldots, |\Gamma|\), we have \({\tt t}(i) = K_i \beta_i\) for some \(K_i \in \ZZ_{>0}\) and \(\beta_i \in \Psi\), with \(\beta_i \succeq \beta_j\) whenever \(i < j\);
\item {\em strict Kostant} provided that for \(i =1, \ldots, |\Gamma|\), we have \({\tt t}(i) = K_i \beta_i\) for some \(K_i \in \ZZ_{>0}\) and \(\beta_i \in \Psi\), with \(\beta_i \succ \beta_j\) whenever \(i < j\). In this setting we write
\begin{align*}
\boldsymbol{K}_\Gamma = (K_1 \beta_1  \mid \dots \mid K_{|\Gamma|} \beta_{|\Gamma|}) \in \Xi(\omega)
\end{align*}
for the Kostant partition associated with the tiling \(\Gamma\).
\end{enumerate}
We will attach these adjectives to tilings as well, saying for instance that \(\Gamma\) is a {\em semicuspidal tiling} provided that a semicuspidal tableau \((\Gamma, {\tt t})\) exists.
\end{Definition}

We have the following by \cite[Theorem~C and Theorem~7.19]{muthtiling}.
\begin{Proposition}
Let \(\btau \in \Lambda^\ell(\omega)\) be a skew diagram. There exists a unique semicuspidal strict Kostant tiling \(\Gamma^{\textup{sc}}_{\btau}\) of \(\btau\). Each semicuspidal tile in \(\Gamma^{\textup{sc}}_{\btau}\) is the union of all same-content tiles in the unique cuspidal tiling \(\Gamma_{\btau}\) of \(\btau\).
\end{Proposition}

We include here a slight upgrade of this result.
\begin{Proposition}\label{lem:uniquescskostant}
Let \(\btau \in \Lambda^\ell(\omega)\) be a skew diagram. Then \(\Gamma^{\textup{sc}}_{\btau}\) is the unique strict Kostant tiling \(\Gamma'\) of \(\btau\) with \(\boldsymbol{K}_{\Gamma'} = \boldsymbol{K}_{\Gamma_{\btau}^{\textup{sc}}}\).
\end{Proposition}

\begin{proof}
Say \(\Gamma_{\btau}\) ends with \(K\) cuspidal ribbons of content \(\beta \in \Psi\). Let \(\bmu\) be the union of these cuspidal ribbons. It follows then that the last tile in \(\Gamma^{\textup{sc}}_{\btau}\) is \(\bmu\), which has content \(K \beta\). Let \(\bmu'\) be any \({\tt SE}\)-removable skew diagram in \(\btau\) with content \(K\beta\). Consider the unique cuspidal Kostant tiling \(\Gamma_{\bmu'}\) of \(\bmu'\). If any ribbon in this tiling has content \(\prec \beta\), it would follow that \(\bmu'\) (and therefore \(\btau\)) has a {\tt SE}-removable ribbon of content \(\prec \beta\). But this contradicts \cref{tilethm}(ii). Thus every ribbon in \(\Gamma_{\bmu'}\) has content \(\succeq \beta\), and thus every ribbon in \(\Gamma_{\bmu'}\) has content \(\beta\) by \cite[Lemma 3.5(i),(ii)]{muthtiling}. It follows then from \cref{tilethm}(i) that \(\Gamma_{\bmu'}\) consists of exactly the \(K\) cuspidal \(\beta\)-ribbons that form \(\bmu\), so \(\bmu' = \bmu\). The claim follows by induction on \(\btau/ \bmu\). 
\end{proof}

For \(\pi \in \Pi(\omega)\), we say that a skew diagram \(\btau\) {\em has a \(\zeta(\pi)\)-tiling} provided that there exists a skew tableau \((\Gamma, {\tt t})\) for \(\btau\) such that 
\begin{align*}
({\tt t}(1) \mid \dots \mid {\tt t}(|\Gamma|)) \sim \zeta(\pi).
\end{align*}
The following two theorems comprise Theorem~\hyperlink{thm:F}{F} from the introduction.

\begin{Theorem}\label{F1}
Let \(\omega \in \ZZ_{\geq 0}I\), \(\bkap\) be a multicharge, and \(\btau \in \Lambda_+^{\bkap}(\omega)\). Assume that \(\pi \in \Pi(\omega)\) is such that \(\btau\) has a \(\zeta(\pi)\)-tiling. Then there exists a nonzero map \(\zS^{\zeta(\pi)} \to \zS^{\btau}\), and we have \([\zS^{\btau} : L(\pi)] = 1\) and \([\zS^{\btau} : L(\sigma)] > 0\) only if \(\sigma \leq_{\textup{bd}} \pi\).
\end{Theorem}

\begin{proof}
Assume that
\[
\pi = (\boldsymbol{K}, \bnu) = ((\beta_1^{K_{\beta_1}} \mid \dots \mid \beta_u^{K_{\beta_u}} \mid \delta^{K_\delta} \mid \beta_{u+1}^{K_{\beta_{u+1}}} \mid \dots \mid \beta_t^{K_{\beta_t}}), \bnu) \in \Pi(\omega)
\]
is a root partition such that \(\btau\) has a \(\zeta(\pi)\)-tiling.
First we show that
\begin{align}\label{claim1234}
\Res_{\boldsymbol{K}}^\omega \zS^{\btau} \cong \widehat{L}_\pi = \zS^{\zeta({\beta_1})^{K_{\beta_1}}} \boxtimes \dots \boxtimes \zS^{\zeta({\beta_u})^{K_{\beta_u}}} \boxtimes
\zS^{\zeta(\bnu)} \boxtimes
\zS^{\zeta({\beta_{u+1}})^{K_{\beta_{u+1}}}} \boxtimes \dots \boxtimes \zS^{\zeta({\beta_t})^{K_{\beta_t}}}.
\end{align}
By iteratively applying \cite[Theorem~5.13]{muthskew}, \(
\Res_{\boldsymbol{K}}^\omega \zS^{\btau}\) has a filtration whose subquotients are all modules of the form 
\[
\zS^{\blam_1} \boxtimes \zS^{\blam_2/{\blam_1}} \boxtimes
\zS^{\blam_3/{\blam_2}}
\boxtimes \dots \boxtimes
\zS^{\blam_u/{\blam_{u-1}}}
\boxtimes
\zS^{\blam_\delta/{\blam_u}}
\boxtimes
\zS^{\blam_{u+1}/{\blam_\delta}}
\boxtimes
\zS^{\blam_{u+2}/{\blam_{u+1}}}\boxtimes
\dots
\boxtimes
\zS^{\blam_{t}/{\blam_{t-1}}},
\]
where \(\blam_{t} = \btau\), each \(\blam_{i}/{\blam_{i-1}}\) has content \(K_{\beta_{i}}\beta_i\), \(\blam_\delta/{\blam_u}\) has content \(K_\delta \delta\), and \(\blam_{u+1}/{\blam_\delta}\) has content \(K_{\beta_{u+1}} \beta_{u+1}\).
Since \(\btau\) has a \(\zeta(\pi)\)-tiling, and each constituent diagram in \(\zeta(\pi)\) is semicuspidal, this implies that \(\boldsymbol{K} = \boldsymbol{K}_{\Gamma_{\btau}^\textup{sc}}\), the Kostant partition associated to the unique semicuspidal strict Kostant tiling of \(\btau\). The existence of the \(\zeta(\pi)\)-tiling also shows that there is a subquotient of \(\Res_{\boldsymbol{K}}^\omega \zS^{\btau}\) isomorphic to \(\widehat{L}_\pi\).
Any other subquotients of the above form would naturally yield another Kostant tiling of \(\btau\) with the same associated Kostant partition, which cannot happen, by \cref{lem:uniquescskostant}. This verifies that (\ref{claim1234}) holds.

Recall from the proof of \cref{thm:klrsimples} that \(\zS^{\zeta(\pi)} \approx \widehat{\Delta}(\pi) \approx \Ind_{\boldsymbol{K}}^\omega \widehat{L}_\pi\).
Then by Frobenius Reciprocity, we have a nonzero map \(\zS^{\zeta(\pi)} \to \zS^{\btau}\) if and only if we have a nonzero map \(\widehat{L}_\pi \to \Res_{\boldsymbol{K}}^\omega \zS^{\btau} \cong \widehat{L}_\pi\).
Thus this nonzero map does indeed exist, and by part (i) of \cref{thm:klrsimples}, the head of the image of this map is a copy of \(L(\pi)\) in \(\zS^{\btau}\).

To see that \(L(\pi)\) cannot appear in \(\zS^{\btau}\) with multiplicity greater than one, we compare
\begin{align}\label{resStau}
\Res_{\boldsymbol{K}}^\omega \zS^{\btau} \cong \widehat{L}_\pi
\quad \text{ and } \quad
\Res_{\boldsymbol{K}}^\omega L(\pi) \cong L_\pi,
\end{align}
by the proof of part (i) of \cref{thm:klrsimples} and \cite[Theorem~6.8(v)]{km17a}, respectively.
Since the multiplicity of \(L_\pi\) in \(\widehat{L}_\pi\) is only one, the result follows. 

Finally, assume \(\sigma = (\boldsymbol{K}', \bnu') \in \Pi(\omega)\) is such that \(L(\sigma)\) arises as a factor of \(\zS^{\btau}\). If  \(\boldsymbol{K}' \neq \boldsymbol{K}\), then it follows from \cite[Theorem 8.5]{muthtiling} that \(\boldsymbol{K}' <_{\textup{b}} \boldsymbol{K}\). On the other hand, if \(\boldsymbol{K}' = \boldsymbol{K}\),  then \(L_{\sigma}\) must be a factor of \(\widehat{L}_\pi\) by (\ref{resStau}). But then we have \(\bnu' \trianglelefteq^D \bnu\) as in the proof of \cref{thm:klrsimples}(ii). Then in any case we have \(\sigma \leq_{\textup{bd}} \pi\), completing the proof.
\end{proof}

The above theorem allows us to identify a simple factor of \(\zS^{\btau}\), together with bounding information on other simple factors that can arise, provided one can find a \(\zeta(\pi)\)-tiling of \(\btau\). However, it is not always easy (or possible) to construct such a tiling. In our final theorem, we note that when \(\btau=\tau\) is an \(e\)-restricted partition (and hence labels a simple module in level one), one can directly construct such a \(\zeta(\pi)\)-tiling, allowing us to implement this program in this setting.

\begin{Theorem}\label{F2}
Let \(\tau \in \Lambda_+^\kappa(\omega)\) be an \(e\)-restricted (level one) partition.
Let \(\Gamma_\tau^{\textup{sc}}\) be the unique semicuspidal strict Kostant tiling of \(\tau\), with associated tableau \(({\tt t}(1), \ldots, {\tt t}(m))\). Then
\begin{align*}
\zeta(\pi) \sim ( {\tt t}(1) \mid \dots \mid {\tt t}(m)) \in \Lambda^m(\omega) 
\end{align*}
for some \(\pi \in \Pi(\omega)\). Then there exists a nonzero map \(\zS^{\zeta(\pi)} \to S^{\tau}\), and we have \([S^{\tau} : L(\pi)] = 1\) and \([S^{\tau} : L(\sigma)] > 0\) only if \(\sigma \leq_{\textup{bd}} \pi\).
\end{Theorem}

\begin{proof}
By \cref{F1}, it is enough to demonstrate that \(\zeta(\pi) \sim ( {\tt t}(1) \mid \dots \mid {\tt t}(m))\) for some \(\pi \in \Pi(\omega)\). Let \(i \in [1,|\Gamma^{\textup{sc}}_\tau|]\). We have then that \({\tt t}(i)\) is a semicuspidal diagram of content \( K_i \beta_i\) for some \(K_i \in \ZZ_{>0}\) and \(\beta \in \Psi\).  If  \(\beta_i \in \Phi_+^\re\), then by \cref{semicuspuniqueness}, we have \({\tt t}(i) \sim \zeta(\beta_i)^{K_i}\).

Now assume \(\beta_i = \delta\). We set \(\omega' = \sum_{q=1}^{i-1} \cont({\tt t}(q))\) and
\begin{align*}
\rho := {\tt t}(1) \sqcup \cdots \sqcup {\tt t}(i-1) \in \Lambda_+^\kappa(\omega')_{\succ \delta}.
\end{align*}
Letting \(\theta\) be the residue permutation realized by \(\succeq\), we have by \cref{RoCKsemipar,RoCKmultitiling} that \(\rho\) is a \((\kappa, \theta)\)-RoCK core partition, and \(\rho \sqcup {\tt t}(i)\) is a \((\kappa, \theta)\)-RoCK multipartition.

In considering the skew diagram \({\tt t}(i)\) up to \(\sim\)-similarity, we will assume without loss of generality that \(K_i \leq \textup{cap}^\theta_\delta(\rho,\kappa)\). Indeed, if this were not the case we could replace \(\rho \sqcup {\tt t}(i)\) with a larger \((\kappa,\theta)\)-RoCK multipartition \(\rho' \sqcup {\tt t}(i)'\) by `stretching' the abacus for \(\rho \sqcup {\tt t}(i)\) as in \cite[\S3.2]{Lyle22}. This process yields a larger \((\kappa, \theta)\)-RoCK core \(\rho'\) with \(\textup{cap}^\theta_\delta(\rho',\kappa)\geq K_i\), and \({\tt t}(i)' \sim {\tt t}(i)\) (the connected components of \({\tt t}(i)\) are merely moved further apart by the stretching process).

Therefore, given the assumption that \(K_i \leq \textup{cap}^\theta_\delta(\rho,\kappa)\), we have by \cref{resirrlist} that \({\tt t}(i) \sim \bnu^\theta_\rho\) for some \(e\)-multipartition \(\widehat\bnu = (\nu^{(0)} \mid \dots \mid \nu^{(e-1)})\). Assume by way of contradiction that \(\nu^{(0)} \neq \varnothing\). If \(\nu^{(0)}\) consists of \(n\) rows, then by \cref{allaboutshapes} the skew (row) diagram \(\xi\) consisting of the \(e\) nodes \(u_1, \ldots, u_e\) at the east end of the \(n\)th row of \(\rho \sqcup {\tt t}(i)\) is a cuspidal \((\theta,0)\)-ribbon, with \(\textup{res}(u_m) =\overline{\theta_{e} + m}\) and \({\tt S}u_m \notin \rho \sqcup {\tt t}(i)\) for \(m=1,\dots, e\). Thus \(\rho \sqcup {\tt t}(i)\) is {\em not} \(e\)-restricted. Since \(\tau\) {\em is} \(e\)-restricted, there must be some tile \({\tt t}(j)\) with \(j>i\) which contains the node \({\tt S}u_1\). Let \(t \geq 0\) be maximal such that \({\tt W}^t {\tt S}u_1 \in {\tt t}(j)\). If \(t \geq e-1\), then \({\tt t}(j)\) would have a {\tt NW}-removable ribbon of content \(\delta\), which contradicts the fact that \({\tt t}(j)\) is a semicuspidal diagram of content \(K_j \beta_j\), with \(\beta_j \prec \delta\). So we have \(t< e-1\). Then set 
\begin{align*}
\mu = \sqcup_{q=0}^{t} {\tt W}^q{\tt S}u_1
\qquad
\textup{and}
\qquad
\mu' = \sqcup_{q=0}^t u_{e - q}.
\end{align*}
Then we have \(\cont(\mu) = \cont(\mu')\).
By cuspidality of \(\xi\) and the fact that \(\mu'\) is {\tt SE}-removable in \(\xi\), we have \(\cont(\mu') \succ \delta\). On the other hand, by cuspidality of \({\tt t}(j)\) and the fact that \(\mu\) is {\tt NW}-removable in \({\tt t}(j)\), we have \(\cont(\mu) \prec \beta_j \prec \delta\), deriving the contradiction.

Thus \(\nu^{(0)} = \varnothing\), so we have an \((e-1)\)-multipartition \(\bnu:= (\nu^{(1)} \mid \dots \mid \nu^{(e-1)})\) such that \(\zeta(\bnu) \sim {\tt t}(i)\) by \cref{allaboutshapes}. Thus \(\zeta(\pi) \sim ( {\tt t}(1) \mid \dots \mid {\tt t}(m))\), as desired, completing the proof.
\end{proof}



\section*{Index of notation}\label{indsec}


For the reader's convenience we conclude with an index of the notation we use in this paper, providing references to the relevant (sub)sections where the notation is first introduced.

{\onehalfspacing
\newlength\colwi
\newlength\colwii
\newlength\colwiii
\setlength\colwi{3.2cm}
\setlength\colwiii{1.2cm}
\setlength\colwii\textwidth
\addtolength\colwii{-\colwi}
\addtolength\colwii{-\colwiii}
\addtolength\colwii{-1em}
\begin{longtable}{@{}p{\colwi}p{\colwii}p{\colwiii}@{}}
$\ZZ_e$&$\ZZ/e\ZZ$, the set of integers modulo $e$&\cref{posrootsec}\\
$I$ & the set of simple roots $\{\alpha_i \mid i\in \ZZ_e\}$, identified with $\ZZ_e$ &\cref{posrootsec}\\
\(\height(\beta)\) & the height of \(\beta \in \ZZ_{\geq 0}I\) & \cref{posrootsec}\\
\(\alpha(t, L)\) & the positive root \(\alpha_{\bar t} + \alpha_{\overline{t+1}} + \dots + \alpha_{\overline{t+L-1}}\) of height \(L\) & \cref{posrootsec}\\
\(\delta\) & the null root \(\alpha_0 + \alpha_1 + \cdots + \alpha_{e-1}\) & \cref{posrootsec}\\
\(I^\omega\) & \(\{ \bi \in I^{\height(\omega)} \mid i_1 + \dots + i_{\height(\omega)} = \omega\}\) & \cref{posrootsec}\\
\(\Phi_+\) & the set of all positive roots & \cref{posrootsec}\\
\(\Phi_+^\re\) & the set of all real positive roots & \cref{posrootsec}\\
\(\Phi_+^\im\) & the set of all positive imaginary roots \(\{m \delta \mid m \in \ZZ_{>0}\}\) & \cref{posrootsec}\\
\(\Psi\) & the set of indivisible roots \(\Phi_+^\re \sqcup \{\delta\}\) & \cref{posrootsec}\\
\(p\) & the `mod \(\delta\)'  map \(p:\Z I \to \Z I^\fin \cong \Z I/ \Z \delta\) & \cref{posrootsec}\\
\(\N\) & the array of nodes \(\ZZ_{>0} \times \ZZ_{>0}\) & \cref{levonedefs}\\
\((x,y) \searrow (x',y')\) & \((x',y')\) is southeast of \((x,y)\) & \cref{levonedefs}\\
\(\la,\mu,\nu,\rho\) & Young diagrams or partitions & \cref{levonedefs}\\
\({\tt N}, {\tt E}, {\tt S},{\tt W}\) & single unit translations of a node in \(\N\) & \cref{levonedefs}\\
\(\res\) & the residue function on nodes & \cref{levonedefs}\\
\(N_\ell\) & \(\bigsqcup_{t \in [1,\ell]} \N = \N^{(1)} \sqcup \dots \sqcup \N^{(\ell)}\) & \cref{highlevdefs}\\
\(\blam/ \bmu, \bxi, \btau \) & skew multipartitions and skew diagrams & \cref{highlevdefs}\\
\(\bkap\) & a multicharge \((\kappa_1,\dots,\kappa_{\ell})\in\ZZ^\ell\) & \cref{highlevdefs}\\
\(\blam,\bmu,\bnu,\brho\) & \(\ell\)-Young diagrams or multipartitions & \cref{highlevdefs}\\
\(\Lambda^\ell(\omega)\) & the set of skew diagrams of content \(\omega\) & \cref{highlevdefs}\\
\(\trianglelefteq^D\) & the dominance order of multipartitions & \cref{highlevdefs}\\
\(\textup{rect}_u\) & the \(\ell\)-multipartition \(\{v \in \N^{(r)} \mid v \searrow u\} \subseteq \N_\ell\) & \cref{highlevdefs}\\
\multirow{2}{\colwi}{\(\xi \NEarrow \xi'\)} & every node in \(\xi'\) is northeast of (or in an earlier component than) every node in \(\xi\) & \multirow{2}{*}{\cref{highlevdefs}}\\
\(\Gamma\) & a skew tiling of a skew diagram & \cref{tiletabsec}\\
\((\Gamma, {\tt t})\) & a skew tableau of a skew diagram & \cref{tiletabsec}\\
\(\Std(\btau)\) & the set of all standard tableaux of \(\btau\) & \cref{tiletabsec}\\
\({\tt t}^{\btau}\) & the (row)-leading standard tableau of \(\btau\) & \cref{tiletabsec}\\
\({\tt Tt}\) &
the standard tableau for \(\blam\) built from \(\tt T\in\Std(\bmu)\) and \(\tt t\in\Std(\blam/\bmu)\)
& {\cref{tiletabsec}}\\
\(\textup{sh}^\downarrow_n({\tt T})\) & the \(\ell\)-multipartition formed by the first \(n\) nodes in \({\tt T}\) & \cref{tiletabsec}\\
\(\cont(S)\) & the content function on a set $S$ of nodes & \cref{subsec:content}\\
\(\Lambda_+^{\bkap}(\omega)\) & the set of multipartitions of content \(\omega\) & \cref{subsec:content}\\
\(\Lambda_{+/+}^{\bkap}(\omega)\) & the set of skew multipartitions of content \(\omega\) & \cref{subsec:content}\\
\(\Lambda_{+/\brho}^{\bkap}(\omega)\) & the set of skew multipartitions \(\blam/\brho\) of content \(\omega\) & \cref{subsec:content}\\
\(\bi^{\tt t}\) & the associated tableau word or residue sequence of \({\tt t}\) & \cref{subsubsec:tabwords}\\
\(\textup{def}_{\bkap}(\btau)\) & the defect of a skew diagram \(\btau \in \Lambda_{+/+}^{\bkap}\) & \cref{subsec:def}\\
\(\B(\blam, \bkap)\) & the \(\bkap\)-beta numbers \((\B^1(\blam, \bkap) \mid \B^2(\blam, \bkap) \mid \dots \mid \B^\ell(\blam, \bkap))\) for \(\blam\) & \cref{cdefs1}\\
\(M^r_i(\blam,\bkap),h^r_{i,j}(\blam, \bkap)\) & data associated with beta numbers & \cref{cdefs1}\\
\(\theta\) & a residue permutation \((\theta_1, \dots, \theta_e)\) of \([0,e-1]\) & \cref{basiccoreRoCKsec}\\
\multirow{2}{\colwi}{\(\Psi_i \blam\)} & the multipartition obtained by removing all removable \((i+1)\)-nodes from \(\blam\) while adding all addable \((i+1)\)-nodes & \multirow{2}{\colwi}{\cref{subsec:rockblocks}}\\
\(\gamma_t^\theta, h^{r,\theta}_t, \gamma^\theta_{[a,b]}\), & \multirow{2}{*}{data associated with beta numbers} & \multirow{2}{*}{\cref{reindsec}}\\
\(h^{r,\theta}_{[a,b]}, h^{\max, \theta}_{[a,b]},
h^{\min, \theta}_{[a,b]}\)\\
\(\textup{cap}^\theta_\delta(\brho, \bkap)\) & the \(\theta\)-capacity of \(\brho\) & \cref{reindsec}\\
\(\textup{cap}^\theta_\delta(\omega, \bkap)\) & \(\min \{\textup{cap}^\theta_\delta(\brho, \bkap) \mid \brho \in \Lambda^{\bkap}_+(\omega)\}\) & \cref{reindsec}\\
\(\remrib^\theta, \addrib^\theta\) & subsets of \(\Z I\) related to removable and addable ribbons of \(\brho\), respectively &\cref{subsec:remaddribs}\\
\(\succeq\) & a convex preorder on \(\Phi_+\)& \cref{subsec:convexpreorders}\\
\(\Gamma_{\btau}\) & the unique cuspidal Kostant tiling of the skew diagram \(\btau\) & \cref{subsec:cuspribtabs}\\
\multirow{2}{\colwi}
{\(\Lambda^{\bkap}_+(\omega)_{\succ \delta},\Lambda^{\bkap}_+(\omega)_{\succeq \delta}\),\\ \(\Lambda^{\bkap}_+(\omega)_{\approx \delta}\)} & {the set of \(\blam \in \Lambda^{\bkap}_+(\omega)\) whose Kostant tilings consist only of tiles of content \(\succ \delta\), \(\succeq \delta\), or \(\approx \delta\), respectively} & \multirow{2}{*}{\cref{subsec:cusptilingsofblocks}}\\
\multirow{2}{\colwi}{\(\bnu^\theta_{\brho}\)} & the skew multipartition associated to a higher level \(e\)-quotient \(\bnu\), an \(e\ell\)-multipartition, and multicore \(\brho\)	 & \multirow{2}{\colwi}{\cref{combimagsec}}\\
\(\bb^{\bnu}_{\theta}\) & the Gelfand--Graev word associated to \(\theta\) and \(\bnu\)& \cref{subsec:ggwords}\\
\(R_\beta\) & the KLR algebra of type \({\tt A}_{e-1}^{(1)}\) & \cref{subsec:KLR}\\
\(j_\beta\) & the unique anti-isomorphism of \(R_\beta\) fixing KLR generators & \cref{subsec:KLR}\\
\(\iota_{\omega,\beta}\) & the embedding \(R_\omega \otimes R_\beta \hookrightarrow 1_{\omega, \beta} R_{\omega + \beta} 1_{\omega, \beta} \subseteq R_{\omega + \beta}\) & \cref{subsec:KLR}\\
\(R^\Lambda_\beta\) & the cyclotomic KLR algebra of type \({\tt A}_{e-1}^{(1)}\) & \cref{subsec:cycKLR}\\
\(\btau \sim \btau'\) & the skew diagrams \(\btau,\btau'\) have identical shapes and residues & \cref{Spechtsec}\\
\(S^{\blam}, S^{\blam/\bmu}\) & (row) Specht modules and skew Specht modules, respectively & \cref{Spechtsec}\\
\(\zS^{\xi}\) & skew Specht modules generated in degree zero & \cref{Spechtsec}\\
\(\deg_{\blam/\bmu}(\tt t)\) 
& the degree of \(\tt t\in\Std(\blam/\bmu)\)
& \cref{Spechtsec}\\
\(\Lambda^{\bkap}_+(\omega)'\)& the set of Kleshchev multipartitions in \(\Lambda^{\bkap}_+(\omega)\) & \cref{celldef}\\
\(D^{\bmu}\) & the simple head of \(S^{\bmu}\), for \(\bmu \in \Lambda^{\bkap}_+(\omega)'\) & \cref{subsec:truncations}\\
\(L(\beta^m)\) & the unique real simple semicuspidal \(R_{m\beta}\)-module for \(\beta \in \Phi_+^\re\) & \cref{stratasec}\\
\(L(\blam)\) & the simple semicuspidal \(R_{m\delta}\)-module for \(\blam\) an \((e-1)\)-multipartition of \(m\) & \cref{stratasec}\\
\(\boldsymbol{K} \) & a Kostant partition & \cref{stratasec}\\
\(\Xi(\omega)\) & the set of all Kostant partitions of \(\omega\) & \cref{stratasec}\\
\(\pi\) & a root partition \((\boldsymbol{K}, \bnu)\) & \cref{stratasec}\\
\(\Pi(\omega)\) & the set of all root partitions of \(\omega\) & \cref{stratasec}\\
\(\bar\Delta(\pi)\) & the proper standard \(R_\omega\)-module indexed by \(\pi \in \Pi(\omega)\) & \cref{stratasec}\\
\(L(\pi)\) & the simple \(R_\omega\)-module indexed by \(\pi \in \Pi(\omega)\) & \cref{stratasec}\\
\multirow{2}{\colwi}{\(\Pi(\omega)_{\succeq \delta}, \Pi(\omega)_{\succ \delta}\)} & the set of root partitions whose Kostant partitions involve only roots \(\succeq \delta\) or \(\succ \delta\), respectively & \multirow{2}{*}{\cref{subsec:cuspsystems}}\\
\(\zeta(\beta)\) & a ribbon constructed from \(\beta\in\Phi_+^\re\) & \cref{subsec:cusprib}\\
\(\zeta_i\) & the cuspidal ribbon of content \(\delta\) and \(i+1\) rows & \cref{subsec:cusprib}\\
\multirow{2}{\colwi}{\(\textup{dil}_{i}(\nu)\)} & an \(i\)-dilation of a partition \(\nu\), a connected skew diagram composed of \(|\nu|\) copies of \(\zeta_i\) & \multirow{2}{*}{\cref{subsec:cusprib}}\\
\(\zeta(\bnu)\) & \((\textup{dil}_{1}(\nu^{(1)})\mid  \textup{dil}_{1}(\nu^{(2)})\mid \dots\mid \textup{dil}_{e-1}(\nu^{(e-1)}))\), for \(\bnu\) an \((e-1)\)-multipartition & \cref{subsec:cusprib}\\
\(\geqslant_{\textup{b}}\) & the bilexicographic partial order on \(\Xi(\omega)\)& \cref{subsec:simplesviaskews}\\
\(\geqslant_{\textup{bd}}\) & the bilexicographic dominance partial order on \(\Pi(\omega)\)& \cref{subsec:simplesviaskews}\\
\(R_{\boldsymbol{K}}\) & the parabolic subalgebra of \(R_\beta\) corresponding to  \(\boldsymbol{K} \in \Xi(\omega)\)& \cref{subsec:simplesviaskews}\\
\(\Ind_{\boldsymbol{K}}^\beta\) & the induction functor from \(R_{\boldsymbol{K}}\)\textup{-mod} to \(R_\beta\)\textup{-mod} & \cref{subsec:simplesviaskews}\\
\(\Res_{\boldsymbol{K}}^\beta\) & the restriction functor from \(R_\beta\)\textup{-mod} to \(R_{\boldsymbol{K}}\)\textup{-mod} & \cref{subsec:simplesviaskews}\\
\(\Gamma_\tau^{\textup{sc}}\) & the unique semicuspidal strict Kostant tiling of the skew diagram \(\btau\) & \cref{subsec:regularize}\\
\end{longtable}
}

\bibliographystyle{lspaper}
\bibliography{master}

\end{document}